\documentclass[11pt,reqno]{amsart}
\pdfoutput=1 

\input{LouisTimoDefs.tex}

\newcommand\newa[1]{#1}

\usepackage{enumitem} 

\allowdisplaybreaks[1]
\begin{document}
 
\title[Busemann process and gRSK]{Intertwining the Busemann process of the directed polymer model} 



\author[E.~Bates]{Erik Bates}
\address{Erik Bates\\ North Carolina State University\\  Department of Mathematics\\ SAS Hall \\ 2311 Stinson Drive \\   Raleigh, NC 27695-8205\\ USA.}
\email{ebates@ncsu.edu}
\urladdr{https://www.ewbates.com/}
\thanks{E.\ Bates was partially supported by National Science Foundation grants DMS-1902734 and DMS-2412473.} 
 
\author[W.-T.~Fan]{Wai-Tong (Louis) Fan}
\address{Wai-Tong (Louis) Fan\\ University of North Carolina at Chapel Hill\\ School of Data Science and Society\\ 211 Manning Dr, 
	\\  Chapel Hill, NC 27514\\ USA.}
\email{louisfan@unc.edu}
\urladdr{https://sites.google.com/site/louisfanmath/}
\thanks{W.-T.\ Fan was partially supported by National Science Foundation grants DMS-2152103 and DMS-2348164, and Office of Naval Research grant N00014-20-1-2411.}

\author[T.~Sepp\"al\"ainen]{Timo Sepp\"al\"ainen}
\address{Timo Sepp\"al\"ainen\\ University of Wisconsin-Madison\\  Mathematics Department\\ Van Vleck Hall\\ 480 Lincoln Dr.\\   Madison, WI 53706-1388\\ USA.}
\email{seppalai@math.wisc.edu}
\urladdr{http://www.math.wisc.edu/~seppalai}
\thanks{T.\ Sepp\"al\"ainen was partially supported by  National Science Foundation grants DMS-1854619 and DMS-2152362, by Simons Foundation grant 1019133, and by  the Wisconsin Alumni Research Foundation.} 

\renewcommand{\subjclassname}{\textup{2020} Mathematics Subject Classification}
\date{\today} 

\begin{abstract}  We study the Busemann process and competition interfaces of the planar directed polymer model with i.i.d.\ weights on the vertices of the planar square lattice, in both the general case and the solvable inverse-gamma case.  We prove new regularity properties of the Busemann process without reliance on unproved assumptions on the shape function.  For example, each nearest-neighbor Busemann function is strictly monotone and has the same random set of discontinuities in the direction variable. 
When all Busemann functions on a horizontal line are viewed together, the Busemann process intertwines with an evolution that obeys a version of the geometric Robinson--Schensted--Knuth correspondence. When specialized to the inverse-gamma case, this relationship enables an explicit distributional description: the Busemann function on a nearest-neighbor edge has independent increments in the direction variable, and its distribution comes from an inhomogeneous planar Poisson process.  The distribution of the asymptotic competition interface direction of the inverse-gamma polymer is discrete and supported on the Busemann discontinuities which---unlike in zero-temperature last-passage percolation---are dense. Further implications follow for the  eternal solutions and the failure of the one force--one solution principle of the discrete stochastic heat equation solved by the polymer partition function. 
\end{abstract}
\maketitle

\maketitle
\tableofcontents

\section{Introduction}


\subsection{Motivation of this paper}  
The investigation of Busemann functions and semi-infinite geodesics in first- and last-passage percolation
has been in progress for three decades, since the seminal work of Newman \cite{newm-icm-95} and Hoffman \cite{hoff-05,hoff-08}.  More recent is the study of the analogous Busemann functions and semi-infinite Gibbs measures in positive-temperature polymer models. 
This work began in \cite{geor-rass-sepp-yilm-15} on the inverse-gamma polymer model introduced in \cite{sepp-12-aop-corr}.   
In  \cite{geor-rass-sepp-16}   Busemann functions were studied as extrema of variational formulas for shape functions and limiting free energy densities.  On the dynamical systems side, \cite{bakh-li-19} utilized Busemann functions and polymer measures to define attractive eternal solutions to a randomly forced Burgers equation in semi-discrete space-time. 

\newa{The present paper considers nearest-neighbor directed polymers on the planar square lattice, for which the Busemann process and semi-infinite Gibbs measures were constructed in \cite{janj-rass-20-aop, janj-rass-20-jsp}.  We establish results in both the general version of this model (i.i.d.~weights with a mild moment assumption) and in the special case of inverse-gamma weights.
Moreover, our methodology and exposition seek to highlight the interplay between these two distinct directions.}


Next we introduce informally  the notions of Busemann function and Busemann process, give  a brief account of the present state of the subject, and then turn to the main novel aspects of this paper.
Rigorous definitions and statements begin in Section~\ref{sec:poly}.  
To keep this introduction to a reasonable length  we refer the reader to the  papers cited above    
  for additional history.     
  Section~\ref{sec:org}  below  summarizes  the organization of the paper.

\subsection{Busemann functions and Busemann process}  \label{sec:i_BB}

 Given a random field $(L_{u,v})_{u,v\tsp\in\tsp\Z^2}$  
 with a metric-like interpretation    
 and  a planar direction vector $\xi$, 
 an individual \textit{Busemann function} $\Bus^\xi\colon \Z^2\times\Z^2\to\R$    
 is a limit of the type 
 \be\label{buse10}
 \Bus^\xi_{x,y} = \lim_{n\to\infty} [ L_{x, v_n} - L_{y, v_n}] \,, \quad x,y \in\Z^2, 
 \ee
 where $(v_n)$ is a sequence of vertices with asymptotic direction $\xi$.  In a first- or last-passage growth model, $L_{u,v}$ is  the passage time between $u$ and  $v$.  In a polymer model, $L_{u,v}$ is the free energy (logarithm of the partition function) of paths between $u$ and $v$. 
 

 The (global, or full) \textit{Busemann process} is a stochastic process $(\Bus^\xi)_{\xi}$ 
 that  combines the individual Busemann functions into a single random object. Since there are uncountably many directions $\xi$, the limits \eqref{buse10} alone do not define this object.  
 But once a global process is constructed, it turns out that the distributional and  regularity properties of the function $\xi\mapsto \Bus^\xi$ capture   useful information about the field $(L_{u,v})_{u,v\tsp\in\tsp\Z^2}$.

   \subsection{Busemann process state of the art}
 The Busemann process can be constructed in broad  generality  in planar growth and polymer models,  with an argument that combines  weak convergence and monotonicity.   In this approach limits \eqref{buse10} are not  the starting point, but instead proved after $\Bus^\rbbullet$ has been constructed.  In the planar corner growth model (CGM), equivalently, in planar directed nearest-neighbor  last-passage percolation  (LPP) this was done  in \cite{geor-rass-sepp-17-buse}, by appeal to weak convergence results from queueing theory.  A general construction for both LPP and the directed nearest-neighbor polymer model was undertaken in \cite{janj-rass-20-aop}, based on the weak convergence argument of \cite{damr-hans-14}.  Recent extensions   to higher dimensions and ergodic weights appear in \cite{groa-janj-rass-23-arxiv, janj-nurb-rass-22}. 
 
 The general construction gives little  insight into the distribution or  the regularity  of the Busemann  process. 
   Explicit  properties of the   joint distribution of the Busemann process    have been established in solvable  
   LPP models:  exponential  
   CGM \cite{fan-sepp-20},  Brownian LPP \cite{seppalainen_sorensen23a}, and the directed landscape \cite{busani_seppalainen_sorensen24}. 
In positive temperature work is in progress on the Kardar--Parisi--Zhang (KPZ)  equation: the construction of the Busemann process and applications to ergodicity and synchronization  in \cite{janj-rass-sepp-22-arxiv}, and distributional properties in  \cite{groathouse_rassoulagha_seppalainen_sorensen25}.   The first lattice polymer case of the Busemann distribution  is developed in the present paper.

In LPP models the Busemann process serves   as an analytic device for studying  infinite geodesics.
A common suite of results has emerged across several models:
\begin{enumerate}[label=\textup{(\alph*)}]

\item \label{busres1} On an event of probability one, there is a Busemann process defined simultaneously across all directions.

\item \label{busres2} The Busemann function in a particular  direction encodes a family of coalescing semi-infinite geodesics. Discontinuities of the Busemann process $\xi\mapsto \Bus_{x,y}^\xi$ correspond to multiple coalescing families in the same asymptotic direction.

\item \label{busres3} When the joint distribution of the Busemann process can be described,  it has revealed that the set of discontinuities is a countable dense subset of directions.

\end{enumerate}

Besides geometric properties, an explicit Busemann process is useful for estimates, such as  bounds on coalescence \cite{sepp-shen-20}  and nonexistence of bi-infinite geodesics \cite{bala-busa-sepp-20}.   Before the   Busemann process,  explicit stationary 
processes were discovered and utilized to establish fluctuation exponents.  The seminal work  \cite{cato-groe-06}  came in Poissonian LPP, followed by the  exponential CGM  \cite{balazs_cator_seppalainen06} and  the inverse-gamma polymer \cite{sepp-12-aop-corr}. 
In positive-temperature polymer models, analogues of objectives~\ref{busres1} and \ref{busres2} above were accomplished in   \cite{janj-rass-20-aop}   for general i.i.d.\ weights. 

\subsection{General goals of this paper}

\newa{In broad strokes, our paper has three primary objectives for directed polymer models:}
\begin{itemize} 
\item \newa{Sharpen the regularity of $(\xi,x,y)\mapsto \Bus^\xi_{x,y}$ in both the direction variable $\xi$ and the lattice pair $(x,y)$.  Our results reveal  distinctions with the zero-temperature case that are not apparent from the construction in \cite{janj-rass-20-aop}. These are described further in subsequent sections.}
\item \newa{Fulfill objective~\ref{busres3} by establishing a characterization of the joint distribution of the Busemann process, valid for the general polymer model.    In the inverse-gamma case this description yields  corollaries whose universality raises open problems for the future.   Moreover, our process-level description goes well beyond the one-point marginals of the inverse-gamma model that were found in \cite{geor-rass-sepp-yilm-15}.}    
\item \newa{Cast these general developments in the framework of the Robinson--Schensted--Knuth (RSK) correspondence that is central to integrable probability.}
\end{itemize}

\newa{An important qualitative feature  of general polymer theory (and of general first- and last-passage percolation) is the degree of reliance on unproved regularity assumptions on the shape function. One commonly used assumption is that the shape function be differentiable at the endpoints of its linear segments. We  explain the significance of  this condition in Remark \ref{diffrmk} of Section \ref{buse_int}, and in various subsequent locations to point out where this assumption makes a difference.}

\newa{A strength of our paper is that we achieve all our main results \textit{without any unproved assumptions}.  
Not only do our regularity results go well beyond those of \cite{janj-rass-20-aop}, but also we 
avoid the unproved differentiability assumption mentioned above.   We invoke this assumption only for the finer points of the competition interface (in part of Theorem~\ref{thm:eta20}) 
and in Appendix~\ref{a:Blim}.  The latter   supplements the construction of the Busemann process from \cite{janj-rass-20-aop} but is not used elsewhere in the paper.}

The next sections 
\ref{sec:i_Bus}--\ref{sec:i_HJ}
provide a more detailed  overview of the contents of this paper.


\subsection{Characterization of the Busemann process of the directed polymer model} \label{sec:i_Bus}  

Our main results for the Busemann process  are the following. 
\begin{enumerate}[label=\textup{(\roman*)}, ref=\textup{(\roman*)}]
\item\label{B.i}  On each lattice edge $(x-\evec_r,x)$, the Busemann process $\xi\mapsto \Bus^\xi_{x-\evec_r,x}$  is \textit{strictly} monotone away from the linear segments of the shape function (Theorem~\ref{thm:78-63}).  The random set of discontinuities is the same on each edge (Theorem~\ref{thm:51-32}).
If the i.i.d.~weights have a continuous distribution, then the random set of discontinuities of $\xi\mapsto\Bus^\xi_{x,y}$ is the same for every pair $x\neq y$ (Theorem~\ref{thm:allxy}).
\newa{Of special note is that the latter two results cannot be predicted from an exactly solvable case such as the inverse-gamma model.}

\item \label{B.ii} The joint distribution of the Busemann process on a lattice level is identified as the invariant distribution of a certain Markov process.
This distribution is shift-ergodic and unique subject to a condition on asymptotic slopes (Theorem~\ref{cpthm}).
The Markovian evolution intertwines with another Markov process that obeys a version of geometric RSK (Section~\ref{sec:grsk}, discussed below in Section~\ref{sec:i_gRSK}). 

\item\label{B.iii} Under inverse-gamma weights, the Busemann process on a lattice edge is realized as a functional of a two-dimensional inhomogeneous Poisson point process (Theorem~\ref{B-th5}).
The discontinuities are countably infinite and dense (Corollary~\ref{cor:ig_aUset}).     In the zero-temperature limit the inverse-gamma Busemann process on a lattice edge converges in distribution to the Busemann process of the exponential CGM (Theorem  \ref{thm:igexp2}).

\end{enumerate}

We point out some distinctions between the positive and zero-temperature cases. 
In contrast with 
items~\ref{B.i} and \ref{B.iii},  in LPP  a  Busemann function is constant on random  open intervals whose union is dense  \cite[Lem.~3.3]{janj-rass-sepp-23}.  The full set of discontinuities does not  appear on a single edge, but any given discontinuity direction is observed at some edge along any bi-infinite down-right path \cite[Lem.~3.6]{janj-rass-sepp-23}. In the case of exponential weights, discontinuities of   $\xi\mapsto \Bus^\xi_{x,y}$  can accumulate only at the extremes $\evec_2$ and $\evec_1$, while  across all $x,y$ discontinuities  are dense \cite{fan-sepp-20}. 
 Item~\ref{B.ii} generalizes invariance, ergodicity, and uniqueness properties of a single  Busemann function from \cite{janj-rass-20-jsp}.
 
 The special case of the joint distribution of two inverse-gamma Busemann functions  from this work has   been in circulation prior to this publication.   
 In earlier collaborative work of the third author, this bivariate case  was applied in \cite{busa-sepp-22-ejp} to prove nonexistence of bi-infinite polymer Gibbs measures and in \cite{rassoulagha_seppalainen_shen24} to derive coalescence estimates for polymers.  

\subsection{Competition interface}  
In LPP, geodesics  from a common point   spread   in a tree-like fashion and divide the lattice into disjoint clusters, depending on the initial steps of the paths.
The boundaries of these clusters, called \textit{competition interfaces}, were introduced in \cite{ferr-pime-05}   and further studied by \cite{cato-pime-13, ferr-mart-pime-09}.
These interfaces convey essential geometric information and are intimately linked to the Busemann process \cite{georgiou_rassoulagha_seppalainen17b, fan-sepp-20, janj-rass-sepp-23, seppalainen_sorensen23b}.

At positive temperature, geodesics are replaced by polymer measures, so the random environment does not by itself generate a tree-forming family of paths.
Instead, one samples from a natural coupling of the quenched polymer measures, thereby adding an additional layer of randomness.
The resulting  competition interface in \cite{geor-rass-sepp-yilm-15} was shown in \cite{janj-rass-20-aop} to have a random asymptotic direction whose distribution is determined by a  Busemann function.

In Section~\ref{sec:cif}, we extend this theme by realizing---in a single coupling---an interface direction from every point on the lattice (Theorem~\ref{thm:eta10}).
\newa{This coupling is new and does not require any unproved assumptions about the shape function.}
Whereas the coupling from \cite{geor-rass-sepp-yilm-15,janj-rass-20-aop} is of finite-volume polymer measures, ours is of semi-infinite polymer measures associated to the Busemann process.
Consequently, the results in item~\ref{B.i} of Section~\ref{sec:i_Bus} allow us to relate the interface directions to discontinuities of the Busemann process (Theorem~\ref{thm:eta20}).
This is similar in spirit to the LPP result \cite[Thm.~3.7]{janj-rass-sepp-23}, but in the polymer case the additional randomness poses a new challenge to establishing the desired relation.


Our results raise questions about the relationship between the geometry of polymer paths and the regularity of the Busemann process (Remark~\ref{rm:open4}).
 We answer some of these questions in the inverse-gamma case in Section~\ref{sec:ig_xi}.  Others remain open.

%

\subsection{Polymers, geometric RSK, and intertwining}  \label{sec:i_gRSK} 

The Robinson--Schensted--Knuth (RSK) correspondence from combinatorics 
is central to the integrable work on LPP models in the KPZ (Kardar--Parisi--Zhang)  class.  The geometric version of the RSK mapping (gRSK), introduced by Kirillov\footnote{Kirillov called his construction tropical RSK. To be consistent with the modern notion  of tropical mathematics,  \cite{corw-ocon-sepp-zygo} renamed the algorithm geometric RSK.}  \cite{kiri-01}  
and  elucidated by Noumi and Yamada \cite{noum-yama-04},  plays the analogous role in  directed polymer models.   The polymer connection of gRSK was initially developed  in \cite{corw-ocon-sepp-zygo, ocon-sepp-zygo-14}. For recent work and references on this theme, see \cite{corwin21}.  
 
Intertwinings of mappings and Markov kernels are typical features of this work.  In \cite{corw-ocon-sepp-zygo},  the application of gRSK to the inverse-gamma polymer and an intertwining argument led to a closed-form expression for the distribution of the polymer partition function.  Subsequently \cite{boro-corw-reme} used this formula to establish the Tracy--Widom limit of the free energy. 

In our paper two Markovian dynamics are intertwined by an explicit mapping (Proposition~\ref{g:itlm5}, Theorem~\ref{twm_thm}). The \textit{sequential process}  is defined by  a gRSK algorithm that produces polymer partition functions on a bi-infinite strip with a boundary condition (Section~\ref{sec:Sgrsk}).   The  \textit{parallel process}  is the dynamics of the Busemann process.  


Under inverse-gamma weights the sequential process has accessible product-form invariant probability  measures (Theorem~\ref{thm-I}). The intertwining map pushes these measures forward into invariant measures of the parallel process.  A uniqueness theorem for the latter identifies these measures as joint distributions of Busemann functions (Theorem~\ref{B_thm9}).  

The analogous zero-temperature intertwining argument appeared in \cite{fan-sepp-20} to describe the Busemann process of the CGM. This development was recast as ``stationary melonization'' by \cite{busani24} 
to derive the universal limit called the \textit{stationary horizon}.
\newa{Relative to \cite{fan-sepp-20}, this paper has several key differences:}
\begin{enumerate}[label=\textup{(\roman*)}]

\item \newa{Adapting the Markovian dynamics from the $(\max, +)$ algebra to the $(+,\aabullet)$ algebra.  This ``de-tropicalization" ultimately leads to the novel strict monotonicity mentioned in item~\ref{B.i} of Section~\ref{sec:i_Bus}, as well as additional regularity that is not present at zero temperature (e.g.~Lemma~\ref{inj_lem}).}

\item \newa{Establishing distributional uniqueness without knowing the explicit shape function.
Because we work with general i.i.d.~weights, we cannot profit from the differentiability and strict convexity that is known in exactly solvable cases.  Ensuing complications are discussed in Remark~\ref{diffrmk}.}

\item \newa{When we do specialize to the solvable inverse-gamma case, the explicit description of the Busemann process is more subtle than in the exponential LPP case.  While its evolution is again governed by a Poisson point process, here the process is two-dimensional (and dense in one coordinate) rather than a marked one-dimensional process; the associated calculations are found in Section~\ref{ss_inv_ga}.
We unify the two Busemann descriptions through the coupling of Theorem~\ref{thm:igexp2}, which offers a way to understand the relationship between the positive-temperature and zero-temperature processes.}

\end{enumerate}



\subsection{Failure of one force--one solution}   \label{sec:i_HJ} 
In  stochastically forced conservation laws such as the stochastic Burgers equation (SBE), the \textit{one force--one solution principle} (1F1S) is the statement that for a given realization of the driving noise and a given value of the conserved quantity, there is a unique eternal solution that is measurable with respect to the history of the noise.  
A connection with polymer models comes from viewing   the polymer free energy    as a solution of   
a stochastically forced viscous Hamilton--Jacobi equation.   In the physics literature this connection goes back to \cite{huse_henley_fisher85,imbr-spen}, while in mathematics an early paper was \cite{kife-97}. 
 
%
%
%
%
%

  In Appendix~\ref{bmcon} we observe  that the exponential of the Busemann process gives eternal solutions to a discrete difference  equation, simultaneously for all values of the conserved quantity on a single event of full probability (Theorem~\ref{thm:V1}).   This equation is a discrete analogue of the stochastic heat equation, which, as is well known, is linked to the KPZ equation and SBE through the Hopf--Cole transform. 
   In the inverse-gamma case our results imply that with probability one,  there is a countable dense set of values of the conserved quantity at which there are at least two eternal solutions (Theorem~\ref{thm:she_ig}). This is the first example of failure of 1F1S in a positive-temperature lattice model.     This failure of 1F1S at the discontinuities of the Busemann process was anticipated in the unpublished manuscript \cite{janj-rass-19-1f1s}.   After the posting of this paper,  the analogous result for the KPZ equation appeared in \cite{groathouse_rassoulagha_seppalainen_sorensen25}.


We refer to the introduction of \cite{janj-rass-sepp-22-arxiv} for further references on this theme and to \cite{bakh-khan-18} for conjectures on the universal behavior of  Hamilton--Jacobi type equations with random forcing.

\subsection{Organization of the paper} \label{sec:org}  
The directed polymer model   is introduced  in Section~\ref{sec:poly}.  Our main results for the general polymer appear in Section~\ref{sec:main1} and for the inverse-gamma polymer in Section~\ref{sec:main2}. 
Some proofs are given straight away, but most appear in Section~\ref{mar_sec}.

Sections~\ref{map_sec}--\ref{2_proc_sec} develop the dynamics of the Busemann process, the intertwining argument, and the Markovian characterization of the joint law of Busemann functions.
The application of these tools to prove four main results comes at the end of Section~\ref{2_proc_sec} and in Section~\ref{sec:pf_disc}.


Section~\ref{sec:grsk} is an interlude that puts the technical development of Section~\ref{mar_sec} in the context of the geometric RSK mapping.   

Section~\ref{sec:Vig} resumes the proofs, focusing on the inverse-gamma model.
Section~\ref{sec:twin_ig} records several consequences of intertwining in this solvable case.
Section~\ref{sec:array} constructs the intertwining through triangular arrays of infinite sequences; Remark~\ref{rmk:a_grsk} makes another contact with gRSK.  
While the result of  Section~\ref{sec:array} holds for general weights, our application in Section~\ref{ss_inv_ga} is to obtain the independent-increments property of the nearest-neighbor Busemann function under inverse-gamma weights.  
  
  Appendix~\ref{app_bp} proves two complements to the general properties of the Busemann process, one of which is needed in the main text (and stated as Theorem~\ref{all_lln}).
  Finally, in Appendix~\ref{bmcon}, we reinterpret some of our results in the language of eternal solutions to a discrete stochastic heat equation, including the failure of 1F1S in the inverse-gamma case.  
 
 \subsection{Notation and conventions}    \label{sec:nota}   We collect here items for quick reference.  Some are reintroduced in appropriate places in the body of the text.    
 

Intervals of integers are written as $\lzb a,b\rzb = \{a,a+1,\dots,b\}$.
Subsets of reals and integers are indicated by subscripts, as in $\Z_{>0}=\{1,2,3,\dotsc\}$ and $\R_{\ge0}=[0,\infty)$.    Spaces of bi-infinite sequences of restricted values are denoted by $\R_{>0}^\Z=(\R_{>0})^\Z$.   On $\R^2$ and $\Z^2$,    
$\zevec=(0,0)$, $\evec_1=(1,0)$ and $\evec_2=(0,1)$. 
In different contexts   
time evolution proceeds either in the vertical $\evec_2$ direction or along anti-diagonal  levels $\level_t=\{x=(a,b)\in\Z^2:\,  a+b=t\}$.

   Inequalities between tuples $I=(I_i)$ and $I'=(I'_i)$ are coordinatewise:  $I\le I'$ means $I_i\le I'_i$ for all $i$, and the strict version $I< I'$ means that  $I_i< I'_i$ for \textit{all} $i$.  
   For points $x=(a,b)$ and $y=(c,d)$  on the plane $\R^2$ or the lattice $\Z^2$,   
 the strict southeast ordering  $x\prec y$   means that $a<c$ and $b>d$. 
 Its  weak version $x\preccurlyeq y$ means that $x\prec y$ or $x=y$. 

 The set of $\ell^1$-unit vectors in the first quadrant will be denoted by $[\evec_2, \evec_1]$, with $\evec_2$ regarded as the minimal element according to southeast ordering.
 Infinite paths  proceed south and west but direction vectors $\xi$ are members of $[\evec_2, \evec_1]$ and so point north and east. In particular, $\Bus^\xi$ will denote a Busemann function associated to the direction   $v_n/n\to-\xi$.  
 
    A range of indices is marked with a colon, for example $x_{\parng{m}{n}} = (x_m,x_{m+1},\dots,x_n)$ or $X^{i, \tspa\parng{m}{n}}=(X^{i,m}, X^{i, m+1},\dotsc, X^{i,n})$. 
   The left tail logarithmic Ces\`aro average 
   is $\ces(I) =\lim_{n\to\infty}  {n}^{-1}\sum_{k=-n+1}^{0}\log I_k$.



The end of a numbered remark or definition is marked with  $\triangle$.

 \subsection{Acknowledgements} The authors thank C.~Janjigian for useful feedback.

 \section{Directed polymer model: definitions and prior results} \label{sec:poly} 
 Polymer models take as input a random environment and produce a family of measures on paths.
 In the standard (1+1)-dimensional discrete model  the random environment consists of  i.i.d.~random variables indexed by the vertices of $\Z^2$ and the paths are up-right nearest-neighbor trajectories on $\Z^2$.

\subsection{Random environment and recovering cocycles}
 Let $(\Omega, \kS, \P)$ be a Polish probability space equipped with a group of 
 continuous\footnote{The authors of \cite{janj-rass-20-aop} communicated to us that this assumption of continuity is needed for their construction, which we cite below as Theorem~\ref{buse_full}.} 
 bijections $\{\shift_x\}_{x\in\Z^2}$ (called \textit{translations}) that map $\Omega\to\Omega$, are measure-preserving ($\P=\P\circ\shift_x$ for all $x\in\Z^2$), and satisfy $\shift_x\circ\shift_y=\shift_{x+y}$.
We then assume
\eeq{ \label{ranenv}
&\text{$(\Yw_x)_{x\tsp\in\tsp\Z^2}$ are strictly  positive, i.i.d.~random variables on $\OSP$ such that} \\
&\text{$\Yw_x(\w) = \Yw_\zevec(\shift_x\w)$, }
\E(\abs{\log \Yw_\zevec}^p)<\infty \text{ for some $p>2$, and }\Vvv(\Yw_\zevec)>0. 
}
It is common to write $\Yw_x = e^{\beta w_x}$ with an inverse temperature parameter  $\beta$.   Our positivity condition comes from having already applied the exponential.

 A \textit{cocycle} on $\Z^2$ is a function $\Bus\colon\Z^2\times\Z^2\to\R$ such that
\begin{subequations} \label{rec_coc}
\eeq{ \label{coc_def}
\Bus(x,y) + \Bus(y,z) = \Bus(x,z) \quad \text{for all $x,y,z\in\Z^2$}.
}
The cocycles of interest to us are those satisfying a second property:
given a realization of the weights $(W_x)_{x\in\Z^2}$, a cocycle $\Bus$ is said to \textit{recover} these weights if
\eeq{ \label{rec_def}
e^{-\Bus(x-\evec_1,x)}+ e^{-\Bus(x-\evec_2,x)} =  \Yw_x^{-1}    \quad \text{for every $x\in\Z^2$}.
}
\end{subequations}
It is generally unclear  if recovering cocycles exist.
The next sections describe how one can furnish a one-parameter family of recovering cocycles known as the Busemann process.
 

\subsection{Path spaces, finite polymer measures, and the limit shape} \label{finm_sec}
A (\textit{directed}) \textit{path} on $\Z^2$  is a sequence of vertices $x_\brbullet = x_{\parng{m}{n}} = (x_i)_{i=m}^n$ such that $x_i - x_{i-1}\in\{\evec_1,\evec_2\}$ for each $i\in\lzb m+1,n\rzb$.
The lattice divides into anti-diagonal \textit{levels},
\eeq{ \label{leveldef} 
\level_n \coloneqq \{x\in\Z^2:\, x\cdot(\evec_1+\evec_2) = n\}, \quad n\in\Z.  
}
We index paths so that $x_i\in\level_i$.
For $u\in\level_m$ and $v\in\level_n$,   the set of paths between $u$ and $v$ is 
\eq{ 
\pathsp_{u,v} = \bigl\{x_{\parng{m}{n}} =(x_i)_{i=m}^n :\,  x_m=u, \, x_n=v, \, x_i-x_{i-1}\in\{\evec_1, \evec_2\} \, \forall i\in\lzb m+1,n\rzb \bigr\}. 
}
This set is nonempty if and only if $u\le v$, by which we mean both $u\cdot\evec_1\leq v\cdot\evec_1$ and $u\cdot\evec_2\leq v\cdot\evec_2$.
The projection random variables on any path space are denoted by $X_m(x_\brbullet) =x_m$ or $X_{\parng{\ell}{m}}(x_\brbullet) = x_{\parng{\ell}{m}}$ whenever the indices make sense (we will always use $\ell \le m \le n$).

Given a collection of weights $(\Yw_x)_{x\in\Z^2}$, we consider the following probability measure on $\pathsp_{u,v}$ (whenever $u\leq v$):
	\be\label{hcq}
	Q_{u,v}(x_{\parng{m}{n}}) =\frac1{Z_{u,v}}   \prod_{i=m+1}^{n} \wgtd_{x_i} 
	\quad\text{for } \  x_{\parng{m}{n}}\in\pathsp_{u,v}. 
	\ee
The normalizing constant  $Z_{u,v}$ is called  the \textit{partition function}:
\be \label{part56}
Z_{u,v}\coloneqq \sum_{x_\brbullet\tspa\in\tspa\pathsp_{u,v}} \prod_{i=m+1}^{n} \wgtd_{x_i}, 
\quad u\in\level_m, v\in\level_n.
\ee
Since all paths terminating at $v$ must pass through either $v-\evec_1$ or $v-\evec_2$, \eqref{part56} can also be thought of as a recursion:
\eeq{ \label{part57}
Z_{u,v} = (Z_{u,v-\evec_1}+Z_{u,v-\evec_2})\Yw_v \quad u\in\level_m, v\in\level_n, m < n, \quad \text{and} \quad Z_{v,v}=1.
}
The marginals of $X_{\parng{m}{n}}$ under $Q_{u,v}$ can be obtained by multiplying partition functions: for any sequence $m < i_1 < \cdots < i_k < n$, we have
\eeq{ \label{qmpf}
Q_{u,v}(X_{i_1} = x_{i_1}, X_{i_2}=x_{i_2}, \dots, X_{i_k} = x_{i_k})
= \frac{Z_{u,x_{i_1}}Z_{x_{i_1},x_{i_2}}\cdots Z_{x_{i_k},v}}{Z_{u,v}}.
}
One usually fixes the starting point  $u$ and studies the polymer  as the terminal point $v$ escapes to infinity   in the northeast quadrant.
We take the opposite (but analogous) perspective of fixing the terminal location $v$ and pulling the starting point $u$ to negative infinity   in the southwest 
quadrant.\footnote{When time proceeds in the up-right diagonal direction, under this convention the Busemann process is related to the environment from the past rather than the future; see \eqref{B-ind}.
This is consistent with the language of SHE and 1F1S in Appendix~\ref{bmcon}.
}
A basic result is a  law of large numbers known as a \textit{shape theorem}.

\begin{theirthm} \textup{\cite[Sec.~2.3]{janj-rass-20-aop}} \label{shape_thm}
Assume \eqref{ranenv}.
Then there exists a nonrandom function $\gpp\colon \R_{\ge0}^2\to\R$ whose restriction to $\Z_{\ge0}^2$ satisfies
\eq{ 
\lim_{n\to\infty}  \; \sup_{x\geq\zevec\tspa: \tspb \abs{x}_1\ge n} 
\frac{\log Z_{-x,\zevec}  -  \gpp(x)}{\abs{x}_1} =0
\qquad \P\text{-almost surely.} 
}
This function $\gpp$ is concave, continuous, and  positively homogeneous in the sense that
\eeq{ \label{homog}
\gpp(c\xi) = c\gpp(\xi) \quad \text{for any scalar $c\geq0$ and  $\xi\in\R_{\ge0}^2$}. 
}
\end{theirthm}

In general, further regularity of $\gpp$ beyond Theorem~\ref{shape_thm} is unknown. 
Here, as in FPP and LPP, curvature and differentiability of the limit shape is a long-standing open problem \cite{auffinger_damron_hanson17}.

\subsection{Infinite polymer measures} \label{infm_sec}
Following Theorem~\ref{shape_thm}, it is natural to ask if the polymer measures  \eqref{hcq}  themselves have  limits. 
Fixing  a root vertex $v\in\level_n$, any such limit is a measure on the space of semi-infinite backward paths: 
	\eq{ 
	\pathsp_v=\{x_{\parng{-\infty}{n}}:\, x_n=v\text{ and }x_i-x_{i-1}\in\{\evec_1, \evec_2\}\text{ for all } i\in\,\rzb-\infty, n\rzb\}. 
	}
This space is equipped with the usual cylindrical $\sigma$-$(+,\abullet)$-algebra.
  In a Gibbsian spirit we desire that the  finite-dimensional conditional distributions of any limiting measure agree with  the pre-limiting measures from \eqref{hcq}. 
So call a probability measure $Q_v$ on $\pathsp_v$ a \textit{semi-infinite polymer measure rooted at} $v\in\level_n$ if, whenever $x_m\leq v$,   
we have  
\begin{subequations} \label{gibbs12}
\eeq{ \label{gibbs_one}
Q_v\givenp{\dd x_{\parng{m}{n}}}{X_m=x_m} = Q_{x_m,v}(\dd x_{\parng{m}{n}}).
}
In words, conditioning the measure $Q_v$ to pass through $x_m\in\level_m$ induces a marginal distribution (on the portion of the path between $x_m$ and $v$) that is exactly the measure from \eqref{hcq}.
To ensure the left-hand side of \eqref{gibbs_one} makes sense, we require the non-degeneracy condition
\eeq{ \label{gibbs_thr}
Q_v(X_m=x_m) > 0\quad \text{whenever $x_m\leq v$}.
}
Another natural requirement is that limiting measures rooted at different vertices are consistent with one another.
Let $(Q_v)_{v\tsp\in\tsp\Z^2}$ be a family of semi-infinite polymer measures, each    $Q_v$   rooted at $v$.  
This family is \textit{consistent} if, whenever $x_m\leq v$,  we have
	\be\label{gibbs_two}   Q_v\givenp{\dd x_{\parng{-\infty}{m}}}{X_m=x_m} = Q_{x_m}(\dd x_{\parng{-\infty}{m}}). \ee
	\end{subequations}
That is, conditioning the measure $Q_{v}$ to pass through $x_m$ induces a marginal distribution (on the portion of the path between $-\infty$ and $x_m$) that is exactly $Q_{x_m}$.

We then have the following (deterministic) relation between consistent families of semi-infinite polymer measures and recovering cocycles.

\begin{theirthm} \textup{\cite[Thm.~5.2]{janj-rass-20-aop}} \label{pococthm} 
Fix any positive weights $(\Yw_x)_{x\tsp\in\tsp\Z^2}$.
There is a bijective correspondence between functions $\Bus$ satisfying \eqref{rec_coc} and families $(Q_v)_{v\in\Z^2}$ satisfying \eqref{gibbs12}, which is realized as follows.  
Each $Q_v$ is the law of the  Markov chain $(X_m)_{m\leq n}$ evolving backward in time with initial state $X_n = v\in\level_n$ and backward transition probabilities
\eeq{\label{back_prob}
Q_v\givenp{X_{m-1}=x-\evec_r}{X_m= x}  
= e^{-\Bus(x-\evec_r,x)}\cdot  \Yw_{x}, \quad r\in\{1,2\}.
}		
\end{theirthm}

This result suggests an entry point to  recovering cocyles.
Observe that for any $x_\ell< x\leq v$ and either $r\in\{1,2\}$, we have
\[
Q_v\givenp{X_{m-1}=x-\evec_r}{X_\ell=x_\ell,X_m = x} 
\stackrefpp{qmpf}{gibbs_one}{=} \frac{Z_{x_\ell,x-\evec_r}Z_{x-\evec_r,\tspb x}}{Z_{x_\ell,\tspb x}} 
= \frac{Z_{x_\ell,x-\evec_r}}{Z_{x_\ell,\tspb x}}\cdot \Yw_x.\]
The ratio $Z_{x_\ell,x-\evec_r}/Z_{x_\ell,x}$ occupies the same role above as 
$e^{-\Bus(x-\evec_r,x)}$  
in  \eqref{back_prob}.   Perhaps  when $x_\ell$   is suitably sent to $-\infty$,  this ratio converges    to 
$e^{-\Bus(x-\evec_r,x)}$ 
for some recovering cocycle $B$. 
The cocycles realized this way are called \textit{Busemann functions}.    Through Theorem~\ref{pococthm} they encode limits of the measures $Q_{u,v}$ from \eqref{hcq} as $u$ is pulled to $-\infty$.

%

\subsection{Busemann process} \label{buse_int}  
We describe features of the shape function $\gpp$.  By homogeneity \eqref{homog},   $\gpp$ is   determined by its restriction to the  closed line segment $[\evec_2,\evec_1]$.     The minimal element of this segment is $\evec_2$, consistently with  \textit{southeast ordering}:  for $\zeta,\xi\in[\evec_2,\evec_1]$ we write  $\zeta\preccurlyeq\xi$ when   $\zeta\cdot\evec_1\leq\xi\cdot\evec_1$, and  $\zeta\prec\xi$ when $\zeta\cdot\evec_1<\xi\cdot\evec_1$.  The relative interior of $[\evec_2,\evec_1]$ is denoted by $]\evec_2,\evec_1[$.  


By concavity  one-sided derivatives exist: 
for $\xi\in\,]\evec_2,\evec_1[\,$, let $\nabla\gpp(\xi+)$ and $\nabla\gpp(\xi-)$ be the vectors in $\R^2$ defined by
\[ \nabla\gpp(\xi\pm)\cdot\evec_1 = \lim_{\eps\searrow0}\frac{\gpp(\xi\pm\eps\evec_1)-\gpp(\xi)}{\pm\eps}
\quad\text{and}\quad 
\nabla\gpp(\xi\pm)\cdot\evec_2 = \lim_{\eps\searrow0}\frac{\gpp(\xi\mp\eps\evec_2)-\gpp(\xi)}{\mp\eps}.  \]
The set of directions of differentiability is 
\eq{
\Udiff \coloneqq \{\xi\in\,]\evec_2,\evec_1[\, :\, \nabla\gpp(\xi+)=\nabla\gpp(\xi-)\}.
}
There may be linear segments of $\gpp$ on either side of a given $\xi\in\,]\evec_2,\evec_1[\,$, which are recorded by the following two closed (nonempty, but possibly degenerate) subintervals:
\eq{ 
\cL_{\xi\pm} \coloneqq \big\{\zeta\in\,]\evec_2,\evec_1[\,:\, \gpp(\zeta)-\gpp(\xi) = \nabla\gpp(\xi\pm)\cdot(\zeta-\xi)\big\}.
}
The endpoints of these intervals will be denoted by
\eeq{ \label{seg_end}
\underline\xi \coloneqq \inf\cL_{\xi-} \quad \text{and} \quad\ \overline\xi \coloneqq \sup\cL_{\xi+} \qquad
\text{for $\xi\in\,]\evec_2,\evec_1[\,$},
}
where the infimum and supremum are taken with respect to the southeast order $\preccurlyeq$ on $[\evec_2,\evec_1]$.
Since $\gpp$ is known to have no linear segment containing $\evec_2$ or $\evec_1$ (see \cite[Lem.~B.1]{janj-rass-20-aop}), we always have $\underline\xi,\overline\xi\in\,]\evec_2,\evec_1[\,$.
Finally, for convenience we will write
\eq{
\cL_\xi \coloneqq \cL_{\xi+}\cup\cL_{\xi-} = [\tsp\underline\xi,\overline\xi\tsp] \quad \text{for $\xi\in\,]\evec_2,\evec_1[$}.
}
We say that $\gpp$ is \textit{strictly concave} at $\xi$ if this interval is degenerate, i.e.~$\underline\xi=\overline\xi=\xi$.

Given $A\subset[\evec_2,\evec_1]$, let us say that a sequence of $x_\ell\in\level_\ell$ is \textit{$A$-directed} as $\ell\to-\infty$ if the set of limit points of $\{x_\ell/\ell\}$ is contained in $A$.

\begin{theirthm} \textup{\cite[Thm.~3.8]{janj-rass-20-aop}} \label{buse_dir}
Assume \eqref{ranenv}, and suppose $\xi\in\Udiff$ is such that $\underline\xi,\overline\xi\in\Udiff$.
Then there is a full-probability 
event $\Omega_\xi\subset\Omega$ on which the following holds.
For each $x,y\in\Z^2$, the following limit exists and is the same for every $\cL_\xi$-directed sequence $(x_\ell)$:
\eeq{ \label{orig_buse}
\Bus^\xi_{x,y} = \Bus^\xi_{x,y}(\w) \coloneqq \lim_{\ell\to-\infty}\big(\log Z_{x_\ell,y} - \log Z_{x_\ell,x}\big), \quad \w\in\Omega_\xi.
}
Furthermore, if $\zeta\in\Udiff$ also satisfies $\underline\zeta,\overline\zeta\in\Udiff$, and has $\zeta\cdot\evec_1<\xi\cdot\evec_1$, then on $\Omega_\xi\cap\Omega_\zeta$ we have the following inequalities for all $x\in\Z^2$:
\eeq{ \label{mono_1}
\Bus^\zeta_{x-\evec_1,x} \ge \Bus^\xi_{x-\evec_1,x} \quad \text{and} \quad
\Bus^\zeta_{x-\evec_2,x} \le \Bus^\xi_{x-\evec_2,x}.
}
\end{theirthm}

Because of the telescoping identity
\eq{
(\log Z_{x_\ell,y} - \log Z_{x_\ell,x}) + (\log Z_{x_\ell,z}-\log Z_{x_\ell,y}) = \log Z_{x_\ell,z}-\log Z_{x_\ell,x},
}
the function $\Bus^{\xi}$ in \eqref{orig_buse} satisfies the cocycle condition \eqref{coc_def}.
It also satisfies the recovery condition \eqref{rec_def}, since \eqref{part57} leads to
\eq{
\frac{Z_{x_\ell,\tspb x-\evec_1}}{Z_{x_\ell,\tspb x}} + \frac{Z_{x_\ell,x-\evec_2}}{Z_{x_\ell,\tspb x}} = \frac{1}{\Yw_x}.
}
Thus Theorem~\ref{buse_dir} produces a recovering cocycle $\Bus^\xi$ for an individual  direction $\xi$.
Crucially, though, the full-probability event in Theorem~\ref{buse_dir} depends on $\xi$.
So in order to realize a cocycle simultaneously for all uncountably many values of $\xi$, \eqref{orig_buse} is not sufficient.
Hence the importance of \eqref{mono_1}, which allows this strategy: 
\begin{enumerate}[label=\textup{\arabic*.}]

\item First realize the nearest-neighbor Busemann functions $\Bus_{x-\evec_r,x}^{\xi}$ for a countable dense collection of direction parameters $\xi$.

\item Next extend to all $\xi\in\, ]\evec_2,\evec_1[$ by taking monotone limits.

\item Finally, extend additively to all of $\Z^2\times\Z^2$ according to \eqref{coc_def}.

\end{enumerate}
Since left and right limits may disagree, this construction results in two Busemann processes: a left-continuous version $(\Bus^{\xi-})_{\xi\in\,]\evec_2,\evec_1[}$ and a right-continuous version $(\Bus^{\xi+})_{\xi\in\,]\evec_2,\evec_1[}$.

\begin{theirthm} \label{buse_full}
\textup{\cite[Thm.~4.7, Lem.~4.13, Thm.~4.14]{janj-rass-20-aop}} 
Assume \eqref{ranenv}. 
Then there exists a family of random variables
\eq{
\Bus^{\xi\sig}_{x,y}\colon\Omega\to\R, \quad \xi\in\,]\evec_2,\evec_1[\,,\ \sigg\in\{-,+\},\ x,y\in\Z^2, 
}
and a full-probability event $\Omega_0\subset\Omega$ with the following properties:

\begin{itemize}
\item Each $\Bus^{\xi\sig}$ is a covariant cocycle on $\Z^2$, the \textit{cocycle} part meaning that 
 \be\label{coc1} 
 \Bus^{\xi\sig}_{x,y}+\Bus^{\xi\sig}_{y, z}= \Bus^{\xi\sig}_{x,z}   
 \quad \text{for all $x,y,z\in\Z^2$},
 \ee
and the \textit{covariant} part meaning that
  \be\label{coc2} 
 \Bus^{\xi\sig}_{x,y}(\shift_u\w) = \Bus^{\xi\sig}_{x+u,y+u}(\w)  
 \quad \text{for all $u,x,y\in\Z^2$, $\w\in\Omega$}.
 \ee
 
\item  Almost surely each $\Bus^{\xi\sig}$ recovers the vertex weights: on the event $\Omega_0$,  
		\be\label{B_reco}  \exp\{-\Bus^{\xi\sig}_{x-\evec_1,x}\}+ \exp\{-\Bus^{\xi\sig}_{x-\evec_2,x}\} =  \Yw_x^{-1}    \quad \text{for all $x\in\Z^2$}. 
		\ee
		
\item When restricted to nearest-neighbor pairs, the Busemann functions exhibit the following monotonicity: if $\zeta\prec\xi\prec\eta$, then for every $x\in\Z^2$ we have
\begin{subequations} \label{B_mono} 
\begin{align}
\label{B_mono1}
\Bus^{\zeta+}_{x-\evec_1,x}&\ge \Bus^{\xi-}_{x-\evec_1,x} \ge \Bus^{\xi+}_{x-\evec_1,x} 
\quad \text{and} \\
\label{B_mono2}
\Bus^{\xi-}_{x-\evec_2,x}&\le \Bus^{\xi+}_{x-\evec_2,x} \le \Bus^{\eta-}_{x-\evec_2,x}.
\end{align}
\end{subequations}

 \item For fixed $\w\in\Omega$ and $x,y\in\Z^2$, the maps $\xi\mapsto \Bus^{\xi-}_{x,y}(\w)$ and $\xi\mapsto \Bus^{\xi+}_{x,y}(\w)$ are the left- and right-continuous versions of each other.
 That is, under the southeast ordering of $\,]\evec_2,\evec_1[\,$, we have these monotone  limits:
 \be\label{B_lim1} 
\lim_{\zeta\nearrow\xi}   \Bus^{\zeta\sig}_{x,y}=\Bus^{\xi-}_{x,y}
\quad\text{and}\quad 
\lim_{\eta\searrow\xi}   \Bus^{\eta\sig}_{x,y}=\Bus^{\xi+}_{x,y} \quad\text{for either $\sigg\in\{-,+\}$}.
\ee
Towards  the endpoints of $[\evec_2, \evec_1]$,  for $r\in\{1,2\}$ and both signs $\sigg\in\{-,+\}$, we have these monotone  limits on the event $\Omega_0$:
\be\label{B_lim} 
\lim_{\xi\to\evec_r} \Bus^{\xi\sig}_{x-\evec_r,x} =\log\Yw_x 
\quad\text{while}\quad 
\lim_{\xi\to\evec_r} \Bus^{\xi\sig}_{x-\evec_{3-r},x} =\infty. 
\ee

\item The Busemann process is constant on linear segments of the shape function: 
\be\label{B==} 
\text{if }\ \zeta\neq\xi\ \text{ and }  \ \nabla\gpp(\zeta\sigg)=\nabla\gpp(\xi\sigg'),   \ \text{ then } \ 
\Bus^{\zeta\sig}_{x,y} = \Bus^{\xi\sig'}_{x,y}  \quad \text{for all $x,y\in\Z^2$.}
\ee


\item  Extended Busemann limits:  on the event $\Omega_0$, for any $\cL_\xi$-directed sequence $(x_\ell)$,  
\begin{subequations} \label{B_monob} 
\begin{align}
\exp\Bus_{x-\evec_1,x}^{\xi-}
\geq \varlimsup_{\ell\to-\infty} \frac{Z_{x_\ell,\tspb x}}{Z_{x_\ell,\tspb x-\evec_1}}
&\geq \varliminf_{\ell\to-\infty} \frac{Z_{x_\ell,\tspb x}}{Z_{x_\ell,\tspb x-\evec_1}}
\geq \exp\Bus_{x-\evec_1,x}^{\xi+}  \quad \text{and} \\
\exp\Bus_{x-\evec_2,x}^{\xi-}
\leq \varliminf_{\ell\to-\infty} \frac{Z_{x_\ell,\tspb x}}{Z_{x_\ell,x-\evec_2}}
&\leq \varlimsup_{\ell\to-\infty} \frac{Z_{x_\ell,\tspb x}}{Z_{x_\ell,x-\evec_2}}
\leq \exp\Bus_{x-\evec_2,x}^{\xi+}.
\end{align}
\end{subequations}

\item For every $\xi\in\,]\evec_2,\evec_1[\,$, $\sigg\in\{-,+\}$, and $x,y\in\Z^2$,   $\Bus^{\xi\sig}_{x,y}\in L^1(\P)$ and  
\be\label{EB}
\E(\Bus^{\xi\sig}_{x,y})  = \nabla\gpp(\xi\sigg)\cdot(y-x). 
\ee

\item 
For any set $A\subset\Z^2$, let $A^{\not\le}=\{u\in\Z^2:\, u\not\le y\text{ for every }y\in A \}$. 
Then we have independence of the following  two collections of random variables:
\be\label{B-ind}   \{ \Yw_u:\,  u\in A^{\not\le}\}  
\quad\indep\quad    
\{  \Yw_y, \Bus^{\xi\sig}_{x,y} :\,  \xi\in\,]\evec_2,\evec_1[\,, \sigg\in\{-,+\},  y\in A, x\le y\}.   
\ee
	

\end{itemize}
\end{theirthm}


\begin{remark}[Busemann process and a regularity assumption] \label{diffrmk} 
 The discussion  
 before Theorem~\ref{buse_full} overlooked the assumptions of Theorem~\ref{buse_dir}.  The condition $\xi\in\Udiff$ is 
harmless because $\Udiff$ is dense by concavity.  
But the requirement  $\underline\xi, \overline\xi\in\Udiff$ is a serious limitation if $\gpp$ has linear segments.
Thus it is common in the literature to assume that if $\gpp$ has linear segments, then it is differentiable at the endpoints of those segments.
Equivalently,
\eeq{ \label{shape_ass}
\text{at every $\xi\in\,]\evec_2,\evec_1[\,$, $\gpp$ is either differentiable or strictly concave}.
}
Under this assumption Theorem~\ref{buse_dir} applies to every $\xi\in\Udiff$.  
This in turn implies that  the entire Busemann process $\Bus^\rbbullet$   is a measurable function of the weights $(\Yw_x)$.

Nevertheless, Theorem~\ref{buse_full} was proved in \cite{janj-rass-20-aop} without \eqref{shape_ass} by an adaptation of  the strategy from \cite{damr-hans-14}.
The shortcoming is that the Busemann process is constructed as a weak limit and is not a function of the original weights $(\Yw_x)$.   
One has to expand the original probability space to accommodate this weak limit.  Hence Theorem~\ref{buse_full} is  properly stated as ``There exists a probability space $\OSP$ such that \eqref{ranenv} holds and..."
We regard the expansion of the probability space as given and will not make any further distinctions.

\newa{As advertised in the introduction}, our main results avoid the assumption \eqref{shape_ass}.  One challenge from 
  this is that we do not know if the Busemann process is ergodic under translations.
We do show (and need to use) that horizontal Busemann increments are ergodic under the $\evec_1$ translation (Theorem~\ref{cpthm}).
 This extends \cite[Thm.~3.5]{janj-rass-20-jsp} to joint distributions with multiple directions.
 \newa{While ergodicity under the $\evec_2$ translation would also be desirable, our arguments (especially in Lemma~\ref{g85n}) show how one can get around this.
 These challenges do not arise in the exactly solvable cases, since then $\Lambda$ is explicit and \eqref{shape_ass} is immediate.}
%
%
%
%
\qedex\end{remark} 

\begin{remark}[Discontinuities and null events]\label{pm_rmk}   Monotonicity \eqref{B_mono} and  \eqref{EB} imply that,  for each  $\xi\in\Diff$,   $ \Bus^{\xi-}= \Bus^{\xi+}$ on a full-probability event $\Omega_\xi$ that \textit{depends} on $\xi$.  In particular, when desirable, any full-probability event $\Omega_0$ can be assumed to satisfy $ \Bus^{\xi-}= \Bus^{\xi+}$ for all $\xi$ in a  fixed   countable subset of $\Diff$.
The construction of $B^\rbbullet$ described above Theorem~\ref{buse_full} relies on this property.  
Another consequence   is that any statement about the distribution of countably many $\Bus^{\xi}$ functions with $\xi\in\Diff$ can drop the signs $\sigg\in\{-,+\}$.  

Random directions $\xi$ of discontinuity $ \Bus^{\xi-}\ne\Bus^{\xi+}$ can arise among the uncountably many differentiability directions.  
 One of the main points of our paper is to describe properties of these directions. 
In Corollary~\ref{cor:ig_aUset} we determine that this set of discontinuities is dense in the inverse-gamma case, thereby providing the first existence result for discontinuities in a positive-temperature lattice model.  We cannot prove this existence in general, but we do present some new properties of the discontinuity set in Section~\ref{pathB_sec}. 

The bounds in \eqref{B_monob} leave open the possibility that when $\xi$ is a jump direction, the Busemann functions $\Bus^{\xi\pm}$ cannot be realized as limits.    
In Appendix~\ref{a:Blim}, we prove this possibility does not occur under the assumption \eqref{shape_ass}.
Under the assumption \eqref{shape_ass},   we prove  in Appendix~\ref{a:Blim}
 That is, the extreme inequalities in  \eqref{B_monob} are in fact equalities for suitably chosen $\cL_\xi$-directed sequences,  simultaneously for all directions $\xi$ (Proposition~\ref{ebiwlp}).
This result helps to enrich the theory from \cite{janj-rass-20-aop}, although it is not needed in the main text.
\qedex\end{remark} 

\begin{remark}[Monotonicity] 
\label{dmnr}
As stated, \eqref{B_mono} is a sure event.
On the almost sure event $\Omega_0$ from Theorem~\ref{buse_full}, the recovery property \eqref{B_reco} allows an upgrade:
\begin{subequations} \label{B_Mono} 
\begin{align}
\label{B_Mono1}
\Bus^{\zeta+}_{x-\evec_1,x}&\ge \Bus^{\xi-}_{x-\evec_1,x} \ge \Bus^{\xi+}_{x-\evec_1,x} > \log\Yw_x
\quad \text{and} \\
\label{B_Mono2}
\log\Yw_x &< 
\Bus^{\xi-}_{x-\evec_2,x}\le \Bus^{\xi+}_{x-\evec_2,x} \le \Bus^{\eta-}_{x-\evec_2,x}.
\end{align}
\end{subequations}
Furthermore,   \eqref{EB} gives  $\E(\Bus^{\xi\sig}_{x-\evec_1,x}) = \gpp(\xi\sigg)\cdot\evec_1$ and $\E(\Bus^{\xi\sig}_{x-\evec_2,x})=\gpp(\xi\sigg)\cdot\evec_2$.   Combined with \eqref{B_Mono} we have this  monotonicity: for $\zeta\prec\xi\prec\eta$,\begin{subequations} \label{inmono} 
\begin{align}
\label{inmono1}
\nabla\gpp(\zeta+)\cdot\evec_1
&\ge \nabla\gpp(\xi-)\cdot\evec_1
\ge \nabla\gpp(\xi+)\cdot\evec_1
> \E[\log\Yw_x] 
\quad \text{and} \\
\label{inmono2}
\E[\log\Yw_x]
&<\nabla\gpp(\xi-)\cdot\evec_2
\le\nabla\gpp(\xi+)\cdot\evec_2
\le\nabla\gpp(\eta-)\cdot\evec_2. \rlap{\hspace*{0.68in}$\triangle$}
\end{align}
\end{subequations}
\end{remark} 


Finally, the Busemann process satisfies a shape theorem simultaneously in all directions.
We prove the following result in Appendix~\ref{a:Bshape} as a small extension of the single-direction result in \cite{janj-rass-20-aop}, recalled as Theorem~\ref{lln_input}.
We stress that this  result does not require the  regularity assumption  \eqref{shape_ass}.
     
 \begin{theorem} \label{all_lln}  Assume \eqref{ranenv}. 
There exists a full-probability 
event on which the following limit  holds simultaneously 
 for each $\xi\in\,]\evec_2,\evec_1[\,$ and $\sigg\in\{-,+\}$:
\eeq{ \label{bfanfx}
\lim_{n\to\infty}\;\max_{|x|_1\le\tsp n}   n^{-1} {|\Bus^{\xi\sig}_{\zevec,x}-\nabla\gpp(\xi\sigg)\cdot x|} = 0 .  
}
\end{theorem}

\section{Main results under general i.i.d.\ weights}  \label{sec:main1}

\subsection{Busemann process indexed by directions} \label{pathB_sec}
Our first result is on the monotonicity of Busemann functions,  proved at the end of Section~\ref{2_proc_sec}.
Combined with \eqref{B==}, it reveals that $\xi\mapsto \Bus^{\xi\pm}_{x-\evec_r,x}$ is constant on linear  segments of $\gpp$ and strictly monotone elsewhere.

\begin{theorem} \label{thm:78-63}
Assume \eqref{ranenv}.
Then there exists a full-probability event on which the following holds. 
For each pair of directions $\zeta\prec\eta$ in $\,]\evec_2, \evec_1[$  that do not lie on the same closed linear segment of $\gpp$, we have the strict inequalities  
 \eeq{ \label{eq:78-63}
 \Bus^{\zeta+}_{x-\evec_1,x}> \Bus^{\eta-}_{x-\evec_1,x} >  \log\Yw_x \quad \text{and} \quad 
\log\Yw_x < \Bus^{\zeta+}_{x-\evec_2,x}< \Bus^{\eta-}_{x-\evec_2,x} \quad\forall\tsp x\in\Z^2.  
 }
\end{theorem}

Next we consider discontinuities of the Busemann process.
Define the $\w$-dependent set of exceptional directions where the Busemann process experiences a jump: 
\eeq{ \label{aUsetdef}
\aUset=\bigl\{ \xi\in\,]\evec_2,\evec_1[\,:\,   \exists\, x,y\in\Z^2,\,  \Bus^{\xi-}_{x,y}(\w)\ne  \Bus^{\xi+}_{x,y}(\w)\bigr\}. 
}
For any sequence of vertices $x=x_0,x_1,\dots,x_k=y$ such that $|x_i - x_{i-1}|_1 = 1$, the cocycle property \eqref{coc1} gives $\Bus^{\xi\pm}_{x,y} \coloneqq \sum_{i = 1}^k  \Bus^{\xi\pm}_{x_{i-1},x_i}$.
Each nearest-neighbor increment  $\Bus^{\xi\pm}_{x_{i-1},x_i}$ is a monotone function of $\xi$ by \eqref{B_mono} and thus has at most countably many discontinuities.
Hence $\aUset$ is at most countable.
Under a differentiability assumption on   $\gpp$,  \cite[Thm.~3.10(c)]{janj-rass-20-aop} implies that $\aUset$ is either empty or infinite.  Membership $\xi\in\aUset$ has implications for the existence and uniqueness of $\xi$-directed polymer Gibbs measures. Such results   under the regularity assumption \eqref{shape_ass} appear in \cite[Thm.~3.10]{janj-rass-20-aop}.   In Remark~\ref{rmk:Q_ig} we state these consequences in the inverse-gamma case.

The following theorem is proved in Section~\ref{sec:pf_disc}.
Part~\ref{thm:51-32a} is the main novelty, as part~\ref{thm:51-32b} is morally contained in \cite[Thm.~3.2]{janj-rass-20-aop}.
 
 \begin{theorem}\label{thm:51-32}   Assume \eqref{ranenv}. Then there exists a full-probability event $\Omega_0$ on which the following statements hold. 
 \begin{enumerate}[label={\rm(\alph*)}, ref={\rm(\alph*)}]  \itemsep=3pt  
 \item \label{thm:51-32a}
 The set of discontinuities of the function $\xi\mapsto \Bus^{\xi\pm}_{x-\evec_r,x}$ is the same for all nearest-neighbor edges. 
 That is,    for each  $\w\in\Omega_0$,  
 \[
\aUset=
\bigl\{\xi\in\,]\evec_2,\evec_1[\,:\,    \Bus^{\xi-}_{x-\evec_r,x}(\w)\ne  \Bus^{\xi+}_{x-\evec_r,x}(\w)\bigr\} 
\quad \forall \tsp 
x\in\Z^2,\  r\in\{1,2\}.  
\] 
\item \label{thm:51-32b}
For each $\w\in\Omega_0$, $\aUset$ contains the set $\,]\evec_2,\evec_1[\,\setminus\Udiff$ of directions $\xi$ at which   the shape function $\gpp(\xi)$ is not differentiable.
\end{enumerate}
\end{theorem} 

It is natural to ask whether Theorem~\hyperref[thm:51-32a]{\ref*{thm:51-32}\ref*{thm:51-32a}} extends from nearest-neighbor edges to all pairs of distinct vertices.
For any $x\neq y$ in $\Z^2$, the cocycle property \eqref{coc1} allows us to write $\xi\mapsto\Bus^{\xi\pm}_{x,y}$ as a sum of nearest-neighbor Busemann functions, all of which share the discontinuity set $\aUset$.
If $x$ and $y$ are connected by a down-right path, then all these discontinuities have the same sign thanks to monotonicity \eqref{B_mono}.  
Hence these discontinuities persist in the sum, and every $\xi\in\aUset$ remains a point of discontinuity for $\xi\mapsto\Bus^{\xi\pm}_{x,y}$.
But if $x$ and $y$ are connected by an up-right path, then it is possible that some discontinuities cancel out.
The next theorem rules out this possibility under the assumption of continuous weights.
The proof is given in Section~\ref{sec:pf_disc}.

\begin{theorem} \label{thm:allxy}
 Assume \eqref{ranenv} and $\P(\Yw_x=s)=0$ for all $s>0$.
Then there exists a full-probability event $\Omega_0$ on which the set of discontinuities of $\xi\mapsto \Bus^{\xi\pm}_{x,y}$ is the same for all pairs of distinct vertices $x,y\in\Z^2$. 
 That is,    for each  $\w\in\Omega_0$,  
 \[
\aUset=
\bigl\{\xi\in\,]\evec_2,\evec_1[\,:\,    \Bus^{\xi-}_{x,y}(\w)\ne  \Bus^{\xi+}_{x,y}(\w)\bigr\} 
\quad \forall \tsp 
x\neq y \text{ in $\Z^2$}.
\] 
\end{theorem}

An interpretation of Theorems~\ref{thm:51-32} and \ref{thm:allxy} in terms of semi-infinite polymer measures is given in Remark~\ref{rmk:pmid}.
Another interpretation in terms of a discrete stochastic heat equation is given in Theorem~\ref{thm:she_ig}.

 \subsection{Joint distribution of  the Busemann process}  \label{distB_sec}

This section gives a preliminary characterization of the joint distribution of the Busemann process, without full details. The complete description requires additional developments and appears in Section~\ref{mar_sec}. 
 
Once the weights $(\Yw_x)_{x\in\Z^2}$ are given, a Busemann function  $\Bus^{\xi\sig}$ is completely determined by its values  $(\Bus^{\xi\sig}_{x-\evec_1,x})_{x\in\Z^2}$  on  horizontal nearest-neighbor edges, by additivity \eqref{coc1}  and    recovery   \eqref{B_reco}. 
For this reason and stationarity, it is sufficient to describe  the joint distribution on any horizontal level.  

On each lattice level $t\in\Z$, define the sequence  ${\brvI}^{\chdir\sig}(t)=({\brvI}^{\chdir\sig}_k(t))_{k\in\Z}$   of exponentiated horizontal nearest-neighbor Busemann increments 
\begin{subequations}\label{IB5}
\be\label{IB5.a}    {\brvI}^{\chdir\sig}_k(t) = 
e^{\Bus^{\chdir\sig}_{(k-1,t),\tsp (k,t)}}, \qquad k\in\Z . \ee 
Fix $N$ directions 
$\xi_1,\dotsc, \xi_N$ in $]\evec_2,\evec_1[\,$  and signs $\sigg_1,\dotsc,\sigg_N\in\{-,+\}$.   Condense the notation of the $N$-tuple of sequences as 
\be\label{IB5.b}    {\brvI}^{(\chdir\sig)_{\parng{1}{N}}}(t) = \bigl( {\brvI}^{\chdir_1\sig_1}(t), {\brvI}^{\chdir_2\sig_2}(t),  \dotsc, {\brvI}^{\chdir_N\sig_N}(t) \bigr) \; \in \; (\R_{>0}^\Z)^N. \ee 
\end{subequations} 
The values  ${\brvI}^{(\chdir\sig)_{\parng{1}{N}}}(t+1)$ at level $t+1$ can be calculated from 
the level-$t$ values   ${\brvI}^{(\chdir\sig)_{\parng{1}{N}}}(t)$ and the level-$(t+1)$ weights $\Yw(t+1)=(\Yw_{(k,t+1)})_{k\in\Z}$  by a  mapping  encoded as  
\be\label{Bmc}
{\brvI}^{(\chdir\sig)_{\parng{1}{N}}}(t+1)=\Taop_{\Yw(t+1)}\bigl({\brvI}^{(\chdir\sig)_{\parng{1}{N}}}(t)\bigr). 
\ee
This mapping $\Taop_\Yw$, called the \textit{parallel transformation}, depends on a given sequence $\Yw$ of weights and acts on $N$-tuples of sequences.  It is defined in equation \eqref{Taop_def} in Section~\ref{twin_sec}.   
Since $\Yw(t+1)$ is independent of  ${\brvI}^{(\chdir\sig)_{\parng{1}{N}}}(t)$,  it follows that the process 
$({\brvI}^{(\chdir\sig)_{\parng{1}{N}}}(t):\,  t\in\Z)$ is an $(\R_{>0}^\Z)^N$-valued stationary Markov chain. 
Here $\R_{>0}^\Z$ and $(\R_{>0}^\Z)^N$ are equipped with the product topologies.

Translation on the sequence space $(\R_{>0}^\Z)^N$ is the operation $\tau$ that acts 
 on  $I=(I^i_k)^{i\in\lzb1,N\rzb}_{k\in\Z}\in(\R_{>0}^\Z)^N$  by shifting the $k$-index: $(\tau I)^i_k=I^i_{k-1}$.    Recall the mean \eqref{EB}.  


\begin{theorem}\label{cpthm}
Assume \eqref{ranenv}. Let $N\in\Z_{>0}$.  
The property   
\eq{    
\E[ \log {\brvI}^{\chdir_i\sig_i}_k(t)  ] =  \E[\Bus^{\xi_i\sig_i}_{(k-1,t),\tsp (k,t)}]  = \nabla\gpp(\xi_i\sigg_i)\cdot\evec_1 \qquad\text{for }  i\in\lzb1,N\rzb \text{ and } k\in\Z
}
 determines uniquely a probability  distribution $\mu$ on the   space $(\R_{>0}^\Z)^N$  that is invariant  for the Markov chain \eqref{Bmc} and stationary and ergodic under the translation $\tau$.  
 In particular, for each $t\in\Z$, the $N$-tuple of sequences  ${\brvI}^{(\chdir\sig)_{\parng{1}{N}}}(t)$ defined in \eqref{IB5}   has distribution $\mu$. 
\end{theorem}

A precise version of this theorem is stated and proved as Theorem~\ref{B_thm9} in Section~\ref{2_proc_sec}.  
Since this theorem concerns a fixed finite set of directions, the sign $\sigg_i$ makes a difference only if $\nabla\gpp(\xi_i-)\ne\nabla\gpp(\xi_i+)$, as explained in Remark~\ref{pm_rmk}. 

\begin{remark}[Vertical increments] 
\label{rmk:vB}    By the  reflection symmetry of i.i.d.~weights,  
Theorem~\ref{cpthm} applies also to  vertical Busemann increments.  
 In particular,   the processes $(I^{\xi+}_k(t):\, k,t\in\Z,\, \xi\in\,]\evec_2,\evec_1])$  
 and  $(J^{\wh\xi-}_t(k):\, k,t\in\Z,\, \xi\in\,]\evec_2,\evec_1])$ are equal in  distribution,  where $J^{\xi\sig}_k(t)$ is defined in \eqref{IB5.8}, and $\wh\xi$ is the reflection of $\xi$ across the $\evec_1+\evec_2$ direction.
This fact, though  intuitive, is not immediately apparent from Theorem~\ref{buse_full}.
\qedex\end{remark}

%
%

\subsection{Competition interface directions}  \label{sec:cif} 
We define the competition interface from \cite{geor-rass-sepp-yilm-15,janj-rass-20-aop}. 
By \eqref{qmpf} the point-to-point polymer measure $Q_{u,v}$ from \eqref{hcq}   is an up-right Markov chain starting at $u$ and ending at $v$, with transition probabilities
\eq{
\pi_v(x,x+\evec_r) = \Yw_{x+\evec_r}\frac{Z_{x+\evec_r,v}}{Z_{x,v}}, \quad x< v,\ r\in\{1,2\}.
}
Given a realization of the weights $(\Yw_x)$, these walks can be coupled  as follows.

For $\w\in\Omega$, let $\Qwalks^\w$ be a probability measure under which the  weights 
have been fixed:
\eeq{ \label{fnv8}
\Qwalks^\w\{\text{$\Yw_x = \Yw_x(\w)$ for all $x\in\Z^2$}\} = 1. 
}
Assume there is a family of random variables $(U_x)_{x\in\Z^2}$ that are i.i.d.~uniform on $(0,1)$ under $\Qwalks^\w$.
Recall the levels $\level_n=\{v\in\Z^2: n=v\cdot(\evec_1+\evec_2)\}$ from \eqref{leveldef}.
For each pair $u<v$ with $u\in\level_\ell$ and $v\in\level_n$, define the path $X^{u,v}_\bbullet = X^{u,v}_{\parng{\ell}{n}}$ starting at $X^{u,v}_\ell = u$ and proceeding up or right according to the following rule.
If $\ell\leq m<n$ and $X^{u,v}_m=x\in\level_m$, then set
\eeq{ \label{buqp1}
  X^{u,v}_{m+1} =  
\begin{cases}  x  +  \evec_1  &\text{if }U_x \leq \pi_v(x, x+\evec_1), \\[4pt]  
 x  +  \evec_2   &\text{if }U_x > \pi_v(x, x+\evec_1). 
\end{cases}  
}
Under $\Qwalks^\w$, $X^{u,v}_{\bbullet}$ has  law $Q_{u,v}$.
Furthermore, if $X^{u_1,v}_m=X^{u_2,v}_m$, then $X^{u_1,v}_{m+1}=X^{u_2,v}_{m+1}$ since the right-hand side of \eqref{buqp1} does not depend on $u$.
For any given $v\in\level_n$, the sets $\{u:\, X^{u,v}_{n-1}=v-\evec_1\}$ and $\{u:\, X^{u,v}_{n-1}=v-\evec_2\}$ are disjoint; so by planarity, these two clusters are separated by a down-left path $\varphi^v = \varphi^v_{\parng{-\infty}{n}}$. (Figure~\ref{Fig:cif} gives an example where $v=\zevec$.) 
This path is   the \textit{competition interface}. Under assumption \eqref{shape_ass}, it was shown in \cite[Thm.~3.12]{janj-rass-20-aop} to have a random asymptotic direction $\xi^*(v)$:  
  for $\P$-almost every $\w$, there is a quenched law of large numbers  
\[  \Qwalks^\w\Big\{ \ddd\lim_{n\to-\infty} n^{-1}\varphi^v_n = \xi^*(v)\Big\}=1\]
with the limit distribution 
\be\label{xilaw5}  
\Qwalks^\w\{ \xi^*(v)  \preccurlyeq \xi \} = \Yw_v \tspb e^{-\Bus^{\xi+}_{v-\evec_1, v}}  , \qquad \xi\in\,]\evec_2, \evec_1[\,.  
\ee
The appearance of the Busemann function in \eqref{xilaw5} suggests a connection to semi-infinite polymer measures, and that is what our paper addresses.

\begin{figure}
\begin{tikzpicture}[>=latex,scale=0.5]
\draw[->] (0,0)--(-10,0);
\draw[->] (0,0)--(0,-10);
\draw(0.38,0.36)node{$\zevec$};
\draw(0.93,-1)node{$-\evec_2$};
\draw(-1,0.5)node{$-\evec_1$};
\draw[line width=1pt](-9,0)--(0,0)--(0,-9); 
\draw[line width=1pt](-2,0)--(-2,-4)--(-5,-4)--(-5,-7)--(-7,-7)--(-7,-9)--(-9,-9);
\draw[line width=1pt](-1,0)--(-1,-3);
\draw[line width=1pt](0,-4)--(-1,-4)--(-1,-5)--(-2,-5)--(-2,-7);
\draw[line width=1pt](-2,-5)--(-4,-5)--(-4,-6);
\draw[line width=1pt](-2,-6)--(-3,-6)--(-3,-8)--(-5,-8)--(-5,-9)--(-6,-9);
\draw[line width=1pt](-3,-7)--(-4,-7);
\draw[line width=1pt](-6,-7)--(-6,-8);
\draw[line width=1pt](0,-6)--(-1,-6)--(-1,-8)--(-2,-8)--(-2,-9)--(-3,-9);
\draw[line width=1pt](-4,-8)--(-4,-9);
\draw[line width=1pt](0,-9)--(-1,-9);
\draw[line width=1pt](-5,-4)--(-7,-4)--(-7,-5)--(-9,-5)--(-9,-8);
\draw[line width=1pt](-6,-4)--(-6,-6)--(-7,-6);
\draw[line width=1pt](-8,-5)--(-8,-8);
\draw[line width=1pt](-2,-1)--(-8,-1)--(-8,-2)--(-9,-2);
\draw[line width=1pt](-2,-2)--(-4,-2)--(-4,-3)--(-6,-3);
\draw[line width=1pt](-9,0)--(-9,-1);

\draw[line width=1pt](-5,-1)--(-5,-2)--(-7,-2)--(-7,-3);
\draw[line width=1pt](-8,-2)--(-8,-4)--(-9,-4);
\draw[line width=1pt](-3,-2)--(-3,-3);
\draw[line width=1pt](-8,-3)--(-9,-3);

\draw[line width=3.0pt](-.5,-.5)--(-0.5,-3.5)--(-1.5,-3.5)--(-1.5,
-4.5)--(-4.5,-4.5)--(-4.5,-7.5)--(-5.5,-7.5)--(-5.5,-8.5)--(-6.5,-8.5)--(-6.5,
-9.5); 
\foreach\x in {0,...,-9}{
 \foreach\y in {0,...,-9}{
\shade[ball color=light-gray](\x,\y)circle(1.5mm); 
}}
\end{tikzpicture}
\caption{\small   A sample of all finite polymer paths terminating at $\zevec$, coupled via \eqref{buqp1}.
The competition interface $\varphi^\zevec$ is the solid line on the dual lattice $\Z^2+(-\tfrac12,-\tfrac12)$. Paths from the west and north of $\varphi^\zevec$ reach $\zevec$ through $-\evec_1$, while paths from the east and south of $\varphi^\zevec$ reach $\zevec$ through $-\evec_2$.}
\label{Fig:cif}
\end{figure}

Consider now the family of Gibbs measures $(Q^{\xi\sig}_v)_{v\in\Z^2}$ associated to the Busemann function $\Bus^{\xi\sig}$ as in Theorem~\ref{pococthm}.
In other words,   $Q^{\xi\sig}_v$  is the quenched distribution of semi-infinite southwest paths rooted at $v\in\level_n$. 
Each $Q^{\xi\sig}_v$  is a down-left Markov chain 
with  transition probabilities
\be\label{pi36}   \pi^{\xi\sig}(x, x-\evec_r) = \Yw_x \tspc e^{-\Bus^{\xi\sig}_{x-\evec_r, x}}  \,, \quad x\in\Z^2, \  r\in\{1,2\}. 
\ee
Note that these transition probabilities inherit the monotonicity of the Busemann process: if either $\zeta\prec\eta$ or $(\zeta\sigg, \eta\sigg')=(\xi-,\xi+)$,  then \eqref{B_mono1} implies
\eeq{ \label{pi_mono}
\pi^{\zeta\sig}(x,x-\evec_1)\leq \pi^{\eta\sig'}(x,x-\evec_1).
}
\begin{remark}[Polymer-measure interpretation of results on discontinuity set] \label{rmk:pmid}
Theorem~\hyperref[thm:51-32a]{\ref*{thm:51-32}\ref*{thm:51-32a}} says that if $\pi^{\xi-}(x,x-\evec_1)<\pi^{\xi+}(x,x-\evec_1)$ for some $x\in\Z^2$, then the same strict inequality holds for \textit{all} $x$.
If in addition the weights $(\Yw_x)_{x\in\Z^2}$ have a continuous distribution, then Theorem~\ref{thm:allxy} implies that $Q_v^{\xi-}$ and $Q_v^{\xi+}$ do not agree on any marginal: assuming $v\in\level_n$ and $m<n$, we have
\[
Q_v^{\xi-}(X_{m}=u) \neq Q_v^{\xi+}(X_{m}=u) \quad \text{for every $u\in\level_m$, $u\le v$, $\xi\in\aUset$} \qedhere.
\]
This is because the probability of reaching $u$ is determined by the value of $\Bus_{u,v}^{\xi\sig}$:
\eq{
Q_v^{\xi\sig}(X_{m}=u) 
= \sum_{x_\brbullet\tspa\in\tspa\pathsp_{u,v}} \prod_{i=m+1}^{n}\pi^{\xi\sig}(x_i,x_{i-1}) 
\stackrel{\eqref{pi36},\eqref{coc1}}{=}e^{-\Bus^{\xi\sig}_{u,v}}\sum_{x_\brbullet\tspa\in\tspa\pathsp_{u,v}} \prod_{i=m+1}^{n}\wgtd_{x_i}
\stackref{part56}{=} e^{-\Bus^{\xi\sig}_{u,v}}Z_{u,v}. 
\rlap{\hspace*{9pt}$\triangle$}
}
\end{remark}



We now proceed to couple all the distributions $(Q^{\xi\sig}_v:\, \xi\in\,]\evec_2, \evec_1[\,,\, \sigg\in\{-,+\}, v\in\Z^2)$.
For each $\w\in\Omega$, let $\Qwalks^\w$ be as in \eqref{fnv8} with the additional guarantee of fixing the values of the Busemann process:\footnote{When \eqref{shape_ass} is assumed, \eqref{fnv9} is implied by \eqref{fnv8} because then the Busemann process is a function of the weights (see Remark~\ref{diffrmk}).}
\eeq{ \label{fnv9}
\Qwalks^\w\{\text{$\Bus^\rbbullet=\Bus^\rbbullet(\w)$}\} = 1.
}
This means the transition probability $\pi^{\xi\sig}(x,x-\evec_r)$ in \eqref{pi36} is deterministic under $\Qwalks^\w$. 
  For each direction $\xi\in\,]\evec_2, \evec_1[$,  sign $\sigg\in\{-,+\}$,  root  vertex $v\in\level_n$,  and tiebreaker  $\tiebr\in\{\evec_1, \evec_2\}$,  define the random path  $X^{v, \xi\sig, \tiebr}_\bbullet = X^{v, \xi\sig, \tiebr}_{\parng{-\infty}{n}}$ inductively as follows.  
Fix the root location $X^{v, \xi\sig, \tiebr}_n=v$.
For $m\le n$, if $X^{v, \xi\sig, \tiebr}_m$ is equal to $x\in\level_m$, then set 
\be\label{X129} 
  X^{v, \xi\sig, \tiebr}_{m-1} =  
\begin{cases}  x  -  \evec_1  &\text{if }U_x < \pi^{\xi\sig}(x, x-\evec_1), \\[4pt]  
 x  -  \evec_2   &\text{if }U_x > \pi^{\xi\sig}(x, x-\evec_1), \\[4pt]  
 x  -  \tiebr  &\text{if }U_x = \pi^{\xi\sig}(x, x-\evec_1). 
\end{cases}  
\ee
%
Under $\Qwalks^\w$, the path  $X^{v, \xi\sig, \tiebr}_\bbullet$ has law $Q^{\xi\sig}_v$ because its transition probability from $x$ to $x-\evec_1$ is clearly $\pi^{\xi\sig}(x,x-\evec_1)$.
The tiebreaker $\tiebr$ is included because $\xi$ takes uncountably many values.
Indeed, for any fixed $\xi\sigg$, we have $\Qwalks^\w\{U_x = \pi^{\xi\sig}(x, x-\evec_1)\}=0$ and so the walks $X^{v, \xi\sig, \evec_1}_\bbullet$ and $X^{v, \xi\sig, \evec_2}_\bbullet$ agree $\Qwalks^\w$-almost surely.
But considering all values of $\xi\sigg$ simultaneously leaves open the possibility that $X^{v, \xi\sig, \evec_1}_\bbullet$ and $X^{v, \xi\sig, \evec_2}_\bbullet$ separate at some lattice vertex.  

Notice that the protocol \eqref{X129} does not depend on $v$.
That is, for given $\xi\sigg$ and $\tiebr$, any two walks $X^{v_1, \xi\sig, \tiebr}_\bbullet$ and $X^{v_2, \xi\sig, \tiebr}_\bbullet$ that meet at some $x\leq v_1\wedge v_2$ will thereupon remain together forever.
Therefore, it suffices to understand the behavior of $X^{x,\xi\sig,\tiebr}_\bbullet$ at $x$, which is the content of the following theorem.


\begin{theorem}\label{thm:eta10}   
Assume \eqref{ranenv}.
For $\P$-almost every $\w$, 
the following holds. Under $\Qwalks^\w$ there exist independent $\,]\evec_2, \evec_1[\,$-valued random directions $(\eta^*(x))_{x\in\Z^2}$ with the following properties. 
\begin{enumerate} [label={\rm(\alph*)}, ref={\rm(\alph*)}] \itemsep=3pt 
\item\label{eta10a}  The marginal distribution is, for $\eta \in\,]\evec_2, \evec_1[\,$, 
\be\label{Q-eta14} 
\Qwalks^\w\{  \eta^*(x)\preccurlyeq \eta\}  =   \pi^{\eta+}(x,x-\evec_1)
\ee

\item\label{eta10b}   Let $x\in\level_m$.    Then $\Qwalks^\w$-almost surely the walks \eqref{X129} behave as follows at $x$.  \\[-9pt]   
\begin{enumerate} [label={\rm(b.\roman*)}, ref={(b.\roman*)}] \itemsep=3pt 

\item\label{eta10bi}  Suppose  $\zeta \prec \eta^*(x)  \prec \eta$. Then  for both signs $\sigg\in\{-,+\}$ and tiebreakers  $\tvec\in\{\evec_1, \evec_2\}$, 
 $X^{x,\zeta\sig, \tiebr}_{m-1} =x-\evec_2$   and   $X^{x,\eta\sig, \tiebr}_{m-1} =x-\evec_1$.  

\item\label{eta10bii}  Suppose $\xi = \eta^*(x)\notin\aUset$.  Then the tiebreaker separates the walks but the $\pm$ distinction has no effect:    for both $\sigg\in\{-,+\}$, 
$  X^{x, \xi\sig, \evec_2}_{m-1}=x-\evec_2$   and   $X^{x, \xi\sig, \evec_1}_{m-1} =x-\evec_1$. 

\item\label{eta10biii}  Suppose $\xi=\eta^*(x)\in\aUset$.    Then the $\pm$ distinction separates the walks but the tiebreaker has no effect: for both $\tvec\in\{\evec_1, \evec_2\}$, 
$X^{x, \xi-, \tvec}_{m-1}=x-\evec_2$   and   $X^{x, \xi+, \tvec}_{m-1} =x-\evec_1$. 

\end{enumerate} 

\end{enumerate} 

\end{theorem}

\begin{remark}[Relation to competition interface] \label{bhg3}
There is an obvious duality between the constructions of $\xi^*(x)$ and $\eta^*(x)$.
The former separates finite up-right paths ending at $x$, while the latter separates semi-infinite down-left paths starting at $x$.
Comparison of \eqref{xilaw5} and \eqref{Q-eta14} shows that the two directions have the same quenched distribution.
One compelling aspect of our construction is that $(\eta^*(x))_{x\in\Z^2}$ is an independent family under $\Qwalks^\w$, whereas $(\xi^*(x))_{x\in\Z^2}$ is not.
This allows us in Theorem~\ref{thm:eta20} below to relate the interface directions to discontinuities of the Busemann process.
Another advantage is that Theorem~\ref{thm:eta10} does not require the regularity assumption \eqref{shape_ass}.
A disadvantage is that there is no canonical way to identify an interface with asymptotic direction $\eta^*(x)$, since two paths $X^{v_1, \zeta\sig, \tiebr}_\bbullet$ and $X^{v_2, \eta\sig', \tiebr'}_\bbullet$ can separate and rejoin several times.

While our presentation has coupled $\xi^*(x)$ and $\eta^*(x)$ through the same auxiliary randomness in \eqref{buqp1} and \eqref{X129}, this is purely for simplicity, and there may be a more natural coupling offering additional insights.
The connections between $\xi^*$, $\eta^*$, the geometry of polymer paths,  and the regularity of the Busemann process are largely left open, elucidated in Remark~\ref{rm:open4}.  
In Section~\ref{sec:ig_xi}  we resolve some of these questions in the inverse-gamma case.  
\qedex\end{remark}

 \begin{remark}[Comparison with zero temperature, part 1] \label{bhg4}
In LPP there is no need for the auxiliary randomness supplied by $(U_x)$, since in that setting the fundamental objects are geodesic paths rather than path measures. 
The finite paths in \eqref{buqp1} are analogous to finite geodesics, while the semi-infinite paths in \eqref{X129} are analogous to semi-infinite geodesics defined by Busemann functions (see \cite[eq.~(2.12)]{janj-rass-sepp-23}).
Those two families of geodesics share the same interface, so there is no distinction between $\xi^*(x)$ and $\eta^*(x)$ at zero temperature.
That interface is defined so as to separate geodesics passing through $x-\evec_1$ from those passing through $x-\evec_2$, just as in Figure~\ref{Fig:cif}.
See \cite[Thm.~A.8]{janj-rass-sepp-23} for a precise accounting of properties that follow.
\qedex\end{remark}


We record further properties of our interface directions in the next theorem.
One of the statements makes the (highly non-trivial) assumption that the Busemann process is pure-jump:
\eeq{ \label{pure_jump}
\Bus^{\zeta-}_{x-\evec_1,x}-\Bus^{\eta+}_{x-\evec_1,x}
= \sum_{\xi\in\aUset\cap[\zeta,\eta]}[\Bus^{\xi-}_{x-\evec_1,x}-\Bus^{\xi+}_{x-\evec_1,x}] \quad
\text{for all $\zeta\prec\eta$ in $]\evec_2,\evec_1[\,$, $x\in\Z^2$}.
}

\begin{theorem}\label{thm:eta20} 
Assume \eqref{ranenv}.
The following holds $\Qwalks^\w$-almost surely, for $\P$-almost every $\w$. 
\begin{enumerate} [label={\rm(\alph*)}, ref={\rm(\alph*)}]
\item \label{thm:eta20b} Each direction $\xi\notin\aUset$ appears at most once among  $\{\eta^*(x):\,  x\in\Z^2\}$.
\item \label{thm:eta20bb}
Under additional assumption \eqref{pure_jump},   $\{\eta^*(x):\,  x\in\Z^2\}\subset\aUset$.
\end{enumerate} 
For the next three statements  assume  \eqref{ranenv} and regularity assumption \eqref{shape_ass}. 
\begin{enumerate} [resume,label={\rm(\alph*)}, resume,ref={\rm(\alph*)}]   \itemsep=2pt
\item  \label{thm:eta20d}  Suppose the pair $(\zeta\sigg, \eta\sigg')$ satisfies one of these two conditions:  
\begin{itemize} \itemsep=1pt
\item $(\zeta\sigg, \eta\sigg')=(\xi-,\xi+)$ for some  $\xi\in\aUset$; or
\item   $\zeta\prec\eta$ do not lie on the same closed linear segment of $\gpp$. 
\end{itemize} 
Then for each $v\in\Z^2$ and any tiebreakers $\tvec, \tvec'\in\{\evec_1, \evec_2\}$,  the walks  $X^{v, \zeta\sig,\tvec}_\bbullet$ and  $X^{v, \eta\sigg',\tvec'}_\bbullet$ eventually separate permanently.
That is, there exists $m>-\infty$ such that 
$X^{v, \eta\sig',\tvec'}_\ell \prec X^{v, \zeta\sig,\tvec}_\ell$ for all $\ell\le m$.  

\item  \label{thm:eta20a} Each discontinuity direction $\xi\in\aUset$ appears infinitely often among $\{\eta^*(x):\,  x\in\Z^2\}$.  
\item \label{thm:eta20c} The set $\{\eta^*(x):\,  x\in\Z^2\}$ is dense in $\,]\evec_2,\evec_1[\,$ in the complement of linear segments of $\gpp$.
\end{enumerate}  
 \end{theorem}

  
  As we will see in the proof, the effect of assumption \eqref{pure_jump} in part~\ref{thm:eta20bb} is to eliminate the third possibility in \eqref{X129}. 
  This renders tiebreakers unnecessary and rules out case~\ref{eta10bii} in Theorem~\ref{thm:eta10}.
  Meanwhile, parts~\ref{thm:eta20d}--\ref{thm:eta20c} utilize the extremality of the polymer Gibbs measures $Q^{\xi\sig}_u$, which presently has been proved only under assumption \eqref{shape_ass} \cite{janj-rass-20-aop}. 

 \begin{remark}[Comparison with zero temperature, part 2] \label{bhg5}
In LPP with continuous weights, the almost-sure uniqueness of finite geodesics implies that once semi-infinite geodesics separate, they cannot meet again.
Theorem~\hyperref[thm:eta20d]{\ref*{thm:eta20}\ref*{thm:eta20d}} is the analogous result here.
It is not possible to eliminate all reunions since the uniform variables $(U_x)$ guiding the polymer walks are chosen independently, which allows any two walks $X^{v, \zeta\sig,\tvec}_\bbullet$ and  $X^{v, \eta\sigg',\tvec'}_\bbullet$ to meet with positive $\Qwalks^\w$-probability even after separating.

Parts~\ref{thm:eta20b} and \ref{thm:eta20a} are similar to the statement in LPP that 
the set $\{\xi^*(x):\,  x\in\Z^2\}$ lies in the union of the supports of the Lebesgue--Stieltjes measures of the Busemann functions $\xi\mapsto \Bus^{\xi+}_{x,y}$ \cite[Thm.~3.7(a)]{janj-rass-sepp-23}.
In the exponential case,  \cite[Thm.~3.4]{fan-sepp-20} shows $\xi\mapsto \Bus^{\xi+}_{x,y}$ are step functions in analogy with \eqref{pure_jump}, so $\{\xi^*(x):\,  x\in\Z^2\}$ is exactly the union of their jump locations \cite[Thm.~3.7(b)]{janj-rass-sepp-23} as in parts~\ref{thm:eta20bb} and \ref{thm:eta20a}.
We prove analogous statements for the inverse-gamma polymer in Theorems~\ref{B-th5} and \ref{thm:ig-cif}.

Finally, part~\ref{thm:eta20c} is a positive-temperature version of \cite[Thm.~3.8(b)]{janj-rass-sepp-23}.
\qedex\end{remark}


\begin{remark}[Open questions]  \label{rm:open4}   $ $ 
\begin{enumerate} [label={\rm(\Roman*)}, ref={\rm(\Roman*)}]  \itemsep=2pt 
\item\label{open4.a}  The fundamental   question is whether the Busemann functions $\xi\mapsto\Bus^{\xi+}_{x-\evec_1,x}$ are continuous.
If not, does the set $\{\eta^*(x):\,  x\in\Z^2\}$ consist entirely of discontinuities of the Busemann functions as in Theorem~\hyperref[thm:eta20]{\ref*{thm:eta20}\ref*{thm:eta20bb}}? 
If so, then the existence and denseness of these discontinuities would follow from Theorem~\hyperref[thm:eta20c]{\ref*{thm:eta20}\ref*{thm:eta20c}}.

\item\label{open4.b}   Do the rich connections between the regularity of the Busemann process and the geometric properties of semi-infinite geodesics in LPP found in \cite[Sec.~3.1]{janj-rass-sepp-23} appear in some form for positive-temperature polymers?   For example, it follows from the coalescence theorem in \cite[App.~A.2]{janj-rass-20-aop} that for each pair  $x,y\in\Z^2$  there exists a dense open subset $\cA\subset\,]\evec_2,\evec_1[\,$ with the following property. 
For each open subinterval  $\,]\zeta, \eta[\,$ of $\cA$, there exists a pair of finite down-right paths that emanate  from $x$ and $y$ and meet at a point $z$, and for each direction $\xi\in\,]\zeta, \eta[\,$, sign $\sigg\in\{-,+\}$ and tiebreaker $\tvec$,  the walks   $X^{x,\xi\sig,\tvec} _\bbullet$ and $X^{y,\xi\sig,\tvec} _\bbullet$  follow these paths to their coalescence point.  Are the coalescence points related  to singularities of the Busemann functions or to the directions $\xi^*(x)$ or $\eta^*(x)$? 
\end{enumerate} 
In Section~\ref{sec:ig_xi} we answer part~\ref{open4.a} in the affirmative for the inverse-gamma polymer, by verifying the pure-jump hypothesis \eqref{pure_jump}. 
Determining whether \eqref{pure_jump} holds in greater generality is an important open problem.
The questions in part~\ref{open4.b} are left for the future even in the exactly solvable case.  
\qedex\end{remark}





The remainder of this section proves Theorems~\ref{thm:eta10} and \ref{thm:eta20}, by appeal to Theorems~\ref{thm:78-63} and \ref{thm:51-32}. 
The proposition below establishes the existence and uniqueness of the directions that dictate where walks split. 
We choose to define our objects in sufficient generality to account for zero-probability events, since that has turned out to be necessary for a full understanding in the zero-temperature case.
Hence below we first define two values $\eta^{*1}_x\preccurlyeq\eta^{*2}_x$ and then show that they agree $\Qwalks^\w$-almost surely for $\P$-almost every $\w$. 

For use below, note that  the limits in \eqref{B_lim} give the degenerate transition kernels 
\be\label{pi-deg} \begin{aligned} 
 \pi^{\evec_r}(x, x-\evec_r)  &=   \lim_{\xi\to\evec_r} \pi^{\xi\sig}(x,x-\evec_r)  =1  \\
 \text{and}\quad  
\pi^{\evec_r}(x, x-\evec_{3-r})  &=   \lim_{\xi\to\evec_r} \pi^{\xi\sig}(x, x-\evec_{3-r})  =0, \quad r\in\{1,2\}. 
\end{aligned} \ee

\begin{proposition} \label{pr:eta6}
For $\P$-almost every $\w$, the following is true.
For any realization of $(U_x)\in(0,1)^{\Z^2}$ and  at each vertex $x$, there exist  unique $\eta^{*1}_x\preccurlyeq\eta^{*2}_x$ in $\,]\evec_2, \evec_1[\,$ such that the following implications are true.
For any $\zeta, \eta\in\,]\evec_2, \evec_1[\,$ and  signs $\sig, \sig'\in\{-,+\}$, 
\begin{subequations} 
\label{Q-eta3} 
\begin{align}
\label{Q-eta3a}
&\zeta \prec \eta^{*1}_x\preccurlyeq\eta^{*2}_x \prec \eta 
\quad\text{implies} \quad 
\pi^{\zeta\sig}(x, x-\evec_1) < U_x <    \pi^{\eta\sig'}(x, x-\evec_1)  \\
\label{Q-eta3b}  
\text{and}\quad & \pi^{\zeta\sig}(x, x-\evec_1) < U_x <    \pi^{\eta\sig'}(x, x-\evec_1) 
\quad\text{implies} \quad 
\zeta \preccurlyeq \eta^{*1}_x\preccurlyeq\eta^{*2}_x \preccurlyeq \eta.
\end{align}
\end{subequations}
Furthermore, we have these inequalities: 
\be\label{Q-eta6} 
\pi^{\eta^{*1}_x-}(x, x-\evec_1) \le  \pi^{\eta^{*2}_x-}(x, x-\evec_1) \le  U_x \le  \pi^{\eta^{*1}_x+}(x, x-\evec_1) \le  \pi^{\eta^{*2}_x+}(x, x-\evec_1)  . \ee
Disagreement $\eta^{*1}_x\ne\eta^{*2}_x$ happens if and only if $[\eta^{*1}_x, \eta^{*2}_x]$ is a maximal  linear segment of  $\gpp$ and  $U_x=\pi^{\xi\sig}(x, x-\evec_1)$ for some {\rm(}and hence any{\rm)} $\xi\in\,]\eta^{*1}_x, \eta^{*2}_x[\,$. 

\end{proposition} 

\begin{proof}  {\it Existence.}    Set
\be\label{eta*5} \begin{aligned} 
\eta^{*1}_x &=\inf\{ \eta\in[\evec_2, \evec_1]:\,   \pi^{\eta\sig}(x, x-\evec_1) \ge U_x\} \\
\quad\text{and}\quad 
   \eta^{*2}_x &=\sup\{ \zeta\in[\evec_2, \evec_1]:\,   \pi^{\zeta\sig'}(x, x-\evec_1) \le U_x\}. 
\end{aligned} \ee
Since $\zeta\mapsto\pi^{\zeta-}$ and $\zeta\mapsto\pi^{\zeta+}$ are the left- and right-continuous versions of the same nondecreasing function, these definitions are independent of the signs $\sig, \sig'\in\{-,+\}$.    
It follows from \eqref{pi-deg} that for $0<U_x<1$, the infimum and the supremum are over nonempty sets and each lies in the open segment $\,]\evec_2, \evec_1[\,$.   Suppose $\eta^{*1}_x\succ\alpha$.   Then $\pi^{\alpha\sig}(x, x-\evec_1) < U_x$, which implies  $\eta^{*2}_x\succcurlyeq\alpha$. Thus $\eta^{*2}_x\succcurlyeq\eta^{*1}_x$. 
The definitions \eqref{eta*5} imply the properties   in \eqref{Q-eta3}.  Thus we have found at least one pair  $\eta^{*1}_x\preccurlyeq\eta^{*2}_x$ satisfying   \eqref{Q-eta3}.   




\smallskip 

{\it Uniqueness.}  
Suppose two pairs, $\eta^{*1}_x\preccurlyeq\eta^{*2}_x$ and  $\zeta^{*1}_x\preccurlyeq\zeta^{*2}_x$,   satisfy   \eqref{Q-eta3}.  We show that $\zeta^{*1}_x\ne\eta^{*1}_x$ leads to a contradiction.  We can assume  $\zeta^{*1}_x\prec\eta^{*1}_x$.  Pick $\alpha$ so that $\zeta^{*1}_x\prec\alpha\prec\eta^{*1}_x$.  Then \eqref{Q-eta3a} applied to $\eta^{*1}_x$ implies 
$\pi^{\alpha\sig}(x, x-\evec_1) < U_x$, while \eqref{Q-eta3b} applied to $\zeta^{*1}_x$ implies $\pi^{\alpha\sig}(x, x-\evec_1) \ge U_x$.
  A similar argument establishes the uniqueness of $\eta^{*2}_x$. 

\smallskip 

{\it Properties.}  The extreme inequalities of  \eqref{Q-eta6} follow from \eqref{pi_mono} since $\eta^{*1}_x\preccurlyeq\eta^{*2}_x$.  The inner inequalities of  \eqref{Q-eta6} follow from letting 
  $\zeta \nearrow\eta^{*2}_x  $ and $\eta\searrow\eta^{*1}_x$ in the definitions in \eqref{eta*5}, because $\xi\mapsto\pi^{\xi-}$ is continuous from the left and $\xi\mapsto\pi^{\xi+}$  from the right.  
  
  Suppose $[\alpha, \beta]$ is a maximal linear segment of $\gpp$ and $U_x=\pi^{\xi\sig}(x, x-\evec_1)$ for some $\xi\in\,]\alpha,\beta[\,$. 
  Then for each $\zeta\prec\alpha$, by the strict inequality of Theorem~\ref{thm:78-63}, we have $\pi^{\zeta\sig}(x, x-\evec_1) < U_x = \pi^{\alpha+}(x, x-\evec_1)$. Hence $\eta^{*1}_x=\alpha$ by definition \eqref{eta*5}.  Similarly $\eta^{*2}_x=\beta$. 
  
  Conversely, suppose $\eta^{*1}_x\prec\eta^{*2}_x$.  
  This implies  
$ \pi^{\eta^{*1}_x+}(x, x-\evec_1)\le    \pi^{\eta^{*2}_x-}(x, x-\evec_1)$ because of \eqref{pi_mono}. 
  Then the  middle inequalities of \eqref{Q-eta6} force  
  $\pi^{\eta^{*1}_x+}(x, x-\evec_1)  =  U_x =  \pi^{\eta^{*2}_x-}(x, x-\evec_1)$.
  Again by the strict inequality of Theorem~\ref{thm:78-63},  $[\eta^{*1}_x, \eta^{*2}_x]$ must be a linear segment for $\gpp$.    
  Moreover, it must be a maximal linear segment because Busemann functions are constant on linear segments by \eqref{B==}, yet $\eta^{*1}_x$, $\eta^{*2}_x$ were chosen in \eqref{eta*5} to be extremal.
  \end{proof}

\begin{proof}[Proof of Theorem~\ref{thm:eta10}]
First we argue that $\Qwalks^\w\{\eta^{*1}_x =\eta^{*2}_x\}=1$ so that we can define 
\eeq{ \label{eta==}
\eta^*(x)=\eta^{*1}_x =\eta^{*2}_x \quad \text{$\Qwalks^\w$-almost surely}.
}
By Proposition~\ref{pr:eta6}, we need to rule out the possibility that $U_x = \pi^{\xi\sig}(x,x-\evec_1)$ for some $\xi$ in an open linear segment  
of the shape function  $\gpp$.
Indeed, there are at most countably many such segments and, by \eqref{B==}, $(\xi, \sigg)\mapsto \pi^{\xi\sig}(x,x-\evec_1)$ is constant on each segment.
So $U_x$ needs to avoid only countably many values (depending on $\w$), which occurs $\Qwalks^\w$-almost surely.

Given $\w$, for  each $x$  the variable $\eta^*(x)$ is a function of $U_x$, a fact which is immediate from \eqref{eta*5} and \eqref{eta==}.
Hence the random variables $(\eta^*(x))_{x\in\Z^2}$ are independent under $\Qwalks^\w$.
To obtain the marginal distribution claimed in \eqref{Q-eta14}, we establish inequalities in both directions.
Utilize \eqref{Q-eta3b} and the right-hand side of  \eqref{Q-eta6} to write 
\[   
\Qwalks^\w\{ U_x <  \pi^{\eta-}(x,x-\evec_1)\}  \le  \Qwalks^\w\{  \eta^*(x)\preccurlyeq \eta\}  \le \Qwalks^\w\{ U_x \le   \pi^{\eta+}(x,x-\evec_1)\}.  
\]  
Since $U_x$ is uniform on $(0,1)$, this says
\[   \pi^{\eta-}(x,x-\evec_1) \le  \Qwalks^\w\{  \eta^*(x)\preccurlyeq \eta\} \le  \pi^{\eta+}(x,x-\evec_1)  .  \]  
The second inequality gives one direction of \eqref{Q-eta14}.
To obtain the other direction, we employ the first inequality:
\eq{ 
\pi^{\eta+}(x,x-\evec_1)
=\lim_{\zeta\searrow\eta} \pi^{\zeta-}(x,x-\evec_1)
\leq \lim_{\zeta\searrow\eta} \Qwalks^\w\{  \eta^*(x)\preccurlyeq \zeta\}
= \Qwalks^\w\{  \eta^*(x)\preccurlyeq \eta\}.  
}
The marginal distribution claimed in part~\ref{eta10a} has been verified.

The final observation we need is that 
\be\label{Q670} 
\Qwalks^\w\bigl\{  U_x \ne \pi^{\xi\sig}(x,x-\evec_1) \; \forall\tsp \xi\in\aUset, \, \sigg\in\{-,+\} \bigr\}=1,    
\ee
which is true because $\aUset$ is at most countable and fixed by $\w$.
In light of \eqref{eta==} and \eqref{Q670}, we infer from \eqref{Q-eta6} that
$\Qwalks^\w$-almost surely 
one of these two cases happens at every $x$:
\begin{subequations} \label{Q679} 
\begin{align}
\label{Q679.1}
&\eta^*(x)\notin\aUset \ \text{ and } \ U_x=\pi^{\eta^*(x)\sig}(x,x-\evec_1)  \ \text{ for } \   \sigg\in\{-,+\}; \\[2pt]
\label{Q679.2}      \text{ or }  \ 
&\eta^*(x)\in\aUset \ \text{ and } \  
\pi^{\eta^*(x)-}(x, x-\evec_1) < U_x <  \pi^{\eta^*(x)+}(x, x-\evec_1) .
\end{align} 
\end{subequations}
The claims~\ref{eta10bi}--\ref{eta10biii} 
follow readily from the dichotomy \eqref{Q679}  and definition \eqref{X129}. 
 \end{proof}  
 
\begin{proof}[Proof of Theorem~\ref{thm:eta20}]  
Part~\ref{thm:eta20b} follows from the fact that under $\Qwalks^\w$  the variables $(\eta^*(x))_{x\in\Z^2}$ are independent and, by \eqref{Q-eta14} and Theorem~\ref{thm:51-32},  each $\eta^*(x)$  has the same set $\aUset$  of atoms.  

\smallskip 

Part~\ref{thm:eta20bb}.
Reinterpret \eqref{pure_jump} in terms of the transition probabilities \eqref{pi36}:
  \eeq{ \label{pi55}
  \pi^{\eta+}(x,x-\evec_1) - \pi^{\zeta-}(x,x-\evec_1) = \sum_{\xi\in\aUset\cap[\zeta,\eta]} [\pi^{\xi+}(x,x-\evec_1)-\pi^{\xi-}(x,x-\evec_1)].
  }
Let $\eta\nearrow\evec_1$, $\zeta\searrow\evec_2$ to obtain
\eeq{ \label{pi56}
1 \stackref{pi-deg}{=} 
\pi^{\evec_1}(x,x-\evec_1) - \pi^{\evec_2}(x,x-\evec_1)
\stackref{pi55}{=} \sum_{\xi\in\aUset} [ \pi^{\xi+}(x,x-\evec_1) - \pi^{\xi-}(x,x-\evec_1)].
}
It follows that $\Qwalks^\w$-almost surely  the uniform variable $U_x$ satisfies $\pi^{\xi-}(x,x-\evec_1) < U_x < \pi^{\xi+}(x,x-\evec_1)$ for some $\xi\in\aUset$.   By the dichotomy \eqref{Q679}, we conclude $\eta^*(x)\in\aUset$.

\smallskip 

 Part~\ref{thm:eta20d}.    We claim there is an event $\Omega_0\subset\Omega$ of full $\P$-probability such that for all $\w\in\Omega_0$,
 \be\label{X450} 
\Qwalks^\w\Big\{\lim_{m\to-\infty}   \frac{Z_{X^{v, \xi\sig, \tvec}_m,\tspb x}}{Z_{X^{v, \xi\sig, \tvec}_m,\tspb x-\evec_1}} =  e^{\Bus^{\xi\sig}_{x-\evec_1,x}} 
\ \forall\ \xi\in\,]\evec_2,\evec_1[\,, \ \sigg\in\{-,+\}, \ \tvec\in\{\evec_1, \evec_2\},\ v,x\in\Z^2 \Big\}=1.
\ee
Indeed, by \cite[Rmk.~5.9]{janj-rass-20-aop},   under assumption \eqref{shape_ass}   there exists $\Omega_0\subset\Omega$ of full $\P$-probability   such that  
for each  $\w\in\Omega_0$,   
$\xi\in\;]\evec_2,\evec_1[\,$,  $\sigg\in\{-,+\}$,  and $v\in\Z^2$,  the path measure $Q^{\xi\sig}_v$ from \eqref{pi36} is extreme among the semi-infinite  Gibbs measures rooted at $v$.   
By \cite[Thm.~3.10(d) and Thm.~5.7]{janj-rass-20-aop}, this extremality implies that for all $x<v$, 
\eq{
Q^{\xi\sig}_v\Bigl\{\text{$X_\bbullet$ is $\cL_\xi$-directed and  }\lim_{m\to-\infty}   \frac{Z_{X_m,\tspb x}}{Z_{X_m,\tspb x-\evec_1}} =  e^{\Bus^{\xi\sig}_{x-\evec_1,x}}  \Bigr\} = 1.
}
Since $X^{v,\xi\sigg,\tiebr}_\bbullet$ has distribution $Q^{\xi\sig}_v$ under $\Qwalks^\w$, 
it follows that for   $\tvec \in\{\evec_1, \evec_2\}$ and $\w\in\Omega_0$,
\eeq{ \label{jra530}
\Qwalks^\w\Bigl\{ \text{$X^{v,\xi\sig,\tvec}_\bbullet$ is $\cL_\xi$-directed and  } \lim_{m\to-\infty}   \frac{Z_{X^{v, \xi\sig, \tvec}_n,\tspb x}}{Z_{X^{v, \xi\sig, \tvec}_m,\tspb x-\evec_1}} =  e^{\Bus^{\xi\sig}_{x-\evec_1,x}}  \Bigr\} = 1.
}
This does not immediately imply \eqref{X450} since the event on the left-hand side of \eqref{jra530} is $\xi$-dependent,
but we will extend it as follows.

Let $\cA^\w$ be a countable dense subset of $\,]\evec_2, \evec_1[\,$ that contains the discontinuity set $\aUset$.   
For $\w\in\Omega_0$, the following occurs with full $\Qwalks^\w$-probability by \eqref{jra530}:
\eeq{ \label{X451}
\lim_{m\to-\infty}   \frac{Z_{X^{v, \xi\sig, \tvec}_m,\tspb x}}{Z_{X^{v, \xi\sig, \tvec}_m,\tspb x-\evec_1}} =  e^{\Bus^{\xi\sig}_{x-\evec_1,x}}
 \quad \text{for all $\xi \in \cA^\w$, $\sigg\in\{-,+\}$, $\tiebr\in\{\evec_1,\evec_2\}$, $v,x\in\Z^2$},
}
and also
\eeq{ \label{X452}
\text{$X^{v,\xi\sig,\tvec}_\bbullet$ is $\cL_\xi$-directed} \quad \text{for all $\xi \in \cA^\w$, $\sigg\in\{-,+\}$, $\tiebr\in\{\evec_1,\evec_2\}$, $v\in\Z^2$}.
}
Consider any $\xi\notin\cA^\w$.  
We necessarily have $\xi\notin\aUset$, so $\Bus^{\xi-}=\Bus^{\xi+}$. 
Pick $\zeta, \eta\in\cA^\w$ so that $\zeta\prec\xi\prec\eta$.
By the monotonicity \eqref{pi_mono} and the decision rule \eqref{X129},  we have
\eeq{ \label{hx63}
X^{v, \eta-, \tvec}_m 
\preccurlyeq X^{v, \xi, \tvec}_m    
\preccurlyeq  X^{v, \zeta+, \tvec}_m. 
} 
This ordering and standard monotonicity of partition function ratios (e.g.~\cite[Lem.~A.2]{busa-sepp-22-ejp}) give 
\eeq{  \label{sabn4}
 \frac{Z_{X^{v, \eta-, \tvec}_m,\tspb x}}{Z_{X^{v, \eta-, \tvec}_m,\tspb x-\evec_1}}
 \le  \frac{Z_{X^{v, \xi, \tvec}_m,\tspb x}}{Z_{X^{v, \xi, \tvec}_m,\tspb x-\evec_1}}
 \le  \frac{Z_{X^{v, \zeta+, \tvec}_m,\tspb x}}{Z_{X^{v, \zeta+, \tvec}_m,\tspb x-\evec_1}} \ \text{whenever $X^{v, \eta-, \tvec}_m, X^{v, \xi, \tvec}_m, X^{v, \zeta+, \tvec}_m \leq x-\evec_1$}.
 }
Since \eqref{X451} applies to the leftmost and rightmost ratios above, the subsequential limits of the middle ratio are caught between $e^{\Bus^{\eta-}_{x-\evec_1,\tspb x}}$ and $e^{\Bus^{\zeta+}_{x-\evec_1,\tspb x}}$.
As we let $\zeta\nearrow\xi$ and $\eta\searrow\xi$,  these exponentials converge to  $e^{\Bus^{\xi+}_{x-\evec_1,\tspb x}}=e^{\Bus^{\xi-}_{x-\evec_1,\tspb x}}$ thanks to \eqref{B_lim1}.
We have thus argued that \eqref{X451} is sufficient to establish the claim \eqref{X450}.
It should be noted that our use of \eqref{sabn4} is permitted because \eqref{X452} implies $X^{v, \eta-, \tvec}_\bbullet$ and $X^{v, \zeta+, \tvec}_\bbullet$ are $\cL_\eta$-directed and $\cL_\zeta$-directed, respectively.
By the curvature result \cite[Lem.~B.1]{janj-rass-20-aop}, the closed intervals $\cL_\eta$ and $\cL_\zeta$ do not contain $\evec_1$ or $\evec_2$, so $X^{v, \eta-, \tvec}_m, X^{v, \zeta+, \tvec}_m \leq x-\evec_1$ for all sufficiently negative $m$. 
The ordering \eqref{hx63} then forces $X^{v,\xi,\tvec}_m \leq x-\evec_1$ as well.


To complete the proof of part~\ref{thm:eta20d},  observe that if $X^{v, \zeta\sig,\tvec}_m=X^{v, \eta\sig',\tvec'}_m$ for  infinitely many $m$, then along this subsequence the limits in   \eqref{X450} give $\Bus^{\zeta\sig}_{x-\evec_1,x} = \Bus^{\eta\sig'}_{x-\evec_1,x}$ for all $x$.
Under the assumptions on the pair $(\zeta\sigg, \eta\sigg')$, this violates either Theorem~\ref{thm:78-63} or \ref{thm:51-32}.

\smallskip 

 Part~\ref{thm:eta20a}.   By part~\ref{thm:eta20d}, for each $\xi\in\aUset$,  from any initial vertex the $\xi\pm$ walks separate. By Theorem~\ref{thm:eta10}\ref{eta10bi} and \ref{eta10biii}, this can happen only if $\eta^*(x)=\xi$ for infinitely many $x$. 
 
\smallskip 

 Part~\ref{thm:eta20c} follows as part~\ref{thm:eta20a}.   By part~\ref{thm:eta20d},   for any open interval $\,]\zeta,\eta[\,$ disjoint from closed linear segments, the walks $X^{v, \zeta\sig,\tvec}_\bbullet$ and  $X^{v, \eta\sigg',\tvec}_\bbullet$ eventually separate. 
 By Theorem~\ref{thm:eta10}\ref{eta10bi},  this can happen only if $\eta^*(x)\in[\zeta,\eta]$ for some $x$.
 \end{proof}  



   \section{Main results under inverse-gamma weights}  \label{sec:main2}
 
 \subsection{Inverse-gamma basics}  
The gamma function is $\Gamma(s)=\int_0^\infty x^{s-1}e^{-x}\,\dd x$.   The digamma   and   the trigamma functions are, respectively,      $\psi_0(s)=\Gamma'(s)/\Gamma(s)$ and $\psi_1(s)=\psi_0'(s)$.   A positive  random variable $X$ has the gamma distribution with parameter $\alpha\in\R_{>0}$, abbreviated $X\sim$ Ga$(\alpha)$, if $X$ has density function $f_X(x)=\frac1{\Gamma(\alpha)}x^{\alpha-1}e^{-x}$ for $x>0$. 
$Y$ has the inverse-gamma distribution with parameter $\alpha$, $Y\sim$ Ga$^{-1}(\alpha)$, if its reciprocal satisfies  $Y^{-1}\sim$ Ga$(\alpha)$. Then $Y$ has density function
 $f_Y(x)=\frac1{\Gamma(\alpha)} x^{-1-\alpha} e^{-x^{-1}}$ for $x>0$ and satisfies the identities 
 $\E[\log Y]=-\psi_0(\alpha)$ and  $\Vvv[\log Y]=\psi_1(\alpha)$.   
 A variable $Z\sim$ Beta$(\alpha, \lambda)$ has density $f_Z(x)= \frac1{\BBet(\alpha, \lambda)} x^{\alpha-1}(1-x)^{\lambda-1}$ for $0<x<1$.

Fix $\alpha>0$ and assume that 
	\be\label{m:exp}\begin{aligned} 
		&\text{the weights $\Yw=(\Yw_x)_{x\tsp\in\tsp\Z^2}$ are i.i.d.\ random variables}\\[-3pt]
		 &\text{with marginal distribution  $\Yw_x\sim$
		Ga$^{-1}(\alpha)$. }  
	\end{aligned}\ee

The 
shape function $\gpp$ is  described as follows (see \cite[eq.~(2.15) and (2.16)]{sepp-12-aop-corr}).     On the axes, $\gpp(s\evec_r)=-s\psi_0(\alpha)$ for $s\ge0$. In the interior, for each $\xi=(\xi_1,\xi_2)\in\R_{>0}^2$ there is a unique real $\rho_{\xi}\in (0,\alpha)$ such that 
	\be\label{gpp} \begin{aligned}
	 \gpp(\xi) &= \inf_{\rho\tspa\in\tspa(0,\alpha)}  \{-\xi_1\tsp\psi_0(\alpha-\rho)-  \xi_2\tspa\psi_0(\rho)\} 
	 = -\xi_1\tsp\psi_0(\alpha-\rho_\xi)-  \xi_2\tsp\psi_0(\rho_\xi)  .  
	\end{aligned} \ee 
	 The minimizer $\rho_\xi$ in \eqref{gpp} is the solution of the equation 
		\be\label{rho_x}\begin{aligned} 
	&\frac{\psi_1(\alpha-\rho_\xi)}{\psi_1(\rho_\xi)}=\frac{\xi_2}{\xi_1} \ 
	\  \iff\ \ \  \xi_1\tsp\psi_1(\alpha-\rho_\xi)-  \xi_2\tsp\psi_1(\rho_\xi) =0. 
	\end{aligned} \ee
The shape function $\gpp$  is continuous on $\R_{\ge0}^2$, and  differentiable and  strictly concave throughout $\R_{>0}^2$. In particular, assumption \eqref{shape_ass} is satisfied. 
	
The correspondence \eqref{rho_x}  gives the following bijective  mapping  
	between direction vectors  $\chdir=(\chdirc_1,\chdirc_2)=(\chdirc_1,1-\chdirc_1)\in[\evec_2,\evec_1]$   and parameters  $\rho\in [0,\alpha]$:  
	\be\label{u-rho}   
	 \chdir=
	\chdir(\rho)=\biggl(\frac{\psi_1(\rho)}{\psi_1(\alpha-\rho)+\psi_1(\rho)}   \,,\, 
	\frac{\psi_1(\alpha-\rho)}{\psi_1(\alpha-\rho)+\psi_1(\rho)}
	\biggr)
	\ \Longleftrightarrow \ \rho=\rho_{\chdir}=\rho({\chdir}). 
	\ee
	 The function $\psi_1$ is strictly positive and strictly decreasing on $\R_{>0}$,  with limits $\psi_1(0+)=\infty$ and $\psi_1(\infty)=0$.  Thus the bijection $\chdir\mapsto \rho(\chdir)$ from  $[\evec_2,\evec_1]$ onto $[0,\alpha]$ is strictly decreasing in the southeast ordering $\prec$ on $[\evec_2,\evec_1]$. 
In particular,  the boundary values are $\rho(\evec_1)=0$ and $\rho(\evec_2)=\alpha$.
	 
\subsection{Global Busemann process} \label{ig1_sec}  

As observed in Section~\ref{distB_sec}, the entire Busemann process can be characterized by the joint distribution of horizontal nearest-neighbor increments on a single lattice level. 
We give here a quick preliminary description of this distribution.  Full details rely on the development of Section~\ref{mar_sec} and are presented in Section~\ref{sec:Vig}.   

We introduce  notation for products of inverse-gamma distributions.  Let  $\lambda_{\parng{1}{N}}=(\lambda_1,\dotsc,\lambda_N)\in\R_{>0}^N$  be an $N$-tuple of positive reals.   Let    $Y^{\parng{1}{N}}=(Y^1, \dotsc,Y^N)\in(\R_{>0}^\Z)^N$    denote  an $N$-tuple  of positive bi-infinite random sequences  $Y^i=(Y^i_k)_{k\in\Z}$.   Then define the probability measure $\nu^{\lambda_{\parng{1}{N}}}$ on $(\R_{>0}^\Z)^N$ as follows: 
	\be\label{nu5} \begin{aligned} 
	&\text{$Y^{\parng{1}{N}}$  has distribution $\nu^{\lambda_{\parng{1}{N}}}$ if all the coordinates $(Y^i_k)^{i\in\lzb1,N\rzb}_{k\in\Z}$ are mutually }\\[-3pt]  
	&\text{independent with marginal distributions $Y^i_k\sim$ Ga$^{-1}(\lambda_i)$.} 
	\end{aligned} \ee
	 To paraphrase  \eqref{nu5}, under $\nu^{\lambda_{\parng{1}{N}}}$ each $Y^i$ is a sequence of i.i.d.\ inverse-gamma variables with parameter $\lambda_i$ and the sequences $Y^1, \dotsc,Y^N$  are mutually independent.

Denote the sequence of level-$t$ weights by $\Yw(t)=(\Yw_{(k,t)})_{k\in\Z}$.   Recall the notation \eqref{IB5} for sequences of exponentiated horizontal 
Busemann increments: 
  $ {\brvI}^{\chdir\sig}_k(t) = 
(e^{\Bus^{\chdir\sig}_{(k-1,t),\tsp (k,t)}})_{k\in\Z}$. 
 Fix directions  $\chdir_1\succ\dotsm\succ\chdir_N$  in $\,]\evec_2,\evec_1[\,$ and signs $\sigg_1,\dotsc,\sigg_N\in\{-,+\}$. 
 There exists a sequence space $\cI_{N+1}^\uparrow\subset(\R_{>0}^\Z)^{N+1}$ that supports the product measure $\nu^{(\alpha, \tspa\alpha-\rho(\chdir_1),\dotsc, \tspa\alpha-\rho(\chdir_N))}$  and a Borel mapping $\Daop^{(N+1)}\colon \cI_{N+1}^\uparrow\to\cI_{N+1}^\uparrow$ such that the following theorem holds.

	\begin{theorem}\label{B1_thm} Assume \eqref{m:exp}. 
	 At  each level  $t\in\Z$,  the joint law $\mu^{(\alpha, \tspa \alpha-\rho(\chdir_1),\dotsc, \tspa\alpha-\rho(\chdir_N))}$ of the $(N+1)$-tuple of sequences   $(\Yw(t),\,{\brvI}^{\chdir_1\sig_1}(t), \dotsc, {\brvI}^{\chdir_N\sig_N}(t))$ satisfies   
	 \eq{ 
	 \mu^{(\alpha, \tspa\alpha-\rho(\chdir_1),\dotsc, \tspa\alpha-\rho(\chdir_N))}=\nu^{(\alpha, \tspa\alpha-\rho(\chdir_1),\dotsc, \tspa\alpha-\rho(\chdir_N))}\circ(\Daop^{(N+1)})^{-1}.
	 } 
	\end{theorem}
	
The theorem states that on a single horizontal level the joint distribution of the original weights and 	
  the Busemann functions  is  a deterministic  pushforward of the distribution of independent inverse-gamma variables \textit{with the same marginal distributions}.     Since $\gpp$ is differentiable, the  signs $\sigg_1,\dotsc,\sigg_N\in\{-,+\}$ are irrelevant  (recall Remark~\ref{pm_rmk}) and included only for completeness.  For this reason  the parametrization of the  measures ignores  the signs. 
  
The   space $\cI_{N+1}^\uparrow$   and the mapping $\Daop^{(N+1)}$ are defined in equations \eqref{in_up_def} and \eqref{Daop_def}.  The precise version of Theorem~\ref{B1_thm} is proved as Theorem~\ref{B2_thm} in Section~\ref{sec:twin_ig}. 
  The mapping $\Daop^{(N+1)}$ preserves the distributions of individual sequence-valued components: 
  \begin{subequations}
  \label{B5x9}
  \be\label{B569} 
  {\brvI}^{\chdir\sig}(t) =  
(e^{\Bus^{\chdir\sig}_{(k-1,t),\tsp (k,t)}})_{k\in\Z} 
\text{  is i.i.d.\ Ga$^{-1}(\alpha-\rho(\chdir))$ distributed. }
\ee
If instead of horizontal increments on a horizontal line, we considered vertical increments on a vertical line, the statement would be this: 
 \be\label{B579} 
  ({\brvJ}^{\chdir\sig}_k(t))_{t\in\Z}  =  
(e^{\Bus^{\chdir\sig}_{(k,t-1),\tsp (k,t)}})_{t\in\Z} 
\text{  is i.i.d.\ Ga$^{-1}(\rho(\chdir))$ distributed. }
\ee
\end{subequations}
These marginal properties \eqref{B5x9}  of the Busemann functions  were derived earlier  in \cite{geor-rass-sepp-yilm-15}.      They follow from Lemma~\ref{Dop-lm4} in Section~\ref{sec:twin_ig}. 

\subsection{Busemann process across an edge} 
We fix a horizontal edge $(x-\evec_1,x)$ and 
describe the Busemann process  $\{\Bus^{\chdir\sig}_{x-\evec_1,x}\}_{\chdir\tsp\in\tspb]\evec_2,\evec_1]}$.  
	 To have a  process indexed by reals, we switch from $\xi$  to the parameter $\rho=\rho(\chdir)\in[0,\alpha)$.  
	 Then   $(\Bus^{\chdir(\rho)-}_{x-\evec_1,x})_{\rho\tsp\in\tsp[0,\alpha)}$  is  an increasing  cadlag  process which has been extended  to the parameter value $\rho=0=\rho(\evec_1)$ by setting ${\Bus^{\evec_1}_{x-\evec_1,\tspb x}}={\Bus^{\evec_1-}_{x-\evec_1,\tspb x}}=\log\Yw_x$.  This process is continuous at  $\rho=0$ by \eqref{B_lim}.  The minus superscript in $\Bus^{\chdir(\rho)-}_{x-\evec_1,x}$ is just for the path regularity. In statements about finite-dimensional distributions we drop it.  
	
Let  $\Prm$ be  the  inhomogeneous  Poisson point process on  $(0,\alpha)\times\R_{>0}$ with intensity measure 
	\eeq{ \label{PPPintensity}
		\widebar\sigma(\dd s, \dd y)=\sigma(s,y)\ \dd s\, \dd y,\qquad \text{where} \qquad
		\sigma(s,y)=
		\frac{e^{-y(\alpha-s)}}{1-e^{-y}}, \quad (s,y)\in(0,\alpha)\times\R_{>0}.
	}
	The Laplace functional of $\Prm$ is given by 
	\begin{equation}\label{LaplaceF}
		\E\big[e^{-\sum_{(s,y)\in \Prm}F(s,y)}\big]=\exp{\Big\{ -\int_0^{\alpha} \dd s \int_0^{\infty} \! \dd y \, (1-\mathrm{e}^{-F(s,y)})\,\sigma(s,y)\Big\}}
	\end{equation}
	for nonnegative Borel functions $F\colon (0,\alpha)\times\R_{>0}\to\R_{\ge0}$.

Define the nondecreasing cadlag process  $(\Zpr(\rho))_{\rho\tspa\in\tspa  [0,\alpha)}$  so that the initial value $\Zpr(0) \sim \log$ Ga$^{-1}$($\alpha$) is   independent of $\Prm$  and
	\begin{equation}\label{Def:X}
		\Zpr(\rho)=\Zpr(0)+\sum_{(s,y)\,\in\,\Prm} \!\! \ind_{(0,\rho]}(s)\cdot y\quad \text{for }\rho\in (0,\alpha).   
	\end{equation}
The sum in \eqref{Def:X} is almost surely finite since $\E[Z(\rho)-Z(0)] = \int_0^\rho\int_0^\infty y\sigma(s,y)\,\dd y\,\dd s<\infty$ for $\rho\in(0,\alpha)$.
See Figure~\ref{f:trajec} for an example sample path.

	\begin{theorem} \label{B-th5}   Assume \eqref{m:exp}.  For each $x\in\Z^2$, the nondecreasing  cadlag processes $(\Bus^{\chdir(\rho)-}_{x-\evec_1,x})_{\rho\tspa\in\tspa  [0,\alpha)}$ and  $(\Zpr(\rho))_{\rho\tspa\in\tspa  [0,\alpha)}$ 
		are equal in distribution.  
	\end{theorem}

	Theorem~\ref{B-th5} is proved in Section~\ref{ss_inv_ga} by establishing  that 
	$\Bus^{\chdir(\rcbullet)-}_{x-\evec_1,x}$ has independent increments as does $\Zpr$,  and by showing that their increments have  identical distributions.  
	Independent increments means that   for $0=\rho_0<\rho_1<\dotsm<\rho_n<\alpha$, the random variables $\log\Yw_x=\Bus^{\chdir(\rho_0)}_{x-\evec_1,x}\,,  \, \Bus^{\chdir(\rho_1)}_{x-\evec_1,x}-\Bus^{\chdir(\rho_0)}_{x-\evec_1,x}\,, \dotsc, \Bus^{\chdir(\rho_n)}_{x-\evec_1,x}-\Bus^{\chdir(\rho_{n-1})}_{x-\evec_1,x}$ are independent.   	
	 From the proof we will see that
	 for $\alpha> \rho>\lambda\geq 0$, 
	 \eeq{ 
	 	 e^{- (\Bus^{\chdir(\rho)}_{x-\evec_1,x}-\Bus^{\chdir(\lambda)}_{x-\evec_1,x})}\sim{\rm Beta}(\alpha-\rho,\,\rho-\lambda),
    }
which is consistent with the expectation following from \eqref{B569}: 
\eq{ 
\E\bigl[  \Bus^{\zeta}_{x-\evec_1,x}- \Bus^{\eta}_{x-\evec_1,x} \bigr]   = \psi_0(\alpha-\rho(\eta))  -  \psi_0(\alpha-\rho(\zeta))  > 0
 \qquad \text{for  }   \evec_2\prec\zeta\prec\eta\prec\evec_1.  
}
 	
	
\begin{figure}
\includegraphics[scale=0.2]{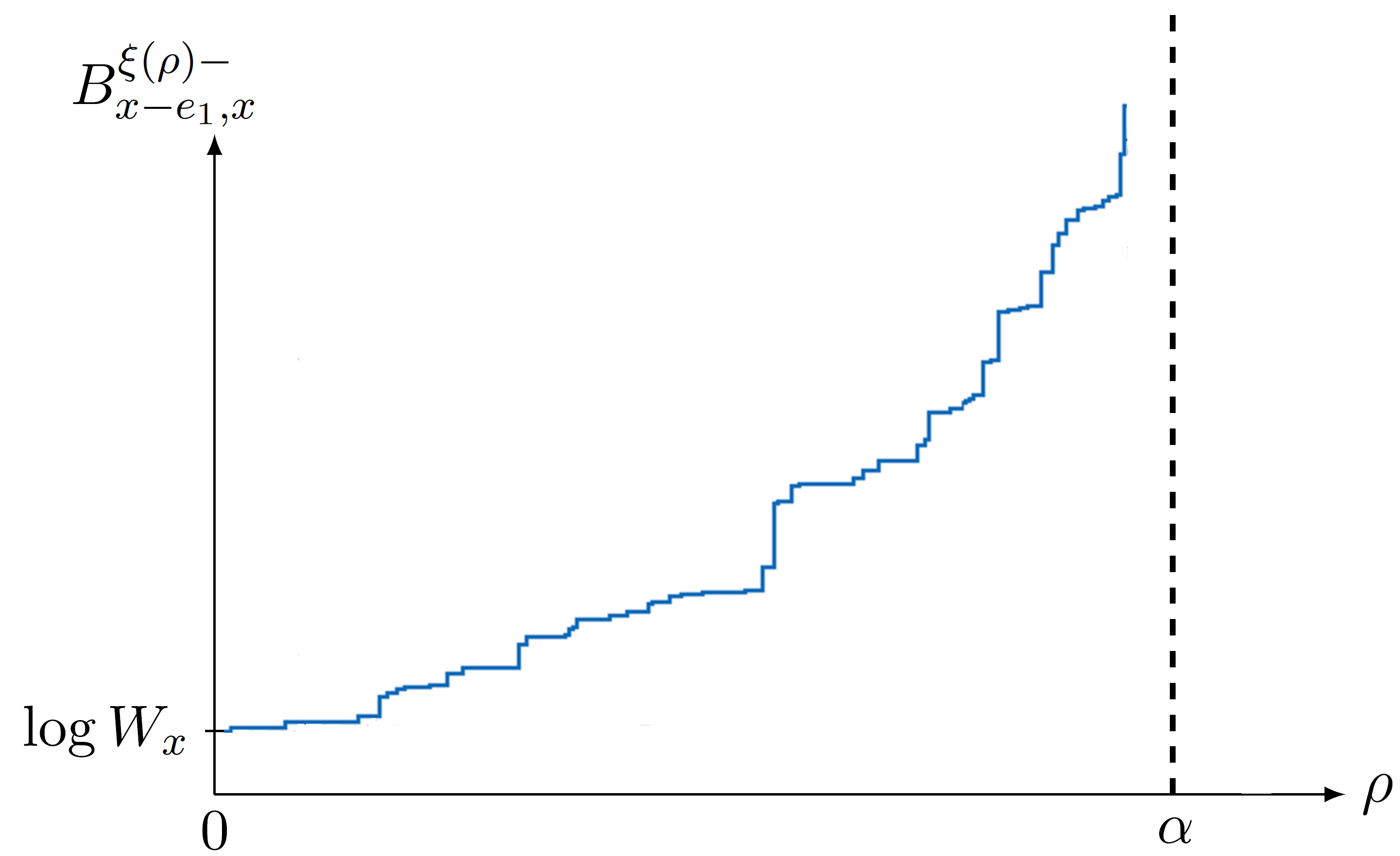}
	\caption{\small A simulated trajectory of the pure jump process $\{\Bus^{\xi(\rho)-}_{x-\evec_1,x}\}_{\rho\,\in\, [0,\alpha)}$, with $\alpha=20$. The initial value is $\log W_x \sim \log$ Ga$^{-1}$($\alpha$) and the jumps are determined by an independent Poisson point process on  $(0,\alpha)\times\R_{>0}$  with	intensity measure  $\frac{e^{-y(\alpha-s)}}{1-e^{-y}}\,\dd s\, \dd y$, according to \eqref{Def:X}.  There are infinitely many jumps on any open interval in $(0,\alpha)$. The process tends to infinity almost surely, as $\rho\nearrow \alpha$.
 }
 	\label{f:trajec} 
\end{figure}

We state a corollary about the jumps of the inverse-gamma Busemann process. 		 
Let $\cM_{\ge\delta}$ be the point process on $\,]\evec_2,\evec_1]$ of downward jumps of size $\ge\delta>0$ of the Busemann function $\xi\mapsto \Bus^{\xi+}_{x-\evec_1,x}$:
\eq{ 
\cM_{\ge\delta}  \bigl(  \,] \zeta, \eta]\bigr)   = \sum_{\xi\,\in \,] \zeta, \eta]} \one\{ \Bus^{\xi-}_{x-\evec_1,x}- \Bus^{\xi+}_{x-\evec_1,x}\, \ge\,\delta \}  \qquad \text{for  }   \evec_2\prec\zeta\prec\eta\preccurlyeq\evec_1. 
}
For distributional statements about $ \cM_{\ge\delta}$ the choice of $x$ is immaterial.  We observe below that large jumps accumulate only at $\evec_2$, while small jumps are dense everywhere. This is consistent with the continuity \eqref{B_lim}  of $\xi\mapsto \Bus^{\xi\sig}_{x-\evec_1,x}$ at the right endpoint $\xi=\evec_1$. 

\begin{corollary} \label{cor:ig_aUset}   Assume   \eqref{m:exp}.
\begin{enumerate}[label={\rm(\alph*)},ref={\rm(\alph*)}] \itemsep=2pt 
\item\label{cor:6.a}  For each $\delta>0$,  $\cM_{\ge\delta}$ is a Poisson process on $\,]\evec_2,\evec_1]$ with intensity measure  
\be\label{cM8} 
\E\bigl[  \tspb  \cM_{\ge\delta}  \bigl(  \,] \zeta, \eta]\bigr) \tspb \bigr]    =  \int_{\rho(\eta)}^{\rho(\zeta)} \dd s \int_\delta^{\infty} \! \dd y \,   \frac{e^{-y(\alpha-s)}}{1-e^{-y}} 
 \qquad \text{for  }   \evec_2\prec\zeta\prec\eta\preccurlyeq\evec_1. 
\ee
In particular, $ \cM_{\ge\delta} (  \,[ \zeta, \evec_1]\tspb)$ is a finite Poisson variable for each $\zeta\in\,]\evec_2,\evec_1[\,$ and so almost surely there is a last jump of size $\ge\delta$ before $\evec_1$. 
By contrast, with probability one, $ \cM_{\ge\delta} (  \,] \evec_2, \eta]\tspb)=\infty$ for each $\eta\in\,]\evec_2,\evec_1[\,$. 

\item\label{cor:6.b}  
With probability one,  the set $\aUset$ of jump directions is dense in $\,]\evec_2,\evec_1[\,$. 
\end{enumerate} 
\end{corollary}  

We prove the corollary at the end of this section after some further remarks.

\begin{remark}[Inverse-gamma polymer Gibbs measures]    \label{rmk:Q_ig}
Here we combine results from \cite{janj-rass-20-aop} with our results  to state facts about the polymer Gibbs measures of the inverse-gamma model. 
A semi-infinite polymer measure $Q_v$ rooted at $v\in\Z^2$ is said to be \textit{$\xi$-directed} if its sample paths have limiting direction $\xi$ with probability one.
That is, $Q_v(\text{$X_\bbullet$ is $\xi$-directed})=1$.

For each  $\xi\in\,]\evec_2, \evec_1[\,$ there is a $\xi$-dependent full-probability event on which $\Bus^{\xi+}=\Bus^{\xi-}$, 
and then there is a unique $\xi$-directed semi-infinite polymer measure rooted at each $x\in\Z^2$. 
This comes from combining  \cite[Thm.~3.7]{janj-rass-20-aop} with the strict concavity and differentiability of the inverse-gamma shape function. 
The unique measure is $Q^\xi_v$, the Gibbs measure supplied by the Busemann process via Theorem~\ref{pococthm}.

There exists a full-probability event $\Omega_0$ on which the following holds for each $x\in\Z^2$:  For each $\xi\tspa\in\,]\evec_2, \evec_1[\,\setminus\aUset$   there is a unique $\xi$-directed semi-infinite polymer measure rooted at $x$, as above.  
On the other hand, for each $\xi\in\aUset$ there are at least two $\xi$-directed semi-infinite polymer measures  rooted at $x$, namely those Gibbs measures $Q_v^{\xi-}$ and $Q_v^{\xi+}$ associated to the Busemann functions $\Bus^{\xi-}$ and $\Bus^{\xi+}$.
These statements come from  \cite[Thm.~3.10(e)--(f)]{janj-rass-20-aop} and the strict concavity  of the inverse-gamma shape function. 

   An important open problem is the number of extreme Gibbs measures at directions $\xi\in\aUset$, rooted at a particular $x\in\Z^2$.  
   This problem has only been solved in two cases, both of which are zero-temperature models: the exponential corner growth model and the directed landscape.
   The statement there  is that in directions of discontinuity of the Busemann process, there are \textit{exactly two}  semi-infinite geodesics from each initial vertex  \cite{coup-11, janj-rass-sepp-23, busani_seppalainen_sorensen24, busani??}. 
   Based on this, the natural conjecture is that,  rooted at each $x$,  there are exactly two  extreme semi-infinite polymer  measures in directions $\xi\in\aUset$. 
\qedex\end{remark}
 

\begin{proof}[Proof of Corollary~\ref{cor:ig_aUset}]  For both processes  $\Bus^{\chdir(\rbbullet)-}_{x-\evec_1,x}$ and  $\Zpr$,  on any compact interval $[0,\lambda]\subset[0,\alpha)$  the finite ordered sequence of  jumps of size $\ge\delta>0$ can be captured with measurable functions of the path.  Thus the processes of such jumps have the same distribution for both $\Bus^{\chdir(\rbbullet)-}_{x-\evec_1,x}$ and  $\Zpr$. For $\Zpr$ the Poisson description of these jumps is clear from \eqref{Def:X}.  Hence the same description works for $\Bus^{\chdir(\rbbullet)-}_{x-\evec_1,x}$. To get the first statement of part~\ref{cor:6.a}, map this Poisson process back to $\,]\evec_2,\evec_1]$ via the decreasing bijection $\rho\mapsto\xi(\rho)$ from \eqref{u-rho}. 
The remaining statements of part~\ref{cor:6.a} follow from the observation that for any $\rho\in(0,\alpha)$,
\eq{
\int_{0}^{\rho} \dd s \int_\delta^{\infty} \! \dd y \,   \frac{e^{-y(\alpha-s)}}{1-e^{-y}} < \infty \qquad \text{while} \qquad
\int_{\rho}^{\alpha} \dd s \int_\delta^{\infty} \! \dd y \,   \frac{e^{-y(\alpha-s)}}{1-e^{-y}} = \infty.
}
Part~\ref{cor:6.b} follows because the inner integral  in \eqref{cM8} diverges  to $+\infty$ as $\delta\searrow0$.
\end{proof}

\subsection{Competition interface under inverse-gamma weights} \label{sec:ig_xi}

In the inverse-gamma case we can answer the questions in Remark~\ref{rm:open4}\ref{open4.a}. 

\begin{theorem}\label{thm:ig-cif} 
 Assume i.i.d.\ inverse-gamma weights \eqref{m:exp}. 
  Then the following hold $\Qwalks^\w$-almost surely, for $\P$-almost every $\w$:  $\{\eta^*(x):\, x\in\Z^2\}=\aUset$ and for each $x\in\Z^2$,   $\xi^*(x)\in\aUset$.
\end{theorem}

 \begin{proof}    
  The  process $Z$ in \eqref{Def:X} is a monotone pure jump process. 
  Therefore, the same is true for $\xi\mapsto\Bus^{\xi+}_{x-\evec_1,x}$ by Theorem~\ref{B-th5}; that is, \eqref{pure_jump} holds.
  The equality $\{\eta^*(x):\, x\in\Z^2\}=\aUset$ now follows from Theorem~\ref{thm:eta20}, parts ~\ref{thm:eta20bb} and \ref{thm:eta20a}.
  The membership $\xi^*(x)\in\aUset$ follows from the observation 
  \[
\Qwalks^\w\{\xi^*(x)\in\aUset\} \overset{\eqref{xilaw5},\eqref{pi36}}{=}
 \sum_{\xi\in\aUset}[\pi^{\xi+}(x,x-\evec_1)-\pi^{\xi-}(x,x-\evec_1)]
 \stackref{pi56}{=} 1. \qedhere
\]
 \end{proof}

\subsection{Zero-temperature limit of the Busemann process}   \label{sec:0temp} 

The zero-temperature limit of the inverse-gamma polymer is the corner growth model (CGM) with exponential weights.  
We write $Y\sim\mathrm{Exp}(\rho)$ when $Y$ is exponentially distributed with rate parameter $\rho>0$, i.e.~$Y$ has density function $f_Y(x) = \rho e^{-\rho x}$ for $x>0$.
In order to stay within the exactly solvable family of inverse-gamma polymers,  we do not  add a separate temperature parameter to the model. Instead we view the parameter $\alpha\in\R_{>0}$ of the weight distribution in \eqref{m:exp} as the temperature and send it to zero.  To describe this, we include $\alpha$ explicitly in the notation and, when necessary,  use the superscript $0$ to identify objects that arise in the $\alpha\searrow0$ limit.    

Let $(\Yw^\alpha_x)_{x\tsp\in\tsp\Z^2}$ denote i.i.d.\ Ga$^{-1}(\alpha)$ weights and $(Y_x)_{x\tsp\in\tsp\Z^2}$  i.i.d.\ Exp$(1)$ weights.   Then as $\alpha\searrow0$, the weights converge in distribution: 
$(\alpha\log\Yw^\alpha_x)_{x\tsp\in\tsp\Z^2} \dto  (Y_x)_{x\tsp\in\tsp\Z^2}$. 
The normalized free energy of the inverse-gamma polymer thus converges to the last-passage value in the Exp$(1)$ CGM:
for $u\le v$ in $\Z^2$, 
\eq{ 
\alpha \log Z^\alpha_{u,v}  
\stackref{part56}{=}  \alpha \log \!\! \!\! \sum_{x_\brbullet\tspa\in\tspa\pathsp_{u,v}} \!\! \!\!   e^{\alpha^{-1} \sum_{i=m+1}^{n}  \alpha\log \Yw^\alpha_{x_i}} 
\ \underset{\alpha\searrow0}{\dtoo} \    \max _{x_\brbullet\tspa\in\tspa\pathsp_{u,v}}  \sum_{i=m+1}^{n}  Y_{x_i} 
\equiv L_{u,v}.
}
In this section we establish an analogous convergence for the Busemann processes.

The Busemann process of the inverse-gamma polymer with weights $\Yw^\alpha$ is now denoted by 
$B^{\alpha, \xi\sig}_{x,y}$ for $\alpha>0$.  
The Busemann process of the Exp(1) CGM is denoted by $\Bus^{0, \xi\sig}_{x,y}$.   
It has properties analogous to those collected in Theorems~\ref{buse_dir} and \ref{buse_full}.  
In particular, for each direction $\xi\in\,]\evec_2, \evec_1[\,$, we have the Busemann limit on a $\xi$-dependent event of full probability:
\[    \Bus^{0, \xi\sig}_{x,y} = \lim_{\ell\to-\infty} [ L_{x_\ell,y}-L_{x_\ell, x}] \quad \text{whenever $\lim_{\ell\to-\infty} x_\ell/\ell=\xi$}.
\]
Furthermore, for each $x\in\Z^2$ the map $\xi\mapsto \Bus^{0,\,\xi+}_{x-\evec_1,x}$ is nonincreasing and right-continuous on $\,]\evec_2,\evec_1]$, blowing up at $\evec_2$ and taking the value $Y_x$ at $\evec_1$.
We will soon give a process-level description in \eqref{hz3f21}.
Further descriptions of $\Bus^0$  appear  in Sections 2 and 3 of \cite{fan-sepp-20} (but note that Busemann variables are parametrized by their means in \cite{fan-sepp-20} rather than directions) and in Section 2 and Appendix A of \cite{janj-rass-sepp-23} (but note that semi-infinite geodesics go northeast in \cite{janj-rass-sepp-23} rather than southwest).

To state our convergence result, we view the process $(\Bus^{\alpha, \xi+}_{x-\evec_1,x}:\, \xi\in\,]\evec_2,\evec_1])$ in the space of real-valued cadlag paths, denoted by $D(\,]\evec_2, \evec_1], \R)$.
We can place a Polish topology on this space by adapting the standard Skorohod topology used for $D([0,\infty),\R)$ (see for example \cite[Sec.~3.5]{ethi-kurt}).   
Namely, a family $(X^\alpha)_{\alpha>0}$ converges as $\alpha\searrow0$ to $X^0$ in $D(\,]\evec_2, \evec_1], \R)$ if and only if there exist increasing Lipschitz bijections $u^\alpha\colon\,]\evec_2, \evec_1]\to\,]\evec_2, \evec_1]$ such that
\begin{subequations}
\label{jc3x1x}
\begin{align} \label{jc3x1xa}
&\lim_{\alpha\searrow0}\sup_{\evec_2\prec\xi\prec\eta\preccurlyeq\evec_1}\Big|\log \frac{|u^\alpha(\xi)-u^\alpha(\eta)|_1}{|\xi-\eta|_1}\Big| = 0, \quad \text{and} \\ \label{jc3x1xb}
&\lim_{\alpha\searrow0} \sup_{\xi\in[\zeta,\evec_1]}|X^\alpha(u^\alpha(\xi)) - X^0(\xi)| = 0 \quad \text{for all $\zeta\in\,]\evec_2, \evec_1]$}.
\end{align}
\end{subequations}
Under this topology, the inverse-gamma polymer Busemann processes converge weakly in the zero-temperature limit to those of Exp(1) CGM.

\begin{theorem}\label{thm:igexp2} 
As $\alpha\searrow0$,   the process 
$\{\alpha  \Bus^{\alpha, \,\xi+}_{x-\evec_1,x} :\, \xi\in\,]\evec_2, \evec_1]\,\}$ converges weakly to the process 
$\{\Bus^{0, \,\xi+}_{x-\evec_1,x} :\, \xi\in\,]\evec_2, \evec_1]\,\}$ in the space   $D(\,]\evec_2, \evec_1]\,, \R)$. 
\end{theorem} 

\begin{remark}[Weak convergence of vertical process]
The analogous theorem holds  for the Busemann process $\{\alpha  \Bus^{\alpha, \,\xi+}_{x-\evec_2,x} : \xi\in[\evec_2, \evec_1[\,\}$ on a vertical edge, with the difference that this process blows up at $\evec_1$.   As indicated by \eqref{u-rho} and \eqref{B5x9}, the reflection of the lattice that switches $\evec_1$ and $\evec_2$ corresponds to replacing the parameter $\rho(\xi)\in[0,\alpha]$  with $\alpha-\rho(\xi)$.   
\qedex\end{remark} 

The proof strategy for Theorem~\ref{thm:igexp2} is to exhibit a coupling of the processes $X^\alpha(\xi) = \alpha\Bus^{\alpha,\,\xi+}_{x-\evec_1,x}$ and $X^0(\xi) = \Bus^{0,\,\xi+}_{x-\evec_1,x}$ which admits increasing Lipschitz bijections $u^\alpha\colon\,]\evec_2, \evec_1]\to\,]\evec_2, \evec_1]$ satisfying \eqref{jc3x1x}.
In fact, our $u^\alpha$ will be deterministic; its role is to reparametrize the bijection $[0,\alpha]\ni\rho\mapsto\xi^\alpha(\rho)\in[\evec_2,\evec_1]$ from \eqref{u-rho} so that the domain does not depend on $\alpha$.
The reparametrization is necessary because our description of the Busemann process---namely, as a functional \eqref{Def:X} of a Poisson point process---passes through this bijection.

Meanwhile, the coupling will be achieved by progressively thinning a single Poisson point process.
The goal is to achieve the correct density \eqref{PPPintensity} for each $\alpha>0$, and then to verify that the fully thinned point process yields the correct Busemann function at zero temperature.
Comparison of  formulas \eqref{newZZ7} and  \eqref{newZZ13}  in the proof below shows  that 
a jump discontinuity in $\Bus^{\alpha,\,\bbullet}_{x-\evec_1, x}$ of size $y$ is  retained in
 $\Bus^{0,\,\bbullet}_{x-\evec_1, x}$ with probability $1-e^{-y}$.
 In particular, jumps of small magnitude $y$ are unlikely to be retained.  This accounts for the major qualitative difference between the positive- and zero-temperature Busemann processes: $\xi\mapsto \Bus^{\alpha, \,\xi+}_{x-\evec_1,x}$ has a dense set of jumps, while the jumps of $\xi\mapsto \Bus^{0, \,\xi+}_{x-\evec_1,x}$ are isolated in $\,]\evec_2, \evec_1[\,$ and accumulate only at $\evec_2$. 
 
 We emphasize that Theorem~\ref{thm:igexp2} and the statement above about jumps are distributional only, and they are made possible by the Poisson point process representation \eqref{Def:X}.  
 We do not presently have a proof based on properties of polymer paths and geodesics. 
 Finding such a proof remains an interesting open problem and may enable one to go beyond the exactly solvable case.   

In preparation, we give the zero-temperature version of the bijection $\rho\mapsto\xi^\alpha(\rho)$ from \eqref{u-rho}.
It is a decreasing map $\xi^0\colon [0,1]\to[\evec_2, \evec_1]$ from rate parameters to direction vectors, given by
\be\label{rhoxi}  \xi^0(\rho)=\bigl( \tfrac{(1-\rho)^2}{\rho^2+(1-\rho)^2}\,,\, \tfrac{\rho^2}{\rho^2+(1-\rho)^2} \bigr) . \ee
With this parametrization,  the marginal distributions of nearest-neighbor Busemann functions are  $\Bus^{0, \xi^0(\rho)\sig}_{x-\evec_1,x}\sim\mathrm{Exp}(1-\rho)$ and $\Bus^{0, \xi^0(\rho)\sig}_{x-\evec_2,x}\sim\mathrm{Exp}(\rho)$
\cite[Cor.~5.1]{geor-rass-sepp-yilm-15}.
For a fixed horizontal edge $(x-\evec_1,x)$, the process $\{\Bus^{0, \,\xi^0(\rho)-}_{x-\evec_1,x} : \rho\in [0,1)\}$ was shown in \cite[Thm.~3.4]{fan-sepp-20} to have the same distribution as the following cadlag process:
\eeq{ \label{hz3f21}
Z^0(\rho) = Y_x + \sum_{(s,y)\,\in\,\Prm^0}\ind_{(0,\rho]}(s)\cdot y, \quad \rho\in[0,1),
}
where $\Prm^0$ is a Poisson point process independent of $Y_x$ with intensity measure\footnote{Since the Busemann processes are parametrized in \cite{fan-sepp-20} by their means rather than directions, one needs to push forward the marked point process $\{(t,Z_t):\, t\in N\}$ in \cite{fan-sepp-20} by the map $[1,\infty)\times\R_{>0}\ni (t,z)\mapsto (s,y)=(1-t^{-1},z)\in[0,1)\times \R_{>0}$ to yield \eqref{cz8g}, and correspondingly take our $Z^0(\rho)$ to be $X((1-\rho)^{-1})$ from \cite[eq.~(3-6)]{fan-sepp-20}.}
\eeq{ \label{cz8g}
e^{-y(1-s)}\ \dd s\,\dd y, \quad (s,y)\in(0,1)\times\R_{>0}.
}
The marginal weak convergence of Busemann processes that points to the correct process-level convergence goes as follows, for each fixed $\rho\in[0,1)$: 
\eq{ 
\alpha  \Bus^{\alpha, \,\xi^\alpha(\alpha\rho)\sig}_{x-\evec_1,x} \; \stackref{B569}{\sim} \;\alpha\log{\rm Ga}^{-1}(\alpha(1-\rho))  
\ \underset{\alpha\searrow0}\dtoo  \ {\rm Exp}(1-\rho)  \; \sim \;   
\Bus^{0, \,\xi^0(\rho)\sig}_{x-\evec_1,x}. 
}
Therefore, the coupling in the proof of Theorem~\ref{thm:igexp2} will use a map $u^\alpha$ that identifies $\xi^\alpha(\alpha\rho)$ with $\xi^0(\rho)$.
To this end, we denote the inverse of \eqref{rhoxi} by $s^0\colon[\evec_2, \evec_1]\to[0,1]$, given by
\eq{
s^0(\xi) = \frac{(1-\xi_1)^{1/2}}{(1-\xi_1)^{1/2}+\xi_1^{1/2}}.
}
We then have the following lemma.

\begin{lemma}\label{lm:ua13}    
Define $u^\alpha: [\evec_2, \evec_1]\to[\evec_2, \evec_1]$ by 
\be\label{ua7}  u^\alpha(\xi)= \xi^\alpha(\alpha \tsp s^0(\xi)). \ee
There exists a constant $C$ so that we have the uniform bound 
\be\label{ua13}     \sup_{\xi\tsp\in[\evec_2, \evec_1]} \abs{ u^\alpha(\xi)-\xi}_1 \le C\alpha^2
\qquad\text{for all }   \alpha>0.  \ee
\end{lemma} 

\begin{proof}   
Since $s^0$ is the inverse of $\xi^0$, we have
\eq{
u^\alpha(\xi) - \xi = \xi^\alpha(\alpha s^0(\xi)) - \xi^0(s^0(\xi)).
} 
Therefore, the desired bound \eqref{ua13} is equivalent to
\be\label{new206}
\sup_{\rho\in [0,1]} \abs{\xi^\alpha(\alpha\rho)- \xi^0(\rho)}_1 \le C \alpha^2 \qquad\text{for all }   \alpha>0.
\ee
We thus proceed to show \eqref{new206}.
Since the functions are continuous in $\rho$, it suffices to consider $\rho\in(0,1)$.
Also, since $|\zeta - \eta|_1 = 2|\zeta_1-\zeta_1|$ for $\zeta,\eta\in[\evec_2,\evec_1]$, it suffices to look at the $\evec_1$ coordinates:
\eq{
\xi^\alpha(\alpha\rho)\cdot \evec_1 
&\stackref{u-rho}{=}  \frac{\psi_1(\alpha\rho)}{\psi_1(\alpha(1-\rho))+\psi_1(\alpha\rho)} \quad \text{and} \quad
\xi^0(\rho)\cdot\evec_1\stackref{rhoxi}=\frac{(1-\rho)^2}{\rho^2+(1-\rho)^2}.
}
 From the series representation of the trigamma function  
\eq{ 
\psi_1(x)=\sum_{k=0}^\infty \frac1{(k+x)^2} =  \frac1{x^2}\Bigl( 1+  \sum_{k=1}^\infty \frac{x^2}{(k+x)^2} \Bigr)
}
we obtain the following with $b=\pi^2/6$:
\eq{ 
x^{-2} \le  \psi_1(x) \le    x^{-2}( 1+   b\tsp x^2)  
\quad\text{for all } x>0. 
}
Now apply the upper bound to the numerator of $\xi^\alpha(\alpha\rho)\cdot \evec_1$, and lower bound to the denominator:\begin{align*}
\xi^\alpha(\alpha\rho)\cdot \evec_1 
\leq \frac{\alpha^{-2}\rho^{-2} (1+ b\tsp\alpha^2)}{  \alpha^{-2}(1-\rho)^{-2} + \alpha^{-2}\rho^{-2} }
=\frac{ (1-\rho)^{2} (1+ b\tsp\alpha^2)}{   (1-\rho)^{2} +  \rho^{2} }
=(1+ b\tsp\alpha^2) \, \xi^0(\rho)\cdot \evec_1 . 
\end{align*}
Then perform the opposite applications:
\begin{align*}
\xi^\alpha(\alpha\rho)\cdot \evec_1 
&\ge \frac{\alpha^{-2}\rho^{-2} }{  \alpha^{-2}(1-\rho)^{-2}(1+ b\tsp\alpha^2(1-\rho)^2) + \alpha^{-2}\rho^{-2}(1+ b\tsp\alpha^2\rho^2) } \\[4pt] 
&\ge  (1+ b\tsp\alpha^2)^{-1} \xi^0(\rho)\cdot  \evec_1 \ge (1- b\tsp\alpha^2) \, \xi^0(\rho)\cdot \evec_1 . 
\end{align*}
Hence \eqref{new206} holds with $C = 2b = \pi^2/3$.
\end{proof}

\begin{proof}[Proof of Theorem~\ref{thm:igexp2}]


 First we define  a single point process $\cZ$ and initial values $\{Z^{\alpha}(0):\alpha\in[0,1]\}$ from  which we will construct versions of the Busemann processes.
 %
For the purpose of thinning, we add a uniform $(0,1)$-valued mark to the process $\Prm$ described in \eqref{PPPintensity} with $\alpha=1$.   (The choice of starting value $\alpha=1$  is arbitrary since in the end we let $\alpha\searrow0$.)
  In other words, we let $\cZ$ be the inhomogeneous Poisson point process on $(0,1)\times\R_{>0}\times(0,1)$  with
	intensity measure   
	\eq{ 
	 \frac{e^{-y(1-s)}}{1-e^{-y}}\ \dd s\, \dd y\, \dd u, \quad (s,y,u)\in(0,1)\times\R_{>0}\times(0,1).
	}
%
%
%
%
%
%
%
%
%
%
  For $\alpha>0$, let $F_\alpha$ be the CDF of the  $\alpha\log{\rm Ga}^{-1}(\alpha)$ distribution, and $F_0$ the CDF of  Exp(1).   
 Let  $U\sim$ Unif$(0,1)$ be independent of $\cZ$.    
For each $\alpha\in[0,1]$,    define  $Z^\alpha(0)=F_\alpha^{-1}(U)$ so that  $Z^\alpha(0)\sim\alpha\log{\rm Ga}^{-1}(\alpha)$ and $Z^0(0)\sim{\rm Exp}(1)$ are independent of $\cZ$, and 
\be\label{newZZ}
\lim_{\alpha\searrow0} Z^\alpha(0) = Z^0(0). 
\ee

Now we construct versions of the Busemann processes.
For $\alpha\in(0,1]$ define the cadlag process 
\begin{equation}\label{newZZ7}\begin{aligned} 
	Z^{\alpha}(\rho) 
	= \Zpr^{\alpha}(0)+\sum_{(s,y,u)\in\cZ} \ind_{\bigl(0\,, \,\tfrac{1-e^{-y}}{1-e^{-y/\alpha}}\bigr]}(u)   
	\cdot  \ind_{(0,\rho]}(s) \cdot y, \quad \rho\in [0,1).
\end{aligned} 	\end{equation}
By thinning, the set $\big\{(s,y)\in(0,1)\times\R_{>0}:\, (s,y,u)\in\cZ \text{ and } u\leq \tfrac{1-e^{-y}}{1-e^{-y/\alpha}}\big\}$ is a Poisson point process with intensity measure
\eq{
\frac{e^{-y(1-s)}}{1-e^{-y/\alpha}}\tspb\ \dd s\, \dd y, \quad (s,y)\in(0,1)\times\R_{>0}.
}
This is precisely the pushforward of \eqref{PPPintensity} under the map $(s,y)\mapsto(s/\alpha,\alpha y)$.
So upon comparison of \eqref{newZZ7} with \eqref{Def:X}, Theorem~\ref{B-th5} can be restated in the equivalent form 
\be\label{new82} 
\{\alpha  \Bus^{\alpha, \,\xi^\alpha(\alpha\rho)-}_{x-\evec_1,x} : \rho\in [0,1)\}
\;\deq\;  \{ Z^{\alpha}(\rho)  : \rho\in [0,1)\} \quad \text{for each $\alpha\in(0,1]$}.
\ee 
For $\alpha=0$  we set
\begin{equation}\label{newZZ13}\begin{aligned} 
	Z^{0}(\rho) = \Zpr^{0}(0)+\sum_{(s,y,u)\in\cZ} \ind_{(0,  1-e^{-y}]}(u)  \cdot  \ind_{(0,\rho]}(s) \cdot y,  
	\quad \rho\in [0,1).
\end{aligned} 	\end{equation}
By thinning, \eqref{newZZ13} is equivalent to our earlier description \eqref{hz3f21} which used intensity measure \eqref{cz8g}, and so
\be\label{new86} 
\{ \Bus^{0, \,\xi^0(\rho)-}_{x-\evec_1,x} : \rho\in [0,1)\}
\;\deq\;  \{ Z^0(\rho)  : \rho\in [0,1)\}. 
\ee 
From \eqref{newZZ7} and \eqref{newZZ13} we have
\eq{
(Z^{\alpha}(\rho)-Z^{\alpha}(0))  - (Z^{0}(\rho) -Z^0(0))
= \sum_{(s,y,u)\in\cZ} \Bigl(\ind_{\bigl(0\,, \,\tfrac{1-e^{-y}}{1-e^{-y/\alpha}}\bigr]}(u)   
-  \ind_{(0,  1-e^{-y}]}(u)   \Bigr) 	\cdot  \ind_{(0,\rho]}(s) \cdot y.
}
The right-hand side is nondecreasing in $\rho$ and vanishes as $\alpha\searrow0$; so for any $\delta>0$,
\eeq{ \label{o8g2x}
&\sup_{\rho\in[0,1-\delta]} \big|(Z^{\alpha}(\rho)-Z^{\alpha}(0))  - (Z^{0}(\rho) -Z^0(0))\big| \\
&= (Z^{\alpha}(1-\delta)-Z^{\alpha}(0))  - (Z^{0}(1-\delta) -Z^0(0)) \to 0 \quad \text{as $\alpha\searrow0$}.
}
Limits \eqref{newZZ} and \eqref{o8g2x} combine to show
\eeq{ \label{gkd93x}
\lim_{\alpha\searrow0} \sup_{\rho\in[0,1-\delta]} |Z^\alpha(\rho)-Z^0(\rho)| = 0.
}

Our final step is to reparametrize. 
For $\alpha>0$ let $s^\alpha\colon [\evec_2, \evec_1]\to[0,\alpha]$ be the decreasing inverse of $\xi^\alpha$ from \eqref{u-rho}. 
Define $X^\alpha(\xi) = Z^\alpha(\alpha^{-1}s^\alpha(\xi)-)$ for $\alpha\in(0,1]$, so that \eqref{new82} reads
\eeq{ \label{nnew82}
\{\alpha  \Bus^{\alpha, \,\xi+}_{x-\evec_1,x} :\, \xi\in\,]\evec_2,\evec_1]\}
\;\deq\;  \{ X^{\alpha}(\xi)  :\, \xi\in\,]\evec_2,\evec_1]\} \quad \text{for each $\alpha\in(0,1]$}.
}
For $\alpha=0$ set $X^0(\xi) = Z^0(s^0(\xi)-)$ so that \eqref{new86} reads
\eeq{ \label{nnew86}
\{ \Bus^{0, \,\xi+}_{x-\evec_1,x} :\, \xi\in\,]\evec_2,\evec_1]\}
\;\deq\;  \{ X^{0}(\xi)  :\, \xi\in\,]\evec_2,\evec_1]\}.
}
Taking $u^\alpha$ as in Lemma~\ref{lm:ua13}, we have
\eq{
X^\alpha(u^\alpha(\xi))-X^0(\xi) 
\stackref{ua7}{=} Z^\alpha(\alpha^{-1}s^\alpha(\xi^\alpha(\alpha s^0(\xi)))-)
- Z^0(s^0(\xi)-)
= Z^\alpha(s^0(\xi)-) - Z^0(s^0(\xi)-).
}
Hence \eqref{gkd93x} implies the uniform limit \eqref{jc3x1xb}.
Meanwhile, \eqref{jc3x1xa} follows from \eqref{ua13}.
We have thus established that $X^\alpha$ converges to $X^0$ in $D(\,]\evec_2, \evec_1], \R)$ as $\alpha\searrow0$.
In light of \eqref{nnew82} and \eqref{nnew86}, this proves the claimed weak convergence.
\end{proof} 

\section{Proofs in the general environment}
\label{mar_sec}

This section identifies the joint distribution of finitely many Busemann functions on a lattice level  as the unique invariant distribution of a Markov chain. 
This Markov chain (the parallel process) intertwines with another Markov chain (the sequential process) which utilizes geometric row insertion. 
Following this development are the proofs of four main results:
\begin{itemize}
\item Theorem~\ref{cpthm} (stated more precisely as Theorem~\ref{B_thm9}) in Section~\ref{2_proc_sec};
\item Theorem~\ref{thm:78-63} also in Section~\ref{2_proc_sec};
\item Theorems~\ref{thm:51-32} and \ref{thm:allxy} in Section~\ref{sec:pf_disc}.
\end{itemize}
The gRSK connection is explained in Section~\ref{sec:grsk}. 
Outcomes of intertwining for the inverse-gamma polymer are pursued in Section~\ref{sec:twin_ig}.

\subsection{Update map} \label{map_sec}

As in the CGM in \cite{fan-sepp-20}, to capture the Busemann process we formulate the directed polymer model on a half-plane. The {update map}  constructs ratios of partition functions from one lattice level to the next. 
Similar mechanics were developed in \cite[Sec.~4]{janj-rass-20-jsp} to study the ergodicity and uniqueness of the distribution of a recovering cocycle. 

Our basic state space is the space of bi-infinite sequences $I = (I_k)_{k\in\Z}$ of strictly  positive real numbers  for which a finite left tail  logarithmic Ces\`aro limit exists:
\eeq{ \label{ces_def}
\ces(I) \coloneqq \lim_{n\to\infty} \frac{1}{n}\sum_{k=-n+1}^{0}\log I_k \in (-\infty,\infty).
}
Let $\cI\subset\R^\Z_{>0}$ denote the space of such sequences. Then  define the space 
\eeq{ \label{i2_up_def}
\cI_2^\uparrow\coloneqq\{(\Yw,I)\in\cI\times\cI:\, \ces(\Yw)<\ces(I)\}.
}
On $\cI_2^\uparrow$  we define the \textit{update map}   $\Dop\colon\cI_2^\uparrow\to\cI$ together with two related maps $\Rop\colon\cI_2^\uparrow\to\cI$ and $\Sop\colon\cI_2^\uparrow\to\R^\Z_{>0}$ 
 that are central to our analysis.
Given  input $(\Yw,I)\in\cI_2^\uparrow$, let us locally denote the outputs of these three maps by
\be
\label{drs_def}   
\wt I=(\wt I_k)_{k\in\Z}=\Dop(\Yw,I),  \quad \wt \Yw=(\wt\Yw_k)_{k\in\Z}=\Rop(\Yw,I) , \quad J=(J_k)_{k\in\Z}=\Sop(\Yw,I). 
\ee
The definitions that follow may seem obscure in origin, but they are manifestations of the dynamics obeyed by the Busemann process.
More specifically, if we make the identifications
\eq{
I_k = e^{\Bus((k-1,t-1),(k,t-1))},\quad
\wt I_k = e^{\Bus((k-1,t),(k,t))}, \quad
J_k = e^{\Bus((k,t-1),(k,t))}, \quad
W_k = W_{(k,t)},
}
then \eqref{J_def} and the first identity in \eqref{I_W_def} are obtained by repeated applications of \eqref{rec_coc}, while the second definition in \eqref{I_W_def} makes intertwining possible via Lemma~\ref{m:D-lm4}.

First define $\Sop$ by setting 
\eeq{ \label{J_def}
J_k 
= \sum_{i=-\infty}^k \Yw_i\prod_{j=i+1}^k \frac{\Yw_j}{I_j}
= \Yw_k + \sum_{i=-\infty}^{k-1} \Yw_i\prod_{j=i+1}^k\frac{\Yw_j}{I_j}.
}
Note that the right-hand side is finite if and only if
\eq{ 
  \sum_{i=-\infty}^{0}  {\Yw_i} \prod_{j=i+1}^{0}\frac{\Yw_j}{I_j}<\infty, \quad \text{equivalently} \quad
  \sum_{i=-\infty}^{0} e^{\sum_{j = i}^0\log \Yw_j-\sum_{j=i+1}^0 \log I_j}<\infty.
}
Consequently, it suffices to have $\ces(\Yw) < \ces(I)$ for $\Sop(\Yw,I)$  to be well-defined.
Since all quantities are positive, it is clear that   $\Sop$ maps $\cI_2^\uparrow$ into $\R_{>0}^\Z$.
Moreover, the definition \eqref{J_def} leads to a recursion:
\eeq{ \label{mja}
J_k 
=\Yw_k\Big(1+\frac{1}{I_k}\sum_{i=-\infty}^{k-1}\Yw_i\prod_{j=i+1}^{k-1}\frac{\Yw_j}{I_j}\Big)
=\Yw_k\Big(1+\frac{J_{k-1}}{I_k}\Big).
}
Finally, define the transformations  $\Dop$ and $\Rop$ in \eqref{drs_def}  by
\eeq{ \label{I_W_def}
\wt I_k \coloneqq \frac{I_{k}J_k}{J_{k-1}} \qquad\text{and}\qquad  \wt \Yw_k \coloneqq (I_k^{-1}+ J_{k-1}^{-1})^{-1} \qquad \text{for }  k\in\Z.
}
The lemma below checks that $\Dop$ and $\Rop$ map $\cI_2^\uparrow$ into $\cI$ and preserve Ces\`aro means.
Afterward we prove additional technical lemmas  for later use.
The reader may proceed to Section~\ref{twin_sec} and return   when needed.  

\begin{lemma} \label{ces_lem}    
For $(\Yw,I)\in\cI_2^\uparrow$, the sequences $\wt I = \Dop(\Yw,I)$ and $\wt\Yw = \Rop(\Yw,I)$ defined in \eqref{I_W_def} satisfy  $\ces(\wt I) = \ces(I)$  and  $\ces(\wt\Yw) = \ces(\Yw)$. 
\end{lemma} 

\begin{proof}
The definition of $\wt I_k$ in \eqref{I_W_def} gives $J_{k}/J_{k-1} = \wt I_k/I_k$.
Similarly, dividing both sides of \eqref{mja} by $J_{k-1}$ gives $J_{k}/J_{k-1} = \Yw_k/\wt\Yw_k$.
From these two equalities of ratios,  
\eq{
\sum_{k=-n+1}^{0}\log\frac{\wt I_k}{I_k} = \sum_{k=-n+1}^{0}\log\frac{\Yw_k}{\wt\Yw_k} = \sum_{k=-n+1}^{0}\log \frac{J_k}{J_{k-1}} = \log J_{0} - \log J_{-n}.
}
Therefore, both statements in the lemma  are implied by
\eeq{ \label{log_J_zero}
\lim_{n\to\infty} 
n^{-1} \log J_{-n}= 0.
} 
The remainder of the proof establishes this limit.

Since $\ces(\Yw)$ exists and is finite, we necessarily have $n^{-1}\log \Yw_{-n}\to0$ as $n\to\infty$.
It thus suffices to show that $(\log J_{-n} - \log \Yw_{-n})/n\to0$.
To this end, for $k<0$ we use \eqref{J_def} to write
\be \label{joy_setup}\begin{aligned} 
\frac{J_k}{\Yw_k} &= 1 + \sum_{i=-\infty}^{k-1} e^{\sum_{j=i}^{k-1} \log\Yw_j \,-\, \sum_{j=i+1}^{k} \log I_j} \\
&= 1 +  e^{-\sum_{j=k}^0 \log\Yw_j \, +\, \sum_{j=k+1}^0 \log I_j} \sum_{i=-\infty}^{k-1} e^{\sum_{j=i}^{0} \log\Yw_j \,-\, \sum_{j=i+1}^{0} \log I_j}.
\end{aligned} \ee
Now, given any $\e>0$, let us identify $k_0$ sufficiently negative that
\eq{
\Big|\frac{1}{k}\Big[\sum_{j = k}^0 \log \Yw_j-\sum_{j={k+1}}^0\log I_j\Big] + \ces(\Yw)-\ces(I)\Big| < \e \quad \text{for all $k\le k_0$}.
}
Applying this estimate inside all the exponentials of \eqref{joy_setup}, we obtain the following for all $k\le k_0$ and $\e<\ces(I)-\ces(\Yw)$:
\eeq{ \label{sandwich}
1 \leq \frac{J_k}{\Yw_k} &\leq 1 + e^{k(\ces(\Yw)-\ces(I)-\e)}\sum_{i={-\infty}}^{k-1}e^{-i(\ces(\Yw)-\ces(I)+\e)} \\
&= 1 + e^{k(\ces(\Yw)-\ces(I)-\e)}\cdot\frac{e^{-(k-1)(\ces(\Yw)-\ces(I)+\e)}}{1-e^{\ces(\Yw)-\ces(I)+\e}} \\
&= 1 + \frac{e^{\ces(\Yw)-\ces(I)-(2k-1)\e}}{1-e^{\ces(\Yw)-\ces(I)+\e}}
= 1+\frac{e^{-2k\e}}{e^{\ces(I)-\ces(\Yw)-\e}-1}.
}
Observe that for any positive constant $C$ and $\eps>0$ we have
\eq{
\lim_{k\to-\infty} |k|^{-1}{\log\big(1+C{e^{-2k\e}}\big)} = 2\eps,
}
so that \eqref{sandwich} implies
\eq{
0\leq \varliminf_{k\to-\infty}|k|^{-1}\log\frac{J_{k}}{\Yw_k} \leq \varlimsup_{k\to-\infty}|k|^{-1}\log\frac{J_{k}}{\Yw_k}\leq 2\eps.
}
Since $\eps$ is arbitrary and $k^{-1}\log \Yw_k\to 0$ by existence of $\ces(\Yw)$, \eqref{log_J_zero} follows.
%
%
%
\end{proof} 

The next lemma shows that $I\mapsto\Dop(\Yw,I)$ is injective for any given weight sequence $W$, unlike its $(\max, +)$ analogue in \cite[eq.~(2-22)]{fan-sepp-20}.

\begin{lemma} \label{inj_lem}
The map $(\Yw,I)\mapsto(\Yw,\Dop(\Yw,I))$ is injective on $\cI_2^\uparrow$ and has a continuous inverse mapping defined on its image.
\end{lemma}

\begin{proof}
Insert the recursion \eqref{mja} into the definition \eqref{I_W_def}  of $\wt I_k$:
\eeq{ \label{m801a}
\wt I_k = \frac{I_k}{J_{k-1}}\cdot \Yw_k\Big(1+\frac{J_{k-1}}{I_k}\Big) = \Yw_k\Big(1 + \frac{I_k}{J_{k-1}}\Big).
}
Solving for $I_k$ results in
\eeq{ \label{ik_first}
I_k = \frac{\wt I_k-\Yw_k}{\Yw_k}\cdot J_{k-1}.
}
Now insert the expression $J_{k-1} = J_{k-2}\wt I_{k-1}/I_{k-1}$ from \eqref{I_W_def} into the right-hand side:
\eeq{ \label{inv_iden}
I_k = \frac{\wt I_k-\Yw_k}{\Yw_k}\cdot\frac{J_{k-2}\wt I_{k-1}}{I_{k-1}}
&\overset{\eqref{ik_first}}= \frac{\wt I_k-\Yw_k}{\Yw_k}\cdot\frac{W_{k-1}\wt I_{k-1}}{\wt I_{k-1}-W_{k-1}}.
}
We note that \eqref{m801a} implies $\wt I_k > \Yw_k$ for all $k$,  so the final expression in \eqref{inv_iden} is well-defined.
Indeed, \eqref{inv_iden} shows that $I$ is uniquely determined by $\Yw$ and $\wt I = \Dop(\Yw,I)$, meaning $I \mapsto \Dop(\Yw,I)$ is injective for any fixed $\Yw$.
Continuity of the inverse map is evident from   \eqref{inv_iden}, since the 
image of $(\Yw,\Dop(\Yw,I))$ is a subset of $\{(\Yw,\wt I\,)\in\cI_2^\uparrow:\, \wt I>\Yw\}$.
\end{proof}

Under a non-explosion condition,  recursions  \eqref{mja} and  \eqref{m801a}  uniquely identify the outputs: 


\begin{lemma} \label{lm:D13}    
Let    $(W,I)\in\cI_2^\uparrow$.  
Assume $\wc J\in\R_{>0}^\Z$ satisfies
\eeq{ \label{b6x3dd}
\varliminf_{k\to-\infty} \frac{\log \wc J_k}{|k|} \leq 0.
}
If $\wc J$ satisfies the recursion
	\be\label{cJ-rec}
	\wc J_k= \Yw_k\biggl(1+  \frac{\wc J_{k-1}}{I_k}   \biggr)\qquad  
	\text{for all } \ k\in\Z,   \ee
then $\wc J=\Sop(\Yw,I)$,
Furthermore, if $\wc I\in\R_{>0}^\Z$ satisfies the recursion 
	\be\label{cJ8}
	\wc I_k= \Yw_k\biggl(1+  \frac{I_k}{\wc J_{k-1}}   \biggr)\qquad  
	\text{for all } \ k\in\Z,   \ee
then $\wc I=\Dop(I, \Yw)$.  	
\end{lemma} 

\begin{proof}  The assumption  $(W,I)\in\cI_2^\uparrow$  
guarantees that $J=\Sop(\Yw,I)$ and $\wt I=\Dop(\Yw,I)$ are well-defined.  
		Performing $n$ iterations of the assumed recursion \eqref{cJ-rec} 
	gives
	\eq{
	\wc J_k
	&= \sum_{i=k-n+1}^k \Yw_i\prod_{j=i+1}^k\frac{\Yw_j}{I_j}
 + \wc J_{k-n}\prod_{j=k-n+1}^k\frac{\Yw_j}{I_j} \\
	&= \sum_{i=k-n+1}^k\Yw_{i}\prod_{j=i+1}^k \frac{\Yw_j}{I_j}
 \,+\,
 \exp\biggl\{  n\biggl(\frac{\log\wc J_{k-n}}{n}  + \frac{1}{n}\sum_{i=k-n+1}^k\log\Yw_i  \, -\, \frac{1}{n} \sum_{i=k-n+1}^k\log I_i\biggr)\biggr\}. 
	}
	By the assumptions, there is a subsequence $n_\ell\to\infty$ along which the second term on  the last line is eventually $\le e^{-n_\ell\delta}$ for some $\delta>0$.  Passing to the limit along this subsequence shows that $\wc J_k$ matches the formula \eqref{J_def} for $J_k$.    Now \eqref{cJ8} agrees with \eqref{m801a} for $\wt I$. 
\end{proof}

The next lemma concerns monotonicity.
The inequalities   are understood coordinatewise: $I'\ge I$ means  $I_k'\ge I_k$ for every $k\in\Z$. 
Similarly, $I'>I$ means $I_k'>I_k$ for every $k\in\Z$.

\begin{lemma} \label{mon_lem}
Let $(\Yw,I)$ be any element of $\cI_2^\uparrow$.
\begin{enumerate}[label=\textup{(\alph*)}]

\item \label{mon_lema}
We have $\Dop(\Yw,I) > \Yw$.

\item \label{mon_lemb}
 If $I'\ge I$, then $\Dop(\Yw,I') \geq \Dop(\Yw,I)$.
If we further know that $I_{k_0}'>I_{k_0}$, then 
\eeq{ \label{dogk}
\Dop(\Yw,I')_k > \Dop(\Yw,I)_k \quad \text{for all $k\geq k_0$}.
}
\end{enumerate}
\end{lemma}

\begin{proof}
Part~\ref{mon_lema} is immediate from \eqref{m801a}.
For part~\ref{mon_lemb}, let us write $\wt I'=\Dop(\Yw,I')$ and
$J' = \Sop(\Yw,I')$.
Then \eqref{J_def} implies $J_k'\leq J_k$, where the inequality is strict as soon as $k \geq k_0$.
In view of \eqref{m801a}, the combination of $I_k'\ge I_k$ and $J_{k-1}'\leq J_{k-1}$ implies $\Dop(\Yw,I') \geq \Dop(\Yw,I)$.
Furthermore, when $k\geq k_0$, at least one of these two inequalities is strict, hence \eqref{dogk}.
\end{proof}

The last lemma shows that when additional control is available,  the update map itself possesses continuity in the product topology.

\begin{lemma}  \label{lm:D319}  Let $(\Yw,I)\in\cI_2^\uparrow$ and let $\{ (\Yw^h,I^h)\}_{h\tspa\in\tspa\Z_{>0}}$ be a sequence of elements of $\cI_2^\uparrow$ such that $(\Yw^h,I^h)\to (\Yw,I)$ coordinatewise as $h\to\infty$.  Assume there is a pair $(\Yw'',I')\in\cI_2^\uparrow$ such that $\Yw^h\le\Yw''$ and $I'\le I^h$ $\forall h\in\Z_{>0}$.   Define the outputs $\wt I=\Dop(\Yw,I)$   and  
$\wt I^h=\Dop(\Yw^h,I^h)$.   Then $\wt I^h\to\wt I$ coordinatewise. 

\end{lemma} 

\begin{proof}   Let   $J=\Sop(\Yw,I)$   and  
$J^h=\Sop(\Yw^h,I^h)$.   
We verify that 
\be\label{J680}   \lim_{h\to\infty}   J^h_k =  J_k \quad \text{for all $k\in\Z$}. \ee
By the recursive formula \eqref{mja}, it suffices to show that \eqref{J680} holds for arbitrarily large negative $k$. 
From \eqref{joy_setup} write 
\be \label{joy_setup6}\begin{aligned} 
\frac{J^h_k}{\Yw^h_k} 
= 1 +  e^{-\sum_{j=k}^0 \log\Yw^h_j \, +\, \sum_{j=k+1}^0 \log I^h_j} \sum_{n=-\infty}^{k-1} e^{\sum_{j=n}^{0} \log\Yw^h_j \,-\, \sum_{j=n+1}^{0} \log I^h_j}.
\end{aligned} \ee
For each $h$ and $n<0$ we have 
\[    e^{\sum_{j=n}^{0} \log\Yw^h_j \,-\, \sum_{j=n+1}^{0} \log I^h_j}  
\le e^{\sum_{j=n}^{0} \log\Yw''_j \,-\, \sum_{j=n+1}^{0} \log I'_j}  
\]  
and the latter terms are summable by the assumption $\ces(\Yw'') < \ces(I')$.   Thus the right-hand side of \eqref{joy_setup6} converges to the same expression without the $h$-superscripts and \eqref{J680}  has been verified.    From \eqref{I_W_def}  follows then that $\wt I^h\to\wt I$. 
\end{proof}

\subsection{Intertwined dynamics on sequences: fixed weight sequence} \label{twin_sec}	
For any positive integer $N$ and real number $\kappa$, define the space
\eeq{ \label{in_ces_def}
\cI_{N\!,\tspb\kappa} \coloneqq \{(I^1,\dots,I^N)\in\cI^N:\, \ces(I^i)>\kappa \text{ for each $i$}\}.
}
To condense notation, we   write
$
I^{\parng{i}{j}} = (I^i,\dots,I^j).
$
%
	Fix a weight sequence $\Yw\in\cI$ with 
	\eeq{ \label{fixed_ces}
	\ces(\Yw) = \kappa.
	} 
	We now define two   $\cI_{N\!,\tspb\kappa}\to\cI_{N\!,\tspb\kappa}$  mappings.
	

\smallskip

	(A) The \textit{parallel transformation}    $\Taop_\Yw\colon\cI_{N\!,\tspb\kappa}\to\cI_{N\!,\tspb\kappa}$ is the simultaneous application of the update map $D$ to several sequences $I^1,\dots,I^N$ with the \textit{same} weight sequence $\Yw$:  
\eeq{ \label{Taop_def}
\Taop_\Yw(I^{\parng{1}{N}}) = \big(\Dop(\Yw,I^1),\dots,\Dop(\Yw,I^N)\big). 
}
This is the transformation we ultimately care about, as it is the one obeyed by Busemann functions.
By Lemma~\ref{ces_lem}, the Ces\`aro limits of the input sequences are all preserved:
\eeq{ \label{T_pres}
\ces\big(\Dop(\Yw,I^i)\big) = \ces(I^i) \quad \text{for each $i\in\lzb 1, N \rzb$}.
}

\smallskip 

(B) The    \textit{sequential transformation}   $\Saop_\Yw\colon\cI_{N\!,\tspb\kappa}\to\cI_{N\!,\tspb\kappa}$ 
  applies  the update map $\Dop$ to each sequence $I^i$, but with weights that are updated between each application: 
\begin{subequations} \label{Saop_def}
\eeq{
\Saop_\Yw(I^{\parng{1}{N}}) = \big(\Dop(\Yw^1,I^1),\dots,\Dop(\Yw^N,I^N)\big), \label{Saop_def_1}
}
where (recall the map $\Rop$ from \eqref{drs_def} and \eqref{I_W_def})
\eeq{ \label{Saop_def_2}
\Yw^1 = \Yw \quad \text{and} \quad \Yw^i = \Rop(\Yw^{i-1},I^{i-1}) \quad \text{for $i\geq2$}.
}
\end{subequations}
 Lemma~\ref{ces_lem} guarantees $\ces(\Yw^1) = \ces(\Yw^{2}) = \cdots = \ces(\Yw^N)$, hence all the operations in \eqref{Saop_def} are well-defined and again preserve Ces\`aro limits:
\eeq{ \label{S_pres}
\ces\big(\Dop(\Yw^i,I^i)\big) = \ces(I^i) \quad \text{for each $i\in\lzb 1, N \rzb$}.
}
The definition \eqref{Saop_def} has also a recursive formulation:
\eeq{ \label{Saop_alt}
\Saop_\Yw(I^{\parng{1}{N}}) = \big(\Dop(\Yw,I^1),\Saop_{\Rop(\Yw,I^1)}(I^{\parng{2}{N}})\big).
}

Next we construct a mapping $\Daop$ that intertwines  $\Taop_\Yw$ and $\Saop_\Yw$.
Its domain is  
 the following 
``ordered" space that generalizes \eqref{i2_up_def}:
\eeq{ \label{in_up_def}
\cI_{N}^\uparrow &\coloneqq \{(I^1,\dots,I^N)\in\cI^N:\, \ces(I^1) < \ces(I^2) < \cdots < \ces(I^N)\}.
}
To begin the construction, Lemma~\ref{ces_lem} allows us to apply the update map $\Dop$ iteratively, as follows.
%
Define $\Dop^{(1)}\colon\cI\to\cI$ as the identity map.  
Take $\Dop^{(2)}\colon\cI_2^\uparrow\to\cI$ to be the map $D$ itself, as in \eqref{drs_def}:  $\Dop^{(2)}(I^1,I^2) \coloneqq \Dop(I^1,I^2)$.
For $i\geq3$ define $\Dop^{(i)}\colon\cI_i^\uparrow\to\cI$   recursively:
	\be\label{repeat_dog}  \begin{aligned}
		\Dop^{(i)}(I^{\parng{1}{i}}) &\coloneqq \Dop\bigl(I^1,\Dop^{(i-1)}(I^{\parng{2}{i}})\bigr).  
	\end{aligned} \ee
By Lemma~\ref{ces_lem} the Ces\`aro means are again preserved: 
\eeq{ \label{D_pres}
\ces\big(\Dop^{(i)}(I^{\parng{1}{i}})\big) = \ces(I^i).
}
 Furthermore, we have this strict monotonicity: 

	\begin{lemma} 	\label{cert_lem}
		For any $I^{\parng{1}{N}}\in\cI_N^\uparrow$, we have
		$\Dop^{(N)}(I^{\parng{1}{N}}) >   \Dop^{(N-1)}(I^{\parng{1}{N-1}})$.  
	\end{lemma} 
	
	\begin{proof}
		The proof uses induction on $N$. The case $N=2$ is Lemma \hyperref[mon_lema]{\ref*{mon_lem}\ref*{mon_lema}}.
		Under the hypothesis $\Dop^{(N-1)}(I^{\parng{2}{N}}) > \Dop^{(N-2)}(I^{\parng{2}{N-1}})$,   Lemma \hyperref[mon_lemb]{\ref*{mon_lem}\ref*{mon_lemb}} gives the following inequality: 
\[
\Dop^{(N)}(I^{\parng{1}{N}})=\Dop\bigl(I^1,\Dop^{(N-1)}(I^{\parng{2}{N}})\bigr)  > \Dop\bigl(I^1,\Dop^{(N-2)}(I^{\parng{2}{N-1}})\bigr) = \Dop^{(N-1)}(I^{\parng{1}{N-1}}).
\qedhere \]
\end{proof}

Finally, define the map  $\Daop=\Daop^{(N)}\colon \cI_{N}^\uparrow\to\cI_N^\uparrow$ by
\eeq{ \label{Daop_def}
\Daop(I^{\parng{1}{N}}) = \big(\Dop^{(1)}(I^1),\Dop^{(2)}(I^{\parng{1}{2}}),\dots,\Dop^{(N)}(I^{\parng{1}{N}})\big).
}
By \eqref{D_pres}, $\Daop$ preserves the Ces\`aro means of the component sequences. 
By Lemma~\ref{cert_lem}, the output is a coordinatewise strictly ordered $N$-tuple of sequences.
It also has an inverse:

\begin{lemma} \label{invc_lem}  Fix $N\in\Z_{>0}$.  
\begin{enumerate}[label={\rm(\alph*)}, ref={\rm(\alph*)}]  \itemsep=3pt  
 \item \label{invc.a}  There exists a Borel set $\cH_N\subset (\R_{>0}^\Z)^N$ and a continuous mapping
$\Haop^{(N)}\colon \cH_N \to (\R_{>0}^\Z)^N$ such that $\Daop^{(N)}(\cI_{N}^\uparrow)\subset\cH_N$ and $\Haop^{(N)}\circ\Daop^{(N)}$ is the identity on $\cI_{N}^\uparrow$. 
 \item  \label{invc.b} Let $\Yw\in\cI$ with $\ces(\Yw)=\kappa$. Then the  maps  $\Saop_\Yw$ and $\Taop_\Yw$ are injective on $\cI_{N\!,\tsp\kappa}$.
\end{enumerate} 
\end{lemma}

\begin{proof}  Part~\ref{invc.a}.  
Our building block will be the inverse map 
 found in Lemma~\ref{inj_lem}.   Let 
 \[ 
 \cA_2=\{(X,Y)\in (\R_{>0}^\Z)^2:\,   X_k < Y_k \,\forall k\in\Z\}. 
 \]  
 Following \eqref{inv_iden}  define the image $I=H(X,Y)$ of the mapping $\Hop\colon \cA_2\to\R_{>0}^\Z$ by 
 \eq{ 
 I_k=\frac{Y_k-X_k}{X_k} \cdot \frac{X_{k-1}Y_{k-1}}{Y_{k-1}-X_{k-1}}\,, \qquad k\in\Z. 
 }
As in the proof of Lemma~\ref{inj_lem}, $\Hop$ is easily seen to be continuous. 
 
 Extend $\Hop$ to a sequence of   mappings $\Hop^{(m)}\colon \cA_m\to\R_{>0}^\Z$  for $m\in\Z_{>0}$   as follows.  
 Let  $\Hop^{(1)}(\brvcp) = \brvcp$ be the identity mapping on $\cA_1=\R_{>0}^\Z$.  Next let 
 $\Hop^{(2)}(\brvcp^{\parng{1}{2}}) = \Hop(\brvcp^{\parng{1}{2}})$ with $\cA_2$ as above. 
For $m\ge3$ define inductively the domain 
\begin{align*}
\cA_m&= \bigl\{ \brvcp^{\parng{1}{m}} \in (\R_{>0}^\Z)^m: 
\brvcp^1<\brvcp^i\ \forall i\in\lzb2,m\rzb,  
\, (\Hop(\brvcp^1,\brvcp^2),\dots,\Hop(\brvcp^1,\brvcp^m)) \in \cA_{m-1} \bigr\}, 
\end{align*} 
and then the map $\Hop^{(m)}\colon \cA_m\to\R_{>0}^\Z$ by 
 \be \label{repeat_hop} \begin{aligned} 
 \Hop^{(m)}(\brvcp^{\parng{1}{m}}) = \Hop^{(m-1)}\big(\Hop(\brvcp^1,\brvcp^2),\dots,\Hop(\brvcp^1,\brvcp^m)\big)  . 
\end{aligned}\ee
One sees inductively that each $\cA_m$ is a Borel set and each $\Hop^{(m)}$ 
is continuous. 

Next we show that 
\be\label{H436} 
\Daop^{(m)}(\cI_{m}^\uparrow)\subset\cA_m \qquad\text{for each $m\ge 1$.}  
\ee
 The $m=1$ case of \eqref{H436} is trivial since $\Daop^{(1)}$ is the identity on $\cI_1^\uparrow=\cI\subset\R_{>0}^\Z=\cA_1$.
For $m\geq2$, assume inductively $\Daop^{(m-1)}(\cI_{m-1}^\uparrow)\subset\cA_{m-1}$ and consider any $I^{\parng{1}{m}}\in\cI_m^\uparrow$.
Because of the definition $\Dop^{(i)}(I^{\parng{1}{i}}) = \Dop\big(I^1,\Dop^{(i-1)}(I^{\parng{2}{i}})\big)$, Lemma~\ref{mon_lem}\ref{mon_lema} shows 
\eeq{ \label{n4u0}
I^1 < \Dop^{(i)}(I^{\parng{1}{i}}) \quad \text{for each $i\in\lzb2,m\rzb$}.
}
In particular, $H\big(I^1,\Dop^{(i)}(I^{\parng{1}{i}})\big)$ is well-defined. 
Furthermore, from the proof of Lemma~\ref{inj_lem}, 
\eeq{
\label{Hop_inv}
\Hop\big(\Yw,\Dop(\Yw,I)\big) = I \quad \text{for any $(\Yw,I)\in\cI_2^\uparrow$}. 
}
Putting these facts together and writing temporarily $Y^i=\Dop^{(i)}(I^{\parng{1}{i}})$, we find
\eq{ 
\Hop(Y^1,Y^i)
\stackref{repeat_dog}{=} \Hop\big(I^1,\Dop\big(I^1,\Dop^{(i-1)}(I^{\parng{2}{i}})\big)\big)
\overset{\eqref{Hop_inv}}{=} \Dop^{(i-1)}(I^{\parng{2}{i}}).
}
As this holds for every $i\in\lzb2,m\rzb$, we have shown
\eq{
\big(\Hop(Y^1,Y^2),\dots,\Hop(Y^1,Y^m)\big)
=\big(\Dop^{(1)}(I^2),\dots,\Dop^{(m-1)}(I^{\parng{2}{m}})\big)
\stackref{Daop_def}{=} \Daop^{(m-1)}(I^{\parng{2}{m}}).
}
By induction, the rightmost expression belongs to $\cA_{m-1}$.
This observation, combined with \eqref{n4u0}, means that $Y^{\parng{1}{m}}\in\cA_m$, thereby verifying \eqref{H436}.
Furthermore,
\be\label{hop700} \begin{aligned}
\Hop^{(m)}(Y^{\parng{1}{m}})
&\overset{\eqref{repeat_hop}}{=} \Hop^{(m-1)}\big(\Hop(Y^1,Y^2),\dots,\Hop(Y^1,Y^m)\big) \\
&\stackrefp{repeat_hop}{=} \Hop^{(m-1)}(\Daop^{(m-1)}(I^{\parng{2}{m}})) . 
\end{aligned}\ee


Finally, define the Borel set
\[  \cH_N= \bigl\{ \brvcp^{\parng{1}{N}} \in (\R_{>0}^\Z)^N: 
\brvcp^{\parng{1}{m}}\in\cA_m \ \forall m\in\lzb1,N\rzb  \bigr\},  \] 
and then combine the maps from \eqref{repeat_hop}  into a continuous mapping $\Haop^{(N)}\colon \cH_N \to(\R_{>0}^\Z)^N$ by
\be  \label{Haop_def}
\Haop^{(N)}(\brvcp^{\parng{1}{N}})=\big(\Hop^{(1)}(\brvcp^1),\Hop^{(2)}(\brvcp^{\parng{1}{2}}),\dots,\Hop^{(N)}(\brvcp^{\parng{1}{N}})\big).
\ee
From the structure of $\Daop^{(N)}$ in \eqref{Daop_def},   $\Daop^{(N)}(I^{\parng{1}{N}})^{\parng{1}{m}} 
=  \Daop^{(m)}(I^{\parng{1}{m}})$ for $m\in\lzb1,N\rzb$.    Thus \eqref{H436} gives $\Daop^{(N)}(\cI_{N}^\uparrow)\subset\cH_N$.

When $N=1$,   $\Haop^{(1)}\circ\Daop^{(1)}$ is a composition of identity maps and hence itself the identity map on $\cI$.  For general $N\geq1$, apply \eqref{hop700} to the definition \eqref{Haop_def}:
\eq{
\Haop^{(N)}\big(\Daop^{(N)}(I^{\parng{1}{N}})\big)
= \big(I^1,\Haop^{(N-1)}\big(\Daop^{(N-1)}(I^{\parng{2}{N}})\big)\big).
}
By induction,   $\Haop^{(N)}\circ\Daop^{(N)}$ is the identity on $\cI_{N}^\uparrow$ for each $N\geq1$.

\smallskip 

Part~\ref{invc.b}.   
It is now clear that $\Taop_\Yw$ has an inverse map given by
\[
\Taop_\Yw^{-1}(\brvcp^{\parng{1}{N}}) = \big(\Hop(\Yw,\brvcp^1),\dots,\Hop(\Yw,\brvcp^N)\big) \quad \text{for $\brvcp^{\parng{1}{N}}\in\Taop_\Yw(\cI_{N\!,\tspb\kappa})$.}    
\]
It is also straightforward to check from \eqref{Saop_alt} that $\Saop_\Yw$ has inverse map given by the recursion
\[
\Saop_\Yw^{-1}(\brvcp^{\parng{1}{N}}) = \big(\Hop(\Yw,\brvcp^1),\Saop^{-1}_{\Rop(\Yw,\Hop(\Yw,\brvcp^1))}(\brvcp^{\parng{2}{N}})\big) \quad \text{for $\brvcp^{\parng{1}{N}}\in\Saop_\Yw(\cI_{N\!,\tspb\kappa})$.}
\qedhere \]
 \end{proof}

The  main goal of this section is  the identity \eqref{twine} below.
In order for its compositions  to make sense, we intersect the domain of $\Taop_\Yw$ and $\Saop_\Yw$ (see \eqref{in_ces_def}) with that of $\Daop$ (see \eqref{in_up_def}):
 \eeq{ \label{kappa_up}
 \cI_{N\!,\tspb\kappa}^\uparrow &\coloneqq \cI_{N\!,\tspb\kappa}\cap\cI_N^\uparrow
 =\{(I^1,\dots,I^N)\in\cI^N:\, \kappa < \ces(I^1) < \ces(I^2) < \cdots < \ces(I^N)\}.
 }
 Because of \eqref{T_pres}, \eqref{S_pres}, and \eqref{D_pres}, all three of  $\Taop_\Yw$, $\Saop_\Yw$, and $\Daop$ map $\cI_{N\!,\tspb\kappa}^\uparrow$ into itself.  So the compositions   in \eqref{twine} are well-defined on this space.

\begin{proposition} \label{g:itlm5}   For any $\Yw\in\cI$ with $\ces(\Yw)=\kappa$, we have the following equality of maps on $\cI_{N\!,\tspb\kappa}^\uparrow$:  
  \eeq{ \label{twine}
  \Taop_\Yw\circ\Daop = \Daop\circ\Saop_\Yw. 
  }
\end{proposition}

The following result from \cite{busa-sepp-22-ejp} is the essential ingredient that leads to our intertwining identity \eqref{twine}.
Originally \eqref{dog3} appeared in its zero-temperature form as  \cite[Lem.~4.4]{fan-sepp-20}. 
Recall the map $(\Yw,I)\mapsto R(\Yw,I)$ from \eqref{drs_def} and defined in \eqref{I_W_def}.

\begin{lemma} \textup{\cite[Lem.~A.5]{busa-sepp-22-ejp}} \label{m:D-lm4}   Given $(\Yw^1,I^1,I^2)\in\cI_3^\uparrow$, set $\Yw^2=\Rop(\Yw^1,I^1)$. 
	Then
	\be\label{dog3}  
	\Dop^{(3)}(\Yw^1,I^1,I^2)
	= \Dop\bigl(\Yw^1, \Dop(I^1, I^2)\bigr)
	=  \Dop\bigl( \Dop(\Yw^1,I^1), \Dop(\Yw^2,I^2)\bigr). 
	\ee 
\end{lemma}

To prove Proposition~\ref{g:itlm5}, we first extend Lemma~\ref{m:D-lm4} by induction.

\begin{lemma} \label{m:D-lm6}  Let $N\ge 2$ and $(\Yw^1, I^1, I^2, \dotsc, I^N)\in\cI_{N+1}^\uparrow$.
As in \eqref{Saop_def_2},  iteratively define  
	\eq{ 
	\Yw^i=\Rop(\Yw^{i-1},I^{i-1}) \quad\text{ for $i\in\lzb 2, N \rzb$.}
	}
	Then the following identity holds whenever $1\le k\le N-1$:  
	\be\label{dogmore}  \begin{aligned}
		\Dop^{(N+1)}(\Yw^1,I^{\parng{1}{N}})
		=\Dop^{(k+1)}\bigl( \Dop( \Yw^1, I^1), \dots, \Dop( \Yw^{k},I^{k}) , \Dop^{(N-k+1)}(\Yw^{k+1},I^{\parng{k+1}{N}})\bigr).  
	\end{aligned}   \ee 
In particular, when $k=N-1$, \eqref{dogmore} becomes 
\be\label{dogpart} \begin{aligned}
	\Dop^{(N+1)}(\Yw^1,I^{\parng{1}{N}}) 
	=\Dop^{(N)}\bigl(\Dop(\Yw^1,I^1), \dots, \Dop( \Yw^{N},I^{N})\bigr).
\end{aligned}  \ee 
\end{lemma}

\begin{proof}
For $k=1$, observe that \eqref{dogmore} is implied by Lemma~\ref{m:D-lm4}:
\eq{
\Dop^{(N+1)}(\Yw^1,I^{\parng{1}{N}})
&\stackref{repeat_dog}{=} \Dop\big(\Yw^1,\Dop^{(N)}(I^{\parng{1}{N}})\big) \\
&\stackref{repeat_dog}{=} \Dop\big(\Yw^1,\Dop\big(I^1,\Dop^{(N-1)}(I^{\parng{2}{N}})\big)\big) \\
&\stackref{dog3}{=} \Dop\big(\Dop(\Yw^1,I^1),\Dop(\Yw^2,\Dop^{(N-1)}(I^{\parng{2}{N}})\big) \\
&\stackref{repeat_dog}{=} \Dop\big(\Dop(\Yw^1,I^1),\Dop^{(N)}(\Yw^2,I^{\parng{2}{N}})\big).
}
Now, in the base case $N=2$, we can only have $k=1$, so there is nothing more to show.
So let us take $N\geq3$ and assume inductively that for each $k \in \lzb 2, N-1 \rzb$, we have
\eeq{ \label{dnk_step}
\Dop^{(N)}(\Yw^2,I^{\parng{2}{N}})
= \Dop^{(k)}\big(\Dop(\Yw^2,I^2),\dots,\Dop(\Yw^k,I^k),\Dop^{(N-k+1)}(\Yw^{k+1},I^{\parng{k+1}{N}})\big).
}
Beginning with the same sequence of equalities as above, we find that
\eq{
&\Dop^{(N+1)}(\Yw^1,I^{\parng{1}{N}}) 
= \Dop\big(\Dop(\Yw^1,I^1),\Dop^{(N)}(\Yw^2,I^{\parng{2}{N}})\big) \\
&\stackref{dnk_step}{=} \Dop\big(\Dop(\Yw^1,I^1),\Dop^{(k)}\big(\Dop(\Yw^2,I^2),\dots,\Dop(\Yw^k,I^k),\Dop^{(N-k+1)}(\Yw^{k+1},I^{\parng{k+1}{N}})\big)\big) \\
&\stackref{repeat_dog}{=} \Dop^{(k+1)}\bigl( \Dop( \Yw^1, I^1), \dots, \Dop( \Yw^{k},I^{k}) , \Dop^{(N-k+1)}(\Yw^{k+1},I^{\parng{k+1}{N}})\bigr). \qedhere
}
\end{proof}

\begin{proof}[Proof of Proposition~\ref{g:itlm5}] 
Given $I^{\parng{1}{N}}\in\cI_{N\!,\tspb\kappa}^\uparrow$, let $(A^1,\dots,A^N) = \Taop_\Yw\big(\Daop(I^{\parng{1}{N}})\big)$.  From \eqref{Daop_def} followed by  \eqref{Taop_def},  we have $A^i =\Dop\bigl( \Yw^1, \Dop^{(i)}(I^{\parng{1}{i}}) \bigr)$. 
    Similarly, let $(B^1,\dots,B^N) = \Daop\big(\Saop_\Yw(I^{\parng{1}{N}})\big)$. From \eqref{Saop_def} followed by \eqref{Daop_def},  we have $B^i = \Dop^{(i)}\bigl( \Dop(\Yw^1, I^1),\dotsc,  \Dop(\Yw^i, I^i)\bigr)$.
    Making use of Lemma~\ref{m:D-lm6}, we conclude
\begin{align*}
 A^i=\Dop\bigl( \Yw^1, \Dop^{(i)}(I^{\parng{1}{i}}) \bigr) 
  &\stackref{repeat_dog}{=} \Dop^{(i+1)}(\Yw^1,I^{\parng{1}{i}})\\
&\stackref{dogpart}{=} \Dop^{(i)}\bigl( \Dop( \Yw^1,I^1),\dotsc,  \Dop(\Yw^i,I^i)\bigr)  = B^i. \qedhere
\end{align*} 
\end{proof}

\subsection{Intertwined dynamics on sequences: random weight sequence} \label{twin_sec_2}
In the previous section, we defined $\Saop_\Yw$ and $\Taop_\Yw$ for any fixed weight sequence $\Yw\in\cI$.
Now we take $\Yw = \Yw(\w)$ to be random, according to the following assumption:
\begin{subequations}  \label{ran_ass}
\eeq{ \label{ran_assa}
&\text{$\Yw=(\Yw_k)_{k\in\Z}$ are positive, i.i.d.~random variables on $\OSP$ such that $\E\abs{\log \Yw_0}<\infty$}.
}
Consequently, the Ces\`aro limit $\ces(\Yw)$ from \eqref{ces_def} almost surely exists and is equal to $\E[\log \Yw_0]$.
Matching the notation from \eqref{fixed_ces}, we set 
\eeq{ \label{ran_assb}
\kappa = \E[\log \Yw_0],
}
\end{subequations}
so that almost surely $\Saop_\Yw$ and $\Taop_\Yw$ are well-defined maps $\cI_{N\!,\tspb\kappa}\to\cI_{N\!,\tspb\kappa}$.
For the purposes of discussing measures below,   $\cI_{N\!,\tspb\kappa}$ inherits the standard product topology of $(\R^\Z)^N$.

Given a probability measure $\mu$ on $\cI_{N\!,\tspb\kappa}$, let $\mu\circ\Saop^{-1}$ be the probability measure on $\cI_{N\!,\tspb\kappa}$ defined by
\eq{ 
[\mu\circ\Saop^{-1}](\cB) = \E\mu\big(\Saop_{\Yw}^{-1}(\cB)\big) \quad \text{for any Borel set $\cB\subset\cI_{N\!,\tspb\kappa}$} 
}
where the expectation $\E$ averages over the random weight sequence $\Yw$.
Similarly define the measure $\mu\circ\Taop^{-1}$ by
\eeq{ \label{push_T}
[\mu\circ\Taop^{-1}](\cB) = \E\mu\big(\Taop_{\Yw}^{-1}(\cB)\big) \quad \text{for any Borel set $\cB\subset\cI_{N\!,\tspb\kappa}$}. 
}
%
%
%
%
%
In other words, if $I^{\parng{1}{N}}$ is a random element of $\cI_{N\!,\tspb\kappa}$ independent of $\Yw$ and distributed according to $\mu$, then $\mu\circ\Saop^{-1}$ and $\mu\circ\Taop^{-1}$ are the laws of $\Saop_\Yw(I^{\parng{1}{N}})$ and $\Taop_\Yw(I^{\parng{1}{N}})$, respectively.
Finally, when $\mu$ is a probability measure on the ordered space $\cI_{N\!,\tspb\kappa}^\uparrow$ from \eqref{kappa_up}, we write $\mu\circ\Daop^{-1}$ for the usual pushforward by $\Daop$.
Because of intertwining, we have the following equivalence.

\begin{theorem} \label{twm_thm}
For any probability measure $\mu$ on $\cI_{N\!,\tspb\kappa}^\uparrow$, we have the following equality of measures on $\cI_{N\!,\tspb\kappa}^\uparrow$:
\eeq{ \label{twine_meas}
\mu\circ\Daop^{-1}\circ\Taop^{-1} = \mu\circ\Saop^{-1}\circ\Daop^{-1}.
}
In particular, if $\nu$ is a probability measure on $\cI_{N\!,\tspb\kappa}^\uparrow$ such that $\nu\circ\Saop^{-1}=\nu$, then the pushforward $\mu=\nu\circ\Daop^{-1}$ satisfies $\mu\circ\Taop^{-1} = \mu$.
\end{theorem}

\begin{proof}
Evaluated at some Borel set $\cB\subset\cI_{N\!,\tspb\kappa}^\uparrow$, the right-hand side of \eqref{twine_meas} gives
\eq{
[\mu\circ\Saop^{-1}]\big(\Daop^{-1}(\cB)\big)
= \E\mu\big[\Saop_\Yw^{-1}\big(\Daop^{-1}(\cB)\big)\big],
}
while the left-hand side gives
\eq{
\E[\mu\circ\Daop^{-1}]\big(\Taop_\Yw^{-1}(\cB)\big)
= \E\mu\big[\Daop^{-1}\big(\Taop_\Yw^{-1}(\cB)\big)\big].
}
By the intertwining identity \eqref{twine}, we have $\Saop_\Yw^{-1}\big(\Daop^{-1}(\cB)\big)=\Daop^{-1}\big(\Taop_\Yw^{-1}(\cB)\big)$, so we are done.
\end{proof}

%
%

 Theorem~\ref{twm_thm}  generates invariant distributions for the parallel transformation $\Taop$ from those of the sequential transformation $\Saop$.
This is useful for inverse-gamma weights discussed in Section~\ref{sec:twin_ig}.   We could go the other direction also, by considering $\Taop$-invariant measures that are supported on the intersection of $\cI_{N\!,\tspb\kappa}$ and the domain of the mapping $\Haop$.  We have presently no use for that direction so we leave it for potential future interest. 

Next we address the issue of uniqueness.
We restrict our attention to distributions that are also stationary with respect to the translation $\tau$:
\eq{
(\tau I)_k \coloneqq I_{k-1} \quad \text{for $I = (I_k)_{k\in\Z}$}.
}
The operator $\tau$ extends to any $N$-tuple of sequences in the obvious way:
\eq{ 
\tau I^{\parng{1}{N}} \coloneqq (\tau I^1,\dots,\tau I^N).
}
We say a probability measure $\mu$ on $(\R^\Z)^N$ is \textit{shift-stationary} if $\mu(\cB) = \mu(\tau^{-1}\cB)$ for every Borel set $\cB \subset (\R^\Z)^N$.
Additionally, a shift-stationary $\mu$ is called \textit{shift-ergodic} if $\mu(\cB)\in\{0,1\}$ whenever $\cB = \tau^{-1}\cB$.
Since the Ces\`aro limits $\ces(I^i)$ from \eqref{ces_def} are preserved under $\tau$, these limits must be deterministic under any shift-ergodic measure $\mu$ on  $\cI_{N\!,\tspb\kappa}$.
It turns out this is enough to separate ergodic components, as the next theorem explains.


Whenever $\mu$ is a probability measure on $\cI_{N\!,\tspb\kappa}$ satisfying
\be \label{log_ass}
\sum_{i=1}^N \,\int_{\cI_{N\!,\tspb\kappa}}|\log I^i_0|\ \mu(\dd I^{\parng{1}{N}}) < \infty,
\ee
define the following average for each $i\in\lzb1,N\rzb$:
\eq{
\cesm{\mu}{i} \coloneqq \int_{\cI_{N\!,\tspb\kappa}} \log I_0^i \ \mu(\dd I^{\parng{1}{N}}).
}

\begin{theorem} \label{uniq_thm}  
Assume \eqref{ran_ass}.
Let  $\kappa_1,\dots,\kappa_N$ be strictly greater than $\kappa$ in \eqref{ran_assb}. 

\begin{enumerate}[label={\rm(\alph*)},ref={\rm(\alph*)}]  \itemsep=3pt 

\item \label{uniq.a}  There exists at most one    shift-ergodic probability measure $\mu$ on $\cI_{N\!,\tspb\kappa}$ such that \eqref{log_ass} holds, 
\be \label{uniqT}
\mu\circ\Taop^{-1}=\mu \qquad \text{and} \qquad \cesm{\mu}{i} = \kappa_i \quad \text{for each $i\in\lzb 1, N \rzb$}.
\ee 
If $\brvcp^{\parng{1}{N}}$ has distribution $\mu$   and $\cesm{\mu}{i} = \cesm{\mu}{j}$, then $\brvcp^i = \brvcp^j$ almost surely.

\item \label{uniq.b}  Assume further  that $\kappa_1,\dots,\kappa_N$ are all  distinct.  
 Then   there exists at most one    shift-ergodic probability measure $\nu$ on $\cI_{N\!,\tspb\kappa}$ such that \eqref{log_ass} holds, 
\be \label{uniqS}
\nu\circ\Saop^{-1}=\nu \qquad \text{and} \qquad \cesm{\mu}{i} = \kappa_i \quad \text{for each $i\in\lzb 1, N \rzb$}.
\ee 
\end{enumerate} 

%
%
\end{theorem}

The second claim of part~\ref{uniq.a} is not valid for $\Saop$.  In the inverse-gamma case the components of an $\Saop$-invariant measure  are independent, regardless of their means (Theorem~\ref{thm-I} below). 

We prove the uniqueness in part~\ref{uniq.a} by  a version of a contraction argument (Proposition~\ref{rho_prop}) originally due to \cite{chan-94}, earlier adapted to the polymer setting in \cite{janj-rass-20-jsp}. 
From this  we deduce the uniqueness in part~\ref{uniq.b} by appeal to Theorem~\ref{twm_thm} and Lemma~\ref{invc_lem}. 
Recall from \cite[Sec.~9.4]{gray09}
 the  ``rho-bar'' distance between  shift-stationary probability measures $\mu_1$ and $\mu_2$ on $\cI_{N\!,\tspb\kappa}$: 
\eeq{ \label{rho_bar}
\bar\rho(\mu_1,\mu_2) \coloneqq \inf_{(\brvcp^{\parng{1}{N}},\brvcpp^{\parng{1}{N}})} \sum_{i=1}^N\E|\log \brvcp_0^i-\log \brvcpp_0^i|,
}
where the infimum is over couplings $(\brvcp^{\parng{1}{N}},\brvcpp^{\parng{1}{N}})=(\brvcp^{\parng{1}{N}}_k,\brvcpp^{\parng{1}{N}}_k)_{k\in\Z}$ such that 
\begin{enumerate}[label=\textup{(\roman*)}]

\item \label{rho_bar1}  $\brvcp^{\parng{1}{N}}$ has distribution $\mu_1$ and $\brvcpp^{\parng{1}{N}}$ has distribution $\mu_2$; and

\item \label{rho_bar2} the joint distribution of $(\brvcp^{\parng{1}{N}},\brvcpp^{\parng{1}{N}})$ on $\cI_{2N\!,\tspb\kappa}$ is shift-stationary.

\end{enumerate}
We can always assume that  these couplings are defined on the same probability space $\OSP$ as the random noise $\Yw$.

\begin{remark}[Ergodic case] \label{inf_rem}
If both $\mu_1$ and $\mu_2$ are also shift-ergodic, then the infimum is achieved by a coupling for which \ref{rho_bar2} is upgraded to 
shift-ergodic.
See the proof of \cite[Thm.~9.2]{gray09}.
\qedex\end{remark}

Since the metric \eqref{rho_bar} is defined only for shift-stationary distributions, the following facts need to be checked.

\begin{lemma} \label{sta_pres} The following statements hold.
\begin{enumerate}[label=\textup{(\alph*)},ref=\textup{(\alph*)}]

\item \label{sta_prea}
$\tau\circ\Daop = \Daop\circ\tau$ as maps $\cI_N^\uparrow\to\cI_N^\uparrow$.

\item \label{sta_presa}
If $\mu$ is a shift-stationary probability measure on $\cI_{N}^\uparrow$, then $\mu\circ\Daop^{-1}$ is also shift-stationary.
The same holds for  shift-ergodicity.

\item  \label{sta_preb}
$\tau\circ\Saop_{\Yw} = \Saop_{\tau\Yw} \circ \tau$ and $\tau\circ\Taop_{\Yw} = \Taop_{\tau\Yw} \circ \tau$
as maps $\cI_{N\!,\tspb\kappa}\to\cI_{N\!,\tspb\kappa}$, for any $\Yw\in\cI$ with $\ces(\Yw) = \kappa$.

\item \label{sta_presb}
Assume \eqref{ran_ass}.
If $\mu$ is a shift-stationary probability measure on $\cI_{N\!,\tspb\kappa}$,
then $\mu\circ\Saop^{-1}$ and $\mu\circ\Taop^{-1}$ are also shift-stationary.  The same holds for  shift-ergodicity.

\end{enumerate}
\end{lemma}


\begin{proof}   
Part~\ref{sta_presa} is immediate from part~\ref{sta_prea}.
Similarly, part~\ref{sta_presb} follows from part~\ref{sta_preb}, since the product of an i.i.d.~distribution (namely, that of $(W_k)_{k\in\Z}$) and a stationary/ergodic one is stationary/ergodic.
So we just prove parts~\ref{sta_prea} and \ref{sta_preb}.

Begin by showing the following three identities for any $(\Yw,I)\in\cI_2^\uparrow$:
\eeq{ \label{tau_out}
\Sop(\tau \Yw,\tau I) = \tau \Sop(\Yw,I), \quad
\Rop(\tau \Yw,\tau I) = \tau \Rop(\Yw,I), \quad
\Dop(\tau \Yw,\tau I) = \tau \Dop(\Yw,I), \quad
}
The first of these is clear from \eqref{J_def}: replacing every $k$ with $k-1$ on the right-hand side yields $J_{k-1}$.
Then the other two identities in \eqref{tau_out} follow by applying similar logic to \eqref{I_W_def}.

The third identity in \eqref{tau_out} easily extends by induction: for any $i\geq2$, if we assume that $\tau\circ\Dop^{(i-1)} = \Dop^{(i-1)}\circ\tau$, then
\eeq{ \label{tau_out_2}
\tau\Dop^{(i)}(I^{\parng{1}{i}})
&\stackref{repeat_dog}{=} \tau\Dop\big(I^1,\Dop^{(i-1)}(I^{\parng{2}{i}})\big) \\
&\stackref{tau_out}{=} \Dop\big(\tau I^1,\tau\Dop^{(i-1)}(I^{\parng{2}{i}})\big) \\
&\stackrefp{tau_out}{=} \Dop\big(\tau I^1,\Dop^{(i-1)}(\tau I^{\parng{2}{i}})\big)
\stackref{repeat_dog}{=} \Dop^{(i)}(\tau I^{\parng{1}{i}}).
}
Now $\tau\circ\Daop = \Daop\circ\tau$ follows by applying \eqref{tau_out_2} to each coordinate in \eqref{Daop_def}.
Similarly, the identity $\tau\circ\Taop_{\Yw} = \Taop_{\tau\Yw}\circ\tau$ is obtained by applying $\tau \Dop(\Yw,I) = \Dop(\tau \Yw,\tau I)$ to each coordinate in \eqref{Taop_def}.
Moreover, the $N=1$ case of $\tau\circ\Saop_{\Yw} = \Saop_{\tau\Yw}\circ\tau$ is handled, since in that case $\Saop_{\Yw}(I) = \Taop_{\Yw}(I) = \Dop(\Yw,I)$.
The general case follows from induction:
 if we assume that $\tau\circ\Saop_{\Yw} = \Saop_{\tau\Yw}\circ\tau$ on $\cI_{N-1,\kappa}$, then
\eq{
\tau\Saop_{\Yw}(I^{\parng{1}{N}})
&\stackref{Saop_alt}{=} \big(\tau\Dop(\Yw,I^1),\tau\Saop_{\Rop(\Yw,I^1)}(I^{\parng{2}{N}})\big) \\
&\stackrefp{Saop_alt}{=} \big(\tau\Dop(\Yw,I^1),\Saop_{\tau\Rop(\Yw,I^1)}(\tau I^{\parng{2}{N}})\big) \\
&\stackref{tau_out}{=} \big(\Dop(\tau\Yw,\tau I^1),\Saop_{\Rop(\tau\Yw,\tau I^1)}(\tau I^{\parng{2}{N}})\big)
\stackref{Saop_alt}{=} \Saop_{\tau\Yw}(\tau I^{\parng{1}{N}}). \qedhere
}
\end{proof}

\begin{proposition}\label{rho_prop}
	Assume \eqref{ran_ass}. Let $\mu_1$ and $\mu_2$ be shift-ergodic probability measures on $\cI_{N, \kappa}$ that satisfy \eqref{log_ass}.
	Then  
	\begin{equation}\label{rho_ineq}
	\bar\rho(\mu_1\circ\Taop^{-1},\mu_2\circ\Taop^{-1})\leq\bar\rho(\mu_1,\mu_2).
	\end{equation} 
	Furthermore, if  $\mu_1\neq\mu_2$ and $\cesm{\mu_1}{i}=\cesm{\mu_2}{i}$ for each $i\in\lzb1,N\rzb$, then this inequality is strict. 
\end{proposition}

\begin{proof}  
Let   $\brvcp^{\parng{1}{N}} = (\brvcp^1,\dots,\brvcp^N)$ and $\brvcpp^{\parng{1}{N}}=(\brvcpp^1,\dots,\brvcpp^N)$ be $\cI_{N, \kappa}$-valued random variables that are independent of $\Yw$  and  satisfy conditions~\ref{rho_bar1} and \ref{rho_bar2} for the definition \eqref{rho_bar}.
By Remark~\ref{inf_rem}, we may assume that
\eq{
\bar\rho(\mu_1,\mu_2) = \sum_{i=1}^N\E|\log \brvcp_0^i-\log \brvcpp_0^i|  
}
and that the joint distribution of $(\brvcp^{\parng{1}{N}},\brvcpp^{\parng{1}{N}})$ is shift-ergodic.
Set  $\wt \brvcp^i = \Dop\big(\Yw,\brvcp^i)$ and  $\wt \brvcpp^i = \Dop\big(\Yw,\brvcpp^i)$. 
Then $(\wt\brvcp^{\parng{1}{N}},\wt\brvcpp^{\parng{1}{N}})$ is a valid coupling for bounding $\bar\rho(\mu_1\circ\Taop^{-1},\mu_2\circ\Taop^{-1})$, by Lemma~\hyperref[sta_presb]{\ref*{sta_pres}\ref*{sta_presb}}.
For \eqref{rho_ineq}  it suffices to show that
\begin{equation}\label{sts_rho}
\sum_{i=1}^N\E|\log \wt\brvcp_0^i-\log \wt\brvcpp_0^i| \leq
\sum_{i=1}^N\E|\log \brvcp_0^i-\log \brvcpp_0^i|. 
\end{equation}
We show that each summand on the left is dominated by the corresponding summand on the right. 

To begin, consider the majorizing process $\brvcz^{\parng{1}{N}}$ defined as $ \brvcz^i_k=\brvcp^i_k \vee  \brvcpp^i_k$.
We  have
\eeq{ \label{gr65}
|\log \brvcp_0^i - \log \brvcpp_0^i| = 2\log \brvcz_0^i - \log \brvcp_0^i - \log \brvcpp_0^i.
}
By shift-ergodicity,  
\eq{ 
\E[\log \brvcz_0^i] = \lim_{n\to\infty}\frac{1}{n}\sum_{k=-n+1}^{0}\log \brvcz_k^i = \ces(\brvcz^i) \quad \mathrm{a.s.}
}
and similarly  $\E[\log \brvcp_0^i] = \ces(\brvcp^i)$ and $\E[\log \brvcpp_0^i] = \ces(\brvcpp^i)$.
Taking expectation in \eqref{gr65} yields
\eeq{ \label{gr75}
\E|\log \brvcp_0^i - \log \brvcpp_0^i| = 2\ces(\brvcz^i) - \ces(\brvcp^i) - \ces(\brvcpp^i) \quad \mathrm{a.s.}
}
Since  $\ces(\brvcz^i) \geq \ces(\brvcp^i)\vee\ces(\brvcpp^i)>\kappa$, the sequence $\wt\brvcz^i = \Dop(\Yw,\brvcz^i)$ is well-defined and 
  by Lemma \hyperref[mon_lemb]{\ref*{mon_lem}\ref*{mon_lemb}} satisfies 
$\wt \brvcz^i \ge \wt \brvcp^i \vee \wt \brvcpp^i$.
This leads to the following inequality:
\eeq{ \label{fq34}
|\log \wt \brvcp_0^i - \log \wt \brvcpp_0^i| &= 2\log(\wt \brvcp_0^i \vee \wt \brvcpp_0^i) - \log \wt \brvcp_0^i - \log \wt \brvcpp_0^i \\
&\leq 2\log \wt \brvcz_0^i - \log \wt \brvcp_0^i - \log \wt \brvcpp_0^i.
}
By joint shift-ergodicity of $(\Yw,\brvcp^{\parng{1}{N}},\brvcpp^{\parng{1}{N}})$, we further have 
\eq{
\ces(\wt \brvcz^i) 
&\stackrefpp{ces_def}{tau_out}{=}  \lim_{n\to\infty}\frac{1}{n}\sum_{k=-n+1}^{0}\log \wt\brvcz_k^i  
=\lim_{n\to\infty}\frac{1}{n}\sum_{k=-n+1}^{0}\log \Dop(\Yw,\brvcz^i)_k  \\
&\stackref{tau_out}{=} \lim_{n\to\infty}\frac{1}{n}\sum_{k=-n+1}^{0}\log \Dop(\tau^{-k}\Yw,\tau^{-k}\brvcz^i)_0
= \E[\log\wt \brvcz^i_0] \quad \mathrm{a.s.}
}
Similarly   $\ces(\wt \brvcp^i) =\E[\log \wt \brvcp_0^i]$ and $\ces(\wt \brvcpp^i) = \E[\log \wt \brvcpp_0^i]$ almost surely.  Now  \eqref{fq34} leads to
	\eeq{ \label{tp17}
	\E|\log \wt\brvcp_0^i-\log \wt\brvcpp_0^i|
	&\leq 2\ces(\wt\brvcz^i) -  \ces(\wt\brvcp^i) -  \ces(\wt\brvcpp^i) \\
	&= 2\ces(\brvcz^i) -  \ces(\brvcp^i) -  \ces(\brvcpp^i)
	\overset{\eqref{gr75}}{=} \E|\log \brvcp_0^i-\log \brvcpp_0^i|,
	}
where the penultimate equality is due to Lemma~\ref{ces_lem}.
This completes the proof of \eqref{rho_ineq}.

For the second part of the proposition,   
  we show that \eqref{sts_rho}  is strict for at least one summand. 

\begin{claim} \label{or43_clm}
If $\mu_1\neq\mu_2$ and $\cesm{\mu_1}{i}=\cesm{\mu_2}{i}$ for each $i\in\lzb 1, N \rzb$, then there is some $i\in\lzb 1, N \rzb$ and $\ell_1,\ell_2\in\Z$ such that
\eeq{ \label{or43}
\P\big(\{\brvcp_{\ell_1}^i>\brvcpp_{\ell_1}^i\}\cap\{\brvcp_{\ell_2}^i<\brvcpp_{\ell_2}^i\}\big) > 0.
}
\end{claim}

\begin{proofclaim}
Suppose   that the claim were false.
Then with probability one, for each $i$ one of the following two events occurs:
\eq{
\bigcap_{\ell\in\Z}\{\brvcp_\ell^i\leq \brvcpp_\ell^i\} \quad \text{or} \quad
\bigcap_{\ell\in\Z}\{\brvcp_\ell^i\geq \brvcpp_\ell^i\}.
}
Each of these events is invariant under translation, so by shift-ergodicity, at least one occurs with probability one.
But because $\E[\log \brvcp_k^i] = \E[\log \brvcpp_k^i]$, this forces $\brvcp_k^i = \brvcpp_k^i$ for all $k\in\Z$, which contradicts the assumption that $\mu_1 \neq \mu_2$.
\end{proofclaim}

Let $i$, $\ell_1$, $\ell_2$ be as in Claim~\ref{or43_clm}.
By \eqref{or43} and shift-ergodicity, with probability one there are infinitely many $k\ge\ell_1\vee\ell_2$ such that the following event occurs:
\[
\{\brvcp^i_{\ell_1-k}>\brvcpp_{\ell_1-k}^i\} \cap \{\brvcp^i_{\ell_2-k}<\brvcpp_{\ell_2-k}^i\}
=\{\brvcz^i_{\ell_1-k}>\brvcpp_{\ell_1-k}^i\} \cap \{\brvcz^i_{\ell_2-k}>\brvcp_{\ell_2-k}^i\}.
\]
On this intersection, by Lemma \hyperref[mon_lemb]{\ref*{mon_lem}\ref*{mon_lemb}}, $\wt \brvcz_0^i > \wt \brvcpp_0^i \vee \wt \brvcp_0^i$.
The inequality in   \eqref{tp17} is now strict.
\end{proof}

%
%
%

\begin{proof}[Proof of Theorem~\ref{uniq_thm}]  

Part~\ref{uniq.a}.  
Proposition~\ref{rho_prop} implies the uniqueness claim. 

Suppose  $\kappa_a=\kappa_{a+1}$. (We can always permute the sequence-valued components to make  the coinciding $\kappa_i$-values adjacent.)    Let shift-ergodic  $\mu$ satisfy \eqref{uniqT}.   Define $\mu'$ on $\cI_{N\!,\tspb\kappa}$ with the same means  
$\cesm{\mu'}{i}=\cesm{\mu}{i}$ 
by 
\[  \int_{\cI_{N\!,\tspb\kappa}}  f(y^{\parng{1}{N}})\,\mu'(\dd y^{\parng{1}{N}})  
=   \int_{\cI_{N\!,\tspb\kappa}}  f(x^{\parng{1}{a}}, x^a, x^{a+2,N} )\,\mu(\dd x^{\parng{1}{N}}). 
\]  
In other words, project $\mu$ to the components $(x^i)^{i\ne a+1}$ and then duplicate $x^a$ to create the (new) component $x^{a+1}$.  These operations preserve shift-ergodicity.    Projection commutes with the parallel mapping, and hence the $\mu$-marginal distribution of $(X^i)^{i\ne a+1}$ is still invariant under $\Taop$.  Duplicating the $X^a$-component also commutes with the parallel mapping, and thereby $\mu'$ is also invariant. The uniqueness part   implies that $\mu=\mu'$, in other words, $\mu(X^a=X^{a+1})=1$. 

\smallskip 

Part~\ref{uniq.b}. 
 Now assume  that the $\kappa_1,\dotsc,\kappa_N$ are all distinct.  
 Suppose $\nu_1$ and $\nu_2$ are shift-ergodic  probability  measures on $\cI_{N\!,\tspb\kappa}$  that satisfy \eqref{uniqS}.  
By permuting the sequence-valued components we can assume $\kappa<\kappa_1<\dotsm<\kappa_N$.  Then the measures $\nu_1$ and $\nu_2$ are supported by the space $\cI_{N\!,\tspb\kappa}^\uparrow$ defined in \eqref{kappa_up}, which is the domain of the mapping $\Daop$.  
Then $\mu_1=\nu_1\circ\Daop^{-1}$ and $\mu_2=\nu_2\circ\Daop^{-1}$ are probability  measures on $\cI_{N\!,\tspb\kappa}^\uparrow$ that satisfy \eqref{uniqT}. Here we use the fact that $\Daop$ preserves Ces\`aro means. Hence  $\mu_1=\mu_2$. By Lemma~\ref{invc_lem}\ref{invc.a}, $\mu_1(\cH_N)=\mu_2(\cH_N)=1$.    Thus for $i\in\{1,2\}$ we can define measures $\nu'_i=\mu_i\circ\Haop^{-1}$ on $(\R_{>0}^\Z)^N$ that also agree.    Again by Lemma~\ref{invc_lem}\ref{invc.a}, $\nu'_i= (\nu_i\circ\Daop^{-1}) \circ\Haop^{-1}  = \nu_i \circ(\Haop\circ\Daop)^{-1} =\nu_i$. 
\end{proof}

\subsection{Sequential process and parallel process} \label{2_proc_sec}

%
As the final step towards the characterization of the distribution of the Busemann process, 
we construct  Markov processes from  the previously defined transformations, by using fresh i.i.d.~driving weights $\Yw$ at each step.
Return to the polymer setting of \eqref{ranenv} with a slightly weaker moment assumption:
\be \label{weak_ass} \begin{aligned} 
&\text{the weights $\Yw=(\Yw_x)_{x\tsp\in\tsp\Z^2}$ are strictly positive, i.i.d.~random variables on $\OSP$}\\
& \text{such that $\Yw_x(\w) = \Yw_\zevec(\shift_x\w)$ and $\E\abs{\log \Yw_0}<\infty$.  
Let $\kappa=\E[\log\Yw_\zevec]$. } 
\end{aligned}\ee
Let $\Yw(t)=(\Yw_{(k,t)})_{k\in\Z}$ denote the sequence of weights at level $t\in\Z$.  
Almost surely $W(t)\in\cI$ with $\ces(W(t)) = \kappa$ for every $t\in\Z$.


Pick an initial time $t_0\in\Z$ and let $\brvY^{\parng{1}{N}}(t_0)$ and $\brvcp^{\parng{1}{N}}(t_0)$ be initial states in the space $\cI_{N\!,\tspb\kappa}$ from \eqref{in_ces_def}.
These states may be random but are presumed independent of the random field $\Yw$.   
 Then  the  {\it sequential process} $\brvY^{\parng{1}{N}}(\aabullet)$  is defined for integer times $t\geq t_0+1$ by the iteration 
\be\label{g:multil}  \brvY^{\parng{1}{N}}(t)=\Saop_{\Yw(t)}(\brvY^{\parng{1}{N}}(t-1)\big). \ee
Similarly  the  {\it parallel process}  $\brvcp^{\parng{1}{N}}(\aabullet)$ is defined by
\be\label{g:cpl}
\brvcp^{\parng{1}{N}}(t)=\Taop_{\Yw(t)}(\brvcp^{\parng{1}{N}}(t-1)). 
\ee
Since $\Saop_\Yw$ and $\Taop_\Yw$ both preserve Ces\`aro limits (recall \eqref{S_pres} and \eqref{T_pres}), the processes $\brvY(\aabullet)$ and $\brvcp(\aabullet)$ can be viewed as discrete-time Markov chains on the  state space $\cI_{N\!,\tspb\kappa}$ or on the smaller   space $\cI_{N\!,\tspb\kappa}^\uparrow$ from \eqref{kappa_up}.   

We begin by stating the immediate corollaries of Theorems~\ref{twm_thm} and \ref{uniq_thm}.  

\begin{corollary} \label{twm_cor}  Assume \eqref{weak_ass}.   If the sequential process has an invariant distribution  $\nu$  on the space $\cI_{N\!,\tspb\kappa}^\uparrow$, then  $\mu=\nu\circ\Daop^{-1}$ is invariant for the parallel process.
\end{corollary}

 As before, the logarithmic mean of the $i$th component under a shift-stationary measure $\mu$ is  denoted by $\cesm{\mu}{i}=\int_{\cI_{N\!,\tspb\kappa}} \log x^i_0\,\mu(\dd x^{\parng{1}{N}})$. 
   
\begin{corollary} \label{uniq_cor}  
Assume \eqref{weak_ass} and let  $\kappa_1,\dots,\kappa_N > \kappa$. 

\begin{enumerate}[label={\rm(\alph*)},ref={\rm(\alph*)}]  \itemsep=3pt 

\item \label{uniq_cor.a}  The parallel process has   at most one    shift-ergodic invariant  measure $\mu$ on $\cI_{N\!,\tspb\kappa}$ such that $\cesm{\mu}{i} = \kappa_i$ for each $i\in\lzb 1, N \rzb$.  
 
\item \label{uniq_cor.b}  Assume further  that $\kappa_1,\dots,\kappa_N$ are distinct.  
 Then   the sequential process has   at most one    shift-ergodic invariant  measure $\nu$ on $\cI_{N\!,\tspb\kappa}$ such that \eqref{log_ass} holds and  $\cesm{\nu}{i} = \kappa_i$ for each $i\in\lzb 1, N \rzb$.  \end{enumerate} 
\end{corollary}

Finally we connect this development back to the Busemann process.
Recall from Section~\ref{distB_sec} the notation for sequences of exponentiated horizontal nearest-neighbor Busemann increments: 
\begin{subequations} \label{IJnot}
\eeq{\label{IB5.5}  
\brvI^{\chdir\sig}(t) &=  (  {\brvI}_k^{\chdir\sig}(t) )_{k\in\Z} \,, 
 \quad  \text{where} \quad {\brvI}_k^{\chdir\sig}(t)  =  e^{\Bus^{\chdir\sig}_{(k-1,t),\tsp (k,t)}}.
}
%
We use similar notation for vertical increments:
\eeq{ \label{IB5.8}
\brvJ^{\chdir\sig}(t) = (\brvJ^{\chdir\sig}_k(t))_{k\in\Z}, \quad \text{where} \quad
\brvJ^{\chdir\sig}_k(t) =   e^{\Bus^{\chdir\sig}_{(k,t-1),\tsp (k,t)}}. 
}
\end{subequations}
The lemma below checks that exponentiated Busemann increments respect the dynamics of the update map.

\begin{lemma} \label{lemresp}
There is a full-probability event on which the following statements hold simultaneously for all $\xi\in\,]\evec_2,\evec_1[$, $\sigg\in\{-,+\}$, and $t\in\Z$:
\begin{enumerate}[label=\textup{(\alph*)}]

\item \label{lemrespa}
$\ces(\brvI^{\chdir\sig}(t)) = \nabla\gpp(\xi\sigg)\cdot\evec_1$

\item \label{lemrespb}
$\brvI^{\chdir\sig}(t)=\Dop( \Yw(t), \brvI^{\chdir\sig}(t-1))$ and $\brvJ^{\chdir\sig}(t)=\Sop( \Yw(t), \brvI^{\chdir\sig}(t-1))$.
\end{enumerate}
\end{lemma}

\begin{proof}
Part~\ref{lemrespa} follows from Theorem~\ref{all_lln}:
\eq{  
\ces({\brvI}^{\chdir\sig}(t))
=  \lim_{n\to\infty} \frac1n \sum_{k=-n+1}^0 \log{\brvI}_k^{\chdir\sig}(t) 
\stackref{coc1}{=}\lim_{n\to\infty} n^{-1}\Bus^{\xi\sig}_{(-n,t),(0,t)}
\stackref{bfanfx}{=} \nabla\gpp(\xi\sigg)\cdot\evec_1.
}
For part~\ref{lemrespb}, observe that in the notation of \eqref{IJnot}, additivity \eqref{coc1} and   recovery \eqref{B_reco}  are re-expressed as 
		\eq{ 
		\brvJ^{\chdir\sig}_k(t) \tspa\brvI^{\chdir\sig}_k(t-1)=\brvI^{\chdir\sig}_k(t) \tspa\brvJ^{\chdir\sig}_{k-1}(t) 
		\qquad\text{and}\qquad
		 \Yw_{(k,t)}^{-1} =  \brvI^{\chdir\sig}_k(t)^{-1}+ \brvJ^{\chdir\sig}_k(t)^{-1}. 
		}
From these one deduces 
\eeq{ \label{zb72dw}    \brvJ^{\chdir\sig}_k(t)= \Yw_{(k,t)}\biggl(1+  \frac{\brvJ^{\chdir\sig}_{k-1}(t)}{\brvI^{\chdir\sig}_k(t-1)}   \biggr)\qquad  	\text{and}  \qquad 
 \brvI^{\chdir\sig}_k(t)= \Yw_{(k,t)}\biggl(1+  \frac{\brvI^{\chdir\sig}_k(t-1)}{\brvJ^{\chdir\sig}_{k-1}(t)}   \biggr).
}
In other words, the recursions \eqref{cJ-rec} and \eqref{cJ8} required by Lemma~\ref{lm:D13} are satisfied. 
That lemma's last remaining hypothesis \eqref{b6x3dd} holds almost surely---and with equality---simply because $(\log J_k^{\xi\sig}(t))_{k\in\Z}$ are identically distributed (thanks to translation invariance \eqref{coc2}) and hence tight.
A priori the almost-sure event $\{\varliminf_{k\to-\infty} |k|^{-1}\log J_k^{\xi\sig}(t) \leq 0\}$ might depend on $\xi$; but thanks to monotonicity \eqref{B_mono2}, it holds for all $\xi$ as soon as it holds for a countable dense set of $\xi$.
Therefore, Lemma~\ref{lm:D13} provides the desired conclusion.
\end{proof}

We can now state and prove a precise version of Theorem~\ref{cpthm} for the Busemann process.
Given directions $\xi_1,\dotsc, \xi_N\in\,]\evec_2,\evec_1[\,$  and signs $\sigg_1,\dotsc,\sigg_N\in\{-,+\}$, we write ${\brvI}^{(\chdir\sig)_{\parng{1}{N}}}(t)$ for the $N$-tuple of sequences $\big({\brvI}^{\chdir_1\sig_1}(t), {\brvI}^{\chdir_2\sig_2}(t),  \dotsc, {\brvI}^{\chdir_N\sig_N}(t)\big)$.


\begin{theorem}\label{B_thm9}
Assume \eqref{ranenv} 
 and let $\kappa=\E[\log\Yw_\zevec]$. 
 \begin{enumerate}[label=\textup{(\alph*)}]

\item \label{B_thm9a}
$\{{\brvI}^{(\chdir\sig)_{\parng{1}{N}}}(t):\, t\in\Z\}$ is a stationary version of the parallel process on the state space $\cI_{N\!,\tspb\kappa}$.

\item \label{B_thm9b}
The law of ${\brvI}^{(\chdir\sig)_{\parng{1}{N}}}(0)$ is the unique shift-ergodic invariant measure of Corollary~\ref{uniq_cor}\ref{uniq_cor.a}  determined by  $\kappa_i=\nabla\gpp(\xi_i\sigg_i)\cdot\evec_1$  for   $i\in\lzb1,N\rzb$. In particular, said invariant distribution exists. 

\end{enumerate}
\end{theorem}

\begin{proof}[Proof of Theorem~\ref{B_thm9}]
Part~\ref{B_thm9a}. 
By Lemma~\hyperref[lemrespa]{\ref*{lemresp}\ref*{lemrespa}}, Ces\`aro averages are deterministic and constant in $t$:
 \eeq{ \label{8cfk4}
\P\Big(\ces\big({\brvI}^{(\chdir\sig)_{\parng{1}{N}}}(t)\big) =\nabla\gpp(\xi_i\sigg_i)\cdot\evec_1 \text{ for all $i\in\lzb1,N\rzb$, $t\in\Z$}\Big) = 1.
 }
Thanks to \eqref{inmono1}, these deterministic values all exceed $\kappa$: 
\eq{
\nabla\gpp(\xi_i\sigg_i)\cdot\evec_1>\E[\log\Yw_\zevec]=\kappa \quad \text{for $i\in\lzb1,N\rzb$}.
} 
It follows from the two previous displays that
${\brvI}^{(\chdir\sig)_{\parng{1}{N}}}(t)$ is almost surely a member of the space $\cI_{N\!,\tspb\kappa}$ from \eqref{in_ces_def}.
Lemma~\hyperref[lemrespb]{\ref*{lemresp}\ref*{lemrespb}} ensures $({\brvI}^{(\chdir\sig)_{\parng{1}{N}}}(t))_{t\in\Z}$ obeys the parallel process \eqref{g:cpl}, where \eqref{B-ind} supplies the independence of $\Yw(t)$ and ${\brvI}^{(\chdir\sig)_{\parng{1}{N}}}(t-1)$ that is assumed in \eqref{g:cpl}.
Finally, $({\brvI}^{(\chdir\sig)_{\parng{1}{N}}}(t))_{t\in\Z}$ is stationary in $t$ by the translation invariance of the Busemann process recorded in \eqref{coc2}.

We prove part~\ref{B_thm9b} in three steps.

 \smallskip

 \textit{Step 1.} We perform an ergodic decomposition.
 Let $\Pscr_e(\cI_{N\!,\tspb\kappa})$ denote the space of shift-ergodic probability measures on $\cI_{N\!,\tspb\kappa}$.
 Write $\mu_0$ for the distribution of ${\brvI}^{(\chdir\sig)_{\parng{1}{N}}}(0)$.  
 This is a shift-stationary measure because of translation invariance of the Busemann process.
 Therefore, by the ergodic decomposition theorem, there exists a probability measure $P$ on $\Pscr_e(\cI_{N\!,\tspb\kappa})$ such that $\mu_0 = \int_{\Pscr_e(\cI_{N\!,\tspb\kappa})}\mu\, P(\dd\mu)$.
 Specializing \eqref{8cfk4} to the case $t=0$, we have
 \eeq{ \label{7fvvc}
 \mu_0\{I^{\parng{1}{N}}\in\cI_{N\!,\tspb\kappa}:\, \ces(I^i)=\nabla\gpp(\xi_i\sigg_i)\cdot\evec_1 \text{ for $i\in\lzb1,N\rzb$}\} = 1.
 }
 Moreover, the integrability assumption \eqref{log_ass} holds with $\mu=\mu_0$ since Busemann functions are integrable (see the sentence containing \eqref{EB}).
Consequently, \eqref{log_ass} holds for $P$-almost every $\mu$, so we are allowed to write $\ces_i(\mu)$. 
Now \eqref{7fvvc} implies
 \eeq{ \label{cesb7}
 \cesm{\mu}{i} =\nabla\gpp(\xi_i\sigg_i)\cdot\evec_1 \quad \text{for $i\in\lzb1,N\rzb$.}
 }
 
 \smallskip
 
  \textit{Step 2.} We show that $P\{\mu:\, \mu\circ\Taop^{-1}=\mu\}=1$.
  For any Borel set $\cB\subset\cI_{N\!,\tspb\kappa}$,
  \eeq{ \label{nfeq}
   \int_{\Pscr_e(\cI_{N\!,\tspb\kappa})}\mu(\cB)\, P(\dd\mu) =
  \mu_0(\cB) 
  &\stackrefp{push_T}{=} \E\mu_0\big(\Taop_{\Yw}^{-1}(\cB)\big) \quad \text{by part \ref{B_thm9a}} \\
  &\stackrefp{push_T}{=}  \int_{\Pscr_e(\cI_{N\!,\tspb\kappa})}\E\mu\big(\Taop_{\Yw}^{-1}(\cB)\big)\, P(\dd\mu)  \\
  &\stackref{push_T}{=}  \int_{\Pscr_e(\cI_{N\!,\tspb\kappa})}[\mu\circ\Taop^{-1}](\cB)\, P(\dd\mu).
  }
  Recall from Lemma \hyperref[sta_presb]{\ref*{sta_pres}\ref*{sta_presb}} that $\mu\circ\Taop^{-1}$ is again a shift-ergodic measure on $\cI_{N\!,\tspb\kappa}$.
  Therefore, by uniqueness in the ergodic decomposition theorem, it follows from \eqref{nfeq} that for any bounded measurable function $f \colon \Pscr_e(\cI_{N\!,\tspb\kappa}) \to \R$,
  \eq{
   \int_{\Pscr_e(\cI_{N\!,\tspb\kappa})} f(\mu)\, P(\dd\mu)
   =  \int_{\Pscr_e(\cI_{N\!,\tspb\kappa})} f(\mu\circ\Taop^{-1})\, P(\dd\mu).
  }
  For instance, choose $f$ given by $f(\mu) = \bar\rho(\mu,\mu\circ\Taop)$, where $\bar\rho$ is the distance in \eqref{rho_bar}.
  This choice leads to
  \eq{
  \int_{\Pscr_e(\cI_{N\!,\tspb\kappa})} \bar\rho(\mu,\mu\circ\Taop^{-1})\, P(\dd\mu)
  =  \int_{\Pscr_e(\cI_{N\!,\tspb\kappa})} \bar\rho(\mu\circ\Taop^{-1},\mu\circ\Taop^{-1}\circ\Taop^{-1})\, P(\dd\mu).
  }
  By Proposition~\ref{rho_prop}, the integrand on the left-hand side pointwise dominates the integrand on the right-hand side.
  Hence $\bar\rho(\mu,\mu\circ\Taop^{-1}) = \bar\rho(\mu\circ\Taop^{-1},\mu\circ\Taop^{-1}\circ\Taop^{-1})$ for $P$-almost every $\mu$.
  Furthermore, since the parallel transformation preserves Ces\`aro limits (recall \eqref{T_pres}), it is always the case that $\cesm{\mu}{i} = \cesm{\mu\circ\Taop^{-1}}{i}$.
  Consequently, the last statement in Proposition~\ref{rho_prop} forces $\mu = \mu\circ\Taop^{-1}$ for $P$-almost every $\mu$.
 
  \smallskip
 
  \textit{Step 3.} We conclude that $\mu_0$ is shift-ergodic.
  Indeed, Step 2 says that $P$ places all its mass on shift-ergodic invariant measures satisfying \eqref{cesb7}.
  Theorem \hyperref[uniq.a]{\ref*{uniq_thm}\ref*{uniq.a}} says there is only one such measure, so it must be $\mu_0$.
 \end{proof}  
 
  \begin{proof}[Proof of Theorem~\ref{thm:78-63}] 
 There are four inequalities in \eqref{eq:78-63}.
 The fourth follows from the first by the recovery property \eqref{B_reco}, and the second and third inequalities already appear in \eqref{B_Mono}.
 So we just prove the first inequality in \eqref{eq:78-63}.
 
For $\zeta\prec\eta$ not belonging to the same linear segment of $\gpp$, we have $\nabla\gpp(\zeta+)\neq\nabla\gpp(\eta-)$. 
By \eqref{inmono}, this means $\nabla\gpp(\zeta+)\cdot\evec_1 > \nabla\gpp(\eta-)\cdot\evec_1$.
The recursion \eqref{Bmc} with $N=2$ says
\eq{ 
(I^{\zeta+}(t+1),I^{\eta-}(t+1)) 
&\stackrefp{Taop_def}{=} \Taop_{W(t+1)}(I^{\zeta+}(t),I^{\eta-}(t)) \\
&\stackref{Taop_def}{=} \big(D(W(t+1),I^{\zeta+}(t)),D(W(t+1),I^{\eta-}(t))\big).
}
By monotonicity \eqref{B_mono1}, we already know $I_k^{\zeta+}(t) \geq I_k^{\eta-}(t)$ for every $(k,t)\in\Z^2$.
Furthermore, for any given $t$, it cannot be the case that equality holds for every $k$, since 
\eq{
&\lim_{n\to\infty}\frac{1}{n}\sum_{k=-n+1}^0 \log I_k^{\zeta+}(t) 
\stackref{coc1}{=} \lim_{n\to\infty}\frac{1}{n} \Bus^{\zeta+}_{(-n,t),(0,t)}
\stackref{bfanfx}{=} \nabla\gpp(\zeta+)\cdot\evec_1 \\
&\qquad\qquad> \nabla\gpp(\eta-)\cdot\evec_1
\stackref{bfanfx}{=}  \lim_{n\to\infty}\frac{1}{n}\Bus^{\eta-}_{(-n,t),(0,t)}
\stackref{coc1}{=} \lim_{n\to\infty}\frac{1}{n}\sum_{k=-n+1}^0 \log I_k^{\eta-}(t). 
}
More specifically, for any positive integer $n$, there is $k_0\leq -n$ such that $I_{k_0}^{\zeta+}(t)>I_{k_0}^{\eta-}(t)$.
It now follows from Lemma \hyperref[mon_lemb]{\ref*{mon_lem}\ref*{mon_lemb}} that $I_k^{\zeta+}(t+1) > I_k^{\eta-}(t+1)$ for all $k\geq k_0$, in particular for $k\geq-n$.
Letting $n\to\infty$, we conclude that $I^{\zeta+}(t+1) > I^{\eta-}(t+1)$.
As $t$ is arbitrary, we have argued that $\Bus^{\zeta+}_{x-\evec_1,x} > \Bus^{\eta-}_{x-\evec_1,x}$ for all $x\in\Z^2$.
%
\end{proof}

\subsection{Discontinuities in the direction variable} 
 \label{sec:pf_disc}

This section proves Theorems~\ref{thm:51-32} and \ref{thm:allxy}. 
Given $x\in\Z^2$, consider the nearest-neighbor Busemann functions $\xi\mapsto\Bus^{\xi\pm}_{x-\evec_r,x}$. 
By monotonicity \eqref{B_mono}, discontinuity at the direction $\xi$ can only occur in one way:
 \be\label{51-25} \begin{aligned}   
 \Bus^{\xi-}_{x-\evec_1, x}\ne \Bus^{\xi+}_{x-\evec_1, x}
 \ &\Longleftrightarrow \ 
 \Bus^{\xi-}_{x-\evec_1, x}> \Bus^{\xi+}_{x-\evec_1, x}  \quad \text{and} \\
    \Bus^{\xi-}_{x-\evec_2, x}\ne \Bus^{\xi+}_{x-\evec_2, x}
 \ &\Longleftrightarrow \ 
 \Bus^{\xi-}_{x-\evec_2, x}< \Bus^{\xi+}_{x-\evec_2, x}.  \end{aligned}\ee
 By recovery \eqref{B_reco}, the two equivalences in \eqref{51-25} happen together or not at all.
   Call $x$ a {\it $\xi$-discrepancy point}  if the statements in \eqref{51-25} hold.    
 Denote the set of $\xi$-discrepancy points by
 \eq{ 
 \bD^\xi=\{x\in\Z^2:\,    \Bus^{\xi-}_{x-\evec_1,x}\ne \Bus^{\xi+}_{x-\evec_1,x} \}.
 }
  By observations just made, the definition is the same if $\evec_1$ is replaced with $\evec_2$.
  Theorem~\hyperref[thm:51-32a]{\ref*{thm:51-32}\ref*{thm:51-32a}} will be obtained from the combination of the next two propositions, which separately provide northeast and southwest propagation of discrepancy points.
  Recall that $y>x$ means $y\cdot\evec_1>x\cdot\evec_1$ and $y\cdot\evec_2>x\cdot\evec_2$.
  
  \begin{proposition} \label{nepg}
  The following holds almost surely:
  for all $\xi\in\,]\evec_2,\evec_1[\,$, if $x\in\bD^\xi$ and $y>x$, then $y\in\bD^\xi$.
  \end{proposition}
  
  \begin{proof}
  Recall the notation $I_k^{\xi\sig}(t) = e^{\Bus^{\xi\sig}_{(k-1,t),(k,t)}}$ and $W(t) = (W_{(k,t)})_{k\in\Z}$.
    Write $x = (k_0,t)$ so that the assumption $x\in\bD^\xi$ means $I_{k_0}^{\xi-}(t) > I_{k_0}^{\xi+}(t)$.
  As observed above, monotonicity \eqref{B_mono1} implies $I^{\xi-}(t)\geq I^{\xi+}(t)$.
The recursion \eqref{Bmc} with $N=2$ says that
\eq{ 
(I^{\xi-}(t+1),I^{\xi+}(t+1)) 
&\stackrefp{Taop_def}{=} \Taop_{W(t+1)}(I^{\xi-}(t),I^{\xi+}(t)) \\
&\stackref{Taop_def}{=} \big(D(W(t+1),I^{\xi-}(t)),D(W(t+1),I^{\xi+}(t))\big).
}
Therefore, Lemma \hyperref[mon_lemb]{\ref*{mon_lem}\ref*{mon_lemb}} shows that $I_k^{\xi-}(t+1) > I_k^{\xi+}(t+1)$ for all $k\geq k_0$.
That is, every $y = (k,t+1)$ with $k\geq k_0$ belongs to $\bD^\xi$.
Inducting on $t$ extends this to all $y>x$. 
  \end{proof}
  

For the second proposition, we must restrict to $\Udiff$, the subset of $]\evec_2,\evec_1[\,$ at which the shape function $\gpp$ is differentiable.

\begin{proposition} \label{lm:51-32}   
The following holds almost surely:
 for all $\xi\in\Udiff$, if $x\in\bD^\xi$, then there exists $z<x$ such that $z\in\bD^\xi$.
 \end{proposition} 
 
The proof 
is quite technical, so we postpone it until after proving Theorems~\ref{thm:51-32} and \ref{thm:allxy}.

 \begin{proof}[Proof of Theorem~\ref{thm:51-32}]  
From Propositions~\ref{nepg} and~\ref{lm:51-32}, the following statement holds almost surely:
for all $\xi\in\Udiff$, the set $\bD^\xi$ is either empty or the entire lattice $\Z^2$.
If we can also show that for all $\xi\notin\Udiff$, the set $\bD^\xi$ is the entire lattice, then both parts of the theorem will have been verified.
So the remainder the proof is to establish this second statement.

There are at most countably many nondifferentiability points, so it suffices to show that for a given $\xi\in\,]\evec_2,\evec_1[\,\setminus\Udiff$,  the set $\bD^\xi$ almost surely equals the entire lattice.  
To that end, note that homogeneity \eqref{homog} implies $\xi\cdot\nabla\gpp(\xi\pm) = \gpp(\xi)$ (see \cite[Lem.~4.6]{janjigian_rassoulagha20a_arxiv}).
In particular $\xi\cdot(\nabla\gpp(\xi-)-\nabla\gpp(\xi+))=0$.
But $\nabla\gpp(\xi-)\neq\nabla\gpp(\xi+)$ since $\xi\notin\Udiff$, so the latter identity must be a consequence of cancelation between a positive term and negative term (see Remark~\ref{dmnr}):
\eq{
\nabla\gpp(\xi-)\neq\nabla\gpp(\xi+) \quad \iff \quad
\nabla\gpp(\xi-)\cdot\evec_1 > \nabla\gpp(\xi+)\cdot\evec_1 \text{ and }
\nabla\gpp(\xi-)\cdot\evec_2 < \nabla\gpp(\xi+)\cdot\evec_2.
}
For an inner product with any direction other than $\xi$, these positive and negative terms cannot fully cancel.
For instance,
\eeq{ \label{nycx}
\nabla\gpp(\xi-)\neq\nabla\gpp(\xi+),\ \zeta \prec \xi \quad \implies \quad
\zeta\cdot(\nabla\gpp(\xi-)-\nabla\gpp(\xi+)) < 0.
}
Now fix some $\zeta\in\,]\evec_2, \xi[\,$ and consider any down-left nearest-neighbor path $(x_n)_{n\leq0}$ such that $x_0=\zevec$ and $x_n/n\to\zeta$ as $n\to-\infty$.
The latter condition implies $\lim_{n\to-\infty}x_n\cdot\evec_1=\lim_{n\to-\infty}x_n\cdot\evec_2=-\infty$, so
\eeq{ \label{hg63}
\text{for any $y\in\Z^2$, there is $n_0$ such that $x_n < y$ for all $n\leq n_0$}.
}
By the cocycle property \eqref{coc1} and the Busemann shape theorem \eqref{bfanfx}, 
\eq{   \lim_{n\to-\infty}\frac{1}{|n|}\sum_{k=n+1}^{0} \Bus^{\xi\pm}_{x_{k-1},x_{k}} =
\lim_{n\to-\infty}\frac{1}{|n|}\Bus^{\xi\pm}_{x_n,\zevec} = \nabla\gpp(\xi\pm)\cdot\zeta.
}
The $\pm$ versions of the right-hand side are distinct because of \eqref{nycx}.
Carrying this distinction over to left-hand side implies
\eeq{ \label{bn78}
\varlimsup_{k\to-\infty}|\Bus^{\xi-}_{x_{k-1},x_k} - \Bus^{\xi+}_{x_{k-1},x_{k}}| > 0.
}
By construction $x_{n-1}\in\{x_n-\evec_1,x_n-\evec_2\}$, so \eqref{bn78} shows there are infinitely many $n$ such that $x_n\in\bD^\xi$.
Thanks to \eqref{hg63} and Proposition~\ref{nepg}, this implies $\bD^\xi$ is all of $\Z^2$.
\end{proof}

\begin{proof}[Proof of Theorem~\ref{thm:allxy}]
Consider any $(a,b)=x\neq y=(k,t)$ in $\Z^2$.
Without loss of generality, we assume $a<k$.
Denote $\aUset_{x,y}=\{\xi\in\,]\evec_2,\evec_1[\,:\,    \Bus^{\xi-}_{x,y}(\w)\ne  \Bus^{\xi+}_{x,y}(\w)\bigr\}$.
We wish to show $\aUset_{x,y}$ is equal to the set $\aUset$ of discontinuity points of nearest-neighbor Busemann functions, defined in \eqref{aUsetdef}.
In general $\aUset_{x,y}\subset\aUset$ since the cocycle property \eqref{coc1} implies $\Bus_{x,y}^{\xi\sig}$ is a sum of nearest-neighbor Busemann functions.
So we just need to show $\aUset_{x,y}\supset\aUset$.

\smallskip

\textit{Case 1.} If $b\geq t$, use the cocycle property \eqref{coc1} to decompose $\Bus^{\xi\sig}_{x,y}$ along a down-right path:
\eq{ 
\Bus^{\xi\sig}_{x,y} 
= \sum_{j=t+1}^{b}\Bus^{\xi\sig}_{(a,j),(a,j-1)} + \sum_{i=a+1}^{k}\Bus^{\xi\sig}_{(i-1,t),(i,t)},
}
where the first sum is empty if $b=t$.
%
If $\xi\in\aUset$, then Theorem~\hyperref[thm:51-32a]{\ref*{thm:51-32}\ref*{thm:51-32a}} together with monotonicity \eqref{B_mono} implies $\Bus^{\xi-}_{(a,j),(a,j-1)}>\Bus^{\xi+}_{(a,j),(a,j-1)}$ for all $j$, and $\Bus^{\xi-}_{(i-1,t),(i,t)}>\Bus^{\xi+}_{(i-1,t),(i,t)}$ for all $i$.
Hence $\Bus^{\xi-}_{x,y}>\Bus^{\xi+}_{x,y}$ and so $\xi\in\aUset_{x,y}$.

\smallskip

\textit{Case 2.} If $b<t$, then $x<y$. 
Suppose $\xi\notin\aUset_{x,y}$, meaning
\eeq{ \label{bjeg30}
0 = \Bus_{x,y}^{\xi+}-\Bus_{x,y}^{\xi-}
\stackref{coc1}{=} \big(\Bus^{\xi+}_{x,y-\evec_2}-\Bus^{\xi-}_{x,y-\evec_2}\big)
+ \big(\Bus^{\xi+}_{y-\evec_2,y}-\Bus^{\xi-}_{y-\evec_2,y}\big).
}
In the notation of \eqref{IJnot}, the equality \eqref{bjeg30} reads
\eq{
\Bus_{x,y-\evec_2}^{\xi-} - \Bus_{x,y-\evec_2}^{\xi+}
= \log J_k^{\xi+}(t)-\log J_k^{\xi-}(t).
}
Applying \eqref{zb72dw} twice yields
\eq{
&\log J_k^{\xi+}(t)-\log J_k^{\xi-}(t)
= \log\Big(1+\frac{J_{k-1}^{\xi+}(t)}{I_k^{\xi+}(t-1)}\Big)-\log\Big(1+\frac{J_{k-1}^{\xi-}(t)}{I_k^{\xi-}(t-1)}\Big) \\
&= \log\bigg(1+\frac{\Yw_{(k-1,t)}}{I_k^{\xi+}(t-1)}\Big(1+\frac{J_{k-2}^{\xi+}(t)}{I_{k-1}^{\xi+}(t-1)}\Big)\bigg)
-\log\bigg(1+\frac{\Yw_{(k-1,t)}}{I_k^{\xi-}(t-1)}\Big(1+\frac{J_{k-2}^{\xi-}(t)}{I_{k-1}^{\xi-}(t-1)}\Big)\bigg).
}
To condense notation, we define
\eq{
\ug^{\xi\sig} = \frac{1}{I_k^{\xi\sig}(t-1)}\Big(1+\frac{J_{k-2}^{\xi\sig}(t)}{I_{k-1}^{\xi\sig}(t-1)}\Big),
}
and then the two previous displays together show
\eeq{ \label{hvqe93}
\Bus_{x,y-\evec_2}^{\xi-} - \Bus_{x,y-\evec_2}^{\xi+} = \log(1+\ug^{\xi+}\Yw_{y-\evec_1}) - \log(1+\ug^{\xi-}\Yw_{y-\evec_1}).
}
Note for later that by Theorem~\hyperref[thm:51-32a]{\ref*{thm:51-32}\ref*{thm:51-32a}} together with monotonicity \eqref{B_mono},
\eeq{ \label{54njc0}
\xi\in\aUset \quad \implies \quad \ug^{\xi+} > \ug^{\xi-}.
}
After algebraic manipulations, \eqref{hvqe93} is equivalent to
\eeq{ \label{hvqe94}
\Yw_{y-\evec_1}\big(\ug^{\xi+}-\ug^{\xi-}(e^{\Bus_{x,y-\evec_2}^{\xi-} - \Bus_{x,y-\evec_2}^{\xi+}})\big)
= e^{\Bus_{x,y-\evec_2}^{\xi-} - \Bus_{x,y-\evec_2}^{\xi+}}-1.
}
In order to show $\xi\notin\aUset$, it now suffices to prove that with probability one, \eqref{hvqe94} fails for every $\xi\in\aUset$.

To this end, let $A = \{v\in\Z^2:\,\text{$v\le y-\evec_2$ or $v\le y-2\evec_1$}\}$, and let $\mathfrak{S}_{A}$ denote the $\sigma$-algebra generated by the random variables 
\eq{
\{\Yw_v, \Bus^{\xi\sig}_{u,v} :\,  \xi\in\,]\evec_2,\evec_1[\,, \sigg\in\{-,+\},  v\in A, u\le v\}.
}
Since $y-\evec_1\not\le v$ for every $v\in A$, the weight $\Yw_{y-\evec_1}$ is independent of $\mathfrak{S}_{A}$ by \eqref{B-ind}.
On the other hand, $\ug^{\xi\sig}$ and $\Bus_{x,y-\evec_2}^{\xi\sig}$  are $\mathfrak{S}_A$-measurable, for any fixed $\xi\in\,]\evec_2,\evec_1[\,$.
We next argue this same measurability when $\xi$ is replaced by a discontinuity direction.

For $\delta>0$ and $\zeta\in\,]\evec_2,\evec_1[$, let $\tau_1^\delta \succ\tau_2^\delta\succ\cdots\succ\tau^\delta_{N_{\delta,\zeta}}$ be the (possibly empty) list of directions in $[\zeta,\evec_1[$ at which $\xi\mapsto\Bus_{x-\evec_1,x}^{\xi+}$ has a jump discontinuity of absolute size $\ge\delta$; this list is finite by monotonicity \eqref{B_Mono1}.
We record these directions in a random sequence padded by a tail of $\evec_2$'s:
\eq{ 
\tau^{\delta,\zeta} = (\tau^{\delta,\zeta}_1,\tau^{\delta,\zeta}_2,\dots)
=(\tau_1^\delta,\tau_2^\delta,\dots,\tau_{N_{\delta,\zeta}}^\delta,\evec_2,\evec_2,\dots)\in[\evec_2,\evec_1[\,^{\{1,2,\dots\}}.
}
That is, $\tau^{\delta,\zeta}_i = \tau^\delta_i$ if $i\leq N_{\delta,\zeta}$, or $\tau^{\delta,\zeta}_i=\evec_2$ if $i>N_{\delta,\zeta}$.
Because $\xi\mapsto\Bus_{x-\evec_1,x}^{\xi+}$ is cadlag, this sequence is a measurable function of $(\Bus_{x-\evec_1,x}^{\xi+}:\,\xi\in\,]\evec_2,\evec_1[\,)$,  and so $\omega\mapsto\tau^{\delta,\zeta}(\omega)$ is $\mathfrak{S}_A$-measurable.\footnote{This statement remains true if $x$ is replaced by any $v\in A$, so the choice of edge $(x-\evec_1,x)$ is not special.} 
Furthermore, Theorem~\hyperref[thm:51-32a]{\ref*{thm:51-32}\ref*{thm:51-32a}} implies that any given $\xi\in\aUset$ will appear as a coordinate of $\tau^{\delta,\zeta}$ when $\delta$ is sufficiently small and $\zeta$ is sufficiently close to $\evec_2$.
More precisely, for any sequences $\delta_j\searrow0$ and $\zeta_j\searrow\evec_2$,
\eeq{ \label{nj5v1}
\aUset = \bigcup_{j=1}^\infty \big\{\tau_1^{\delta_j},\tau_2^{\delta_j},\cdots,\tau^{\delta_j}_{N_{\delta_j,\zeta_j}}\big\}.
}
Meanwhile, $\xi\mapsto\Bus_{x,y-\evec_2}^{\xi-}-\Bus_{x,y-\evec_2}^{\xi+}$ is the difference of two $\mathfrak{S}_A$-measurable random functions, the first caglad and the second cadlag.
Therefore, evaluating this difference at any $\mathfrak{S}_A$-measurable random value yields a $\mathfrak{S}_A$-measurable random variable. 
In particular, the random sequence
\eq{
\Delta^{\delta,\zeta}
&=(\Delta^{\delta,\zeta}_1,\Delta^{\delta,\zeta}_2,\dots) \\
&=\big(\Bus_{x,y-\evec_2}^{\tau_1^\delta-}-\Bus_{x,y-\evec_2}^{\tau_1^\delta+},\Bus_{x,y-\evec_2}^{\tau_2^\delta-}-\Bus_{x,y-\evec_2}^{\tau_2^\delta+},\dots,\Bus_{x,y-\evec_2}^{\tau^\delta_{N_{\delta,\zeta}}-}-\Bus_{x,y-\evec_2}^{\tau^\delta_{N_{\delta,\zeta}}+},0,0,\dots\big)\in\R_{\ge0}^{\{1,2,\dots\}}
}
is $\mathfrak{S}_A$-measurable.
By similar reasoning, the random sequences
\eq{
\ug^{\delta,\zeta,+}
&=(\ug^{\delta,\zeta,+}_1,\ug^{\delta,\zeta,+}_2,\dots)
=(\ug^{\tau_1^\delta+},\ug^{\tau_2^\delta+},\dots,\ug^{\tau_{N_{\delta,\zeta}}^\delta+},1,1,\dots)\in\R_{\ge0}^{\{1,2,\dots\}} \\
\text{and}\quad 
\ug^{\delta,\zeta,-}
&= (\ug^{\delta,\zeta,-}_1,\ug^{\delta,\zeta,-}_2,\dots)=(\ug^{\tau_1^\delta-},\ug^{\tau_2^\delta-},\dots,\ug^{\tau_{N_{\delta,\zeta}}^\delta-},0,0,\dots)\in\R_{\ge0}^{\{1,2,\dots\}}
}
are $\mathfrak{S}_A$-measurable.
Since $\tau_i^{\delta,\zeta}\in\aUset$ whenever $i\le N_{\delta,\zeta}$, \eqref{54njc0} implies
\eeq{ \label{bzm7m0}
\ug^{\delta,\zeta,+}_i-\ug^{\delta,\zeta,-}_i>0 \quad \text{for all $i\in\{1,2,\dots\}$}.
}
The equality \eqref{nj5v1} implies 
\eeq{ \label{nj5v2}
\{\text{\eqref{hvqe94} holds for some $\xi\in\aUset$}\}
\subset \bigcup_{j=1}^\infty\bigcup_{i=1}^\infty \Big\{\Yw_{y-\evec_1}\big(\ug^{\delta_j,\zeta_j,+}_i-\ug^{\delta_j,\zeta_j,-}_ie^{\Delta^{\delta_j,\zeta_j}_i}\big)
= e^{\Delta^{\delta_j,\zeta_j}_i} - 1\Big\}.
}

Now recall the assumption that $\Yw_{y-\evec_1}$ has a continuous distribution.
To capitalize on this assumption, we claim that for any given value of the triple $(V_i^{\delta_j,\zeta_j,+},V_i^{\delta_j,\zeta_j,-},\Delta_i^{\delta_j,\zeta_j})$, there is at most one strictly positive value of $\Yw_{y-\evec_1}$ that solves the equation
\eeq{ \label{6bg2bx}
\Yw_{y-\evec_1}\big(\ug^{\delta_j,\zeta_j,+}_i-\ug^{\delta_j,\zeta_j,-}_ie^{\Delta^{\delta_j,\zeta_j}_i}\big)
= e^{\Delta^{\delta_j,\zeta_j}_i} - 1.
}
Indeed, if $\Delta^{\delta_j,\zeta_j}_i\neq 0$, then the right-hand side is nonzero, and the claim is clear.
If $\Delta^{\delta_j,\zeta_j}_i= 0$, then \eqref{6bg2bx} fails: the right-hand side is 0 while the left-hand side is positive thanks to \eqref{bzm7m0}.
Since the triple $(V_i^{\delta_j,\zeta_j,+},V_i^{\delta_j,\zeta_j,+},\Delta_i^{\delta_j,\zeta_j})$ is $\mathfrak{S}_A$-measurable and thus independent of $\Yw_{y-\evec_1}$, it follows from our claim---and the continuous distribution assumption---that \eqref{6bg2bx} fails with probability one.
As this holds for every $i$ and $j$, the right-hand side of \eqref{nj5v2} has probability zero.
Hence the left-hand side does too, and the proof is complete. 
\end{proof}

 To prove Proposition~\ref{lm:51-32}, we will need some additional notation and three lemmas.
 Define the jumps at $x$ in direction $\xi$ as 
\eeq{ \label{ge71}
S^{\xi,\evec_1}_x = \Bus^{\xi-}_{x-\evec_1, x} -  \Bus^{\xi+}_{x-\evec_1, x} \quad \text{and} \quad
S^{\xi,\evec_2}_x = \Bus^{\xi+}_{x-\evec_2, x} -  \Bus^{\xi-}_{x-\evec_2, x}.
}
By \eqref{51-25}, these quantities are 
nonnegative.
Denote the total jump at $x$ in direction $\xi$ by
\be\label{S4950}   S^\xi_x = S^{\xi,\evec_1}_x + S^{\xi,\evec_2}_x.
    \ee 
By the discussion following \eqref{51-25}, membership  $x\in\bD^\xi$ is equivalent to $S_x^\xi > 0$.

The first two lemmas involve deterministic statements.
 
 \begin{lemma} \label{lm:51-31} 
 If $x\in\bD^\xi$, then the following statements hold.
 \begin{enumerate}[label=\textup{(\alph*)}]
\item \label{lm:51-31a}
At least one of  $x-\evec_ 1$ and $x-\evec_2$ belongs to $\bD^\xi$.
\item \label{lm:51-31b}
If $x-\evec_ 2\notin\bD^\xi$, then $S^{\xi,\evec_2}_{x-\evec_1} = S^\xi_x$.
Similarly, if $x-\evec_ 1\notin\bD^\xi$, then $S^{\xi,\evec_1}_{x-\evec_2} = S^\xi_x$.
\end{enumerate}
  \end{lemma} 
 
 \begin{proof}  
 Both parts of the lemma are immediate from the identity
   \be\label{51-31.1} 
   S^{\xi,\evec_2}_{x-\evec_1}+ S^{\xi,\evec_1}_{x-\evec_2} = S^\xi_x,
\ee  
which we will show is valid for all $x\in\Z^2$.
Start by applying the definitions \eqref{ge71} to the left-hand side:
\eq{
S^{\xi,\evec_2}_{x-\evec_1}+ S^{\xi,\evec_1}_{x-\evec_2}
&= \Bus^{\xi+}_{x-\evec_1-\evec_2, x-\evec_1} -  \Bus^{\xi-}_{x-\evec_1-\evec_2, x-\evec_1} \\
&\phantom{=} -\Bus^{\xi+}_{x-\evec_2-\evec_1, x-\evec_2}  + \Bus^{\xi-}_{x-\evec_2-\evec_1, x-\evec_2}.
}
Now add the terms vertically on the right-hand side, according to the cocycle rule \eqref{coc1}:
\eq{
S^{\xi,\evec_2}_{x-\evec_1}+ S^{\xi,\evec_1}_{x-\evec_2} = \Bus^{\xi+}_{x-\evec_2,x-\evec_1} + \Bus^{\xi-}_{x-\evec_1,x-\evec_2}.
}
Use \eqref{coc1} again to expand each term on the right-hand side:
\eq{
S^{\xi,\evec_2}_{x-\evec_1}+ S^{\xi,\evec_1}_{x-\evec_2}
&= \Bus^{\xi+}_{x-\evec_2,x} + \Bus^{\xi+}_{x,x-\evec_1} + \Bus^{\xi-}_{x-\evec_1,x} + \Bus^{\xi-}_{x,x-\evec_2} \\
&= \Bus^{\xi+}_{x-\evec_2,x} - \Bus^{\xi+}_{x-\evec_1,x} + \Bus^{\xi-}_{x-\evec_1,x} - \Bus^{\xi-}_{x-\evec_2,x}.
}
The right-hand side is exactly \eqref{S4950}, so we have proved \eqref{51-31.1}.
\end{proof} 

\begin{lemma} \label{rjco}
Almost surely the following implication is true for all $x\in\Z^2$, $\xi\in \,]\evec_2,\evec_1[\, $, and $r\in\{1,2\}$.
If $|\log \Yw_x| \leq L$, $|\Bus^{\xi\sigg}_{x-\evec_r,x}| \leq L$, and $\Bus^{\xi\sigg}_{x-\evec_r,x}-\log\Yw_x\geq 1/L$ for both signs $\sigg\in\{-,+\}$ and some $L\geq1$, then $S^{\xi,\evec_r}_x \geq e^{-(2L+\log L)}S^{\xi,\evec_{3-r}}_x$.
\end{lemma}

\begin{proof}
Assume for simplicity that $r=1$, since the $r=2$ case is analogous.
Consider any $x$ for which the hypotheses are true.
By the recovery property \eqref{B_reco}, we have
\eeq{ \label{ncy6}
e^{-\Bus^{\xi-}_{x-\evec_1,\tspb x}} + e^{-\Bus^{\xi-}_{x-\evec_2,\tspb x}}
= \Yw_x^{-1} = 
e^{-\Bus^{\xi+}_{x-\evec_1,\tspb x}} + e^{-\Bus^{\xi_+}_{x-\evec_2,\tspb x}}.
}
Solving for the $\evec_2$ terms results in
\eq{
e^{-\Bus^{\xi\pm}_{x-\evec_2,\tspb x}} &= \int_{\log \Yw_x}^{\Bus^{\xi\pm}_{x-\evec_1,\tspb x}} e^{-s}\ \dd s 
\geq (\Bus^{\xi\pm}_{x-\evec_1,x}-\log\Yw_x)e^{-L}
\geq \frac{1}{L}e^{-L} = e^{-L-\log L}.
}
Now take logarithms to see that $\Bus^{\xi\pm}_{x-\evec_2,x} \leq L+\log L$.
Thanks to \eqref{B_Mono2}, we also have $\Bus^{\xi\pm}_{x-\evec_2,x}  > \log \Yw_x \geq -L$, so $|\Bus^{\xi\pm}_{x-\evec_2,x}| \leq L+\log L$.

Next manipulate \eqref{ncy6} in a different way: put $\evec_1$ terms on the right-hand side and $\evec_2$ terms on the left-hand side:
\eeq{ \label{nx74}
e^{-\Bus^{\xi-}_{x-\evec_2,\tspb x}} - e^{-\Bus^{\xi+}_{x-\evec_2,\tspb x}} 
= e^{-\Bus^{\xi+}_{x-\evec_1,\tspb x}} -  e^{-\Bus^{\xi-}_{x-\evec_1,\tspb x}} .
}
By the hypothesis $|\Bus^{\xi\pm}_{x-\evec_1,x}| \leq L$,
\eq{
\text{R.H.S. of \eqref{nx74}} \leq e^{L}(\Bus^{\xi-}_{x-\evec_1,x}-\Bus^{\xi+}_{x-\evec_1,x}) = e^L S^{\xi,\evec_1}_x.
}
On the other hand, thanks to our earlier finding $|\Bus^{\xi\pm}_{x-\evec_2,x}| \leq L+\log L$,
\eq{
\text{L.H.S. of \eqref{nx74}} \geq e^{-(L+\log L)}(\Bus^{\xi+}_{x-\evec_2,x}-\Bus^{\xi-}_{x-\evec_2,x}) = e^{-(L+\log L)}S^{\xi,\evec_2}_x.
}
The combination of these two statements proves the claimed inequality.
\end{proof}

 The third and final lemma shows that the hypotheses of Lemma~\ref{rjco} are satisfied at a positive density of vertices.
 
  \begin{lemma} \label{g85n}
 Given $r\in\{1,2\}$ and $x\in\Z^2$, consider the straight-line path $(x_k)_{k\leq0}$ given by $x_k = x-k\evec_r$.
 There is a family of positive constants $(L^\xi:\, \xi\in\,]\evec_2,\evec_1[\,)$ such that the following holds almost surely: for every $\xi\in\,]\evec_2,\evec_1[\,$ and $\sigg\in\{-,+\}$,
 \eeq{ \label{gnb8}
 \varliminf_{n\to-\infty}\frac{1}{|n|}\sum_{k=n+1}^0 \one\Big\{
 |\Bus^{\xi\sig}_{x_{k-1},x_k}|\leq L^\xi, \
 |\log\Yw_{x_k}|\leq L^\xi, \
 \Bus^{\xi\sig}_{x_{k-1},x_k}-\log\Yw_{x_k}\geq\frac{1}{L^\xi} \Big\} 
 \geq \frac{1}{L^\xi}.
 }
 \end{lemma}
 

\begin{proof}
We will assume $r=1$, since the $r=2$ case follows by symmetry (see Remark~\ref{rmk:vB}).
We may work on a compact subinterval $[\zeta,\eta] \subset \,]\evec_2,\evec_1[\,$, as the full result follows by taking a countable sequence $\zeta_k\searrow\evec_2$ and $\eta_k\nearrow\evec_1$.

Having fixed $\zeta$ and $\eta$, define the following positive number:
\eeq{ \label{bqeqs}
\delta = 
\nabla\gpp(\eta+)\cdot\evec_1 - \E[\log\Yw_x] \stackref{inmono1}{>} 0.
}
We know from \eqref{EB} and \eqref{ranenv} that $\Bus^{\zeta-}_{x-\evec_1,x}$ and $\log\Yw_x$ are integrable.
So for any $\eps>0$, there is $L\geq1$ large enough that
\eq{
\E\big(|\Bus^{\zeta-}_{x-\evec_1,x}|\cdot\one\{|\Bus^{\zeta-}_{x-\evec_1,x}|\geq L\}\big) \leq \eps 
\quad \text{and} \quad
\E\big(|\log \Yw_x|\cdot\one\{|\log \Yw_x|\geq L\}\big) \leq \eps. 
}
By the ergodicity in Theorem~\ref{B_thm9}, it follows that almost surely
\begin{subequations} \label{bne15}
\eeq{ \label{bne7}
\varlimsup_{n\to-\infty}\frac{1}{|n|}\sum_{k=n+1}^0 |\Bus^{\zeta-}_{x_{k-1},x_k}|\cdot \one\{|\Bus^{\zeta-}_{x_{k-1},x_k}|\geq L\} \leq \eps.
}
Similarly, because the weights $(\Yw_{x_k})$ are i.i.d.~and hence ergodic, almost surely
\eeq{ \label{bne8}
\varlimsup_{n\to-\infty}\frac{1}{|n|}\sum_{k=n+1}^0 |\log \Yw_{x_k}| \cdot \one\{|\log \Yw_{x_k}|\geq L\} \leq \eps.
}
\end{subequations}
Because we assumed $L\geq1$, these inequalities 
still hold if the multiplicative factors are dropped:
\begin{subequations} \label{bne9}
\begin{align}
\label{bne9a}
\varlimsup_{n\to-\infty}\frac{1}{|n|}\sum_{k=n+1}^0 \one\{|\Bus^{\zeta-}_{x_{k-1},x_k}|\geq L\} &\leq \eps \quad \text{and} \\
\label{bne9b}
\varlimsup_{n\to-\infty}\frac{1}{|n|}\sum_{k=n+1}^0 \one\{|\log\Yw_{x_k}|\geq L\} &\leq \eps.
\end{align}
\end{subequations}

Now consider any $\xi\in[\zeta,\eta]$ and $\sigg\in\{-,+\}$.
The constant $L^\xi$ in the statement of the lemma will be realized as $L^\xi = \max\{L_1,18L_2/\delta,L_3\}$, where $L_1$, $L_2$, $L_3$ will be specified below and depend only on $\zeta$ and $\eta$.
Define the quantity
 \begin{align}
 \label{g8o0}
 \Asf_n &= \frac{1}{|n|}\sum_{k=n+1}^0 \one\Big\{\Bus^{\xi\sig}_{x_{k-1},x_k} - \log\Yw_{x_k} \geq \frac{1}{L_1}\Big\}.
 \end{align}
To understand the asymptotics of $\Asf_n$ as $n\to\infty$, we introduce auxiliary quantities
\begin{align}
\label{g8o1}
 \Bsf_{n,1} &= \frac{1}{|n|}\sum_{k=n+1}^0 (\Bus^{\xi\sig}_{x_{k-1},x_k}- \log\Yw_{x_k}) \cdot \one\Big\{\frac{1}{L_1} \leq \Bus^{\xi\sig}_{x_{k-1},x_k} - \log\Yw_{x_k} < 2L_2\Big\}, \\
 \label{g8o2}
  \Bsf_{n,2} &= \frac{1}{|n|}\sum_{k=n+1}^0 (\Bus^{\xi\sig}_{x_{k-1},x_k}- \log\Yw_{x_k}) \cdot \one\Big\{\Bus^{\xi\sig}_{x_{k-1},x_k} - \log\Yw_{x_k} \geq 2L_2\Big\}, \\
  \label{g8o3}
 \Bsf_{n,3} &= \frac{1}{|n|}\sum_{k=n+1}^0 (\Bus^{\xi\sig}_{x_{k-1},x_k}- \log\Yw_{x_k}) \cdot \one\Big\{\Bus^{\xi\sig}_{x_{k-1},x_k} - \log\Yw_{x_k} < \frac{1}{L_1}\Big\}.
 \end{align}
 Since the indicator variables add to $1$ for every $k$, we have
\eq{ 
  \Bsf_{n,1}+\Bsf_{n,2}+\Bsf_{n,3} &= \frac{1}{|n|}\sum_{k=n+1}^0( \Bus^{\xi\sig}_{x_{k-1},x_k}-\log \Yw_{x_k})
 \stackref{coc1}{=}\frac{1}{|n|}\Bus^{\xi\sig}_{x_n,x_0} - \frac{1}{|n|}\sum_{k=n+1}^0\log\Yw_{x_k}.
}
  Since $x_n=x_0-n\evec_1$, \eqref{bfanfx} guarantees that
 \eq{ 
 \lim_{n\to-\infty}{|n|^{-1}}\Bus^{\xi\sig}_{x_n,x_0} = \nabla\gpp(\xi\sigg)\cdot\evec_1.
 }
 In addition, the i.i.d.~random variables $(\Yw_{x_k})_{k\leq0}$ almost surely obey their own law of large numbers, resulting in a smaller limit:
 \eq{
 \lim_{n\to\infty}\frac{1}{|n|}\sum_{k=n+1}^0\log\Yw_{x_k} = \E[\log \Yw_{x_k}] \stackref{inmono}{<}\nabla\gpp(\xi\sigg)\cdot\evec_1.
 }
The three previous displays lead to
\eeq{ \label{c7r5}
\lim_{n\to-\infty}( \Bsf_{n,1}+\Bsf_{n,2}+\Bsf_{n,3})
&\stackrefp{inmono1}{=} \nabla\gpp(\xi\sigg)\cdot\evec_1 - \E[\log\Yw_{x_k}] \\
&\stackref{inmono1}{\geq}
\nabla\gpp(\eta+)\cdot\evec_1 - \E[\log\Yw_{x_k}] \stackref{bqeqs}{=} \delta.
}
From the definition \eqref{g8o3}, it is trivial that $\Bsf_{n,3} < 1/L_1$.
So choose $L_1$ large enough that $1/L_1\leq \delta/3$, and then \eqref{c7r5} can be revised as
\eeq{ \label{c7r6}
\varliminf_{n\to-\infty} \Bsf_{n,1} \geq \frac{2}{3}\delta - \varlimsup_{n\to-\infty}\Bsf_{n,2}.
}
Our next step is to show that $\Bsf_{n,2}$ is small.

By monotonicity \eqref{B_mono1}, each summand in \eqref{g8o2} admits the following upper bound:
\eq{
&(\Bus^{\xi\sig}_{x_{k-1},x_k}- \log\Yw_{x_k}) \cdot \one\big\{\Bus^{\xi\sig}_{x_{k-1},x_k} - \log\Yw_{x_k} \geq 2L_2\big\} \\
&\leq 
(\Bus^{\zeta-}_{x_{k-1},x_k}- \log\Yw_{x_k}) \cdot \one\big\{\Bus^{\zeta-}_{x_{k-1},x_k} - \log\Yw_{x_k} \geq 2L_2\big\}.
}
The indicator on the right-hand side can be further bounded from above:
\eq{
&\one\big\{\Bus^{\zeta-}_{x_{k-1},x_k} - \log\Yw_{x_k} \geq 2L_2\big\}
\leq \one\big\{\max\big(|\Bus^{\zeta-}_{x_{k-1},x_k}|,|\log\Yw_{x_k}|\big) \geq L_2 \big\} \\
&= \one\big\{|\Bus^{\zeta-}_{x_{k-1},x_k}| > |\log\Yw_{x_k}|,\  |\Bus^{\zeta-}_{x_{k-1},x_k}| \geq L_2 \big\}
+\one\big\{|\Bus^{\zeta-}_{x_{k-1},x_k}| \leq |\log\Yw_{x_k}|,\  |\log\Yw_{x_k}| \geq L_2 \big\}.
}
Now multiply each of the last two indicators by the difference $\Bus^{\zeta-}_{x_{k-1},x_k}- \log\Yw_{x_k}$.
The resulting products trivially satisfy
\eq{
(\Bus^{\zeta-}_{x_{k-1},x_k}- \log\Yw_{x_k})  \cdot \one\big\{|\Bus^{\zeta-}_{x_{k-1},x_k}| > |\log\Yw_{x_k}|,\  |\Bus^{\zeta-}_{x_{k-1},x_k}| \geq L_2 \Big\}
&\leq 2|\Bus^{\zeta-}_{x_{k-1},x_k}|\cdot \one\big\{|\Bus^{\zeta-}_{x_{k-1},x_k}| \geq L_2 \big\}, \\
(\Bus^{\zeta-}_{x_{k-1},x_k}- \log\Yw_{x_k}) \cdot \one\big\{|\Bus^{\zeta-}_{x_{k-1},x_k}| \leq |\log\Yw_{x_k}|,\ |\log\Yw_{x_k}|\geq L_2\big\}
&\leq 2 |\log\Yw_{x_k}| \cdot \one\big\{|\log\Yw_{x_k}| \geq L_2 \big\}.
}
Now choose $L_2$ large enough that \eqref{bne15} applies with $\eps=\delta/12$.
Then the cumulative result of the three previous displays is
\eq{
\varlimsup_{n\to-\infty} \Bsf_{n,2} \leq \frac{4}{12}\delta.
}
Inserting this estimate into \eqref{c7r6} results in
\eq{
\varliminf_{n\to-\infty}\Bsf_{n,1} \geq \frac{1}{3}\delta.
}
Comparing the definitions \eqref{g8o1} and \eqref{g8o0}, we see $\Bsf_{n,1} \leq 2L_2\Asf_n$, and so
\eeq{ \label{bnw4}
\varliminf_{n\to-\infty}\Asf_{n} \geq \frac{1}{6L_2}\delta.
}
Finally, choose $L_3$ so that \eqref{bne9} applies with $\eps = \delta/(36L_2)$.
Since $\Bus^{\zeta-}_{x_{k-1},x_k} \geq \Bus^{\xi\sig}_{x_{k-1},x_k}>\log\Yw_{x_k}$ by \eqref{B_Mono1}, the two statements in \eqref{bne9} together yield
\eeq{ \label{gbc4}
\varlimsup_{n\to-\infty}\frac{1}{|n|}\sum_{k=n+1}^0 \one\{|\Bus^{\xi\sig}_{x_{k-1},x_k}|\geq L_3\} &\leq \frac{1}{18L_2}\delta.
}
Of course, \eqref{bne9b} in isolation says
\eeq{ \label{gbc8}
\varlimsup_{n\to-\infty}\frac{1}{|n|}\sum_{k=n+1}^0 \one\{|\log \Yw_{x_k}|\geq L_3\} &\leq \frac{1}{18L_2}\delta.
}
Finally, observe that
\eq{
&\one\Big\{
 |\Bus^{\xi\sig}_{x_{k-1},x_k}| < L_3, \
 |\log\Yw_{x_k}| < L_3, \
 \Bus^{\xi\sig}_{x_{k-1},x_k}-\log\Yw_{x_k}\geq\frac{1}{L_1} \Big\} \\
 &\geq
\one\Big\{\Bus^{\xi\sig}_{x_{k-1},x_k} - \log\Yw_{x_k} \geq \frac{1}{L_1}\Big\}
- \one\{|\Bus^{\xi\sig}_{x_{k-1},x_k}|\geq L_3\}
- \one\{|\log \Yw_{x_k}|\geq L_3\}
}
So subtracting \eqref{gbc4} and \eqref{gbc8} from \eqref{bnw4} results in
\eq{
 \varliminf_{n\to-\infty}\frac{1}{|n|}\sum_{k=n+1}^0 
 \one\Big\{
 |\Bus^{\xi\sig}_{x_{k-1},x_k}| < L_3, \
 |\log\Yw_{x_k}| < L_3, \
 \Bus^{\xi\sig}_{x_{k-1},x_k}-\log\Yw_{x_k}\geq\frac{1}{L_1} \Big\}
  \geq \frac{\delta}{18L_2}.
}
Since the left-hand side is nondecreasing in $L_1$ and $L_3$ while the right-hand side is decreasing in $L_2$, we may set $L = \max\{L_1,L_3,18L_2/\delta\}$ and obtain \eqref{gnb8}.
\end{proof}

\begin{proof}[Proof of Proposition~\ref{lm:51-32}]  
Consider $\xi\in\Udiff$ and $x\in\bD^\xi$.
By Lemma \hyperref[lm:51-31a]{\ref*{lm:51-31}\ref*{lm:51-31a}}, we must have $x-\evec_r\in\bD^\xi$ for some $r\in\{1,2\}$.
Assume $r=1$ without loss of generality, since the case $r=2$ is analogous.
Now suppose toward a contradiction that there is no $z<x$ such that $z\in\bD^\xi$.
In particular, $x-\evec_1-\evec_2$ does not belong to $\bD^\xi$, so part~\ref{lm:51-31a} of Lemma~\ref{lm:51-31} forces $x-2\evec_1\in\bD^\xi$, while part~\ref{lm:51-31b} says $S^{\xi,\evec_2}_{x-2\evec_1} = S^\xi_{x-\evec_1}$.
Repeating this logic results in
\eq{
0 < S^\xi_{x-\evec_1} 
= S^{\xi,\evec_2}_{x-2\evec_1}
= S^{\xi}_{x-2\evec_1}
= S^{\xi,\evec_2}_{x-3\evec_1}
= S^{\xi}_{x-3\evec_1}
=\cdots
}
Set $\delta = S^\xi_{x-\evec_1} > 0$.

Henceforth we use the notation $x_k = x-k\evec_1$.
Let $L = L^\xi$ be the constant from Lemma~\ref{g85n}, which we assume to be greater than $1$.
Consider the indicator variable
\eq{
\Isf_k = \one\Big\{
 |\Bus^{\xi\sig}_{x_{k-1},x_k}|\leq L, \
 |\log\Yw_{x_k}|\leq L, \
 \Bus^{\xi\sig}_{x_{k-1},x_k}-\log\Yw_{x_k}\geq\frac{1}{L} \Big\}.
}
The inequality \eqref{gnb8} says
\eeq{ \label{gnb9}
\varliminf_{n\to-\infty}\frac{1}{|n|}\sum_{k=n+1}^0 \Isf_k \geq \frac{1}{L}.
}
When $\Isf_k=1$, Lemma~\ref{rjco} guarantees $S_{x_k}^{\xi,\evec_1} \geq e^{-(2L+\log L)}S_{x_k}^{\xi,\evec_2}= \delta e^{-(2L+\log L)}$.
When $\Isf_k=0$, we still have the trivial bound $S_{x_k}^{\xi,\evec_1}\geq0$.
Therefore, it follows from \eqref{gnb9} that
\eeq{ \label{xnd6}
\varliminf_{n\to-\infty}\frac{1}{|n|}\sum_{k=n+1}^0 S^{\xi,\evec_1}_{x_k} \geq \frac{\delta e^{-(2L+\log L)}}{L} > 0.
}
On the other hand, by the cocycle property \eqref{coc1},
\eq{
\frac{1}{|n|}\sum_{k=n+1}^0 S^{\xi,\evec_1}_{x_k} = \frac{\Bus^{\xi-}_{x_n,x_0}-\Bus^{\xi+}_{x_n,x_0}}{|n|}.
}
By \eqref{bfanfx}, the right-hand side converges as $n\to-\infty$ to $\nabla\gpp(\xi-)\cdot\evec_1 - \nabla\gpp(\xi+)\cdot\evec_1$, but this difference is zero since $\xi$ was assumed to be a direction of differentiability for $\gpp$.
This contradicts \eqref{xnd6}.
\end{proof}




\section{Polymer dynamics and geometric RSK} \label{sec:grsk}  

This section reformulates  the sequential process to make explicit  the appearance of the gRSK correspondence.  We start with a brief introduction to gRSK, without aiming for a complete description. We follow the conventions of \cite{corw-ocon-sepp-zygo}.   This section can be skipped without loss of continuity.

%

%

\subsection{Polymers and gRSK} 

For $m,n\in\Z_{>0}$,   
gRSK  is a  bijection between $m\times n$ matrices $d=(d_{ij}:\,  1\le i\le m, 1\le j\le n)$  with positive 
entries and pairs of triangular arrays $(z,w)$ of positive reals, indexed as $z=(z_{k\ell}:\,  1\le k\le n,  1\le\ell\le k\wedge m)$ and  $w=(w_{k\ell}:\,  1\le k\le m,  1\le\ell\le k\wedge n)$, whose bottom rows 
agree:  $(z_{n1},\dotsc,z_{n,m\wedge n})=(w_{m1},\dotsc,w_{m,m\wedge n})$. Pictorially,  $z$ consists of rows $z_{k\tsp\bbullet}$  indexed by $k$ from top to bottom and southeast-pointing
diagonals $z_{\rcbullet\tsp\ell}$ indexed by $\ell$ from right to left.  See Figures~\ref{fig:z5} and  \ref{fig:z8} for examples. 

   The connection with polymers is that $z_{k1}$ equals the partition function $Z_{(1,1),(m,k)}$  of polymer paths from $(1,1)$ to $(m,k)$ with weights $d_{ij}$. 
Furthermore,  for $\ell=2,\dotsc,k\wedge m$,  
$z_{k\ell}=\tau_{k\ell}/\tau_{k,\ell-1}$ is a ratio  where $\tau_{k\ell}$ is the partition function of $\ell$-tuples $(\pi_1,\dotsc,\pi_\ell)$ of pairwise disjoint paths such that $\pi_r$ goes from $(1,r)$ to $(m,n-\ell+r)$.  This fact makes the restriction $\ell\le k\wedge m$ natural.  


The utility of the array representation is that  $z$ can be constructed in an alternative way by a procedure called geometric row insertion.  Starting with an empty array $\varnothing$, the rows $d_{i\tsp\bbullet}$ of the matrix $d$ are row-inserted into the growing array one by one. This procedure is denoted by 
\be\label{zd7} 
z = \varnothing \leftarrow  d_{1\tsp\bbullet}  \leftarrow d_{2\tsp\bbullet}\leftarrow \dotsm \leftarrow d_{m\tsp\bbullet}. 
\ee
The array $w$ is constructed by applying the same process to the transpose $d^{\tspb T}$.   
This alternative construction is a key part of the integrability of the inverse-gamma polymer; see \cite{corw-ocon-sepp-zygo}.
We explain some details of the construction  next.  

\begin{figure}
	\centering
	\begin{minipage}{235pt}
		\begin{equation*}
			{
			\begin{array}{ccccccccc}
				&& z_{11} &&\\
				& z_{22} && z_{21} &\\
				z_{33} && z_{32} && z_{31}\\
				&z_{43} && z_{42} && z_{41}\\
				&&z_{53} && z_{52} && z_{51}
			\end{array}
			}
		\end{equation*}
		\captionsetup{width=.9\linewidth}
		\caption{\small The form of the $z$ array in the case $m=3$ and $n=5$. The first diagonal is 
			$z_{\bullet 1}=(z_{11}, z_{21}, z_{31}, z_{41}, z_{51})$ 
			and the second one $z_{\bullet 2}=(z_{22}, z_{32}, z_{42}, z_{52})$.}
		\label{fig:z5}
	\end{minipage}
	\hfill
	\begin{minipage}{235pt}
		\begin{equation*}
			{
			\begin{array}{ccccccccc}
				&&&& z_{11} &&\\
				&&& z_{22} && z_{21} &\\
				&&  z_{33} && z_{32} && z_{31}\\
				&  z_{44} && z_{43} && z_{42} && z_{41}\\
				z_{55} &&z_{54}&&z_{53} && z_{52} && z_{51}
			\end{array}
			}
		\end{equation*}
		\captionsetup{width=.9\linewidth}
		\caption{\small The form of a fully triangular array $z$ in the case   $m=n=N=5$. From right to left there are five diagonals $z_{\bullet\ell}=(z_{\ell\ell},\dotsc,z_{5\ell})$ for $\ell=1,2,\dotsc,5$.}
		\label{fig:z8}
	\end{minipage}
\end{figure}

%

The basic building block of this process is the row insertion of a single word (a vector of positive reals) into another, defined as follows. 

\begin{definition}\label{NYdef}
Let $1\le \ell\le N$. Consider two {\it words} $\xi=(\xi_{\ell},\ldots,\xi_N)$ and $b=(b_\ell,\ldots, b_N)$ with strictly positive real entries. {\it Geometric row insertion} of the word $b$ into the word $\xi$    transforms  $(\xi,b)$ into a new pair $(\xi',b')$ where $\xi'=(\xi'_{\ell},\ldots, \xi'_N)$ and $b'=(b'_{\ell+1},\ldots, b'_N)$. The notation and definition are  as follows:
\begin{equation}
\begin{array}{ccc}
& b & \\
\xi & \cross & \xi' \\[3pt] 
& b'
\end{array}
\qquad
\textrm{where} \qquad
\begin{cases}
\xi'_\ell=b_\ell\xi_\ell, &\\[4pt]
\xi'_k =b_k(\xi_{k-1}' + \xi_k), & \ell+1\le k\le N\\[5pt]
b'_{k} = b_k \ddd\frac{ \xi_{k}\xi'_{k-1}}{\xi_{k-1}\xi'_k}, & \ell+1\leq k\leq N.
\end{cases}
\label{g-row-ins}\end{equation}
Transforming $b\mapsto b'$ produces a word shorter by one position.    If $\ell=N$, then $b'$ is empty and we write $b'=\es$.  
\qedex
\end{definition}


Next, a sequence of row insertions are combined to update an array, diagonal by diagonal.   

\begin{definition}\label{NYdef-z}
 Let  $\zarr=(\zarr_{k\ell}:\,  1\le \ell\le k\le N)$   be an array with $N$ rows and $N$ diagonals.  (That is,   $m=n=N$ and $z$ is the full triangle in Figure~\ref{fig:z8}.)    Let  $b\in\R_{>0}^N$ be an $N$-word.  Geometric row insertion of   $b$ into   $\zarr$  produces a new triangular array $\zarr'= \zarr\leftarrow b$ with $N$ rows and $N$ diagonals. This procedure  consists of $N$ successive  basic row insertions. 
Set $a_1=b$. For $\ell=1,\ldots , N$  iteratively apply the row insertion map \eqref{g-row-ins} to the diagonal words $\zarr_{\rcbullet\tsp\ell}  = (\zarr_{\ell\ell},\ldots, \zarr_{N\ell})$  of $\zarr$:
\[
\begin{array}{ccc}
& a_\ell & \\
\zarr_{\rcbullet\tsp\ell} & \cross & \zarr'_{\rcbullet\tsp\ell} \\
& a_{\ell+1}
\end{array}
\]
where $a_{\ell+1}=a'_\ell$ is one position shorter than $a_\ell$.  The last output $a_{N+1}$ is empty.
The new array  $\zarr'=(\zarr'_{k\ell}:\,  1\le \ell\le k\le N)$ is formed from the diagonals 
$\zarr'_{\rcbullet\tsp\ell} = (\zarr'_{\ell\ell},\ldots, \zarr'_{N\ell})$. 
\qedex
\end{definition}

Shown below is an example when $N=3$.
At each step $a_\ell$ is inserted into the diagonal  $\zarr_{\rcbullet\tsp\ell}$ with outputs $z'_{\rcbullet\tsp\ell}$ and $a_{\ell+1}$:
\eeq{ \label{NYdef_eg}
\begin{array}{ccc}
& a_1=b & \\
\zarr_{\rcbullet\tsp1} & \cross & z'_{\rcbullet\tsp1} \\
& a_2 &\\
\zarr_{\rcbullet\tsp2} & \cross & z'_{\rcbullet\tsp2} \\
& a_3 &\\
\zarr_{\rcbullet\tsp3} & \cross & z'_{\rcbullet\tsp3} \\ &a_4=\varnothing&
\end{array}
}


This description does not cover the construction   \eqref{zd7} of the array $z$  from an empty one.  Separate rules are needed for insertion into an empty array and into an array that is not fully triangular as in Figure~\ref{fig:z5}.  However, these details are not needed for our subsequent discussion and we refer the reader to \cite{corw-ocon-sepp-zygo} for the rest.

Once the array from  \eqref{zd7}  is full (that is, has $N=m=n$ rows and diagonals, as in Figure~\ref{fig:z8}),  we keep $n=N$ fixed and let $m$ grow to define a temporal evolution $z(m)$  of the array.  At each time step 
$m=n+1, n+2, n+3,\dotsc$, the input is the next row $d_{m\tsp\bbullet}$ from the now semi-infinite weight matrix 
$d=(d_{ij}:\,  i\ge 1, 1\le j\le n)$ and the next array $z(m)= z(m-1) \leftarrow d_{m\tsp\bbullet}$ is computed as  in Definition~\ref{NYdef-z}.      The size of $z(m)$ remains fixed at $n=N$ rows and diagonals,  and the polymer interpretations of $z_{k\ell}$ for 
$1\le \ell\le k\le n$ explained above are valid for each $m\ge N$.   Figure~\ref{NYfig3} illustrates diagrammatically  the temporal evolution $z(\aabullet)$ of a full array. 

{\small
 \begin{figure}
\begin{equation*}
\begin{array}{cccccccc}
           &a_1(1) &           & a_1(2)     &        & a_1(3) & &      \\[2pt]
\zarr_{\rcbullet\tsp1}(0)        &\cross         &\zarr_{\rcbullet\tsp1}(1)     &\cross              & \zarr_{\rcbullet\tsp1}(2) &\cross          & \zarr_{\rcbullet\tsp1}(3)   & \cdots\\[3pt]
           &a_2(1)           &           & a_2(2)             &        & a_2(3)         & &      \\[2pt]
\zarr_{\rcbullet\tsp2}(0)         &\cross         &\zarr_{\rcbullet\tsp2}(1)        &\cross              & \zarr_{\rcbullet\tsp2}(2) &\cross          & \zarr_{\rcbullet\tsp2}(3)   & \cdots\\[3pt]
           &a_3(1)            &           & a_3(2)              &        & a_3(3)         & &      \\[2pt]
\zarr_{\rcbullet\tsp3}(0)       &\cross         &\zarr_{\rcbullet\tsp3}(1)         &\cross              & \zarr_{\rcbullet\tsp3}(2)     &\cross          & \zarr_{\rcbullet\tsp3}(3)   & \cdots\\[3pt]
     &a_4(1)            &           & a_4(2)              &        & a_4(3)         & &      \\
  \vdots         &\vdots           & \vdots          & \vdots               & \vdots       & \vdots           &\vdots & \\[3pt]
               &a_N(1)            &           & a_N(2)              &        & a_N(3)         & &      \\[2pt]
\zarr_{\rcbullet\tsp N}(0)       &\cross         &\zarr_{\rcbullet\tsp N}(1)         &\cross              & \zarr_{\rcbullet\tsp N}(2)     &\cross          & \zarr_{\rcbullet\tsp N}(3)   & \cdots\\[3pt]
               &\es           &           &\es              &        & \es       & &      \\
\end{array}
\end{equation*}
\caption{\small Evolution of a triangular array $\zarr(m)$  with $N$ rows and diagonals  over time $m=0,1,2,\dotsc$.  
   The initial state
$\zarr(0)$ is on the left edge and time progresses from left to right.   At time $m$, the driving weights   
come from  row $m$ of the  $d$-matrix: $a_1(m)=d_{m\bullet}=(d_{m,1},\dotsc, d_{m,N})$. The update of $z(m-1)$ to $z(m)$  diagonal by diagonal is represented by the downward  vertical progression of row insertions. Each cross reduces the length of $a_\ell(m)$ by one and  after $N$ steps the last output $a_{N+1}(m)$ is empty. }
 \label{NYfig3}
\end{figure}
}

\subsection{Geometric row insertion in the sequential transformation}  \label{sec:Sgrsk}
Structurally, the triangular  form of the output $z$ with shrinking diagonals towards the left is tied to the shortening in the $b$ to $b'$ mapping in \eqref{g-row-ins}.  We utilize the same row insertion  \eqref{g-row-ins} but in the sequence of row insertions, such as in the example in Figure~\ref{NYfiga}, the shortening of the outputs $a_\ell$ is countered by the addition of a weight from a boundary condition.   Thus the end result is not triangular but rectangular.  Additionally, we formulate the process for a matrix that extends bi-infinitely left and right.  Our procedure is represented by the diagram in Figure~\ref{fig:Z200}. Each cross $\ \scaleobj{0.7}\cross \ $ is an instance of the transformation in  \eqref{g-row-ins} that reduces length along its vertical arrow.  But before the next cross below, the outputted $\Yw$-vector is augmented with an $I$-weight from the boundary condition, thus restoring the original length of the input.  

\begin{figure}   
{\footnotesize
\[  \begin{array}{cccccc}
&&  ( \boldsymbol{I^1_{(k,0)}}, \Yw^1_{(k,\tsp\parng{1}{M})} )  & &  ( \boldsymbol{I^1_{(k+1,0)}}, \Yw^1_{(k+1,\tsp\parng{1}{M})} )  \\[8pt]
\dotsm
&Z^1_{(k-1,\tsp\parng{0}{M})}  & \cross &  Z^1_{(k,\tsp\parng{0}{M})}  & \cross &  Z^1_{(k+1,\tsp\parng{0}{M})} \  
\dotsm  \\[8pt] 
&& \Yw^{2}_{(k,\tsp\parng{1}{M})}  & & \Yw^{2}_{(k+1,\tsp\parng{1}{M})} \\[7pt] 
&&\downarrow  &&\downarrow \\[7pt] 
&&  ( \boldsymbol{I^2_{(k,0)}}, \Yw^2_{(k,\tsp\parng{1}{M})} )  & &  ( \boldsymbol{I^2_{(k+1,0)}}, \Yw^2_{(k+1,\tsp\parng{1}{M})} )  \\[8pt]
\dotsm &Z^2_{(k-1,\tsp\parng{0}{M})}  & \cross &  Z^2_{(k,\tsp\parng{0}{M})}  & \cross &  Z^2_{(k+1,\tsp\parng{0}{M})} \  \dotsm \\[8pt] 
&& \Yw^{3}_{(k,\tsp\parng{1}{M})} & & \Yw^{3}_{(k+1,\tsp\parng{1}{M})}  \\[7pt]  
&&\downarrow   &&\downarrow\\
&&\vdots   &&\vdots \\
&&\downarrow  &&\downarrow\\[7pt] 
&&  ( \boldsymbol{I^N_{(k,0)}}, \Yw^i_{(N,\tsp\parng{1}{M})} )  &  &  ( \boldsymbol{I^N_{(k+1,0)}}, \Yw^N_{(k+1,\tsp\parng{1}{M})} ) \\[8pt]
\dotsm &Z^N_{(k-1,\tsp\parng{0}{M})}  & \cross &  Z^N_{(k,\tsp\parng{0}{M})} & \cross &  Z^N_{(k+1,\tsp\parng{0}{M})} \  \dotsm \\[8pt] 
&& \Yw^{N+1}_{(k,\tsp\parng{1}{M})}   & & \Yw^{N+1}_{(k+1,\tsp\parng{1}{M})}
\end{array} 
 \]
 }
\caption{\small    The bi-infinite geometric row insertion procedure with boundary.   Index $i=1,\dotsc,N$ runs vertically down and index $k\in\Z$ horizontally from left to right.  The ratio variables  $\{\boldsymbol{I^{i}_{(k, 0)}}\}$ are boldfaced to highlight   that they are  initially given boundary conditions.  The weights $\Yw^1$ are the initial dynamical input.   On row $i\in\lzb1,N\rzb$,   instance $k$ of the geometric row insertion marked by crossed arrows updates the vector  $Z^i_{(k-1,\parng{0}{M})}$  to $Z^i_{(k,\parng{0}{M})}$ and outputs the dual weight vector $\Yw^{i+1}_{(k,\parng{1}{M})}$. If $i<N$, the latter is  then combined with the initially given ratio weight $\boldsymbol{I^{i+1}_{(k, 0)}}$ and fed into  instance $k$ of the geometric row insertion on row $i+1$.   The evolution began in the infinite past of the $k$-index on the left and progresses into the infinite future on the right.  The final dual weights $\Yw^{N+1}_{(k,\tsp\parng{1}{M})}$ are left unused in this picture, but index $i$ can also be extended indefinitely beyond $N$. 
 }
\label{fig:Z200}
\end{figure}

We now reformulate the update map so that we can express the sequential transformation in terms of geometric row insertion. 

For $x\in\Z\times\Z_{\ge0}$, define a vector   $Z_x=(Z^1_x,\dotsc, Z^N_x)$  of partition functions with a boundary condition  as follows. 
  On the bottom level $\Z\times\{0\}$ we have $N$  given  boundary functions  $\{Z^i_{(k,0)}\}_{k\in\Z}$ for $i\in\lzb1,N\rzb$.   
  In the bulk $\Z\times\Z_{>0}$ the weights $\Yw^1=(\Yw^1_x)_{x \tsp\in\tsp\Z\times\Z_{>0}}$ are given.  
For $i=1,\dotsc,N$ iterate the following  two-step construction.    

\smallskip 

\textit{Step 1.} 
  For $(k,t)\in\Z\times\Z_{>0}$ define  
	\begin{equation}\label{Def:Z^i}
		Z^i_{(k,t)}=\sum_{j:\,j\leq k}Z^i_{(j,0)} \, Z^i_{(j,1),(k,t)}\,,
	\end{equation}
	where $\{Z^i_{x,y}: x\le y\}$ is the partition function   with weights $\Yw^i=(\Yw^i_x)_{x \in\Z\times\Z_{>0}}$:  
\eq{ 
Z^i_{x,y}= \sum_{x_\brbullet\tspa\in\tspa\pathsp_{x,\tspb y}} \prod_{j=m}^{n} \Yw^i_{x_j}   
\quad \text{for $x\in\level_m$,  $y\in\level_n$, $m\leq n$}.
}
   (The difference with the partition function in  \eqref{part56} is that now the initial weight at $x$ is included.)  Assume that the series in \eqref{Def:Z^i} always converges.   
   
   \smallskip 
   
 \textit{Step 2.}    For $k\in\Z$, $s\in\Z_{\ge0}$ and $t\in\Z_{>0}$ define the weights 
	\be\label{G-eta} 
	I^i_{(k,s)}=\frac{Z^i_{(k,s)} }{Z^i_{(k-1,s)}} \, , \qquad 
	J^i_{(k,t)}=\frac{Z^i_{(k,t)} }{Z^i_{(k,t-1)}} \, , 
	\qquad\text{and}\qquad  
	\Yw^{i+1}_{(k,t)} =     \frac{1}{ \frac1{I^i_{(k, t-1)}}+\frac1{J^i_{(k-1,t)}}} .	\ee  
If $i<N$, return to Step 1 with $i+1$ and use the weights $\Yw^{i+1}$ just constructed.  	

\smallskip 
   
  The reader can check that we have replicated the construction in Section~\ref{map_sec}. Namely, on each level $t\in\Z_{>0}$, 
  \eq{
		Z^i_{(k,t)}=\sum_{m:\,m\leq k}Z^i_{(m,t-1)}  \prod_{j=m}^k \Yw^i_{(j,t)}\,, \quad k\in\Z,  
	}
  and   the sequences in \eqref{G-eta}  obey the transformations \eqref{drs_def}: 
 \be\label{Z285} 
I^i_{(\bbullet\tsp,\tspa t)}=\Dop( \Yw^i_{(\bbullet\tsp,\tspa t)}, I^i_{(\bbullet\tsp,\tspa t-1)}), \quad
J^i_{(\bbullet\tsp,\tspa t)}=\Sop( \Yw^i_{(\bbullet\tsp,\tspa t)}, I^i_{(\bbullet\tsp,\tspa t-1)}) \quad \text{and}\quad 
\Yw^{i+1}_{(\bbullet\tsp,\tspa t)}=\Rop( \Yw^i_{(\bbullet\tsp,\tspa t)}, I^i_{(\bbullet\tsp,\tspa t-1)}). 
\ee	
Moreover, for each $t\in\Z_{>0}$, the $N$-tuple $I^{\parng{1}{N}}_{(\bbullet\tsp,\tspa t)} \in(\R_{>0}^\Z)^N$ is an output of the sequential transformation from \eqref{Saop_def}:
$I^{\parng{1}{N}}_{(\bbullet\tsp,\tspa t)}  = \Saop_{\Yw^1_{(\bbullet\tsp,\tspa t)}}( I^{\parng{1}{N}}_{(\bbullet\tsp,\tspa t-1)}  )$.  
In particular,  $(I^{\parng{1}{N}}_{(\bbullet\tsp,\tspa t)}:\, t\in\Z_{\ge0})$ is an instance of the sequential process defined in \eqref{g:multil}.  

Fix $M>0$  and for a given $i\in\lzb1,N\rzb$   consider the   partition functions  $(Z^i_{(\bbullet, t)} :\, t\in\lzb0,M\rzb)$  restricted to  $M+1$ lattice levels. The evolution of the $(M+1)$-vector $Z^i_{(k,\tsp\parng{0}{M})}=(Z^i_{(k,t)}:\,  t\in\lzb0,M\rzb)$ and the $M$-vector 
$\Yw^i_{(k,\parng{1}{M})}=(\Yw^i_{(k,t)}:\,  t\in\lzb1,M\rzb)$    from left to right, as $k$ ranges over $\Z$,   obeys these equations: 
\be\label{Z300} \begin{aligned} 
Z^i_{(k,0)}&= Z^i_{(k-1,0)}  I^i_{(k,0)}, \\[4pt]
Z^i_{(k,t)}&= (Z^i_{(k,t-1)}+Z^i_{(k-1,t)})  \Yw^i_{(k,t)}\,, \quad t\in\lzb1,M\rzb, \\[4pt] 
\Yw^{i+1}_{(k,t)}&=    \Yw^i_{(k,t)} \tspb\frac{Z^i_{(k,t-1)}Z^i_{(k-1,t)}}{Z^i_{(k-1,t-1)}Z^i_{(k,t)}} \, , \quad t\in\lzb1,M\rzb. 
\end{aligned}\ee
The first equation above is the definition of $I^i_{(k,0)}$ from \eqref{G-eta}. 
The middle equation is deduced from \eqref{Def:Z^i}.  
The last equation above is a rewriting of the last equation of \eqref{Z285}. 
Now note that equation \eqref{Z300} is exactly the geometric  row insertion 
\be\label{Z308} 
{\small
\begin{array}{ccc}
&  ( I^i_{(k,0)}, \Yw^i_{(k,\tsp\parng{1}{M})} )  & \\[4pt]
Z^i_{(k-1,\tsp\parng{0}{M})}  & \cross &  Z^i_{(k,\tsp\parng{0}{M})}  \\[6pt] 
& \Yw^{i+1}_{(k,\tsp\parng{1}{M})}
\end{array}
}
\ee

Lastly, we combine these geometric row insertions from \eqref{Z308} over all $i\in\lzb1,N\rzb$ and $k\in\Z$ into a bi-infinite network that represents the two-step construction of the partition functions $Z^i_x$ for $x\in\Z\times\lzb0,M\rzb$. The network is  depicted in Figure~\ref{fig:Z200}. The   boundary ratio weight  $I^i_{(k,0)}$  is inserted into the network before the cross $\ \scaleobj{0.7}\cross \ $ that marks  the $(k,i)$  row insertion step. 




\section{Proofs in the inverse-gamma environment} \label{sec:Vig}  

\subsection{Intertwining under inverse-gamma weights} \label{sec:twin_ig}
This section applies the results of Section~\ref{mar_sec}  to  
i.i.d.\ inverse-gamma weights $\Yw_x\sim{\rm Ga}^{-1}(\alpha)$, as assumed in \eqref{m:exp}.   The logarithmic mean of the weights is now $\kappa=-\psi_0(\alpha)$.
The lemma recalled below captures a central feature of inverse-gamma distributions. 
A partial version of it appeared as \cite[Lem.~3.13]{corw-ocon-sepp-zygo}  in the context  of invariant distributions of gRSK. 

\begin{lemma} {\textup{\cite[Lem.~B.2]{busa-sepp-22-ejp}}} \label{lm:invg2}  
Let $\lambda_1>\lambda_2>0$.
Let $\Yw=(\Yw_j)_{j\in\Z}$ and $I=(I_k)_{k\in\Z}$  be mutually independent random variables such that $\Yw_j\sim{\rm Ga}^{-1}(\lambda_1)$ and $I_k\sim{\rm Ga}^{-1}(\lambda_2)$.  
Let 
	\[  \wt I=\Dop(\Yw,I)\,\quad \wt\Yw=R(\Yw,I)\quad\text{and}\quad J=S(\Yw,I).\]      Let $\gpp_k= (\{\wt I_j\}_{j\leq k}, \,J_k, \,\{\wt \Yw_j\}_{j\leq k})$.  
	Then the following statements hold.
	\begin{enumerate}[label={\rm(\alph*)}, ref={\rm(\alph*)}]   \itemsep=3pt
		\item\label{lm:invg2.a}   $\{\gpp_k\}_{k\in\Z}$ is a stationary, ergodic process.  
		For each $k\in\Z$, the random variables $\{\wt I_j\}_{j\leq k}$, $\,J_k$, and $\{\wt \Yw_j\}_{j\leq k} $ are mutually independent with marginal distributions  
		\[  \text{ $\wt\Yw_j\sim{\rm Ga}^{-1}(\lambda_1)$, \ \ $\wt I_j\sim{\rm Ga}^{-1}(\lambda_2)$ \ and \   $J_k\sim{\rm Ga}^{-1}(\lambda_1-\lambda_2)$.   } \] 
		
		\item\label{lm:invg2.b}  $\wt\Yw$ and $\wt I$  are  mutually  independent sequences of i.i.d.\ variables.  
	\end{enumerate}
\end{lemma}

Induction leads to the following generalization.


	\begin{lemma}\label{Dop-lm4} 
	Let $\lambda_1>\dotsm > \lambda_N>0$.
	If $I^{\parng{1}{N}}\in(\R_{>0}^\Z)^N$ has the product inverse-gamma distribution $\nu^{\lambda_{\parng{1}{N}}}$ defined in \eqref{nu5}, then $\Dop^{(N)}(I^{\parng{1}{N}})$ has distribution $\nu^{\lambda_N}$.  In other words,  $\Dop^{(N)}(I^{\parng{1}{N}})\in\R_{>0}^\Z$ is a sequence of i.i.d.\  ${\rm Ga}^{-1}(\lambda_N)$ random variables.  
	\end{lemma} 
	
	


We now identify invariant distributions for the sequential process. 

	\begin{theorem} \label{thm-I} Assume   \eqref{m:exp} and   $\lambda_{\parng{1}{N}}=(\lambda_1,\dotsc,\lambda_N)\in (0,\alpha)^N$.  The product measure  $\nu^{\lambda_{\parng{1}{N}}}$  in \eqref{nu5}  is invariant for the sequential process $Y^{\parng{1}{N}}(\aabullet)$ defined in \eqref{g:multil}.  
	\end{theorem}
	
\begin{proof}    We assume
$(\Yw^{1}, I^{\parng{1}{N}})\sim\nu^{(\alpha, \lambda_{\parng{1}{N}})}$ and then apply the sequential map \eqref{Saop_def}.
	Utilizing Lemma~\ref{lm:invg2}\ref{lm:invg2.b},  induction on $k$ shows that $\Dop(\Yw^1,I^1),\dotsc, \Dop(\Yw^k,I^k)$, $\Yw^{k+1}$, $I^{k+1},\dotsc, I^N$ are independent with $\Dop(\Yw^i,I^i)\sim\nu^{\lambda_i}$, $\Yw^{k+1}\sim\nu^{\alpha}$, $I^j\sim\nu^{\lambda_j}$.   The case $k=N$ is the claim. 
\end{proof} 	

We have partial uniqueness for Theorem~\ref{thm-I}. Namely,  $\nu^{\lambda_{\parng{1}{N}}}$ is the unique invariant measure among shift-ergodic measures $\nu$ with means 
$\int_{\cI_{N\!,\tsp\kappa}} \log x^i_0\,\nu(\dd x^{\parng{1}{N}})  =  -\psi_0(\lambda_i)$ 
 under two different restricted settings:
\begin{enumerate} [label={\rm(\alph*)}, ref={\rm(\alph*)}] \itemsep=3pt 
 \item  if $\lambda_1,\dots,\lambda_N$ are all distinct, by Corollary~\ref{uniq_cor}\ref{uniq_cor.b};  and 
 \item  if we consider measures whose sequence-valued components are independent, for then each component must be i.i.d.\ inverse-gamma, by the uniqueness in the case $N=1$ applied to each component and by Lemma~\ref{lm:invg2}\ref{lm:invg2.b}. 
 \end{enumerate} 
 We leave further uniqueness as an open problem. 
	

Our next task is to apply Theorem~\ref{B_thm9} to the inverse-gamma case.   We wish to include the original weights in this description, as stated in the preliminary Theorem~\ref{B1_thm}.  This will be achieved by taking the  limit \eqref{B_lim} at the level of measures.  

 With  $\lambda_{\parng{1}{N}}=(\lambda_1,\dotsc,\lambda_N)\in \R_{>0}^N$ such that  $\lambda_1>\dotsm > \lambda_N>0$,  $\nu^{\lambda_{\parng{1}{N}}}$ as in \eqref{nu5}, and the transformation 
$\Daop^{(N)}\colon \cI_{N}^\uparrow\to\cI_N^\uparrow$  as in \eqref{Daop_def}, 
  define these probability  measures on $ \cI_{N}^\uparrow$:  
\be\label{mu_ig} 
\mu^{\lambda_{\parng{1}{N}}}=\nu^{\lambda_{\parng{1}{N}}}\circ(\Daop^{(N)})^{-1}.  \ee 
For the continuity claim below we endow the product space $(\R_{>0}^\Z)^N$ and its subspaces with the product topology. 
  
  	\begin{theorem}\label{murho-th1}    The probability measure $\mu^{\lambda_{\parng{1}{N}}}$ is shift-ergodic and  has the following  properties. 
		
		{\rm (Continuity.)}  The probability measure $\mu^{\lambda_{\parng{1}{N}}}$ is weakly continuous as a function of ${\lambda_{\parng{1}{N}}}$ on the set of vectors that satisfy $\lambda_1>\lambda_2>\dotsm >\lambda_N>0$. 
		
		{\rm (Consistency.)}   If $(\brvX^1, \dotsc, \brvX^N)\sim\mu^{(\lambda_1,\dotsc,\lambda_N)}$, then   for all $j\in\lzb1,N\rzb$, we have
		\eq{
		(\brvX^1, \dotsc,  \brvX^{j-1}, \brvX^{j+1},\dotsc, \brvX^N) \sim \mu^{(\lambda_1,\dotsc,  \lambda_{j-1}, \lambda_{j+1},\dotsc,  \lambda_N)}.
		}
		
	\end{theorem}

We prove Theorem~\ref{murho-th1}  after completing the main result of this section and thereby proving Theorem~\ref{B1_thm}.  
Recall the notation  $\Yw(t)=(\Yw_{(k,t)})_{k\in\Z}$  and   $ {\brvI}^{\chdir\sig}(t) = 
(e^{\Bus^{\chdir\sig}_{(k-1,t),\tsp (k,t)}})_{k\in\Z}$.

	\begin{theorem}\label{B2_thm} Assume \eqref{m:exp} and let $N\in\Z_{>0}$.  Let  $\chdir_1\succ\dotsm\succ\chdir_N$ be directions  in $\,]\evec_2,\evec_1[\,$ and  $\sigg_1,\dotsc,\sigg_N$ signs in $\{-,+\}$. 
 	 Then at  each level  $t\in\Z$,  we have 
	 \eeq{ \label{jc9fa}
	 (\Yw(t),\,{\brvI}^{\chdir_1\sig_1}(t),\dotsc, {\brvI}^{\chdir_N\sig_N}(t)) \sim
	 \mu^{(\alpha, \tspa \alpha-\rho(\chdir_1),\dotsc, \tspa\alpha-\rho(\chdir_N))}.
	 } 
	\end{theorem}

\begin{proof}  Pick one more direction $\chdir_0\in\,]\chdir_1, \evec_1[$ and sign $\sigg_0 \in \{-,+\}$.  Think of $\lambda_{\parng{0}{N}}=(\alpha-\rho(\chdir_0), \tspa \alpha-\rho(\chdir_1),\dotsc, \tspa\alpha-\rho(\chdir_N))$ as a function of $\chdir_0$ while $\chdir_{\parng{1}{N}}$ are held fixed.    By Theorem~\ref{thm-I}, $\nu^{\lambda_{\parng{0}{N}}}$ is invariant for the sequential process with $N+1$ components.  By Corollaries~\ref{twm_cor} and \ref{uniq_cor}\ref{uniq_cor.a},   $\mu^{\lambda_{\parng{0}{N}}}$ of \eqref{mu_ig}  is the unique shift-ergodic invariant distribution of  the parallel process, with the given logarithmic means.   By Theorem~\ref{B_thm9},  $\mu^{\lambda_{\parng{0}{N}}}$   is the   distribution of ${\brvI}^{(\chdir\sig)_{\parng{0}{N}}}(t)$.   

As the final step, let $\chdir_0\nearrow\evec_1$.   Then $\lambda_{\parng{0}{N}}\to (\alpha, \tspa \alpha-\rho(\chdir_1),\dotsc, \tspa\alpha-\rho(\chdir_N))$ and by Theorem~\ref{murho-th1}, 
$\mu^{\lambda_{\parng{0}{N}}}\to\mu^{(\alpha, \tspa \alpha-\rho(\chdir_1),\dotsc, \tspa\alpha-\rho(\chdir_N))}$.  
By \eqref{B_lim},   ${\brvI}^{(\chdir\sig)_{\parng{0}{N}}}(t)\to (\Yw(t),\,{\brvI}^{(\chdir\sig)_{\parng{1}{N}}}(t))$ almost surely.   Thus in the limit we obtain \eqref{jc9fa}. 
\end{proof} 

\begin{proof}[Proof of Theorem~\ref{murho-th1}]  
Shift-ergodicity of $\mu^{\lambda_{\parng{1}{N}}}$ follows from shift-ergodicity of  $\nu^{\lambda_{\parng{1}{N}}}$, because of Lemma~\hyperref[sta_prea]{\ref*{sta_pres}\ref*{sta_prea}}.
Consistency follows from the uniqueness of $\mu^{\lambda_{\parng{1}{N}}}$ as the invariant distribution of the parallel transformation because the projection in question commutes with the transformation.  
 We prove the continuity claim by constructing coupled configurations that converge almost surely. 

 Fix $\lambda_{\parng{1}{N}}=(\lambda_1,\dotsc,\lambda_N)$ such that  $\lambda_1>\dotsm>\lambda_N>0$.   Let  $\{\lambda_{\parng{1}{N}}^h\}_{h\in\Z_{>0}}$  be a sequence of parameter vectors such that  $\lambda_{\parng{1}{N}}^h=(\lambda^h_1,\dotsc,\lambda^h_N)\to(\lambda_1,\dotsc,\lambda_N)$ as $h\to\infty$. 
 %
%
%
Let $\{U^i_k\}^{i\in\lzb1,N\rzb}_{k\in\mathbb{Z}}$ be i.i.d.\ uniform variables on $(0,1)$.  
For $\lambda\in (0,\infty)$,  let $F^{-1}_{\lambda}$ be the inverse of the cumulative distribution function of the Ga$^{-1}$($\lambda$) distribution. To obtain sequences $I^{\parng{1}{N}}=(I^1,\dotsc,I^N)\sim\nu^{\lambda_{\parng{1}{N}}}$ and $I^{h, \tspa\parng{1}{N}}=(I^{h,1},\dotsc,I^{h,N})\sim\nu^{\lambda_{\parng{1}{N}}^h}$, set $I^i_k=F^{-1}_{\lambda_i}(U^i_k)$ and $I^{h,i}_k=F^{-1}_{\lambda^h_i}(U^i_k)$.
Then we have the pointwise limits $I^{h,i}_k\to I^i_k$ for all $i\in\lzb1,N\rzb$ and $k\in\Z$ as $h\to\infty$.   

Define the outputs $X^{h, \tspa\parng{1}{N}}=\Daop^{(N)}(I^{h, \tspa\parng{1}{N}})\sim\mu^{\lambda_{\parng{1}{N}}^h}$ and  $X^{\parng{1}{N}}=\Daop^{(N)}(I^{\parng{1}{N}})\sim\mu^{\lambda_{\parng{1}{N}}}$. 
To show $\mu^{\lambda_{\parng{1}{N}}^h}\to\mu^{\lambda_{\parng{1}{N}}}$ weakly, we verify that $X^{h, \tspa\parng{1}{N}}\to X^{\parng{1}{N}}$ coordinatewise almost surely, as $h\to\infty$.     For the latter we turn to Lemma~\ref{lm:D319}.   To satisfy its hypothesis, for each $i\in\lzb1,N-1\rzb$ fix intermediate parameter values $\wh\lambda_i$ and $\wc\lambda_i$ so that 
$\lambda^h_i>\wh\lambda_i>\wc\lambda_i>\lambda^h_{i+1}$ holds for large enough $h$.    Define intermediate weight sequences by 
$\wh I^{i}_k=F^{-1}_{\wh\lambda_i}(U^i_k)$ and $\wc I^{i}_k=F^{-1}_{\wc\lambda_i}(U^i_k)$.   
Then 
\begin{subequations} \label{gcerf3}
\be\label{B4512}   \bigl(\wh I^{i} , \,  \wc I^{i} \bigr) \in \cI_2^\uparrow  \quad \text{for all $i\in\lzb1,N-1\rzb$}  
\ee
and for  large enough  $h$ we  have the inequalities 
\be\label{B4509} 
I^{h,i}_k < \wh I^{i}_k <  \wc I^{i}_k < I^{h,i+1}_k \quad \text{for all $i\in\lzb1,N-1\rzb$, $k\in\Z$}. 
\ee
\end{subequations}
These follow because $\lambda\mapsto F^{-1}_{\lambda}(u)$ is strictly decreasing.  

We verify the desired limits $X^{h, \tspa\parng{1}{N}}\to X^{\parng{1}{N}}$ inductively. 
  
		
		\smallskip 
		
		(1)  $X^{h,1}=I^{h,1}\to I^1=X^1$ needs no proof.   

		\smallskip 
				
		(2) For each $i\in\lzb1,N-1\rzb$  apply Lemma~\ref{lm:D319} to the pair $(\Yw,I)=(I^{h,i}, I^{h,i+1})$ with $(\Yw'', I')=(\wh I^{i} ,  \wc I^{i}) $.  The hypotheses of Lemma~\ref{lm:D319} are in \eqref{gcerf3}. 
  This gives the limit  $\Dop(I^{h,i}, I^{h,i+1}) \to  \Dop(I^{i}, I^{i+1})$ and in particular, $X^{h,2}=\Dop(I^{h,1}, I^{h,2}) \to  \Dop(I^{1}, I^{2})=X^2$.

		\smallskip

		(3) \textit{Induction step.}   Suppose we have the limits 
$\Dop^{(k)}(I^{h,\tspa \parng{i}{i+k-1}}) \to \Dop^{(k)}(I^{ \parng{i}{i+k-1}})$ for $i\in\lzb1,N-k+1\rzb$.   For each $i\in\lzb1,N-k\rzb$  apply Lemma~\ref{lm:D319} to the pair $(\Yw,I)=(I^{h,i}, \Dop^{(k)}(I^{h,\tspa\parng{i+1}{i+k}}))$ again with $(\Yw'', I')=(\wh I^{i} ,  \wc I^{i})$.     From \eqref{B4509} and an inductive application of Lemma~\ref{mon_lem} we have 
\[    I^{h,i} < \wh I^{i}  =  \Yw'' < I' = \wc I^{i} <  I^{h,i+1} <  \Dop^{(k)}(I^{h,\tspa\parng{i+1}{i+k}}).  \]
The   hypotheses of Lemma~\ref{lm:D319} are met, so we get the limits 
\eeq{ \label{4vx4c}  
\Dop(I^{h,i}, \Dop^{(k)}(I^{h,\tspa\parng{i+1}{i+k}})) 
\to    \Dop(I^{i}, \Dop^{(k)}(I^{\parng{i+1}{i+k}})) 
}
 for $i\in\lzb1,N-k+1\rzb$.   The case $i=1$ is $X^{h, k+1}\to X^{k+1}$. 
 Since the left-hand side of \eqref{4vx4c} is $\Dop^{(k+1)}(I^{h,\tspa\parng{i}{i+k}})$, while the right-hand side is $\Dop^{(k+1)}(I^{ \parng{i}{i+k}})$, the induction is complete.
\end{proof}

\subsection{Triangular array construction of the intertwining mapping} \label{sec:array}  
	
	To extract further properties of the law of the Busemann process, we  develop a triangular array  description of the mapping $\arrX=\Daop^{(N)}(I)$ of \eqref{Daop_def}.   Figure~\ref{fig-ar}   represents  the resulting arrays  graphically according to a matrix convention.   There is no probability in this section and the weights  are arbitrary strictly positive reals.  Still, we place this section here 
	because its application to  inverse-gamma weights comes immediately in the next section. 
The proofs of this section are structurally \newa{similar} to those in \cite{fan-sepp-20} for last-passage percolation, after de-tropicalization, that is, after replacement of the max-plus operations of \cite{fan-sepp-20} with standard $(+,\aabullet)$ algebra.

	
\begin{definition}[Array algorithm]  \label{def:arr} 	
	Assume given $I^{\parng{1}{N}}=(I^1,\dotsc,I^N)\in \cI_N^\uparrow$.  
	Define arrays  $\{\arrX^{i,j}:\,  1\le j\le i\le N\}$ and $\{\arrV^{i,j}:\,  1\le j\le i\le N\}$ of elements of $\R_{>0}^\Z$  as follows.  In the inductive definition below index $i$ increases from 1 to $N$, and for each fixed $i$ the second  index $j$ increases from $1$ to $i$.   The $\arrV$ variables are passed from one $i$ level to the next. 
	\begin{enumerate}  [label={\rm(\alph*)}, ref={\rm(\alph*)}]   \itemsep=4pt
		\item \label{def:arra} For $i=1$ set  $\arrX^{1,1}=I^1=\arrV^{1,1}$. 
		\item For $i=2,3,\dotsc,N$, 
		\be\label{m:ar5} \begin{aligned}
			&\ \arrX^{i,1}=I^i, \\
			&\begin{cases}  \arrX^{i,j} =\Dop( \arrV^{i-1,j-1}, \arrX^{i,j-1})\\[3pt]
				\arrV^{i,j-1}= \Rop( \arrV^{i-1,j-1}, \arrX^{i,j-1})  \end{cases} 
			\quad\text{for } j=2,3 \dotsc, i, \\
			&\; \,\arrV^{i,i}=\arrX^{i,i}.  
		\end{aligned}\ee
		Step $i$ takes inputs from two sources:  from the outside it takes $I^i$, and from step $i-1$ it takes the    configuration  $\arrV^{i-1, \tspa\parng{1}{i-1}}=(\arrV^{i-1, 1},  \arrV^{i-1, 2}, \dotsc,  \arrV^{i-1,i-2},  \arrV^{i-1,i-1}=\arrX^{i-1,i-1})$.   \qedex
	\end{enumerate}  
\end{definition} 	
	
 \begin{figure}
 \[
 \begin{matrix}
 \arrX^{1,1} \\[2pt] 
 \arrX^{2,1} &\arrX^{2,2} \\[2pt]
  \arrX^{3,1} &\arrX^{3,2} & \arrX^{3,3} \\
 \vdots& \vdots& \vdots&  \ddots \\[1pt]  
   \arrX^{N,1} &\arrX^{N,2} & \arrX^{N,3}  & \dotsm&  \arrX^{N,N}
  \end{matrix}  
  \qquad \qquad 
  \begin{matrix}
 \arrV^{1,1} \\[2pt]
 \arrV^{2,1} &\arrV^{2,2} \\[2pt]
  \arrV^{3,1} &\arrV^{3,2} & \arrV^{3,3} \\
 \vdots& \vdots& \vdots&  \ddots \\[1pt] 
   \arrV^{N,1} &\arrV^{N,2} & \arrV^{N,3}  & \dotsm&  \arrV^{N,N}
  \end{matrix}  
\] 
\caption{\small Arrays $\{\arrX^{i,j}:\,  1\le j\le i\le N\}$ and $\{\arrV^{i,j}:\,  1\le j\le i\le N\}$. The input  $I^{\parng{1}{N}}=(I^1,\dotsc,I^N)$ enters on the left edge of the $\arrX$-array as the first column $(\arrX^{1,1}, \arrX^{2,1},\dotsc,\arrX^{N,1})=(I^1, I^2,\dotsc,I^N)$. The output appears in the rightmost diagonal of both arrays as $(\arrX^{1,1}, \arrX^{2,2},\dotsc,\arrX^{N,N})=(\arrV^{1,1}, \arrV^{2,2},\dotsc,\arrV^{N,N})=\Daop^{(N)}(I^{\parng{1}{N}})$, as proved in Lemma~\ref{lm:eta5}.}
\label{fig-ar} 
\medskip 
\end{figure}


%
%
%
	
	Lemma~\ref{ces_lem} ensures that the arrays are well-defined for $I^{\parng{1}{N}}\in \cI_N^\uparrow$.  
	The inputs $I^1,\dotsc,I^N$ enter the algorithm one by one in order.  If the process is stopped after the step  $i=m$ is completed for some $m<N$, it produces  the arrays for $(I^1,\dotsc,I^m)\in \cI_m^\uparrow$.   
	
The description in \eqref{m:ar5} constructs the arrays row by row.  Observing the $\arrX$-array column by column from left to right, one sees the sequential transformation in action.  For $j\in\lzb2,N\rzb$,   the mapping from column $\arrX^{\parng{j-1}{N},j-1}$ to column  $\arrX^{\parng{j}{N},j}$ is the sequential transformation  
\be\label{arrX-S}     \arrX^{\parng{j}{N},j}=\Saop_{\arrX^{j-1,j-1}}(\arrX^{\parng{j}{N},j-1})  \ee
on $(N-j+1)$-tuples of sequences,  with the first input sequence $\arrX^{j-1,j-1}$ used as the driving weights. 	
	
	\begin{lemma}\label{lm:eta5}  Let   $I=(I^1,\dotsc,I^N)\in \cI_N^\uparrow$.  
		Let  
		$(\wt\arrX^1,\dotsc,\wt\arrX^N)=\Daop^{(N)}(I^1,\dotsc,I^N)$
		be given by the   mapping   \eqref{Daop_def}.  Let $\{\arrX^{i,j}\}$ and $\{\arrV^{i,j}\}$ be the arrays defined in \eqref{m:ar5}  above.   Then $\wt\arrX^i=\arrX^{i,i}=\arrV^{i,i}$ for $i=1,\dotsc,N$. 
	\end{lemma}  
	
	\begin{proof}  
		It suffices to prove $\wt\arrX^N=\arrX^{N,N}$  because the same proof applies to all $i$. 
		
		Let  $\ell\in\lzb 1,N-1\rzb$.  In the $\arrX$-array of Figure~\ref{fig-ar},  consider  the step from column $\ell$ to column $\ell+1$.   This is done by  transforming  the $(N-\ell+1)$-vector 
		$\bigl( \arrX^{\ell,\ell}, \arrX^{\ell+1,\ell} , \dotsc,   \arrX^{N, \ell} \bigr)$
		into  the $(N-\ell)$-vector 
		\be\label{n-ell}\begin{aligned}
			&\bigl(\arrX^{\ell+1,\ell+1},  \arrX^{\ell+2,\ell+1} , \dotsc,  \arrX^{N, \ell+1} \bigr)  \\
			& =   \bigl(  \Dop(\arrV^{\ell,\ell}, \arrX^{\ell+1, \ell}),  \Dop(\arrV^{\ell+1,\ell}, \arrX^{\ell+2, \ell}) , \dotsc ,  \Dop(\arrV^{N-1,\ell}, \arrX^{N, \ell})  \bigr).
		\end{aligned}\ee
		The $\arrV$-variables above satisfy 
		\begin{align*}
			&\arrV^{\ell,\ell}=\arrX^{\ell,\ell}, \quad 
			\arrV^{\ell+1,\ell}=\Rop(\arrV^{\ell,\ell}, \arrX^{\ell+1, \ell}), \quad \dotsc   
			\quad \arrV^{N-1,\ell}=  \Rop(\arrV^{N-2,\ell}, \arrX^{N-1, \ell}). 
		\end{align*}
		Invoking \eqref{dogpart} and then \eqref{n-ell} 
		gives
		\be\label{m:ar8} \begin{aligned}
			&\Dop^{(N-\ell+1)}\bigl( \arrX^{\ell,\ell}, \arrX^{\ell+1,\ell}  , \dotsc ,    \arrX^{N, \ell} \bigr) \\
			&= \Dop^{(N-\ell)}\bigl( \Dop(\arrV^{\ell,\ell}, \arrX^{\ell+1, \ell}),   \Dop(\arrV^{\ell+1,\ell}, \arrX^{\ell+2, \ell})  , \dotsc,  \Dop(\arrV^{N-1,\ell}, \arrX^{N, \ell})\bigr)\\   
			&=  \Dop^{(N-\ell)}\bigl(\arrX^{\ell+1,\ell+1},  \arrX^{\ell+2,\ell+1} , \dotsc , \arrX^{N, \ell+1} \bigr). 
		\end{aligned}\ee
		In the derivation below, use the first line of \eqref{m:ar5} to replace each $I^i$ with $\arrX^{i,1}$.  Then iterate \eqref{m:ar8}  from $\ell=1$  to $\ell=N-2$ to obtain  
		\begin{align*}
			&\wt\arrX^{N}=\Dop^{(N)}(I^1, I^2, \dotsc, I^N)  
			=  \Dop^{(N)}\bigl(\arrX^{1, 1},  \arrX^{2, 1}, \dotsc, \arrX^{N, 1} \bigr)\\
			&=  \Dop^{(N-1)}\bigl(\arrX^{2,2},  \arrX^{3,2} , \dotsc,  \arrX^{N, 2} \bigr)\\
			&=\dotsm= \Dop^{(3)}(\arrX^{N-2,N-2} , \arrX^{N-1,N-2}, \arrX^{N,N-2} ) = \Dop( \arrX^{N-1,N-1}, \arrX^{N,N-1} )=\arrX^{N,N}. 
			\qedhere \end{align*} 
	\end{proof}

Before turning to inverse-gamma weights, we make an observation about geometric RSK. 

\begin{remark}[Ingredients of geometric row insertion]   \label{rmk:a_grsk}  As in Section~\ref{sec:Sgrsk}, to observe  the geometric row insertion in algorithm \eqref{m:ar5}, we switch from ratio variables $\arrX^{i,j}_m$ to polymer partition functions $\arrZ^{i,j}_m$.
  Since step~\ref{def:arra} in Definition~\ref{def:arr}  is just a straightforward assignment  for $i=1$, let $i\ge 2$.  

For each $i\ge2$ repeat these steps.  Given the input $I^i$, pick an initial sequence  $\arrZ^{i,1}$ that satisfies  
$\arrZ^{i,1}_k/\arrZ^{i,1}_{k-1}=I^i_k$.   Then, with the additional input $\arrV^{i-1,\tspa\parng{1}{i-1}}$ from the previous round  $i-1$, for  $j=2,\dotsc,i $   and  $m\in\Z$  define partition functions 
	\eq{ 
	\arrZ^{i,j}_{m}=\sum_{\ell:\,\ell\le m}    \arrZ^{i,j-1}_\ell\,\prod_{k=\ell}^{m} \arrV^{i-1,j-1}_k \,.
	}
The outputs $\arrX^{i,j}$ are the ratio  variables $\arrX^{i,j}_m=\arrZ^{i,j}_m/\arrZ^{i,j}_{m-1}$.  Along the way, construct the auxiliary outputs $\arrV^{i, \tspa\parng{1}{i}}$ as in \eqref{m:ar5}.  
	
In the variables $(\arrZ, \arrV)$, equations \eqref{m:ar5} can be represented by the following iteration as the $m$-index runs from $-\infty$ to $\infty$:  
\be\label{Z359} \begin{aligned}     \arrZ^{i,1}_m &=  \arrZ^{i,1}_{m-1} \tspa I^i_m,  \\
 \arrZ^{i,j}_{m} &=  (\arrZ^{i,j}_{m-1}  + \arrZ^{i,j-1}_{m}) \arrV^{i-1,j-1}_m , \quad  j=2,\dotsc,i, \\
 \arrV^{i,j-1}_m&=  \arrV^{i-1,j-1}_m \tspa \frac{\arrZ^{i,j}_{m-1} \arrZ^{i,j-1}_m}{{\arrZ^{i,j-1}_{m-1}}    {\arrZ^{i,j}_m}   } 
  , \quad  j=2,\dotsc,i, \\
 \arrV^{i,i}_m&=    \frac{\arrZ^{i,i}_m}{{\arrZ^{i,i}_{m-1}}  } .  
  \end{aligned}\ee
Comparison with \eqref{g-row-ins} shows that the first three lines of \eqref{Z359} constitute  the geometric  row insertion 
\[
{\small \begin{array}{ccc}
&  ( I^i_m, \arrV^{i-1,\tspa\parng{1}{i-1}}_m  )  & \\[4pt]
Z^{i, \parng{1}{i}}_{m-1}   & \cross &  Z^{i, \parng{1}{i}}_m  \\[6pt] 
& \arrV^{i,\tspa\parng{1}{i-1}}_m
\end{array} } 
\]
In a network in the style of Figure~\ref{fig:Z200}, the next row insertion below would be 
\[
{\small \begin{array}{ccc}
&  ( I^{i+1}_m, \arrV^{i,\tspa\parng{1}{i}}_m  )  & \\[4pt]
Z^{i+1, \parng{1}{i+1}}_{m-1}   & \cross &  Z^{i+1, \parng{1}{i+1}}_m  \\[6pt] 
& \arrV^{i+1,\tspa\parng{1}{i}}_m
\end{array} } 
\]
As we go vertically down from line $i$ to line $i+1$,   the length of the $Z$-vectors increases from  $i$ to $i+1$. To match this length, the output $ \arrV^{i,\tspa\parng{1}{i-1}}_m$ of length $i-1$ from line $i$  is augmented by the inclusion of $I^{i+1}_m$ from the initial input and by $\arrV^{i,i}_m$ from the fourth line of  equation \eqref{Z359}, and then fed into the row insertion at line $i+1$. 
\qedex\end{remark}


\subsection{Array with inverse-gamma weights} \label{ss_inv_ga}
	
This section derives properties of the array  under inverse-gamma weights and culminates in the proof of Theorem~\ref{B-th5}.

	\begin{lemma}\label{lm:ar-dist}  
	   Fix  $N\in\Z_{>0}$ and  $\lambda_1>\dotsm>\lambda_N>0$.  
	Let   $I^{\parng{1}{N}}=(I^1, \dotsc, I^N)$ have law 
		$\nu^{(\lambda_1,\dotsc, \lambda_N)}$.  Then the following hold for the arrays  $\{\arrX^{i,j}\}$ and $\{\arrV^{i,j}\}$ from \eqref{m:ar5}.
		\begin{enumerate}  [label={\rm(\alph*)}, ref={\rm(\alph*)}]   \itemsep=3pt 
			\item\label{lm:ar-dist.i}  Both arrays have the law $\mu^{(\lambda_1,\dotsc,\lambda_N)}$  on the right diagonal. That is,
			$$(\arrX^{1,1},\dotsc,\arrX^{N,N})=(\arrV^{1,1},\dotsc,\arrV^{N,N})\sim \mu^{(\lambda_1,\dotsc,\lambda_N)}.$$
			\item\label{lm:ar-dist.ii} For each $i\in\lzb 1,N\rzb$,   the row $(\arrV^{i,1}, \arrV^{i,2},  \dotsc, \arrV^{i,i})$ has  law $\nu^{(\lambda_1, \,\lambda_{2},\dotsc, \,\lambda_i)}$. 
			\item\label{lm:ar-dist.iii} For each  $j\in\lzb 1,N\rzb$,   the column $(\arrX^{j,j}, \arrX^{j+1,j},  \dotsc, \arrX^{N,j})$ has  law  $\nu^{(\lambda_j,\dotsc,\lambda_N)}$.
		\end{enumerate}
	\end{lemma}
	
	\begin{proof}  
		{Part~\ref{lm:ar-dist.i}. } This part follows from Lemma~\ref{lm:eta5} and the definition of $\mu^{(\lambda_1,\dotsc, \lambda_N)}$ as the push-forward of $\nu^{(\lambda_1,\dotsc, \lambda_N)}$ under the mapping $\Daop^{(N)}$.

		\smallskip 
		
		{Part~\ref{lm:ar-dist.ii}. } We shall show that $\arrX^{i,j} \sim \nu^{\lambda_i}$ and $(\arrV^{i,1}, \arrV^{i,2},  \dotsc, \arrV^{i,i}) \sim \nu^{(\lambda_1, \,\lambda_{2},\dotsc, \,\lambda_i)}$.
		
		The claims are immediate for $i=1$ because there is just one sequence $\arrX^{1,1}=I^1=\arrV^{1,1}$ that has distribution $\nu^{\lambda_1}$.  Let $i\in\lzb 2,N\rzb$ and assume inductively that 
		\be\label{m:ar30} \begin{aligned} 
			&\text{elements $\arrV^{i-1,1}, \arrV^{i-1,2},  \dotsc, \arrV^{i-1,i-1}$ of $\R_{>0}^\Z$ 
			are independent, and $\arrV^{i-1,j}\sim\nu^{\lambda_{j}}$. }
		\end{aligned}\ee 
		
		We extend \eqref{m:ar30}  from $i-1$ to $i$.  
		By construction,  $\arrX^{i,1}=I^i\sim\nu^{\lambda_i}$ is independent of $\arrV^{i-1,\rcbullet}$. 
		Run $j$-induction upward through $j=2 \dotsc, i$.  The first pair 
		\[  \begin{cases}  
		\arrX^{i,2}=\Dop(\arrV^{i-1,1}, \arrX^{i,1}) =\Dop(\arrV^{i-1,1}, I^{i})\\
		\arrV^{i,1} =\Rop(\arrV^{i-1,1}, \arrX^{i,1})=\Rop(\arrV^{i-1,1}, I^{i})
		\end{cases} \]
		is independent of $\arrV^{i-1,2},  \dotsc, \arrV^{i-1,i-1}$.   According to Lemma~\ref{lm:invg2}, $\arrX^{i,2}$ and $\arrV^{i,1}$ are independent,  $\arrV^{i,1}$ inherits the law $\nu^{\lambda_{1}}$  of $\arrV^{i-1,1}$, while   $\arrX^{i,2}$ inherits the law $\nu^{\lambda_i}$ of  $\arrX^{i,1}$.  
		
		Inside this $i$-step we do induction on  $j\in\lzb 1, i-1\rzb$.  Induction assumption:  after constructing the pair $(\arrV^{i,j}, \arrX^{i,j+1})$,  the  sequences  
		\be\label{m:ar31}  
		\arrV^{i,1}, \dotsc, \arrV^{i,j-1}, (\arrV^{i,j}, \arrX^{i,j+1}) ,\arrV^{i-1, j+1}, \arrV^{i-1, j+2}, \dotsc, \arrV^{i-1,i-1} 
		\ee
		are independent, and the marginal distributions are $\arrV^{i,\ell}\sim\nu^{\lambda_{\ell}}$ for $\ell\in\lzb 1, j\rzb$, $\arrX^{i,j}\sim\nu^{\lambda_i}$, and $\arrV^{i-1,r}\sim\nu^{\lambda_{r}}$ for $r\in\lzb j+1,i-1\rzb$ (the last one inherited from the induction assumption on $i-1$).   The induction assumption was just verified for $j=1$ in the previous paragraph.  
		
		The tail   $\arrV^{i-1, j+2},  \dotsc, \arrV^{i-1,i-1}$ of \eqref{m:ar31} consists of those row  $i-1$  elements   that have not yet been used to construct row $i$ elements.  
		Next construct the pair 
		\[  \begin{cases}  \arrX^{i,j+2} =\Dop(\arrV^{i-1,j+1}, \arrX^{i,j+1})\\[3pt]
		\arrV^{i,j+1}= \Rop(\arrV^{i-1,j+1}, \arrX^{i,j+1}).  \end{cases}  \]
		This transforms the independent pair $(\arrX^{i,j+1}, \arrV^{i-1,j+1})$ in the middle of \eqref{m:ar31} into the  independent pair 
		$ (\arrV^{i,j+1}, \arrX^{i,j+2})$.  Again by Lemma~\ref{lm:invg2}, $\arrV^{i,j+1}$ inherits the distribution $\nu^{\lambda_{j+1}}$  of $\arrV^{i-1,j+1}$ and $\arrX^{i,j+2}$ inherits the distribution $\nu^{\lambda_i}$ of  $\arrX^{i,j+1}$. 
		Thus the induction assumption \eqref{m:ar31} has been advanced from $j$ to $j+1$.  
		
		At the end of the $j$-induction at $j=i-1$ we have constructed the pair $ (\arrV^{i,i-1}, \arrX^{i,i})$
		and \eqref{m:ar31} has been transformed into 
		\[ 
		\arrV^{i,1},\arrV^{i,2}, \dotsc, \arrV^{i,i-1}, \arrX^{i,i}. 
		\]
		Finally recall that $\arrV^{i,i}= \arrX^{i,i}$.  Induction assumption \eqref{m:ar30} has been advanced from $i-1$ to $i$. 
		
		\smallskip 
		
		{Part~\ref{lm:ar-dist.iii}. }    Since the columns of the $\arrX$-array follow the sequential transformation \eqref{arrX-S}, this follows from the invariance of product inverse-gammas in Theorem~\ref{thm-I}. 
	\end{proof}


For the remainder of the section, we introduce alternative notation for the mappings \eqref{drs_def}: $\wt I^{\tspb\Yw,\,I}=\Dop(\Yw,I)$,  $J^{\tspb\Yw,\,I}=\Sop(\Yw,I)$ and $\wt\Yw^{\tspb\Yw,\,I}=\Rop(\Yw,I)$. 
	

	\begin{lemma}\label{lm:eta12}    Fix  $\lambda_1>\dotsm>\lambda_N>0$ 
		and let   $I^{\parng{1}{N}}=(I^1, \dotsc, I^N)$ have law 
		$\nu^{(\lambda_1,\dotsc, \lambda_N)}$.    Let $\arrX^{\parng{1}{N}}=(\arrX^1,\dotsc,\arrX^N)=\Daop^{(N)}(I^{\parng{1}{N}})$ and let $\{\arrX^{i,j}\}$ and $\{\arrV^{i,j}\}$ be the arrays from \eqref{m:ar5}.    Then for each $m\in\lzb 2,N\rzb$ and $k\in\Z$, the following random variables are independent: 
		\[ 
		\{\arrV^{m,1}_i\}_{i\le k}, \{\arrV^{m,2}_i\}_{i\le k}, \dotsc, \{\arrV^{m,m-1}_i\}_{i\le k}, \{\arrX^m_i\}_{i\le k-1}, \,
		\frac{\arrX^m_k}{\arrX^{m-1}_k}\,, \, \frac{\arrX^{m-1}_k}{\arrX^{m-2}_k}\,, \dotsc,   \frac{\arrX^2_k}{\arrX^1_k}\,,
		\, \arrX^1_k.
		\]
	\end{lemma} 
	
	\begin{proof}  The index $k$ is fixed throughout.  Recall the connection $\arrX^i=\arrX^{i,i}=\arrV^{i,i}$ from Lemma~\ref{lm:eta5}.   We begin with the case $m=2$ and then undertake two nested loops of induction. 
		
		
		By the definitions and Lemma~\ref{lm:invg2},    $\arrX^1=I^1\sim\nu^{\lambda_1}$, 
		\[  \arrV^{2,1}=\Rop(\arrX^{2,1}, \arrV^{1,1})=\Rop(I^1\!, I^2)=\wt\Yw^{I^1\!,\, I^2} \ \text{ and }\ \arrX^2=\Dop(I^1\!, I^2)=\wt I^{I^1\!,\, I^2}\sim\nu^{\lambda_2}. \]   
		Lemma~\ref{lm:invg2}\ref{lm:invg2.a} gives the mutual independence of  
		\eq{
		\{\wt I^{I^1\!,\, I^2}_i\}_{i\le k-1}, \quad  J^{I^1\!,\, I^2}_{k-1}, \quad \text{and} \quad \{\wt\Yw^{I^1\!,\, I^2}_i\}_{i\le k-1}.
		}   
		These are functions of $\{I^1_i, I^2_i\}_{i\le k-1}$, and thereby  independent of $I^1_k, I^2_k$.  Thus we have the mutual independence of   $\{\arrV^{2,1}_i\}_{i\le k-1}$, $\{\arrX^2_i\}_{i\le k-1}$,   $\arrX^1_k$  and the pair $(J^{I^1\!,\, I^2}_{k-1}, I^2_k)$. 
The reciprocals $\bigl((J^{I^1\!,\, I^2}_{k-1})^{-1}, (I^2_k)^{-1}\bigr)$ of this last pair are an independent 
$({\rm Ga}(\lambda_1-\lambda_2), {\rm Ga}(\lambda_2))$ pair. 		
Then the beta-gamma algebra of random variables  \cite[Exercise 6.50, p.~244]{ande-sepp-valk}   implies  the independence of  
		\be\label{m:ar26} 
		\begin{aligned} 
(\arrV^{2,1}_k)^{-1}&\overset{\rm\eqref{I_W_def}}= {(I^2_k)^{-1} + (J^{I^1\!,\, I^2}_{k-1})^{-1}}  \sim {\rm Ga}(\lambda_1)  \\
 \text{and}\qquad \frac{\arrX^1_k}{\arrX^2_k}&\stackrefp{I_W_def}{=}\frac{I^1_k}{\wt I^{I^1\!,\, I^2}_k} 
\overset{\rm\eqref{m801a}}=\frac{(I^2_k)^{-1}}{(I^2_k)^{-1} + (J^{I^1\!,\, I^2}_{k-1})^{-1}}  \sim {\rm Beta}(\lambda_2, \lambda_1-\lambda_2).  			
		\end{aligned} 
		\ee
		We have the independence of   $\{\arrV^{2,1}_i\}_{i\le k}$, $\{\arrX^2_i\}_{i\le k-1}$,   $\arrX^2_k/\arrX^1_k$,\, $\arrX^1_k$.   This concludes the case $m=2$ of the lemma.

		Now let $m\ge 3$ and make an induction assumption: 
		\be\label{m:ar44} 
		\begin{aligned}
			&\text{$\{\arrV^{m-1,1}_i\}_{i\le k}, \dotsc,\{\arrV^{m-1,m-2}_i\}_{i\le k}$,  } \\
			&\text{ $\{\arrX^{m-1}_i\}_{i\le k-1}, \,\arrX^{m-1}_k/\arrX^{m-2}_k,\dotsc,\arrX^2_k/\arrX^1_k,\, \arrX^1_k$ are independent. }
		\end{aligned}\ee
		The previous paragraph verified   this assumption for $m=3$. (Note that the meaning of $m$ shifted by one.)  Our task is to verify this statement with $m-1$ replaced by $m$. 
		
		
		Since $\arrX^{m,1}=I^m$ is independent of all the variables in \eqref{m:ar44},  apply Lemma~\ref{lm:invg2}\ref{lm:invg2.b} to the pair  $\arrV^{m,1}=\Rop( \arrV^{m-1,1}, \arrX^{m,1})$,  $\arrX^{m,2}=\Dop( \arrV^{m-1,1}, \arrX^{m,1})$ and  \eqref{m:ar31}   to conclude the independence of 
		\be\label{m:ar46} 
		\begin{aligned}
			& \bigl(\{\arrV^{m,1}_i\}_{i\le k},   \{\arrX^{m,2}_i\}_{i\le k}\bigr),  \{\arrV^{m-1,2}_i\}_{i\le k}, \dotsc,\{\arrV^{m-1,m-2}_i\}_{i\le k},   \\
			&\{\arrX^{m-1}_i\}_{i\le k-1}, \arrX^{m-1}_k/\arrX^{m-2}_k,\dotsc,\arrX^2_k/\arrX^1_k,\, \arrX^1_k. 
		\end{aligned}\ee
		
		This starts an inner   induction loop on $j=1, 2, \dotsc, m-2$, whose induction assumption is the independence of 
		\be\label{m:ar48} 
		\begin{aligned}
			& \{\arrV^{m,1}_i\}_{i\le k}, \dotsc, \{\arrV^{m,j-1}_i\}_{i\le k}, \bigl(\{\arrV^{m,j}_i\}_{i\le k},   \{\arrX^{m,j+1}_i\}_{i\le k}\bigr),  \{\arrV^{m-1,j+1}_i\}_{i\le k}, \dotsc,  \\
			&\{\arrV^{m-1,m-2}_i\}_{i\le k}, \{\arrX^{m-1}_i\}_{i\le k-1}, \arrX^{m-1}_k/\arrX^{m-2}_k,\dotsc,\arrX^2_k/\arrX^1_k,\, \arrX^1_k. 
		\end{aligned}\ee
		The base case $j=1$ is \eqref{m:ar46} above.  
		The induction step is an application of Lemma~\ref{lm:invg2}\ref{lm:invg2.b} to the pair  
		\eq{
		\arrV^{m,j+1}=\Rop(\arrV^{m-1,j+1}, \, \arrX^{m,j+1}),  \quad \arrX^{m,j+2}=\Dop(\arrV^{m-1,j+1}, \, \arrX^{m,j+1})
		}
		to advance the  induction assumption \eqref{m:ar48} from $j$ to $j+1$. 
		At the end of the $j$-induction at $j=m-2$   all the  $\arrV^{m-1, \tspb\rcbullet}$ sequences  have been converted to  $\arrV^{m, \rcbullet}$ sequences, and we have  independence of 
		\be\label{m:ar52} 
		\begin{aligned}
			& \{\arrV^{m,1}_i\}_{i\le k}, \dotsc,  \{\arrV^{m,m-3}_i\}_{i\le k}, \{\arrV^{m,m-2}_i\}_{i\le k},\{\arrX^{m,m-1}_i\}_{i\le k} ,   \\
			&\{\arrX^{m-1}_i\}_{i\le k-1}, \arrX^{m-1}_k/\arrX^{m-2}_k,\dotsc,\arrX^2_k/\arrX^1_k,\, \arrX^1_k. 
		\end{aligned}\ee
		
		
		We return to advancing the induction assumption \eqref{m:ar44}   from $m-1$ to $m$. 
		Separate $ \{\arrX^{m,m-1}_i\}_{i\le k}$ into  $ \{\arrX^{m,m-1}_i\}_{i\le k-1}$ and $\arrX^{m,m-1}_k$, which are independent by Lemma~\ref{lm:ar-dist}\ref{lm:ar-dist.iii}. 
		Combine the former with $\{\arrX^{m-1}_i\}_{i\le k-1}$,  Lemma~\ref{lm:invg2}\ref{lm:invg2.a}, and the transformations 
		\[  \begin{cases} \arrV^{m,m-1}=\Rop(\arrX^{m-1}, \arrX^{m,m-1})\\[3pt]
		\arrX^{m}=\Dop(\arrX^{m-1}, \arrX^{m,m-1})  \end{cases}  \]
		to form the independent variables $\{\arrV^{m,m-1}_i\}_{i\le k-1}$, $\{\arrX^{m}_i\}_{i\le k-1}$,  $J^{\arrX^{m-1}, \arrX^{m,m-1}}_{k-1}$.    
		
		As above in \eqref{m:ar26}, transform the independent pair $(\arrX^{m,m-1}_k,  J^{\arrX^{m-1}, \arrX^{m,m-1}}_{k-1})$  into 
		the  independent pair  
		\[  \frac1{\arrV^{m,m-1}_k}= \frac1{\arrX^{m,m-1}_k}+ \frac1{J^{\arrX^{m-1}, \arrX^{m,m-1}}_{k-1}} 
		\quad\text{and}\quad 
		\frac{\arrX^m_k}{\arrX^{m-1}_k}=1+ \frac{\arrX^{m,m-1}_k}{ J^{\arrX^{m-1}, \arrX^{m,m-1}}_{k-1}}.  
		 \]
		  Attach $\arrV^{m,m-1}_k$ to the sequence $\{\arrV^{m,m-1}_i\}_{i\le k-1}$.  After these steps, the independent variables of  \eqref{m:ar52}  have been transformed into the independent variables 
		 \eq{
		 \{\arrV^{m,1}_i\}_{i\le k}, \dotsc,\{\arrV^{m,m-1}_i\}_{i\le k},   \{\arrX^{m}_i\}_{i\le k-1},
		 \frac{\arrX^{m}_k}{\arrX^{m-1}_k}, \,\frac{\arrX^{m-1}_k}{\arrX^{m-2}_k},\dotsc,\frac{\arrX^2_k}{\arrX^1_k},\, \arrX^1_k. 
}
		Thus the induction assumption \eqref{m:ar44} has been advanced from $m-1$ to $m$. 
	\end{proof}

\begin{proof}[Proof of Theorem~\ref{B-th5}]
It suffices to show the equality in distribution 
\be\label{X780} \begin{aligned} 
& \bigl( \log \Yw_x, \, \Bus^{\chdir(\rho_1)}_{x-\evec_1,x}- \log\Yw_x,\, \Bus^{\chdir(\rho_2)}_{x-\evec_1,x}-\Bus^{\chdir(\rho_1)}_{x-\evec_1,x}\,, \dotsc, \Bus^{\chdir(\rho_{N})}_{x-\evec_1,x}-\Bus^{\chdir(\rho_{N-1})}_{x-\evec_1,x} \bigr) \\
&\qquad\qquad
\,\deq \, 	 \bigl(\Zpr(0), \Zpr(\rho_1)-\Zpr(0), \Zpr(\rho_2)-\Zpr(\rho_1),\dotsc, \Zpr(\rho_N)-\Zpr(\rho_{N-1})\bigr) 	  	
\end{aligned} \ee
for arbitrary but henceforth fixed parameters $0< \rho_1 <\dotsm <\rho_N< \alpha$.  The initial values at $\rho=0$  satisfy $\Bus^{\chdir(0)}_{x-\evec_1,x}=\log\Yw_x\deq\Zpr(0)\sim\log {\rm Ga}^{-1}(\alpha)$ by the definition. 

 We represent the law of the Busemann process as the image of independent inverse-gamma weights.  Let  the $(\R_{>0}^\Z)^{N+1}$-valued  configuration $I^{\parng{0}{N}}$ 
 have law 
$\nu^{(\alpha, \alpha-\rho_1,\dotsc, \alpha-\rho_N)}$ and let 
$\arrX^{\parng{0}{N}}=(\arrX^0,\dotsc,\arrX^N)=\Daop^{(N+1)}(I^{\parng{0}{N}})$.
		By Theorem~\ref{B1_thm},  	
		\eq{
		(\Yw(t),{\brvI}^{\chdir(\rho_1)}(t), \dotsc, {\brvI}^{\chdir(\rho_N)}(t))\deq\arrX^{\parng{0}{N}}\sim\mu^{(\alpha, \alpha-\rho_1,\dotsc, \alpha-\rho_N)}.
		}  
Taking logarithms of the coordinates  gives
		\be\label{X784} \begin{aligned} 
			&\bigl(\log \Yw_x, \, \Bus^{\chdir(\rho_1)}_{x-\evec_1,x}- \log\Yw_x,\, \Bus^{\chdir(\rho_2)}_{x-\evec_1,x}-\Bus^{\chdir(\rho_1)}_{x-\evec_1,x}\,, \dotsc, \Bus^{\chdir(\rho_{N})}_{x-\evec_1,x}-\Bus^{\chdir(\rho_{N-1})}_{x-\evec_1,x} \bigr)  \\
			&\qquad\qquad
			 \deq \,  \big( \log\arrX^0_k, \, \log(\arrX^1_k/\arrX^0_k), \, \log(\arrX^2_k/\arrX^1_k), \dotsc, \log(\arrX^{N}_k/\arrX^{N-1}_k)\big). 
		\end{aligned}\ee 
The choices of the lattice locations  $x\in\Z^2$, $t\in\Z$ and  $k\in\Z$  above  are entirely  arbitrary  because all the distributions are invariant under lattice translations. 

 Lemma~\ref{lm:eta12} and \eqref{X784} give the independence of the coordinates on the left-hand side of \eqref{X780}. On the right of \eqref{X780} the independence of the $\Zpr$-increments follows from the definition \eqref{Def:X}.   Thus it remains to check the distributional equality of a single increment: 
\be\label{X788}		
	\log(\arrX^m_k/\arrX^{m-1}_k)\,\deq\,\Zpr(\rho_{m})-\Zpr(\rho_{m-1}).  
	\ee

The distribution  of  $\arrX^m_k/\arrX^{m-1}_k$  comes from the 2-component mapping 
\eq{
(\arrX^{m-1}, \arrX^m)=\Daop^{(2)}(I^{m-1}, I^m)=(I^{m-1}, \Dop(I^{m-1}, I^m)),
}
where  $(I^{m-1}, I^{m})\sim\nu^{\alpha-\rho_{m-1},\,\alpha-\rho_m}$.  This was stated in \eqref{m:ar26} for the reciprocal:   
\eeq{ \label{X792}	 
{\arrX^{m-1}_k}/{\arrX^m_k}  \,\sim\, {\rm Beta}(\alpha-\rho_m, \rho_m-\rho_{m-1}).
}  
We now examine the right-hand side of \eqref{X788}.
By definition \eqref{Def:X},
\[\Zpr(\rho_{m})-\Zpr(\rho_{m-1}) =\sum_{(s,y)\in\,\Prm}F(s,y)
\quad\text{where}\quad 
F(s,y)=y\ind_{(\rho_{m-1},\rho_m]}(s).
 \]    
Apply \eqref{LaplaceF}   to compute the Laplace transform of $\Zpr(\rho_{m})-\Zpr(\rho_{m-1})$ for $t\ge0$: 
\begin{align*}
&\E\big[e^{-t\tspa(\Zpr(\rho_{m})-\Zpr(\rho_{m-1}))}\big]
= \exp\Big\{-\int_0^\alpha \dd s \int_0^{\infty} \dd y\, (1-e^{-tF(s,y)})\,\sigma(s,y)\Big\}\\
&\quad = \exp\Big\{-\int_{\rho_{m-1}}^{\rho_m} \dd s \int_0^\infty\dd y\,(1-\mathrm{e}^{-t\tspa y})\,\frac{e^{-y(\alpha-s)}}{1-e^{-y}}\Big\} 
\\&\quad 
= \exp\Big\{\int_{\rho_{m-1}}^{{\rho_m}}\bigl[ \psi_0(\alpha-s)-\psi_0(\alpha-s+t)\bigr] \,\dd s\Big\}\\
&\quad = \exp\Big\{\log\frac{\Gamma(\alpha-\rho_{m-1})}{\Gamma(\alpha-{\rho_m})} 
- \log\frac{\Gamma(\alpha-\rho_{m-1}+t)}{\Gamma(\alpha-\rho_m+t)} \Big\} 
= \frac{\BBet(\alpha-\rho_m+t,\,  \rho_m-\rho_{m-1})}{\BBet(\alpha-{\rho_m},\,{\rho_m}-\rho_{m-1})}\\
&\quad = \frac1{\BBet(\alpha-{\rho_m},\,{\rho_m}-\rho_{m-1})}  \int_0^1  e^{-t \log u^{-1}}  \tspb u^{\alpha-\rho_m}\tspb (1-u)^{\rho_m-\rho_{m-1}}\,\dd u . 
\end{align*}
Above we used   $\frac{\dd }{\dd s}\log\Gamma(s)=\psi_0(s)=\int_0^{\infty}\bigl(\frac{e^{-r}}{r}-\frac{e^{-sr}}{1-e^{-r}}\bigr)\,\dd r$. 
The calculation establishes $\Zpr(\rho_{m})-\Zpr(\rho_{m-1})\sim \log{\rm Beta}^{-1}(\alpha-\rho_m, \rho_m-\rho_{m-1})$, which together with \eqref{X792} verifies \eqref{X788}. 
\end{proof}


\appendix

 \section{Busemann process} \label{app_bp} 
 
 We present  two complements to the general properties of the Busemann process.

\subsection{Shape theorem for Busemann functions}  \label{a:Bshape}

This section shows that the shape theorem  holds simultaneously for all Busemann functions on a single full-probability event.
We recall the statement of Theorem~\ref{all_lln}.

 \begin{theorem} \label{all_lln_app}  Assume \eqref{ranenv}. 
There exists a full-probability 
event on which the following limit  holds simultaneously 
 for each $\xi\in\,]\evec_2,\evec_1[\,$ and $\sigg\in\{-,+\}$:
\eeq{ \label{bfanfx_app}
\lim_{n\to\infty}\;\max_{|x|_1\le\tsp n}   n^{-1} {|\Bus^{\xi\sig}_{\zevec,x}-\nabla\gpp(\xi\sigg)\cdot x|} = 0 .  
}
\end{theorem}

This improves the following input.

\begin{theirthm} \textup{\cite[Thm.~4.4, Lem.~4.12]{janj-rass-20-aop}} \label{lln_input}
For each $\xi\in\,]\evec_2,\evec_1[\,$, there exists a full-probability 
event $\Omega_\xi$ on which \eqref{bfanfx_app} holds for both signs $\sigg\in\{-,+\}$.
\end{theirthm}

\begin{proof}[Proof of Theorem~\ref{all_lln_app}]
Let $\Udiff_0$ be a countable dense subset of $\Udiff$, the directions of differentiability for $\gpp$.
Since $\gpp$ is concave, the set $\Udiff^\cc = \,]\evec_2,\evec_1[\,\setminus\Udiff$ is countable, so we can consider the countable set $\cC \coloneqq \Udiff_0\cup\Udiff^\cc$.
For each $\zeta\in\cC$, let $\Omega_\zeta$ be the full-probability event from Theorem~\ref{lln_input}.
For convenience, when $\zeta\in\Udiff_0$, we will assume that $\Omega_\zeta\subset\{\Bus^{\zeta-}=\Bus^{\zeta+}\}$.
Let   $\Omega_0 \coloneqq \bigcap_{\zeta\in\cC}\Omega_\zeta$,  again a full-probability event. 
We show that on $\Omega_0$, the limit \eqref{bfanfx_app} holds for every direction $\xi\in\,]\evec_2,\evec_1[\,$ and both signs $\sigg\in\{-,+\}$.
We may assume $\xi\in\Udiff$ since $\Udiff^\cc\subset\cC$.

Given $\xi\in\Udiff$ and some $\eps>0$, choose directions $\zeta,\eta\in\Udiff_0$ such that $\zeta\prec\xi\prec\eta$ and
\eeq{ \label{cdirec}
|\nabla\gpp(\zeta) - \nabla\gpp(\xi)|_1 \leq \eps \quad \text{and} \quad
|\nabla\gpp(\xi) - \nabla\gpp(\eta)|_1 \leq \eps.
}
We show that the following quantity is $o(n)$ on the event $\Omega_0$: 
\eq{
M_\xi(n) \coloneqq \max_{|x|_1\le \tsp n\tsp,\tspc \sig\tsp\in\tsp\{-,+\}} |\Bus^{\xi\sig}_{\zevec,x}-\nabla\gpp(\xi)\cdot x|. 
}
Let $x = (a,b)\in\Z^2$ satisfy $|x|_1\le n$. 
For ease of exposition, assume that $x$ lies in the first quadrant so that $a$ and $b$ are nonnegative. (Along the way, we indicate what changes if this is not true.)

Decompose $\Bus^{\xi\sig}_{\zevec,x}$ into horizontal and vertical increments:
\eeq{ \label{vhvh}
\Bus^{\xi\sig}_{\zevec,x} = \Bus^{\xi\sig}_{\zevec,a\evec_1} + \Bus^{\xi\sig}_{a\evec_1, a\evec_1+b\evec_2}.
}
For the horizontal increments, apply monotonicity \eqref{B_mono1}:
\eeq{ \label{huwq}
\Bus^{\zeta}_{\zevec,a\evec_1} \geq \Bus^{\xi\sig}_{\zevec,a\evec_1} \geq \Bus^\eta_{\zevec,a\evec_1}.
}
The upper bound admits a further sequence of inequalities:
\begin{subequations} \label{fmpo}
\eeq{ \label{fmpo1}
\Bus^\zeta_{\zevec, a\evec_1} \leq \nabla\gpp(\zeta)\cdot(a\evec_1) + M_\zeta(n) \stackref{cdirec}{\leq} \nabla\gpp(\xi)\cdot(a\evec_1) + a\eps + M_\zeta(n).
}
Similarly, the lower bound in \eqref{huwq} satisfies
\eeq{ \label{fmpo2}
\Bus^\eta_{\zevec, a\evec_1} \geq \nabla\gpp(\eta)\cdot(a\evec_1) - M_\eta(n) \geq \nabla\gpp(\xi)\cdot(a\evec_1) - a\eps - M_\eta(n).
}
\end{subequations}
Together \eqref{huwq}--\eqref{fmpo} yield 
\eeq{ \label{hoxi}
|\Bus^{\xi\sig}_{\zevec, a\evec_1} - \nabla\gpp(\xi)\cdot(a\evec_1)| \leq M_\zeta(n) + M_\eta(n) + a\eps.
}
If $a<0$, exchange $\zeta$ and $\eta$: \eqref{huwq} is replaced by
\eq{
\Bus^{\zeta}_{\zevec,a\evec_1} \leq \Bus^{\xi\sig}_{\zevec,a\evec_1} \leq \Bus^\eta_{\zevec,a\evec_1} \quad \text{for $a<0$}, 
}
and then \eqref{fmpo1} converted to further lower bounds and \eqref{fmpo2} to further upper bounds.
The replacement to \eqref{hoxi} would then be
\eq{
|\Bus^{\xi\sig}_{\zevec, a\evec_1} - \nabla\gpp(\xi)\cdot(a\evec_1)| \leq M_\zeta(n) + M_\eta(n) + |a|\eps.
}

Next we address  the vertical increment in \eqref{vhvh}. 
By monotonicity \eqref{B_mono2},
\eeq{ \label{huwp}
\Bus^\zeta_{a\evec_1, a\evec_1+b\evec_2} \leq \Bus^{\xi\sig}_{a\evec_1, a\evec_1+b\evec_2} \leq \Bus^\eta_{a\evec_1, a\evec_1+b\evec_2},
}
where the lower bound satisfies
\eq{
\Bus^\zeta_{a\evec_1, a\evec_1+b\evec_2} &= \Bus^\zeta_{\zevec,a\evec_1+b\evec_2} - \Bus^\zeta_{\zevec, a\evec_1} \\
&\geq [\nabla\gpp(\zeta)\cdot(a\evec_1+b\evec_2) - M_\zeta(n)] - [\nabla\gpp(\zeta)\cdot(a\evec_1) + M_\zeta(n)] \\
&= \nabla\gpp(\zeta)\cdot(b\evec_2) - 2M_\zeta(n)  
\geq \nabla\gpp(\xi)\cdot(b\evec_2) - b\eps - 2M_\zeta(n).
}
By analogous reasoning, the upper bound in \eqref{huwp} satisfies
\eq{
\Bus^\eta_{a\evec_1, a\evec_1+b\evec_2} \leq \nabla\gpp(\xi)\cdot(b\evec_2) + b\eps + 2M_\eta(n).
}
Together, the three previous displays imply
\eeq{ \label{vexi}
|\Bus^{\xi\sig}_{a\evec_1, a\evec_1+b\evec_2} - \nabla\gpp(\xi)\cdot(b\evec_2)| \leq 2M_\zeta(n) + 2M_\eta(n) + b\eps.
}
Similar to before, if $b$ were negative, replace $b\eps$ with $|b|\eps$ on the right-hand side.

Combining \eqref{vhvh}, \eqref{hoxi}, and \eqref{vexi}, we have
\eq{ 
|\Bus^{\xi\sig}_{\zevec,x} - \nabla\gpp(\xi)\cdot x| \leq 3M_\zeta(n) + 3M_\eta(n) + n\eps.
}
By virtue of $\zeta, \eta\in\Udiff_0\subset\cC$,  we have $M_\zeta(n) + M_\eta(n)=o(n)$ on the event $\Omega_0$. As $\eps>0$ is arbitrary, \eqref{bfanfx_app} follows.
\end{proof}

\subsection{Busemann limit}\label{a:Blim}
 This section refines the asymptotic Busemann bounds \eqref{B_monob} by showing that even in jump directions, the (exponentiated)  Busemann function is a limit of partition function ratios. 

\begin{proposition} \label{ebiwlp}
Assume \eqref{ranenv} and  \eqref{shape_ass}.  
Then the following holds almost surely. 
For every $\xi\in\,]\evec_2,\evec_1[\,$, $\sigg\in\{-,+\}$, $x\in\Z^2$, and $r\in\{1,2\}$, 
there exists an $\cL_\xi$-directed sequence $(x_\ell)$ such that
\eeq{ \label{exidn}
e^{\Bus^{\xi\sig}_{x-\evec_r,\tspb x}} = \lim_{\ell\to-\infty} \frac{Z_{x_\ell,\tspb x}}{Z_{x_\ell,x-\evec_r}}.
}
\end{proposition}

The following lemma is a consequence of the concavity of $\gpp$.  
Recall the definitions of $\underline\xi$ and $\overline\xi$ from \eqref{seg_end}. 

\begin{lemma} \label{leftdc}
The map $\xi\mapsto\underline\xi$ is left-continuous on $\,]\evec_2,\evec_1[\,$, while $\xi\mapsto\overline\xi$ is  right-continuous.
\end{lemma}

\begin{proof}   We prove the left-continuity of $\xi\mapsto\underline\xi$, as right-continuity of $\xi\mapsto\overline\xi$ is analogous.
Fix  $\xi\in\,]\evec_2,\evec_1[\,$.  We have two cases to consider.  

\smallskip 

\textit{Case 1.}  $\underline\xi<\xi$.    Then $\xi$ belongs to a linear segment of $\gpp$, and $\underline{\zeta}=\underline\xi$ for all $\zeta\in\,]\underline\xi,\xi]$.
In particular, $\zeta\mapsto\underline\zeta$ is left-continuous at $\xi$.

\smallskip 

\textit{Case 2.} 
$\underline\xi=\xi$.  Now,  according to definition \eqref{seg_end} 
and  concavity,  
\eq{
\gpp(\xi-)\cdot(\xi-\zeta) < \gpp(\xi)-\gpp(\zeta) \quad \text{for all $\zeta\in\,]\evec_2,\xi[\,$}.
}
Let $\zeta_0\in\,]\evec_2,\xi[\,$.  Since both sides of the above inequality are left-continuous in $\xi$, there is some $\zeta_1 \in\,]\zeta_0,\xi[\,$ such that
\eq{
\gpp(\zeta_1-)\cdot(\zeta_1-\zeta_0) < \gpp(\zeta_1)-\gpp(\zeta_0).
}
Hence $\zeta_0\prec\underline{\zeta_1}$ (again by definition \eqref{seg_end}), which forces the following for every $\zeta\in[\zeta_1,\xi]$:
\eq{
\zeta_0\prec\underline{\zeta_1}\preccurlyeq \underline{\zeta} \preccurlyeq \underline{\xi}\preccurlyeq\xi.
}
Since $\zeta_0$ can be chosen arbitrarily close to $\xi$, we have verified that $\zeta\mapsto\underline\zeta$ is left-continuous at $\xi$.
\end{proof}

\begin{proof}[Proof of Proposition~\ref{ebiwlp}]   We prove the case $(\sigg,r)=(-,1)$, as the three other cases are analogous. 
Let $\Udiff_0$ be a countable dense subset of $\Udiff$, the directions of differentiability for $\gpp$.
Since we have assumed \eqref{shape_ass}, the hypotheses of Theorem~\ref{buse_dir} are satisfied for every $\zeta\in\Udiff_0$.
So take $\Omega_\zeta$ to be the full-probability event from Theorem~\ref{buse_dir}, on which
\eeq{ \label{zetgo}
e^{\Bus^{\zeta-}_{x,\tspb y}}=e^{\Bus^{\zeta+}_{x,\tspb y}}=\lim_{\ell\to-\infty}\frac{Z_{y_\ell,\tspb y}}{Z_{y_\ell,\tspb x}} \quad \text{for all $x,y\in\Z^2$ and any $\cL_\zeta$-directed sequence $(y_\ell)$}.
}
In addition, let $\Omega_0$ be the full-probability from Theorem~\ref{buse_full}.
We will prove the claim of the proposition on the event $\Omega_1 \coloneqq \Omega_0\cap\big(\bigcap_{\zeta\in\Udiff_0}\Omega_\zeta\big)$. 


Let $\xi\in\,]\evec_2,\evec_1[\,$ and $x\in\Z^2$ be given. 
Take a sequence $(\zeta_k)_k$ in $\Udiff_0$ such that $\zeta_k\nearrow\xi$.
By \eqref{B_lim1},  
\eeq{ \label{budcon}
 \lim_{k\to\infty}e^{\Bus_{x-\evec_1,x}^{\zeta_k}} = e^{\Bus_{x-\evec_1,x}^{\xi-}}.
}
For each $k$, choose any $\cL_{\zeta_k}$-directed sequence $(y^{(k)}_\ell)_\ell$, meaning that
\eq{
\underline{\zeta_k}\cdot\evec_1 
\leq \varliminf_{\ell\to-\infty} \frac{y^{(k)}_\ell}{\ell}\cdot\evec_1
\leq \varlimsup_{\ell\to-\infty} \frac{y^{(k)}_\ell}{\ell}\cdot\evec_1
\leq \overline{\zeta_k}\cdot\evec_1
\leq \overline\xi\cdot\evec_1.
}
No matter our choice of sequence, \eqref{zetgo} ensures that
\eq{ 
\lim_{\ell\to-\infty}\frac{Z_{y^{(k)}_\ell,\tspb x}}{Z_{y^{(k)}_\ell,\tspb x-\evec_1}} = e^{\Bus^{\zeta_k}_{x-\evec_1,\tspb x}}.
}
We now inductively construct a decreasing sequence of integers $(\ell_k)_{k\geq1}$ as follows.
The initial value $\ell_1$ can be chosen arbitrarily.
For each $k\geq2$,  invoke the two previous displays to choose some $\ell_{k}< \ell_{k-1}$ such that
\eeq{ \label{dircon}
\underline{\zeta_k}\cdot\evec_1 - \frac{1}{k} \leq \frac{y_{\ell}^{(k)}}{\ell}\cdot \evec_1
\leq \overline\xi\cdot\evec_1 + \frac{1}{k} \quad \text{for all $\ell\leq\ell_k$} 
}
and 
\eeq{ \label{bnew}
\bigg|\frac{Z_{y^{(k)}_\ell,\tspb x}}{Z_{y^{(k)}_\ell,\tspb x-\evec_1}} - e^{\Bus^{\zeta_k}_{x-\evec_1,\tspb x}} \bigg| \leq \frac{1}{k}
 \qquad \text{for all $\ell\leq\ell_k$}.
}
Now consider the sequence $(x_\ell)_\ell$ defined by
\eq{
x_\ell = y_\ell^{(k)} \quad \text{when $\ell_{k+1}<\ell\leq\ell_k$}.
}
Since $\underline{\zeta_k}\nearrow\underline\xi$ as $k\to\infty$ by Lemma~\ref{leftdc}, it follows from \eqref{dircon} that
\eq{
\underline\xi \cdot\evec_1 
\leq \varliminf_{\ell\to-\infty}\frac{x_\ell}{\ell}
\leq \varlimsup_{\ell\to-\infty}\frac{x_\ell}{\ell}
\leq \overline\xi\cdot\evec_1.
}
That is, $(x_\ell)_\ell$ is $\cL_\xi$-directed.
The combination of \eqref{budcon} and \eqref{bnew} produces \eqref{exidn}.
\end{proof}



\section{Discrete stochastic heat equation}  \label{bmcon}

This appendix records  implications of our results for a  lattice version  of the stochastic heat equation (SHE).
There are four subsections.  
Section~\ref{suppB1} is purely for context; it briefly discusses the standard stochastic heat equation (SHE) and the related Kardar--Parisi--Zhang (KPZ) and stochastic Burgers equations (SBE).
Section~\ref{suppB2} introduces the discrete SHE on the lattice that is solved by polymer partition functions.
The associated eternal solutions are seen to be in correspondence with recovering cocycles (Lemma~\ref{lm:sheB7}).
Since the Busemann functions are in fact recovering cocycles, this leads to the existence of eternal solutions (Theorem~\ref{thm:V1}) and---more novelly---to the failure of 1F1S (Theorem~\ref{thm:she_ig}).
Section~\ref{suppB3} proves the results from Section~\ref{suppB2}.
The final Section~\ref{suppB4} offers a different representation of eternal solutions, identifying them with semi-infinite polymer measures (Theorem~\ref{thm:V-mu}).


\subsection{Polymers, SHE, KPZ and SBE} \label{suppB1}
In continuous time and space, the SHE with multiplicative space-time white noise $\dot W$ is the stochastic partial differential equation 
\eeq{ \label{sheq}
\partial_t\tsp\cZ = \tfrac{1}{2}\partial_{xx}\tsp\cZ + \cZ\tsp\dot W. 
}
With point mass initial condition $\cZ(0,x) = \delta_0(x)$, \eqref{sheq} is formally solved by the rescaled partition function of the continuum directed random polymer (CDRP) \cite{alberts_khanin_quastel14b}:  
\eq{ 
 \cZ(t,x)= \rho(t,x)  \tsp E\Bigl[:\!\exp\!: \!\Bigl(\; \int_0^t  \dot W(s, b(s))\,\dd s\Bigr)\Bigr],
}
where the expectation $E$ is over Brownian bridges $b(\cdot)$ from $b(0) = 0$ to $b(t) = x$,   $:\!\exp\!:$  is  the Wick exponential, and 
 $
\rho(t,x) = \frac{1}{\sqrt{2\pi t}} e^{- \frac{x^2}{2t}} \one\{t\in(0,\infty)\}
$
is the heat kernel.


Switching to the free energy $\cH = \log\cZ$ ($\cZ=e^\cH$ is also called the Hopf--Cole transform) takes us formally from SHE to 
  to the Kardar--Parisi--Zhang (KPZ) equation  
\eeq{ \label{kpzeq}
\partial_t \cH = \tfrac{1}{2}\partial_{xx}\cH + \tfrac{1}{2}(\partial_x \cH)^2 + \dot W.
}
Originally proposed in \cite{kardar_parisi_zhang86} as a model for the height profile of a growing interface, \eqref{kpzeq} is the universal scaling limit of various 1+1 dimensional stochastic  models under the so-called intermediate disorder scaling and is itself a member of the   KPZ universality class; see \cite{corwin12} for a survey.
 
Upon formally  taking a  spatial derivative $\cU = \partial_x\log\cZ$ 
we arrive at  the (viscous) stochastic Burgers equation (SBE)  
\eeq{ \label{stbur}
\partial_t\, \cU = \tfrac{1}{2}\partial_{xx}\, \cU + \cU\tsp\partial_x\,\cU + \partial_x \dot W.
}
  The  \textit{one force--one solution principle} (1F1S) is concerned with existence and uniqueness of eternal solutions  to \eqref{stbur} and its inviscid counterpart. This program was initiated  by Sinai \cite{sina-91}.

\medskip 
 
\subsection{Discrete SHE} \label{suppB2}
The directed polymer model of our paper is associated with a particular discretization of \eqref{sheq} on the planar integer lattice  $\Z^2$.
  Given an assignment $\Yw=(\Yw_v)_{v\tsp\in\tsp\Z^2}$ of  strictly positive weights, consider solutions  $\gs$  
of the equation
\eeq{ \label{origin_eq}
\gs(x) = \Yw_x\big[\gs(x-\evec_1)+\gs(x-\evec_2)\big]. 
}

\begin{remark}[Relation to usual SHE] 
 \label{rmk:dshe3} 
Equation \eqref{origin_eq} is a natural discrete counterpart of \eqref{sheq} because both are equations for  polymer partition functions.  We can also render \eqref{origin_eq} formally similar to \eqref{sheq} by choosing suitable  variables.    Let the forward diagonal $\etime \coloneqq \evec_1+\evec_2$ represent the   time direction and  $\espace \coloneqq \evec_1-\evec_2$ the positive spatial direction.  
Suppose first that $\Yw_x = 1/2$ for every $x$.  Then several applications of   \eqref{origin_eq}  yield
 \eeq{ \label{finite_difference_heat}
\gs(x+\etime )-\gs(x) = \tfrac{1}{4}\big[\gs(x+\espace)+\gs(x-\espace)-2\gs(x)\big].
}
This is a finite difference version of the heat equation  $\gs_t = \frac{1}{2}\gs_{xx}$.  Next, let 
  $\Yw_x = 1/2+ \wb\Yw_x$ for i.i.d.\  mean zero random variables $\wb\Yw_x$.  Then the right-hand side of \eqref{finite_difference_heat} acquires an additional  term which is a linear combination of the $\gs$-terms on the right with mean-zero random coefficients.   This is   a discrete, though somewhat complicated,  version of the multiplicative noise term in \eqref{sheq}.
 \qedex\end{remark} 

With partition functions defined as in \eqref{part56},  equation \eqref{origin_eq} extends    across multiple levels:   
\be\label{g77} 
\gs(x) =\sum_{u\tsp\in\tsp\level_m}\gs(u)\tsp Z_{u,x}
\quad \text{ for all $m<n$ and $x\in\level_n$}. 
\ee
 Equation \eqref{g77} prescribes how to calculate, from an initial condition $\gs|_{\level_m}$, the unique solution on all later levels $\level_n$, $n>m$.  Instead of an initial value problem, we consider  eternal solutions.  An   \textit{eternal solution}  is a function $\gs\colon\Z^2\to\R$ such that \eqref{origin_eq} (equivalently, \eqref{g77})  holds at \textit{every} $x\in\Z^2$.
 The first lemma below gives a deterministic relationship between strictly positive eternal solutions and recovering cocycles.
 Recall that a recovering cocycle is a function $B\colon \Z^2\times\Z^2\to\R$ that satisfies properties \eqref{rec_coc}, with  the given weights $\Yw$ appearing in \eqref{rec_def}. 

\begin{lemma}  \label{lm:sheB7}
Let $(\Yw_x)_{x\in\Z^2}$ be strictly positive weights, and fix $u\in\Z^2$.  
Then eternal solutions $\gvv>0$ of \eqref{g77} such that $\gvv(u)=1$ are in bijective correspondence with recovering cocycles  $B$ via $\gvv(x)=e^{B(u,x)}$. 
\end{lemma}

  Existence and uniqueness questions of eternal solutions are typically posed under  given weights $\Yw$ and for a given value of a conserved quantity. 
Equation \eqref{origin_eq} has a natural conserved quantity in the asymptotic logarithmic slope.  

\begin{lemma} \label{nx79}
Let $(\Yw_x)_{x\in\Z^2}$ be strictly positive weights satifying 
\eeq{ \label{she129}  
\lim_{\abs k\to\infty}  \abs k^{-1}\log \Yw_{(k, t-k)} = 0   \quad \text{for all $t\in\Z$}.
}
Then the quantity 
\be\label{she139}  \lambda = \lim_{\abs k\to\infty}  k^{-1}\log \gs(k, t-k)  \;\in\; [-\infty, \infty]      \ee
is preserved by the evolution \eqref{origin_eq}. That is, if   the limit \eqref{she139} holds at level $t$, it continues to hold at all subsequent levels. 
\end{lemma}

As discussed in Theorem~\ref{buse_full}, the Busemann process is a family of recovering cocycles $(\Bus^{\xi\sig}:\, \xi\in\,]\evec_2,\evec_1[\,,   \sigg\in\{-,+\})$.
 So in light of Lemma~\ref{lm:sheB7}, we obtain the following theorem on  the almost sure  existence of eternal solutions under i.i.d.\ random weights.



\begin{theorem}\label{thm:V1}  Assume  \eqref{ranenv}.
 There exists a full-probability event $\Omega_0$ such that 
  for each  $\w\in\Omega_0$,  $\xi\in\,]\evec_2,\evec_1[\,$,   $\sigg\in\{-,+\}$, and  $u\in\Z^2$, the function 
 $\gs^{\w,\tsp\xi\sig}_u\colon \Z^2\to\R$ defined by  
 \eq{ 
 \gs^{\w,\tsp\xi\sig}_u(x) =   \exp\{ \Bus^{\xi\sig}_{u,x}(\w)\} , \quad x\in\Z^2, 
   }
 satisfies the following properties. 
\begin{enumerate}  [label={\rm(\roman*)}, ref={\rm(\roman*)}]   \itemsep=3pt 
\item\label{thm:V1.i}    $\gs^{\w,\tsp\xi\sig}_u$ is an eternal solution of \eqref{g77}  normalized by $\gs^{\w,\tsp\xi\sig}_u(u)=1$.   \item\label{thm:V1.ii}  
The following limit holds for all choices of the parameters: 
\eq{ 
\lim_{\abs x_1\to\infty} \frac{\log \gs^{\w,\tsp\xi\sig}_u( x)-\nabla\gpp(\xi\sigg)\cdot x}{\abs x_1}=0. 
}
 \item\label{thm:V1.iii}   Under the additional assumption \eqref{shape_ass}, for each $t\in\Z$, the  ratios $\bigl\{  \frac{\gs^{\w,\tsp\xi\sig}_u(k, t-k)}{\gs^{\w,\tsp\xi\sig}_u(\ell, t-\ell)} :\, k,\ell\in\Z\bigr\}$ on lattice level $\level_t$ are measurable functions of the weights $\{\Yw_x:\,  x\cdot   \etime\le t\}$ in the past. 
 
\end{enumerate}
\end{theorem}
  
  Further properties of the eternal solutions  $\gs^{\w,\tsp\xi\sig}_u$ can of course be inferred from the properties of the Busemann functions. 
 Some comments on the theorem follow.   Part~\ref{thm:V1.i}  follows from Lemma~\ref{lm:sheB7} together with Theorem~\ref{buse_full}.   
 Part~\ref{thm:V1.ii} is a restatement of Theorem~\ref{all_lln} by identifying the conserved quantity in \eqref{she139}  for the solution $\gs^{\w,\tsp\xi\sig}_u$  as  $\lambda=\nabla\gpp(\xi\sigg)\cdot(\evec_1-\evec_2)$.

The  eternal solutions of the conservation law  required by  1F1S 
must  depend only on the past of the weights.  In our setting this is the past measurability of the ratios  in part~\ref{thm:V1.iii}.    This is the natural statement, for  if we imitate the connection from  SHE to SBE, then the differences $\cU^{\w,\tsp\xi\sig}_u(k, t-k) = \log \gs^{\w,\tsp\xi\sig}_u(k, t-k) - \log \gs^{\w,\tsp\xi\sig}_u(k-1, t-k+1)$ are the discrete counterpart of the solution to SBE \eqref{stbur}. 
The solution $\gs^{\w,\tsp\xi\sig}_u$ itself is  determined by the past weights only up to a multiplicative constant.  Part~\ref{thm:V1.iii} is a consequence of the construction of the Busemann process described below Theorem~\ref{buse_dir}.
This construction realizes the Busemann function $\xi\mapsto\Bus^{\xi\sig}$ from countably many limits of the form \eqref{orig_buse}, and each of these limits is determined only by weights in the past.  
But this strategy requires assumption \eqref{shape_ass} (see Remark~\ref{diffrmk}), hence this assumption's appearance in part~\ref{thm:V1.iii}.  

Theorem~\ref{thm:V1} opens the possibility of failure of 1F1S.
In the inverse-gamma case we have a theorem.

\begin{theorem}\label{thm:she_ig} 
 Assume \eqref{m:exp}.  Then there exists a full-probability event $\Omega_0$ with the following property. 
 For each $\w\in\Omega_0$ there exists a countably infinite dense set $\aUset\subset\,]\evec_2, \evec_1[\,$ such that for each $\xi\in\aUset$ and each base point $u\in\Z^2$,  $\gs^{\w,\tsp\xi-}_u$ and $\gs^{\w,\tsp\xi+}_u$ are two distinct eternal solutions with the same conserved quantity $\lambda=\nabla\gpp(\xi)\cdot(\evec_1-\evec_2)$, and $\gs^{\w,\tsp\xi-}_u(x)\neq \gs^{\w,\tsp\xi+}_u(x)$ for all $x\neq u$.
\end{theorem}

We cannot state Theorem~\ref{thm:she_ig} for general weights because we do not presently know whether in general $\aUset$ is nonempty.
In the inverse-gamma case, the denseness of $\aUset$ follows from Corollary~\ref{cor:ig_aUset}, and we use differentiability of the inverse-gamma polymer shape function \eqref{gpp} in expressing $\lambda$.
 The final claim $\gs^{\w,\tsp\xi-}_u(x)\neq \gs^{\w,\tsp\xi+}_u(x)$ for all $\xi\in\aUset$ follows from Theorem~\ref{thm:allxy} and thus requires only the fact that the weights are continuous.


\subsection{Proofs of lemmas} \label{suppB3}

As stated above, Theorems~\ref{thm:V1} and \ref{thm:she_ig} are immediate from earlier results in the main text.
So we just prove the preceding lemmas.

\begin{proof}[Proof of Lemma~\ref{lm:sheB7}]   Let $B$ be a recovering cocycle and define $\gvv(x) = e^{B(u,x)}$.
For this function $\gvv$, first verify \eqref{g77} for $m=n-1$: 
\eeq{
\label{VB70} 
\gvv(x-\evec_1)Z_{x-\evec_1,x}+\gvv(x-\evec_2)Z_{x-\evec_2,x} 
&\stackrefpp{part56}{coc_def}{=} (e^{B(u,\tspb x-\evec_1)}+ e^{B(u,x-\evec_2)})\Yw_x   \\
&\stackref{coc_def}{=}  e^{B(u,\tspb x)}    (e^{-B(x-\evec_1,x)}+ e^{-B(x-\evec_2, x)})\Yw_x   \\
&\stackref{rec_def}{=}  e^{B(u,\tspb x)}  . 
} 
To verify \eqref{g77} for $m\le n-2$, split the partition function $Z_{y,x}$ into two parts and then apply induction:
\eq{
\sum_{y\tsp\in\tsp\level_m} e^{B(u,\tspb y)}  Z_{y,x} 
&= 
\biggl\{\, \sum_{y\tsp\in\tsp\level_m} e^{B(u,\tspb y)}  Z_{y,x-\evec_1}   + \sum_{y\tsp\in\tsp\level_m} e^{B(u,\tspb y)}  Z_{y,x-\evec_2} \biggr\}  \Yw_x \\
&= 
(e^{B(u,\tspb x-\evec_1)}+ e^{B(u,x-\evec_2)})\Yw_x  
\stackref{VB70}{=} e^{B(u,\tspb x)}  . 
}
Thus $\gvv(x)=e^{B(u,x)}$ is an eternal solution. 
Furthermore, any cocycle must have $B(u,u) = 0$, and so $\gvv(u) = 1$.

Now suppose $\gvv>0$ is an eternal solution  
and define $B$ via 
$e^{B(x,y)}= \gvv(y)/\gvv(x)$. 
The cocycle property \eqref{coc_def} is  immediate.   
The recovery property \eqref{rec_def} follows from \eqref{g77} with $m=n-1$:
\begin{align*}
e^{-\Bus(x-\evec_1,x)}+ e^{-\Bus(x-\evec_2,x)} = \frac{\gvv(x-\evec_1)+\gvv(x-\evec_2)}{\gvv(x)}  =  \Yw_x^{-1}  . 
\end{align*}
Thus $B$ is a recovering cocycle.

Finally, check that these mappings are inverses of each other. 
In one direction, map $B$ to $\gvv(x) = e^{B(u,x)}$, and then map $\gvv$ to $\wt B$ defined by $e^{\wt B(x,y)} = \gvv(y)/\gvv(x)$.
This results in
\[
e^{\wt B(x,y)} = \frac{\gvv(y)}{\gvv(x)} = \frac{e^{B(u,y)}}{e^{B(u,x)}} = e^{B(x,u)+B(u,y)} = e^{B(x,y)}.
 \]
In the other direction, let $\gvv>0$ be an eternal solution such that $\gvv(u)=1$. Map $\gvv$ to $B$ defined by  $e^{B(x,y)}= \gvv(y)/\gvv(x)$, and then map $B$ to  $\wt \gvv(x)=e^{B(u,x)}$.
This results in 
\[
\wt \gvv(x)=e^{B(u,x)} =   \frac{\gvv(x)}{\gvv(u)} = \gvv(x).  
\qedhere \]  
\end{proof}

\begin{proof}[Proof of Lemma~\ref{nx79}]
Assuming \eqref{she129} and \eqref{she139}, we will show that
\eeq{ \label{she140}
\lim_{|k|\to\infty}k^{-1}\log\gs(k,t+1-k) = \lambda\in[-\infty,\infty].
}
Note that by replacing $k$ with $k-1$, we can also write \eqref{she139} as
\eeq{ \label{she140.3}
\lim_{|k|\to\infty} k^{-1}\log\gs(k-1,t+1-k) = \lambda.
}
Since for any $a,b>0$ we have
\eq{
\log a \vee \log b \leq \log(a+b) \leq \log(2(a\vee b)) = \log 2 + (\log a \vee \log b)
}
it follows from \eqref{she139} and \eqref{she140.3} that
\eeq{ \label{she141}
\lim_{|k|\to\infty}k^{-1}\log[\gs(k,t-k) + \gs(k-1,t+1-k)] = \lambda.
}
When $x = (k,t+1-k)$, equation \eqref{origin_eq} becomes
\eq{
\gs(k,t+1-k) = \Yw_{(k,t+1-k)}\big[\gs(k-1,t+1-k)+\gs(k,t-k)\big].
}
Now \eqref{she140} follows by taking logarithms, dividing by $k$, and applying \eqref{she129} and \eqref{she141}.
\end{proof}

\subsection{Correspondence with Gibbs measures} \label{suppB4}

The theorem below relates eternal solutions to consistent families of rooted semi-infinite polymer measures.
 Recall that such a family $(Q_v)_{v\in\Z^2}$ satisfies \eqref{gibbs12}.

 \begin{theorem}\label{thm:V-mu}
		There is a bijective correspondence between strictly positive eternal solutions  of \eqref{g77}  up to a constant multiplicative factor  and consistent families of rooted semi-infinite polymers measures.  This correspondence is formulated  as follows. 
		
\smallskip 		
		
		{\rm (a)}  Given a strictly positive eternal solution $\gs$ of \eqref{g77}, the consistent family $\{Q_v\}_{v\tsp\in\tsp\Z^2}$  of Gibbs measures associated  to $\gs$ is defined through their finite-dimensional marginals as follows: 
\be\label{g83}    Q_v(X_{\parng{m}{n}}=x_{\parng{m}{n}}) =  \one \{x_n=v\} \, \frac{\gs(x_{m})}{\gs(v)}  \prod_{i=m+1}^{n} \Yw_{x_i} 
\ee 
for $m\le n=v\cdot(\evec_1+\evec_2)$ and paths $x_{\parng{m}{n}}\in\pathsp_{x_m, v}$.  

\smallskip 

{\rm (b)}  Given a consistent family $(Q_v)_{v\tsp\in\tsp\Z^2}$  of Gibbs measures and any vertex $u\in\Z^2$, the strictly positive eternal solution $\gs$ that satisfies  $\gs(u)=1$ and is  associated to the   family $(Q_v)_{v\tsp\in\tsp\Z^2}$    is given by
\be\label{V55.5} \begin{aligned}   \gs(x) 
= \frac{Q_v(X_m=x)}{Z_{x,v}} \cdot \frac{Z_{u,v}}{Q_v(X_{m'}=u)} 
\quad\text{whenever $x\in\level_m$, $u\in\level_{m'}$, $v\ge x\vee u$}.
\end{aligned} \ee
		
\end{theorem} 		

\begin{remark}[Random walk in a random environment]
Another way to state \eqref{g83} is that $Q_v$ is the  Markov chain evolving backward in time   with initial state $v\in\level_n$ and transition probability 
\eeq{ \label{qpx8}
Q_v\givenp{X_{m-1}=x-\evec_r}{X_m= x}  
=  \frac{\gs(x-\evec_r)}{\gs(x)}  \Yw_{x} \quad \text{for $x\in\level_m$, $r\in\{1,2\}$, $m\le n$}.
}
 If we denote the particular function defined in \eqref{V55.5}  by $\gs_u(x)$, then it follows that 
$\gs_{a}(x)= \gs_{a}(u)\gs_u(x)$ for all $a,u,x\in\Z^2$.
That is,  $\gs_a$ and $\gs_{u}$ are constant multiples of each other, and so the transition probabilities do not depend on the choice of $u$.

The representation \eqref{qpx8} is also found in Theorem~\ref{pococthm} as \eqref{back_prob}, but with the ratio $\gs(x-\evec_r)/\gs(x)$ replaced by the increment of a recovering cocycle.
In this way, Theorem~\ref{thm:V-mu} could be inferred from Lemma~\ref{lm:sheB7}.
Nevertheless, we provide a direct proof.
\qedex\end{remark}

\begin{proof}[Proof of Theorem~\ref{thm:V-mu}]



  \textit{Step 1.} Given a strictly positive eternal solution $\gs$ of \eqref{g77}, we show that 
\eqref{g83}  defines a consistent family of polymer Gibbs measures. 
First we check that \eqref{g83} gives a well-defined probability measure on $\pathsp_v$. 
Namely, we need to verify that (i) the finite-dimensional marginals are consistent; and (ii) the total mass is $1$.
This is done by induction on the distance from the root $v$.
First, we have the base case
\eeq{ \label{mass_1}
Q_v(X_n = v) = \one \{v=v\} \frac{\gs(v)}{\gs(v)} = 1.
}
Second, observe that for any nearest neighbor path $x_{\parng{m}{n}}$, we have
\eq{
&Q_v(X_{m-1} = x_m-\evec_1, X_{\parng{m}{n}} = x_{\parng{m}{n}})
+ Q_v(X_{m-1} = x_m-\evec_2, X_{\parng{m}{n}} = x_{\parng{m}{n}}) \\
&\overset{\eqref{g83}}{=} \one \{x_n=v\} \frac{\gs(x_m-\evec_1)+\gs(x_m-\evec_2)}{\gs(v)}\prod_{i=m}^n \Yw_{x_i} \\
&\stackrefpp{g77}{g83}{=} \one \{x_n=v\} \frac{\gs(x_m)}{\gs(v)} \prod_{i=m+1}^n \Yw_{x_i} 
\overset{\eqref{g83}}{=} Q_v(X_{\parng{m}{n}}=x_{\parng{m}{n}}).
}
That is, the marginal on paths from level $m-1$ is consistent with that from level $m$.
By induction and \eqref{mass_1}, $Q_v$ is indeed a well-defined probability measure on $\pathsp_v$.

Next we check that $Q_v$ is a semi-infinite polymer measure; that is, $Q_v$ satisfies \eqref{gibbs_one}. 
As an intermediate step, we calculate the finite-dimensional marginals:
\eeq{ \label{mult_mar}
Q_v(x_{\parng{\ell}{m}}) 
&\overset{\hphantom{\eqref{g83},\eqref{part56}}}{=} \sum_{x_\brbullet\in\pathsp_{x_m,v}}Q_v(x_{\parng{\ell}{n}})
\stackref{g83}{=} \sum_{x_\brbullet\in\pathsp_{x_m,v}}\frac{\gs(x_\ell)}{\gs(v)}\prod_{i=\ell+1}^n\Yw_{x_i} \\
&\overset{\hphantom{\eqref{g83},\eqref{part56}}}{=} \frac{\gs(x_\ell)}{\gs(x_m)}\prod_{i=\ell+1}^m\Yw_{x_i}\sum_{x_\brbullet\in\pathsp_{x_m,v}}\frac{\gs(x_m)}{\gs(v)}\prod_{i=m+1}^n\Yw_{x_i}  \\
&\overset{\eqref{g83},\eqref{part56}}{=} Q_{x_m}(x_{\parng{\ell}{m}})\frac{\gs(x_m)}{\gs(v)}Z_{x_m,v}.
}
With this (using the case $\ell=m$) we can check the Gibbs property \eqref{gibbs_one}:  with $x_n=v$, we have
\eeq{ \label{use_mult}
Q_v\givenp{x_{\parng{m}{n}}}{x_m} 
&= \frac{Q_v(x_{\parng{m}{n}})}{Q_v(x_m)} 
\overset{\eqref{g83},\eqref{mult_mar}}{=} \frac{\gs(v)^{-1}\gs(x_m)\prod_{i=m+1}^n \Yw_{x_i}}{\gs(v)^{-1}\gs(x_m)Z_{x_m,u}} \\
&=
\frac{\prod_{i=m+1}^n \Yw_{x_i}}{Z_{x_m,v}}.
}
Finally, we verify that $(Q_v)_{v\in\Z^2}$ is a consistent family; that is, \eqref{gibbs_two} holds.
Indeed, given any $\ell\le m\le n$ and $x_m$ such that $\pathsp_{x_m,v}$ is nonempty, we can verify the desired equality:
\eq{
Q_v\givenp{x_{\parng{\ell}{m}}}{x_m}
= \frac{Q_v(x_{\parng{\ell}{m}})}{Q_v(x_m)} 
\stackref{mult_mar} =Q_{x_m}(x_{\parng{\ell}{m}}). 
}
We have verified that \eqref{g83}  defines a consistent family of polymer Gibbs measures. 

\smallskip

  \textit{Step 2.}  Fix $v\in\Z^2$.  Given a semi-infinite Gibbs measure $Q_v$ rooted at $v$, we check that 
\be\label{V45}   \gs_v(x)=\frac{Q_v(x)}{Z_{x,v}}  \quad \text{for $x\le v$}  \ee
 defines a solution $\gs_v$ of \eqref{g77} on the southwest quadrant $\{x\in\Z^2:\, x\leq v\}$. 
 The key observation is that whenever $u\le x\le v$, we have
 \eeq{ \label{mid_part}
Q_{u,v}(x) = \frac{Z_{u,x}Z_{x,v}}{Z_{u,v}}.
}
Now start from the right-hand side of \eqref{g77}: for $m<n=x\cdot \etime $, we have
 \eq{ 
 \sum_{u\tsp\in\tsp\level_m} \gs_v(u)\tsp Z_{u,x} 
&\stackrefp{gibbs_one}{=} \sum_{u\tsp\in\tsp\level_m} \frac{Q_v(u)}{Z_{u,v}}\tsp Z_{u,x}  
\overset{\eqref{mid_part}}{=} \sum_{u\tsp\in\tsp\level_m} \frac{Q_v(u)}{Z_{x,v}}\tsp Q_{u,v}(x)   \\
&\stackref{gibbs_one}{=} \sum_{u\tsp\in\tsp\level_m} \frac{Q_v(u)}{Z_{x,v}}\tsp Q_v\givenp{x}{u}
= \sum_{u\tsp\in\tsp\level_m} \frac{Q_v(u,x)}{Z_{x,v}}   = \frac{Q_v(x)}{Z_{x,v}} =\gs_v(x).
 }

\smallskip

  \textit{Step 3.} Suppose we have a consistent family $(Q_v)_{v\in\Z^2}$ of semi-infinite rooted Gibbs measures, and fixed $u\in\Z^2$.
We show that the formula given in \eqref{V55.5}, namely
\be\label{V54.67}    \gs(x) 
= \frac{Q_v(x)}{Z_{x,v}} \cdot \frac{Z_{u,v}}{Q_v(u)}  
\qquad\text{for any $v\ge x\vee u$},
\ee 
is independent of $v$ and  defines an eternal solution. 
Indeed, in terms of definition \eqref{V45},  the formula \eqref{V54.67}  is 
\be\label{V55}   \gs(x)= \frac{\gs_v(x)}{\gs_v(u)}. \ee
Therefore, we wish to show that
\eeq{ \label{root_irr}
\frac{\gs_v(x)}{\gs_v(u)} = \frac{\gs_{v'}(x)}{\gs_{v'}(u)} \quad \text{whenever $v\wedge v' \geq x\vee u$}.
}
Given such $v,v'$, take any $w\in\Z^2$ such that  $w\ge v \vee v'$.
Since $w \ge v \ge x$, we can write
\eeq{ \label{preroot1}
\gs_v(x)
\stackref{V45}{=} \frac{Q_v(x)}{Z_{x,v}}  
&\stackref{gibbs_two}{=}\frac{Q_w\givenp{x}{v}}{Z_{x,v}}
=\frac{Q_w(x)Q_w\givenp{v}{x}}{Q_w(v)\tsp Z_{x,v}} \\
&\stackref{gibbs_one}{=}\frac{Q_w(x)Q_{x,w}(v)}{Q_w(v)\tsp Z_{x,v}}
\stackref{mid_part}{=} \frac{Q_w(x)\tsp Z_{v,w}}{Q_w(v)\tsp Z_{x,w}}
\stackref{V45}{=} \frac{\gs_w(x)}{\gs_w(v)}. 
}
But then the same sequence of equations holds with $u$ replacing $x$ and/or $v'$ replacing $v$, and so 
\eeq{ \label{preroot2}
\gs_v(u) = \frac{\gs_w(u)}{\gs_w(v)}, \qquad \gs_{v'}(x) = \frac{\gs_w(x)}{\gs_w(v')}, \qquad \gs_{v'}(u) = \frac{\gs_w(u)}{\gs_w(v')}.
}
The desired equality \eqref{root_irr} is   immediate from \eqref{preroot1} and \eqref{preroot2}, with both sides equal to $\gs_w(x)/\gs_w(u)$.
Furthermore, since $\gs$ is a constant multiple of $\gs_v$, $\gs$ is a solution on 
the quadrant $\{x\in\Z^2:\, x\le v\}$ by Step 2.
Since $v$ is now   arbitrary,  $\gs$ is a solution on the entire lattice $\Z^2$.

\smallskip

  \textit{Step 4.}   We show that the mappings constructed above are inverses of each other when solutions are restricted to those satisfying $\gs(u)=1$ for a fixed base vertex $u\in\Z^2$.  
In one direction, let $\gs$ be a eternal solution such that $\gs(u)=1$. Then 
let  $(Q_v)_{v\in\Z^2}$  be the image of $\gs$ from \eqref{g83}, and let $\wt\gs$ be the image of $(Q_v)_{v\in\Z^2}$ under \eqref{V54.67}.   For $v\ge x\vee u$, we have
\begin{align*}
   \wt \gs(x)   \overset{\eqref{V54.67}}{=} \frac{Q_v(x)}{Z_{x,v}} \cdot \frac{Z_{u,v}}{Q_v(u)} 
    \overset{\eqref{mult_mar}}{=} \frac{\gs(x)}{\gs(v)} \cdot \biggl( \frac{\gs(u)}{\gs(v)}\biggr)^{-1} =  \frac{\gs(x)}{\gs(u)}  = \gs(x).  
\end{align*} 

In the other direction, let $\gs$ be the image of $(Q_v)_{v\in\Z^2}$ under \eqref{V54.67}, and then let  $(\wt Q_v)_{v\in\Z^2}$  be the image of $\gs$ from \eqref{g83}.  Let  $v\in\level_n$, $m\le n$, and   $x_{\parng{m}{n}}\in\pathsp_{x_m,v}$.
Choose some $w\ge v\vee u$.    Then
\eq{
 \wt Q_v(x_{\parng{m}{n}}) 
& \stackrefpp{g83}{use_mult}{=}   \frac{\gs(x_{m})}{\gs(v)}  \prod_{i=m+1}^{n} \Yw_{x_i}
 \overset{\eqref{V55}}{=}  \frac{\gs_w(x_m)}{\gs_w(u)} \cdot \biggl( \frac{\gs_w(v)}{\gs_w(u)}\biggr)^{-1}  \prod_{i=m+1}^n \Yw_{x_i} 
 \\ &\stackrefp{g83}{=}  \frac{\gs_w(x_m)}{\gs_w(v)} \prod_{i=m+1}^n \Yw_{x_i} 
 \overset{\eqref{root_irr}}{=}  \frac{\gs_v(x_m)}{\gs_v(v)} \prod_{i=m+1}^n \Yw_{x_i} 
 \overset{\eqref{V45}}{=}\frac{Q_v(x_m)}{Z_{x_m,v}} \prod_{i=m+1}^n \Yw_{x_i}  \\
 & \overset{\eqref{use_mult}}{=}  Q_v\givenp{x_{\parng{m}{n}}}{x_m}Q_v(x_m)
 = Q_v(x_{\parng{m}{n}}) . 
}
 This completes the proof. 
 \end{proof}  

\bibliographystyle{plain}
\bibliography{polymerLGv2,firasbib2010,erikbib,timobib,MultiCGM}

\begin{thebibliography}{10}

\bibitem{alberts_khanin_quastel14b}
Tom Alberts, Konstantin Khanin, and Jeremy Quastel.
\newblock The continuum directed random polymer.
\newblock {\em J. Stat. Phys.}, 154(1-2):305--326, 2014.

\bibitem{ande-sepp-valk}
David~F. Anderson, Timo Sepp\"al\"ainen, and Benedek Valk\'o.
\newblock {\em Introduction to probability}.
\newblock Cambridge Mathematical Textbooks. Cambridge University Press, Cambridge, 2018.

\bibitem{auffinger_damron_hanson17}
Antonio Auffinger, Michael Damron, and Jack Hanson.
\newblock {\em 50 years of first-passage percolation}, volume~68 of {\em University Lecture Series}.
\newblock American Mathematical Society, Providence, RI, 2017.

\bibitem{bakh-khan-18}
Yuri Bakhtin and Konstantin Khanin.
\newblock On global solutions of the random {H}amilton-{J}acobi equations and the {KPZ} problem.
\newblock {\em Nonlinearity}, 31(4):R93--R121, 2018.

\bibitem{bakh-li-19}
Yuri Bakhtin and Liying Li.
\newblock Thermodynamic limit for directed polymers and stationary solutions of the {B}urgers equation.
\newblock {\em Comm. Pure Appl. Math.}, 72(3):536--619, 2019.

\bibitem{bala-busa-sepp-20}
M\'{a}rton Bal\'{a}zs, Ofer Busani, and Timo Sepp\"{a}l\"{a}inen.
\newblock Non-existence of bi-infinite geodesics in the exponential corner growth model.
\newblock {\em Forum Math. Sigma}, 8:Paper No. e46, 34, 2020.

\bibitem{balazs_cator_seppalainen06}
M\'{a}rton Bal{\'a}zs, Eric Cator, and Timo Sepp{\"a}l{\"a}inen.
\newblock Cube root fluctuations for the corner growth model associated to the exclusion process.
\newblock {\em Electron. J. Probab.}, 11:no. 42, 1094--1132, 2006.

\bibitem{boro-corw-reme}
Alexei Borodin, Ivan Corwin, and Daniel Remenik.
\newblock Log-gamma polymer free energy fluctuations via a {F}redholm determinant identity.
\newblock {\em Comm. Math. Phys.}, 324(1):215--232, 2013.

\bibitem{busani??}
Ofer Busani.
\newblock Non-existence of three non-coalescing infinite geodesics with the same direction in the directed landscape.
\newblock 2023.
\newblock \texttt{arXiv:2401.00513}, 38 pages.

\bibitem{busani24}
Ofer Busani.
\newblock Diffusive scaling limit of the {B}usemann process in last passage percolation.
\newblock {\em Ann. Probab.}, 52(5):1650--1712, 2024.

\bibitem{busa-sepp-22-ejp}
Ofer Busani and Timo Sepp\"{a}l\"{a}inen.
\newblock Non-existence of bi-infinite polymers.
\newblock {\em Electron. J. Probab.}, 27:Paper No. 14, 40, 2022.

\bibitem{busani_seppalainen_sorensen24}
Ofer Busani, Timo Sepp\"{a}l\"{a}inen, and Evan Sorensen.
\newblock The stationary horizon and semi-infinite geodesics in the directed landscape.
\newblock {\em Ann. Probab.}, 52(1):1--66, 2024.

\bibitem{cato-groe-06}
Eric Cator and Piet Groeneboom.
\newblock Second class particles and cube root asymptotics for {H}ammersley's process.
\newblock {\em Ann. Probab.}, 34(4):1273--1295, 2006.

\bibitem{cato-pime-13}
Eric Cator and Leandro P.~R. Pimentel.
\newblock Busemann functions and the speed of a second class particle in the rarefaction fan.
\newblock {\em Ann. Probab.}, 41(4):2401--2425, 2013.

\bibitem{chan-94}
Cheng~Shang Chang.
\newblock On the input-output map of a {$G/G/1$} queue.
\newblock {\em J. Appl. Probab.}, 31(4):1128--1133, 1994.

\bibitem{corwin12}
Ivan Corwin.
\newblock The {K}ardar-{P}arisi-{Z}hang equation and universality class.
\newblock {\em Random Matrices Theory Appl.}, 1(1):1130001, 76, 2012.

\bibitem{corwin21}
Ivan Corwin.
\newblock Invariance of polymer partition functions under the geometric {RSK} correspondence.
\newblock In {\em Stochastic analysis, random fields and integrable probability---{F}ukuoka 2019}, volume~87 of {\em Adv. Stud. Pure Math.}, pages 89--136. Math. Soc. Japan, Tokyo, 2021.

\bibitem{corw-ocon-sepp-zygo}
Ivan Corwin, Neil O'Connell, Timo Sepp{\"a}l{\"a}inen, and Nikolaos Zygouras.
\newblock Tropical combinatorics and {W}hittaker functions.
\newblock {\em Duke Math. J.}, 163(3):513--563, 2014.

\bibitem{coup-11}
David Coupier.
\newblock Multiple geodesics with the same direction.
\newblock {\em Electron. Commun. Probab.}, 16:517--527, 2011.

\bibitem{damr-hans-14}
Michael Damron and Jack Hanson.
\newblock Busemann functions and infinite geodesics in two-dimensional first-passage percolation.
\newblock {\em Comm. Math. Phys.}, 325(3):917--963, 2014.

\bibitem{ethi-kurt}
Stewart~N. Ethier and Thomas~G. Kurtz.
\newblock {\em Markov processes: Characterization and convergence}.
\newblock Wiley Series in Probability and Mathematical Statistics. John Wiley \& Sons Inc., New York, 1986.

\bibitem{fan-sepp-20}
Wai-Tong~(Louis) Fan and Timo Sepp\"{a}l\"{a}inen.
\newblock Joint distribution of {B}usemann functions in the exactly solvable corner growth model.
\newblock {\em Probab. Math. Phys.}, 1(1):55--100, 2020.

\bibitem{ferr-mart-pime-09}
Pablo~A. Ferrari, James~B. Martin, and Leandro P.~R. Pimentel.
\newblock A phase transition for competition interfaces.
\newblock {\em Ann. Appl. Probab.}, 19(1):281--317, 2009.

\bibitem{ferr-pime-05}
Pablo~A. Ferrari and Leandro P.~R. Pimentel.
\newblock Competition interfaces and second class particles.
\newblock {\em Ann. Probab.}, 33(4):1235--1254, 2005.

\bibitem{geor-rass-sepp-16}
Nicos Georgiou, Firas Rassoul-Agha, and Timo Sepp{\"a}l{\"a}inen.
\newblock Variational formulas and cocycle solutions for directed polymer and percolation models.
\newblock {\em Comm. Math. Phys.}, 346(2):741--779, 2016.

\bibitem{georgiou_rassoulagha_seppalainen17b}
Nicos Georgiou, Firas Rassoul-Agha, and Timo Sepp\"{a}l\"{a}inen.
\newblock Geodesics and the competition interface for the corner growth model.
\newblock {\em Probab. Theory Related Fields}, 169(1-2):223--255, 2017.

\bibitem{geor-rass-sepp-17-buse}
Nicos Georgiou, Firas Rassoul-Agha, and Timo Sepp\"al\"ainen.
\newblock Stationary cocycles and {B}usemann functions for the corner growth model.
\newblock {\em Probab. Theory Related Fields}, 169(1-2):177--222, 2017.

\bibitem{geor-rass-sepp-yilm-15}
Nicos Georgiou, Firas Rassoul-Agha, Timo Sepp{\"a}l{\"a}inen, and Atilla Yilmaz.
\newblock Ratios of partition functions for the log-gamma polymer.
\newblock {\em Ann. Probab.}, 43(5):2282--2331, 2015.

\bibitem{gray09}
Robert~M. Gray.
\newblock {\em Probability, random processes, and ergodic properties}.
\newblock Springer, Dordrecht, second edition, 2009.

\bibitem{groa-janj-rass-23-arxiv}
Sean Groathouse, Christopher Janjigian, and Firas Rassoul-Agha.
\newblock Existence of generalized {B}usemann functions and {G}ibbs measures for random walks in random potentials.
\newblock 2023.
\newblock {\tt arXiv:2306.17714}, 51 pages.

\bibitem{groathouse_rassoulagha_seppalainen_sorensen25}
Sean Groathouse, Firas Rassoul-Agha, Timo Sepp{\"a}l{\"a}inen, and Evan Sorensen.
\newblock Jointly invariant measures for the {K}ardar-{P}arisi-{Z}hang {E}quation.
\newblock {\em Probab. Theory Related Fields}.
\newblock To appear, available at \texttt{arXiv:2309.17276}.

\bibitem{hoff-05}
Christopher Hoffman.
\newblock Coexistence for {R}ichardson type competing spatial growth models.
\newblock {\em Ann. Appl. Probab.}, 15(1B):739--747, 2005.

\bibitem{hoff-08}
Christopher Hoffman.
\newblock Geodesics in first passage percolation.
\newblock {\em Ann. Appl. Probab.}, 18(5):1944--1969, 2008.

\bibitem{huse_henley_fisher85}
David~A Huse, Christopher~L Henley, and Daniel~S Fisher.
\newblock Huse, {H}enley, and {F}isher respond.
\newblock {\em Phys. Rev. Lett.}, 55(26):2924, 1985.

\bibitem{imbr-spen}
J.~Z. Imbrie and T.~Spencer.
\newblock Diffusion of directed polymers in a random environment.
\newblock {\em J. Statist. Phys.}, 52(3-4):609--626, 1988.

\bibitem{janj-nurb-rass-22}
Christopher Janjigian, Sergazy Nurbavliyev, and Firas Rassoul-Agha.
\newblock A shape theorem and a variational formula for the quenched {L}yapunov exponent of random walk in a random potential.
\newblock {\em Ann. Inst. Henri Poincar\'{e} Probab. Stat.}, 58(2):1010--1040, 2022.

\bibitem{janjigian_rassoulagha20a_arxiv}
Christopher Janjigian and Firas Rassoul-Agha.
\newblock Busemann functions and {G}ibbs measures in directed polymer models on {$\Bbb Z^2$}.
\newblock Extended version, available at \href{https://arxiv.org/abs/1810.03580}{arXiv:1810.03580}.

\bibitem{janj-rass-19-1f1s}
Christopher Janjigian and Firas Rassoul-Agha.
\newblock Existence, uniqueness, and stability of global solutions of a discrete stochastic {B}urgers equation.
\newblock 2019.
\newblock Unpublished manuscript.

\bibitem{janj-rass-20-aop}
Christopher Janjigian and Firas Rassoul-Agha.
\newblock Busemann functions and {G}ibbs measures in directed polymer models on {$\Bbb Z^2$}.
\newblock {\em Ann. Probab.}, 48(2):778--816, 2020.

\bibitem{janj-rass-20-jsp}
Christopher Janjigian and Firas Rassoul-Agha.
\newblock Uniqueness and ergodicity of stationary directed polymers on {$\Bbb{Z}^2$}.
\newblock {\em J. Stat. Phys.}, 179(3):672--689, 2020.

\bibitem{janj-rass-sepp-22-arxiv}
Christopher Janjigian, Firas Rassoul-Agha, and Timo Sepp{\"a}l{\"a}inen.
\newblock Ergodicity and synchronization of the {K}ardar-{P}arisi-{Z}hang equation.
\newblock 2022.
\newblock {\tt arXiv:2211.06779}, 94 pages.

\bibitem{janj-rass-sepp-23}
Christopher Janjigian, Firas Rassoul-Agha, and Timo Sepp\"{a}l\"{a}inen.
\newblock Geometry of geodesics through {B}usemann measures in directed last-passage percolation.
\newblock {\em J. Eur. Math. Soc. (JEMS)}, 25(7):2573--2639, 2023.

\bibitem{kardar_parisi_zhang86}
Mehran Kardar, Giorgio Parisi, and Yi-Cheng Zhang.
\newblock Dynamic scaling of growing interfaces.
\newblock {\em Phys. Rev. Lett.}, 56:889--892, Mar 1986.

\bibitem{kife-97}
Yuri Kifer.
\newblock The {B}urgers equation with a random force and a general model for directed polymers in random environments.
\newblock {\em Probab. Theory Related Fields}, 108(1):29--65, 1997.

\bibitem{kiri-01}
Anatol~N. Kirillov.
\newblock Introduction to tropical combinatorics.
\newblock In {\em Physics and combinatorics, 2000 ({N}agoya)}, pages 82--150. World Sci. Publ., River Edge, NJ, 2001.

\bibitem{newm-icm-95}
Charles~M. Newman.
\newblock A surface view of first-passage percolation.
\newblock In {\em Proceedings of the {I}nternational {C}ongress of {M}athematicians, {V}ol.\ 1, 2 ({Z}\"urich, 1994)}, pages 1017--1023, Basel, 1995. Birkh\"auser.

\bibitem{noum-yama-04}
Masatoshi Noumi and Yasuhiko Yamada.
\newblock Tropical {R}obinson-{S}chensted-{K}nuth correspondence and birational {W}eyl group actions.
\newblock In {\em Representation theory of algebraic groups and quantum groups}, volume~40 of {\em Adv. Stud. Pure Math.}, pages 371--442. Math. Soc. Japan, Tokyo, 2004.

\bibitem{ocon-sepp-zygo-14}
Neil O'Connell, Timo Sepp{\"a}l{\"a}inen, and Nikos Zygouras.
\newblock Geometric {RSK} correspondence, {W}hittaker functions and symmetrized random polymers.
\newblock {\em Invent. Math.}, 197(2):361--416, 2014.

\bibitem{rassoulagha_seppalainen_shen24}
Firas Rassoul-Agha, Timo Sepp\"{a}l\"{a}inen, and Xiao Shen.
\newblock Coalescence and total-variation distance of semi-infinite inverse-gamma polymers.
\newblock {\em J. Lond. Math. Soc. (2)}, 110(1):Paper No. e12955, 58, 2024.

\bibitem{sepp-12-aop-corr}
Timo Sepp{\"a}l{\"a}inen.
\newblock Scaling for a one-dimensional directed polymer with boundary conditions.
\newblock {\em Ann. Probab.}, 40(1):19--73, 2012.
\newblock Corrected version available at {\tt arXiv:0911.2446}.

\bibitem{sepp-shen-20}
Timo Sepp\"al\"ainen and Xiao Shen.
\newblock Coalescence estimates for the corner growth model with exponential weights.
\newblock {\em Electron. J. Probab.}, 25:1--31, 2020.
\newblock Corrected version available at {\tt arXiv:1911.03792}.

\bibitem{seppalainen_sorensen23a}
Timo Sepp\"{a}l\"{a}inen and Evan Sorensen.
\newblock Busemann process and semi-infinite geodesics in {B}rownian last-passage percolation.
\newblock {\em Ann. Inst. Henri Poincar\'{e} Probab. Stat.}, 59(1):117--165, 2023.

\bibitem{seppalainen_sorensen23b}
Timo Sepp\"{a}l\"{a}inen and Evan Sorensen.
\newblock Global structure of semi-infinite geodesics and competition interfaces in {B}rownian last-passage percolation.
\newblock {\em Probab. Math. Phys.}, 4(3):667--760, 2023.

\bibitem{sina-91}
Ya.~G. Sina\u{\i}.
\newblock Two results concerning asymptotic behavior of solutions of the {B}urgers equation with force.
\newblock {\em J. Statist. Phys.}, 64(1-2):1--12, 1991.

\end{thebibliography}

\end{document}